\documentclass[b5paper]{article}
\usepackage{curves,ascmac,fancybox,amsxtra,bm,comment}
\usepackage{amsmath,amssymb,amsfonts,amsthm,latexsym,amstext,amscd}
\usepackage{epsfig,graphics,graphicx,color}
\usepackage{enumerate}
\usepackage{hyperref}
\makeatletter \thm@headfont{\bfseries\scshape} \makeatother
\allowdisplaybreaks

\newtheorem{thm}{Theorem}       [section]
\newtheorem{prop}{Proposition}  [section]
\newtheorem{lem}{Lemma}         [section]
\newtheorem{cor}{Corollary}     [section]

\newtheorem{conj}{Conjecture} 

\title{Conjectures on the distribution of roots modulo a prime of a polynomial}
\author{Yoshiyuki Kitaoka}

\date{}
\begin{document}
\maketitle
\tableofcontents
\section{Introduction and Conjectures}
In this monograph, a polynomial means always  a monic one over the ring $\mathbb Z$ 
of integers 
and the letter $p$ denotes a prime number, unless specified.
Let 
\begin{equation}\label{eq1}
f(x)=x^n+a_{n-1}x^{n-1}+\dots+a_0
\end{equation} 
be  a polynomial of degree $n$ with complex roots $\alpha_1,\dots,\alpha_n$.
We fix the numbering of roots once and for all, and
define  vector spaces $LR,LR_0$ over the rational number field $\mathbb{Q}$ by
\begin{equation}\label{eq2}
\begin{array}{ll}
LR&\hspace{-3mm}:=\{(l_1,\dots,l_{n+1})\in {\mathbb Q}^{n+1}\mid \sum_{i=1}^n l_i\alpha_i = l_{n+1}\},
\\[2mm]
LR_0&\hspace{-3mm}:=\{(l_1,\dots,l_n)\in\mathbb{Q}^n\mid \sum_{i=1}^nl_i\alpha_i\in\mathbb{Q}\}
\end{array}
\end{equation}
which depend on the numbering of roots $\alpha_i$.
The mapping $\iota$ defined by 
$$
\iota:(l_1,\dots,l_n)\mapsto (l_1,\dots,l_n,\sum_{i=1}^n l_i\alpha_i) 
$$ 
is an isomorphism  from $LR_0$ on $LR$ over 
$\mathbb{Q}$ and more strongly from $LR_0\cap\mathbb{Z}^{n}$ on $LR\cap\mathbb{Z}^{n+1}$
over $\mathbb{Z}$,
because numbers $\alpha_i$ are algebraic integers.
Since there is  a trivial linear relation $\sum \alpha_i=-a_{n-1}$,
the non-zero vector $(1,\dots,1,-a_{n-1})$ is always in $LR$, 
hence we see clearly 
$$
1\le t:=\dim_{\mathbb Q} LR=\dim_{\mathbb Q} LR_0\le n.
$$
The  condition $t=n$ holds if and only if $f(x)$ is a product of linear forms in $\mathbb{Q}[x]$.
The dimension $\dim_{\mathbb{Q}}\langle\alpha_1,\dots,\alpha_n,1\rangle_{\mathbb{Q}}$ is $n+1-t$, 
hence $\dim_{\mathbb{Q}}\langle\alpha_1,\dots,\alpha_n\rangle_{\mathbb{Q}}=n+1-t$ or $n-t$.
The equality $\dim_{\mathbb{Q}}\langle\alpha_1,\dots,\alpha_n,1\rangle_{\mathbb{Q}}=\dim_{\mathbb{Q}}\langle\alpha_1,\dots,\alpha_n\rangle_{\mathbb{Q}}$ is equivalent to 
$-a_{n-1}=\frac{tr(\sum_i\alpha_i)}{n}
\ne0$ if the polynomial $f(x)$ is irreducible.

We note that the vector $(0,\dots,0,1,0,\dots,0)\in\mathbb{Z}^n$ is in $LR_0$ if and only 
if   $f(x)$ has a rational root,
and  $(0,\dots,0,1,0,\dots,0,-1,0,\dots,0)  \in LR_0$ holds 
 if and only if $\alpha_i-\alpha_j\in\mathbb{Q}$ holds  for $i\not=j$, that is  $f(x)$ is of the form 
$f(x)=g(x)g(x+r)F(x)$ for some
polynomials $g(x),F(x)\in\mathbb{Q}[x]$ and an integer $r$.

We say that a polynomial $f(x)$ has a non-trivial linear relation among roots if $t>1$.
We know that if there is no non-trivial linear relation, then the polynomial $f(x)$ is irreducible,
and that if $f(x)$ is irreducible and $\deg f(x)=n$ is prime, then there is no non-trivial linear relation
among roots (cf. {\S}3).

We take and fix a $\mathbb{Z}$-basis of $LR_0\cap\mathbb{Z}^n$
$$
{\bm{m}}_j:=(m_{j,1},\dots,m_{j,n})\,\,(j=1,\dots,t),
$$ 
and put
$$
\hat{\bm{m}}_j:=\iota({\bm{m}}_j)=({\bm{m}}_j,m_j),\,m_j=\sum_{i=1}^nm_{j,i}\alpha_i\,\,(j=1,\dots,t).
$$ 
It is clear that  the vectors $\hat{\bm{m}}_1,\dots,\hat{\bm{m}}_t$  are  a basis of $LR\cap\mathbb{Z}^{n+1}$
over $\mathbb{Z}$.
If a vector $(l_1,\dots,l_n)\in\mathbb{Z}^n$ satisfies  $\sum_{i=1}^nl_i\alpha_i\in\mathbb{Z}$, 
then  we see easily $\sum_{i=1}^n(l_i/l)\alpha_i\in\mathbb{Z}$  for $l=\gcd(l_1,\dots,l_n)$,
hence the matrix whose rows are $\bm{m}_1,\dots,\bm{m}_t$ is primitive,
that is 
elementary divisors of the $(t,n)$-matrix with rows $\bm{m}_i$ are $1$ only, hence
it is complemented to an integral $(n,n)$-matrix with determinant $1$.

We introduce two groups associated with them:
\begin{align*}
\hat {\bm G} 
:&=\{\sigma \in S_n\mid \sigma(LR)\subset LR\}
\\
&=
\{\sigma\in S_n\mid
\langle\sigma(\hat{\bm{m}}_1),\dots,\sigma(\hat{\bm{m}}_t)\rangle_{\mathbb Z}
=\langle\hat{\bm{m}}_1,\dots,\hat{\bm{m}}_t\rangle_{\mathbb Z}\},
\\
{\bm G}
:&=\{\sigma \in S_n\mid \sigma(LR_0)\subset LR_0\}
\\
&=
\{\sigma\in S_n \mid \langle\sigma(\bm{m}_1),\dots,\sigma(\bm{m}_t)\rangle_{\mathbb{Z}}=
\langle\bm{m}_1,\dots,\bm{m}_t\rangle_{\mathbb{Z}}\}.
\end{align*}
Here, $\langle \bm{z}_1,\dots,\bm{z}_k\rangle_{\mathbb Z}:=\{a_1\bm{z}_1+\dots+a_k\bm{z}_k\mid 
a_1,\dots,a_k\in\mathbb{Z}\}$ 
and we let, for $\bm{x}=(x_1,\dots,x_n)\in\mathbb{R}^n,\,x\in\mathbb R$ and a permutation $\sigma \in S_n$
\begin{align}\label{eq3}
\sigma(\bm{x}):=(x_{\sigma^{-1}(1)},\dots,x_{\sigma^{-1}(n)}),\hspace{1mm}
\sigma(({\bm{x}},x))&:=(\sigma(\bm{x}),x).
\end{align}
Groups $\hat {\bm G},{\bm G}$ are independent of the choice of bases $\hat{\bm{m}}_i$ and
$\hat {\bm G}$ is   a subgroup of ${\bm G}$.
If $f(x)$ is irreducible, then they are identical.
However, they are not necessarily equal for a reducible polynomial,
because $\sum_il_i\alpha_i, \sum_i l_i \alpha_{\sigma(i)}$ may be different even if they are rational.

If $f(x)$ has no non-trivial linear relation among roots, 
then  $(1,\dots,1,-a_{n-1})$  is a basis of $LR\cap\mathbb{Z}^{n+1}$, hence $\hat {\bm G} = {\bm G} =S_n$.

We fix the basis $\hat{\bm{m}}_j$ $(j=1,\dots,t)$ together with roots $\alpha_i$ once and for all.

Next, we put
\begin{equation*}
Spl_X(f):=\left\{p\le X\mid f(x) \text{ is fully splitting modulo
}p\right\}
\end{equation*}
for a positive number $X$ and $Spl(f):=Spl_\infty(f)$.
We know that $Spl(f)$ is an infinite set and that the density theorem due to Chebotarev holds,
that is
$$
\lim_{X\to\infty}\frac{\#Spl_X(f)}{\# \{p\le X\} }=\frac{1}{[\mathbb{Q}(f):\mathbb{Q}]},
$$
where $\mathbb{Q}(f):=
\mathbb{Q}(\alpha_1,\dots,\alpha_n)$ is a finite Galois extension field of $\mathbb{Q}$ generated by all roots of $f(x)$ (\cite{Se}).
We note that except finitely many primes, three conditions  (i) $p\in Spl(f)$, (ii) a prime ideal $(p)$ of
$\mathbb{Q}$
splits completely in $\mathbb{Q}(f)$ and (iii) $f(x)$ is a product of $n$ linear forms over $\mathbb{Q}_p$ are equivalent.

We require the following conditions on the local roots
$r_1,\dots,r_n\,( \in \mathbb{Z})$
 of $f(x)\equiv 0\bmod p$ for   a prime $p\in Spl(f)$\,:
\begin{align}
\label{eq4}
&f(x) \equiv \prod_{i=1}^n(x-r_i) \bmod p,
\\
\label{eq5}
&0\le r_1\le r_2\le\dots\le r_n<p.
\end{align}
The condition \eqref{eq4} is nothing but  the definition of $p\in Spl(f)$,
and \eqref{eq5}  determines local roots $r_i$ uniquely.
Through this monograph, local roots $r_i$ are supposed to satisfy two conditions  \eqref{eq4} and \eqref{eq5}.
The global ordering \eqref{eq5} to local roots is a key.

We know that 
for a sufficiently large prime $p\in Spl(f)$, there is at least one permutation $\sigma \in S_n$
dependent on $p$
such that 
\begin{equation}\label{eq6}
\sum_{i=1}^n m_{j,i}r_{\sigma(i)}
=\sum_{i=1}^n m_{j,\sigma^{-1}(i)}r_{i}
\equiv m_j \bmod p\quad(1\le {}^\forall j \le t),
\end{equation}
hence there are  some permutation $\sigma\in S_n$ and integers $k_j\,(j=1,\dots,t)$ dependent on the prime $p$ satisfying equations
\begin{equation}\label{eq7}
\sum_{i=1}^n m_{j,i}r_{\sigma(i)}=\sum_{i=1}^n m_{j,\sigma^{-1}(i)}r_{i} =m_j +k_jp\quad(1\le {}^\forall j \le t).
\end{equation}
The  possibility of integers $k_j$ is finite
by the inequality $|k_j|p\le\sum|m_{j,i}|p+|m_j|$.    

%
%

Just to make sure, we give a proof of \eqref{eq6} here:
If a prime $p\in Spl(f)$ is sufficiently large, then it is completely decomposable at $\mathbb{Q}(f)$,
 hence there is a prime ideal $\frak{p}$ of degree $1$ over $p$.
Therefore there are rational integers $r_i'$ such that $\alpha_i\equiv r_i'\bmod \frak{p}$, hence $\sum_i m_{j,i}r_i'\equiv m_j\bmod p$ for $j=1,\dots,t$.
On the other hand, we see that $f(x)=\prod(x-\alpha_i)\equiv \prod (x-r_i')
\bmod \frak{p}$
and $f(x)\equiv\prod(x-r_i)\bmod p$ imply that two sets  $\{r_i'\bmod p\}$ and 
$\{r_i\bmod p\}$ are equal.
Hence there is a permutation $\sigma$ such that $r_i'\equiv r_{\sigma(i)}\bmod p$ $(i=1,\dots,n)$.

We note that  if  there are infinitely many primes
$p\in Spl(f)$ satisfying
$$
\sum_i r_i = -a_{n-1} + kp
$$
for an integer $k$ with $k\le0$ or $k\ge n$,
then the polynomial $f(x)$ is a product of linear forms,
because the above equation implies 
$0\le r_1\le\dots\le r_n\le\sum r_i\le -a_{n-1}$ if $k\le0$,
or  $0>r_{n}-p\ge\dots\ge r_1-p\ge\sum(r_i-p)\ge -a_{n-1}$
if $k\ge n$,
hence $f(x)\equiv\prod(x-s_i)\bmod p$ with some integers
$s_i$ satisfying $0\le| s_i|\le|a_{n-1}|$
for infinitely many primes $p$, which means that $f(x)$ is 
equal to $ \prod(x-s_i)$.

And let us give one more remark: Suppose  that a polynomial $f(x)$ is of the form
$\prod_{i=1}^{k_1}(x-a_i)^{c_i} \prod_{j=1}^{k_2}(x-b_j)^{d_j}\cdot g(x)$, 
where all $a_i,b_j$ are integers and $a_1<\dots<a_{k_1}<0
\le b_1<\dots<b_{k_2}$ with all $c_i,d_j$ being positive, and that the polynomial $g(x)$ 
of degree $l$  has no integer root.
Then we see that for a prime $p\in Spl(f)=Spl(g)$, local roots of $f(x)$ 
 are,  with finitely many exceptions $p$ 
\begin{align*}
0&\le\overbrace{ b_1=\dots =b_1}^{d_1}<\dots <\overbrace{b_{k_2}=\dots=b_{k_2}}^{d_{k_2}}
\\
&<r_1\le\dots\le r_l
\\
&<\underbrace{p+a_1=\dots=p+a_1}_{c_1}<\dots<\underbrace{p+a_{k_1}=\dots=p+a_{k_1}}_{c_{k_1}}<p,
\end{align*}
where $r_1,\dots,r_l$ are local roots of $g(x)$.

To state  conjectures,
we introduce following notations corresponding  to conditions \eqref{eq6}, \eqref{eq7}:
\begin{gather*}
Spl_X(f,\sigma):= 
\{ p\in Spl_X(f)\mid\sum_{i=1}^nm_{j,i}r_{\sigma(i)}\equiv m_j\bmod p\,(1\le{}^\forall j\le t)\},
\\
Spl_X(f,\sigma,\{k_j\}):
= \{ p\in Spl_X(f,\sigma)\mid\sum_{i=1}^n m_{j,i}r_{\sigma(i)}=m_j+k_jp\,(1\le{}^\forall j\le t)\},
\end{gather*}
where $Spl_X(f,\sigma)$ does not depend on choice of the basis ${\bm{m}}_i$, but 
$Spl_X(f,\sigma,\{k_j\})$ does.
Furthermore, for integers $L\,(>1),R_1,\dots,R_n$ we put
\begin{gather*}
Spl_X(f,\sigma,\{k_j\},L,\{R_i\})\!
:= \{ p\in Spl_X(f,\sigma,\{k_j\})\mid r_i\equiv R_i\bmod L \,(1\le{}^\forall i\le n)   \}
\end{gather*}
and we define densities by
\begin{gather*}
Pr(f,\sigma):=\lim_{X\to\infty}\frac{\#Spl_X(f,\sigma)}{\#Spl_X(f)},
\\
Pr(f,\sigma,\{k_j\}):=\lim_{X\to\infty}\frac{\#Spl_X(f,\sigma,\{k_j\})}{\#Spl_X(f,\sigma)},
\\
Pr(f,\sigma,\{k_j\},L,\{R_i\}):=\lim_{X\to\infty}\frac{\#Spl_X(f,\sigma,\{k_j\},L,\{R_i\})}{\#Spl_X(f,\sigma,\{k_j\})},
\end{gather*}
where for the last two, the denominators $\#Spl_X(f,\sigma),\#Spl_X(f,\sigma,\{k_j\})$
of the right-hand sides are supposed to tend to the infinity,
and {\it we suppose that all the limits  exist}.
The existence of the limits is supported by computer experiment (cf. \S7).

If there is no non-trivial linear relation among roots, 
then $Spl_X(f,\sigma)=Spl_X(f)$, i.e. $Pr(f,\sigma)=1$
and $Spl_X(f,\sigma,\{k_j\})$ is independent of the permutation  $\sigma$.

Next, we introduce geometric objects
\begin{align}\nonumber
\Delta:&=\{\bm{x}=(x_1,\dots,x_n)\in\mathbb{R}^n\mid0\le x_1\le\dots\le x_n\le1\},
\\
\hat{\mathfrak{D}}_n:&=
\{
(x_1\dots,x_n)\in \Delta\mid
\sum_{i=1}^n x_i\in\mathbb{Z}
\},
\\\nonumber
{\mathfrak{D}}(f,\sigma):&=
\{
(x_1\dots,x_n)\in \Delta\mid
\sum_{i=1}^n m_{j,i}\,x_{\sigma(i)}\in\mathbb{Z}\:\:(1\le{}^\forall j\le t)
\},
 \\\nonumber
 \mathfrak{D}(f,\sigma,\{k_j\}):&=\{(x_1,\dots,x_n)\in\Delta\mid\sum_{i=1}^n m_{j,i}x_{\sigma(i)}=k_j\:(1\le{}^\forall j\le t)\},
\end{align}
where ${\mathfrak{D}}(f,\sigma)$ does not depend on the choice of basis ${\bm{m}}_i$,
but $\mathfrak{D}(f,\sigma,\{k_j\})$ does.
We note  that  $\mathfrak{D}(f,\sigma,\{k_j\})\subset\mathfrak{D}(f,\sigma)\subset 
\hat{\mathfrak{D}}_n\subset\Delta$ and
\begin{equation}\label{eq9}
{\mathfrak{D}}(f,\sigma)=\cup_{\{k_j\}\in\mathbb{Z}^t}\mathfrak{D}(f,\sigma,\{k_j\})
\quad\text{(a finite sum)},
\end{equation}
hence ${\mathfrak{D}}(f,\sigma)$ is disconnected in general.

In case that the set  $Spl(f,\sigma,\{k_j\})$ is an infinite set,  any accumulation point 
$\bm{x}=(x_1,\dots,x_n)$ of $(r_1/p,\dots,r_n/p)\in[0,1)^n$ with $p\in Spl(f,\sigma,\{k_j\})$  
 satisfies $\bm{x}\in\Delta,\sum_i m_{j,i}x_{\sigma(i)}=k_j$
  $(1\le{}^\forall j\le t)$, i.e. $\bm{x}\in\frak{D}(f,\sigma,\{k_j\})$ by 
  $\sum_i m_{j,i}r_{\sigma(i)}=m_j+k_jp$.
The condition $x_n\le1$ can not be replaced by $x_n<1$.
The dimension of $\hat{\mathfrak{D}}_n$ is $ n-1$ and that of ${\mathfrak{D}}(f,\sigma)$ and
${\mathfrak{D}}(f,\sigma,\{k_j\})$
 is less than or equal to $n-t$.
In the following, the volume $vol$ of sets ${\mathfrak{D}}(f,\sigma)$ and  
${\mathfrak{D}}(f,\sigma,\{k_j\})$
  means that as an $(n-t)$-dimensional set,
that is for a set $S$ in $\bm{v}+\mathbb{R}[\bm{v}_1,\dots,\bm{v}_{n-t}](\subset\mathbb{R}^n)$ with orthonormal vectors
$\bm{v}_1,\dots,\bm{v}_{n-t}\in \mathbb{R}^n$,
we identify $\bm{v}+\sum y_i\bm{v}_i\in S$ with a point $(y_1,\dots,y_{n-t})\in\mathbb{R}^{n-t}$.
So,  $vol({\mathfrak{D}}(f,\sigma))=0$ holds if   $\dim{\mathfrak{D}}(f,\sigma)<n-t$.

In case that $f(x)$ is a product of linear forms, which is equivalent to $t=n$,
$\dim{\mathfrak{D}}(f,\sigma)=0$ holds, 
hence  the discussion about ${\mathfrak{D}}(f,\sigma)$  is  meaningless.

If there is no non-trivial linear relation among roots, 
then 
${\mathfrak{D}}(f,\sigma)$ is equal to $\hat{\mathfrak{D}}_n$ for every permutation $\sigma$,
and
$\hat{\mathfrak{D}}_n$ is equal to
$$\{(0,\dots,0),(1,\dots,1)\}\cup\cup_{k=1}^{n-1}
\{\bm{x} \in \Delta\mid \sum x_i=k\}.
$$

The first conjecture is

\noindent
\begin{conj}\label{conj1}
 For a permutation $\sigma$ with $Pr(f,\sigma)>0 $,
the ratio 
\begin{equation}\label{eq10}
c=\frac{{vol}({\mathfrak{D}}(f,\sigma))}{Pr(f,\sigma)}
\end{equation}
is independent of $\sigma$.
 If $\bm{G}=\hat{\bm{G}}$ holds, then two conditions $Pr(f,\sigma)>0$ and ${vol}({\mathfrak{D}}(f,\sigma))>0 $ are equivalent.
 \end{conj}
 
 If (i)
 $f(x)$ is not a product of linear forms,  (ii) $\alpha_i-\alpha_j \not\in\mathbb{Q}$ if $i\ne j$,
and (iii)
 $vol(\mathfrak{D}(f,\sigma))=cPr(f,\sigma)$ holds for every $\sigma\in S_n$,  then  we have
 (cf. Prop. \ref{prop5.3})
$$
c=\sqrt{\det((\bm{m}_i,\bm{m}_j))}/\#\hat{\bm{G}}.
$$ 

There is an example of a reducible polynomial $f(x)$ such that $Pr(f,\sigma)=0$ but ${vol}({\mathfrak{D}}(f,\sigma))>0$ for some permutation $\sigma$ (cf. {\bf{\ref{counterEx}}}),
and so Expectation $1''$ in \cite{K8}, \cite{K11} should be revised as above.
The author does not know the necessary and/or sufficient condition to that
 two conditions   $Pr(f,\sigma)>0$,  
${vol}({\mathfrak{D}}(f,\sigma))>0$ are  equivalent (cf. Proposition \ref{prop4.5}).  
As above, they seem to be equivalent for an irreducible polynomial $f(x)$ of $\deg f>1$.
Incidentally, it is likely that  
$f(x)$ is a product of linear forms if and only if 
$vol(\mathfrak{D}(f,\sigma))=0$ holds for every permutation $\sigma$.

If $f(x)$ has no non-trivial linear relation among roots, then  $Pr(f,\sigma)$ is equal
to $1$ as above, and by the property ${\mathfrak{D}}(f,\sigma)=\hat{\mathfrak{D}}_n$ for all
$\sigma\in S_n$ Conjecture \ref{conj1} is true (cf. Proposition \ref{prop4.6}).
If $f(x)$ is a product of linear forms, then \eqref{eq10} is true for $c=0$, but it is futile.

The second  is

\noindent
\begin{conj}\label{conj2}
Suppose that $Pr(f,\sigma)>0$ and ${vol}(  \mathfrak{D}(f,\sigma))>0$ for  a permutation $\sigma$; then
for a set $D$ satisfying that $D=\overline{D^\circ}$ and $ D\cap \mathfrak{D}(f,\sigma)$ is 
the closure of the intersection of $D^\circ$ and the interior of $\mathfrak{D}(f,\sigma)$, we have
\begin{align}\nonumber
{Pr}_D(f,\sigma):&=\lim_{X\to\infty}\displaystyle
\frac{\#\{p\in {Spl}_X(f,\sigma) \mid (r_1/p,\dots,r_n/p)\in D\}}
{\#{Spl}_X(f,\sigma)}
\\\label{eq11}
&=\displaystyle\frac{ {vol}(  {D}\cap{\mathfrak{D}}(f,\sigma))}
{{vol}({\mathfrak{D}}(f,\sigma))}.
\end{align}
\end{conj}

About the condition  on $D$ above, except the measurability:
In case of $m_1=\dots=m_t=0$, every point $(r_1/p,\dots,r_n/p)$ is on $\mathfrak{D}(f,\sigma)$, 
and so we may assume that  the set $D$ is  in $\mathfrak{D}(f,\sigma)$ from the beginning.
This is usual and there is no problem to generalize the Weyl criterion formally, although the proof is
another big problem.
Suppose not, i.e. $m_i\ne0$ for some $i$; then any point 
$(r_1/p,\dots,r_n/p)$ is not on $\mathfrak{D}(f,\sigma)$,
hence for example, for $D=[0,1)^n\setminus \mathfrak{D}(f,\sigma)$ the left-hand side of \eqref{eq11} 
is $1$, but the right-hand side is $0$.
So we must put some restriction on $D$ (cf. \S6.2).
Although \eqref{eq11} matches numerical data for the set $D=\{\bm{x}\mid x_i\le a\}$,  
it seems to be too restrictive.
The closed set whose boundaries are given by linear forms except 
$\sum_i m_{j,\sigma^{-1}(i)}x_i\in\mathbb{Z}$ may also be a candidate.

It is rather complicated to see whether $vol(D\cap\mathfrak{D}(f,\sigma))>0$ or not. 
Let us give an example when $f(x)$ has no non-trivial linear relation among roots:
Let integers $k,J$ satisfy $1\le k\le n-1,1\le J\le n-k$ and 
$$
D =\left\{
\bm{x}\in \mathbb{R}^n\left|
\begin{array}{cl}
|x_j|<\delta&( ^\forall j< J),
\\
|x_j-\frac{k}{n-J+1}|<\delta&(^\forall j\ge J)
\end{array}
\right.
\right\}
\quad(^\forall\delta>0).
$$
Then defining $\bm{x}$  by
$$
x_j=
\left\{
\begin{array}{ll}
\frac{1}{M^{n-j}} &\text{ if } j<J,
\\
\frac{k}{n-J+1}-\frac{1}{N^j}  &\text{ if } J\le j < n,
\\
\frac{k}{n-J+1}-\sum_{l =1}^{J-1}\frac{1}{M^{n-l}}+\sum_{l=J}^{n-1}\frac{1}{N^{l}}
&\text{ if }j=n,
\end{array}
\right.
$$
we see that $\sum x_j = k$ and $\bm{x}\in D$ is an inner point of $\Delta$ if $N,\frac{M}{N}$ are sufficiently large, hence  $vol(D\cap\mathfrak{D}(f,\sigma,k))>0$.

The equation \eqref{eq11}  implies
\begin{align}\label{eq12}
&Pr(f,\sigma,\{k_j\})=\frac{vol(\mathfrak{D}(f,\sigma,\{k_j\}))}{vol(\mathfrak{D}(f,\sigma))},
\end{align}
applying to
$$
D=\{(x_1,\dots,x_n)\mid|\sum_{i} m_{j,i}x_{\sigma(i)}-k_j|\le1/3\,(1\le{}^\forall j\le t)\}.
$$
\noindent
The following may be better, although they are equivalent.

\noindent
\begin{conj}\label{conj3}
Suppose that $Pr(f,\sigma,\{k_j\})>0$ and ${vol}(  \mathfrak{D}(f,\sigma,\{k_j\}))>0$ for  a permutation $\sigma$; then
for a set $D$ in
$$
\{(x_1,\dots,x_n)\in[0,1)^n\mid|\sum_{i} m_{j,i}x_{\sigma(i)}-k_j|\le1/3\,(1\le{}^\forall j\le t)\}
$$
satisfying the same condition in Conjecture \ref{conj2}
we have
\begin{align}\nonumber
{Pr}_D(f,\sigma,\{k_j\}):&=\lim_{X\to\infty}\displaystyle
\frac{\#\{p\in {Spl}_X(f,\sigma,\{k_j\}) \mid (r_1/p,\dots,r_n/p)\in D\}}
{\#{Spl}_X(f,\sigma,\{k_j\})}
\\\label{eq13}
&=\displaystyle\frac{ {vol}(  {D}\cap{\mathfrak{D}}(f,\sigma,\{k_j\}))}
{{vol}({\mathfrak{D}}(f,\sigma,\{k_j\}))}.
\end{align}
\end{conj}
Conjecture \ref{conj2} implies Conjecture \ref{conj3} as follows :
\begin{align*}
&{Pr}_D(f,\sigma,\{k_j\})
\\
=&\lim_{X\to\infty}\displaystyle
\frac{\#\{p\in {Spl}_X(f,\sigma,\{k_j\}) \mid (r_1/p,\dots,r_n/p)\in D\}}
{\#{Spl}_X(f,\sigma,\{k_j\})}
\\
=&\lim_{X\to\infty}\frac{\#\{p\in {Spl}_X(f,\sigma) \mid (r_1/p,\dots,r_n/p)\in D\}}
{\#{Spl}_X(f,\sigma)}\times
\frac{\#{Spl}_X(f,\sigma)}{\#{Spl}_X(f,\sigma,\{k_j\})}
\\
=&
\displaystyle\frac{ {vol}(  {D}\cap{\mathfrak{D}}(f,\sigma))}
{{vol}({\mathfrak{D}}(f,\sigma))}\times\frac{1}{Pr(f,\sigma,\{k_j\})}\hspace{21mm}(\text{by }\eqref{eq11})
\\
=&
\displaystyle\frac{ {vol}(  {D}\cap{\mathfrak{D}}(f,\sigma,\{k_j\}))}
{{vol}({\mathfrak{D}}(f,\sigma))}\times\frac{vol(\mathfrak{D}(f,\sigma))}{vol(\mathfrak{D}(f,\sigma,\{k_j\}))}\hspace{10mm}( \text{by }\eqref{eq12})
\\
=&
\displaystyle\frac{ {vol}(  {D}\cap{\mathfrak{D}}(f,\sigma,\{k_j\}))}
{vol(\mathfrak{D}(f,\sigma,\{k_j\}))},
\end{align*}
and vice versa.
Conjecture \ref{conj2} is a kind of the equi-distribution (cf. \S\ref{sec5}). 
In fact, in case that there is no non-trivial linear relation among roots,
 \eqref{eq11} implies, for $1\le {}^\forall i \le n$, $0\le {}^\forall a <1$
\begin{align}\nonumber
&(n-1)!\lim_{X\to\infty}\frac{
\#\{p\in Spl_X(f)\mid r_i/p<a\}}
{\#Spl_X(f)}=
\\ \label{eq14}
&
\sum_{0\le h \le n\atop 1\le l \le n-1}
\sum_{k =i}^n(-1)^{h+k+n}{n\choose k}\sum_{m=1}^{n-1}{k \choose n-h-m+l}
{n-k \choose m-l}M(l-ha)^{n-1},
\end{align}
where the binomial coefficient $A\choose B$ is supposed to vanish  unless $0\le B \le A$,
and $M(x):=\max(x,0)$,
from which follows
the one-dimensional equi-distribution 
\begin{equation}\label{eq15}
\lim_{X\to\infty}\frac{\sum_{p\in Spl_X(f)}\#\{i\mid 0\le r_i/p\le a\}}{n\#Spl_X(f)}=a.
\end{equation}
In case that a polynomial $f(x)$ is quadratic and irreducible in particular,  
it is true by \cite{DFI}, \cite{T}.
In case that a polynomial $f(x)$ has a non-trivial linear relation among roots,
see \S\ref{subsec5.3}.

If $f(x)$ has no non-trivial linear relation among roots, then \eqref{eq12} implies  with Proposition \ref{prop4.6} that for $1\le k \le n-1$,
\begin{equation}\label{eq16}
\left\{
\begin{array}{cl}
Pr(f,\sigma,k)&=\lim_{X\to\infty}\frac{\#\{p\in Spl_X(f)\mid (\sum_{i=1}^n r_i+a_{n-1})/p=k\}}{\#Spl_X(f)}
\\
&=E_n(k),
\\
vol(\hat{\frak{D}}_n) &=\frac{\sqrt{n}}{n!},
\\
vol(\frak{D}(f,\sigma,k))&=  Pr(f,\sigma,k)   vol(\hat{\frak{D}}_n),
\end{array}
\right.
\end{equation}
where  $E_n(k)$ is the volume of 
$$
\{\bm{x}\in[0,1)^{n-1} \mid \lceil x_1+\dots+x_{n-1}\rceil=k\}
$$ 
and  it is equal to $A(n-1,k)/(n-1)!$ where Eulerian numbers $A(n,k)$ $(1\le k\le n)$ are defined 
 recursively by
 $$
 A(1,1)=1,A(n,k)=(n-k+1)A(n-1,k-1)+kA(n-1,k),
 $$
 and $\lceil x \rceil$ is the integer satisfying $x \le\lceil x \rceil<x+1$.

Next, to state the conjecture on $Pr(f,\sigma,\{k_j\},L,\{R_i\})$ for integers $L\,(>1)$, 
$R_i$ $(i=1,\dots,n)$ as above, 
we introduce the  following condition $(C_1)$, which is a necessary condition to 
$\#Spl_\infty(f,\sigma,\{k_j\},L,\{R_i\})=\infty$:
\begin{quote}$(C_1)$ :
$(k_j,L)=(\sum_i m_{j,i} R_{\sigma(i)}-m_j,L)(=d_j$, say) for ${}^\forall j=1,\dots,t$ 
and there is an integer $q$
which satisfies (i) $q$ is independent of $j$, (ii) $q$ is relatively prime to $L$,
(iii) $\sum_i m_{j,i} R_{\sigma(i)}-m_j\equiv k_j\cdot q \bmod L$ $({}^\forall j)$
and (iv) $[[q]]=[[1]]$ on $\mathbb{Q}(f)\cap\mathbb{Q}(\zeta_{L/d_j})$ $({}^\forall j)$.
\end{quote}
Here, $\zeta_l$ is a primitive $l$th root of unity, and 
 for a subfield $F$ in $\mathbb{Q}(\zeta_l)$ and an integer $q$
relatively prime to $l$,
$[[q]]$ denotes the automorphism of $F$ induced by
$\zeta_l\to\zeta_l^q$.
The condition $[[q]]=[[1]]$ on $\mathbb{Q}(f)\cap\mathbb{Q}(\zeta_{L/d_j})$ is trivially
satisfied if $\mathbb{Q}(f)\cap\mathbb{Q}(\zeta_{L/d_j})=\mathbb{Q}$.
We put
\begin{gather*}
\mathfrak{R}(f,\sigma,\{k_j\},L):=\{\{R_i\}\in[0,L-1]^n\mid (C_1)\text{ is satisfied}\}.
\end{gather*}
The last conjecture is

\noindent
\begin{conj}\label{conj4}
Under the assumption   $\#Spl_\infty(f,\sigma,\{k_j\})=\infty$,
\begin{align}\nonumber
&Pr(f,\sigma,\{k_j\},L,\{R_i\})
\\\label{eq17}
=&
\begin{cases}
\displaystyle
\frac{ 1}{\#\mathfrak{R}(f,\sigma,\{k_j\},L)}&\hspace{1mm}\text{ if }
\{R_i\}\in \mathfrak{R}(f,\sigma,\{k_j\},L),
\\
0&\hspace{1mm}\text{ otherwise}.
\end{cases}
\end{align}
\end{conj}
If the polynomial $f(x)$ has no  non-trivial linear relation among roots in particular,
then we see that $t=1$ and   $\sum_i m_{j,i}R_{\sigma(i)}-m_j=\sum_i R_i+a_{n-1}$, hence
\begin{gather*}
\#\mathfrak{R}(f,\sigma,k,L)=L^{n-1}[\mathbb{Q}(\zeta_{L/d}):\mathbb{Q}(f)\cap\mathbb{Q}(\zeta_{L/d})] \quad(d:=(k,L)).
\end{gather*}

Let us see the case of degree 1, i.e. $f(x) =x+a$ :
Then we see that  the permutation $\sigma$ is the identity, $t=1$, $\hat{\bm{m}}_1=(1, -a)$,
and
the local root $r_1$ is equal to $-a+k_1p$ for $k_1=0,1$ according to $a\le0$, $a>0$ if $p>|a|$.
Then  the condition $(C_1)$  is $R_1+a\equiv 0 \bmod L $, and $(R_1+a,L)=1$ according to $a\le0$, 
and  $a>0$, respectively.
Thus we have only to consider the case  $k_1=0,a\le0$ or $ k_1=1,a>0$,
and  neglect a finite number of primes $p$ less than or equal to $|a|$;
then we see that
\begin{gather*}
Spl_X(f,id)= \{ p\le X\},
\\
Spl_X(f,id,k_1)= \{ p\le X\},
\\
Spl_X(f,id,k_1,L,R_1)= \{ p\le X\mid -a+k_1p\equiv R_1\bmod L\},
\end{gather*}
hence $Pr(f,id)=Pr(f,id,k_1)=1$ and 
$$
Pr(f,id,k_1,L,R_1)=
\begin{cases}
1& \text{ if } a\le0,-a\equiv R_1\bmod L,
\\
0& \text{ if } a\le0,-a\not\equiv R_1\bmod L,
\\
\frac{1}{\varphi(L)}&\text{ if } a>0,(R_1+a,L)=1,
\\
0&\text{ if } a>0,(R_1+a,L)\ne1
\end{cases}
$$
by Dirichlet's theorem, and
\begin{gather*}
\#\mathfrak{R}(f,id,k_1,L)=
\begin{cases}
1&\text{ if }a\le 0,
\\
\varphi(L)&\text{ if } a>0.
\end{cases}
\end{gather*}
Thus the conjecture \eqref{eq17} is nothing but Dirichlet's theorem.

In case that $f(x)=x^2+ax+b$ is an irreducible quadratic polynomial, 
we have $t=1$, and   if $\#Spl(f,\sigma,k)=\infty$, then  $k=1$ holds, hence
$$
\#\frak{R}(f,\sigma,1,L)=
\left\{
\begin{array}{ll}
\varphi(L^2)/2 &\text{ if } D\mid L,
\\
\varphi(L^2) &\text{ if } D\nmid L,
\end{array}
\right.
$$
where $D$ is the discriminant of the quadratic field $\mathbb{Q}(f)$.
Suppose that $L=2$ moreover; 
then  $\#\frak{R}(f,\sigma,1,2)=2$, and  the equation $r_1+r_2=-a+p$ for a large prime $p$
implies $R_1+R_2\equiv-a+1\bmod 2$, i.e. either $\{R_1,R_2\}=\{0,1\}$ in the case of 
$a\equiv0\bmod 2$
 or $R_1=R_2\,(=0$ or $1)$ in the case of $a\equiv 1\bmod 2$.
So, the conjecture says that  the density of primes $p\in Spl(f)$ with $r_1$ being odd is equal to that with $r_1$ being even.

In  case of $m_1=\dots=m_t=0$, the point $(r_1/p,\dots,r_n/p)$ is on $\mathfrak{D}(f,\sigma,\{k_j\})$ for every $p\in Spl(f,\sigma,\{k_j\})$, hence the situation may be a bit simplified, although it does not mean ``easy'' at all:
Put $\beta_i:=[\mathbb{Q}(f):\mathbb{Q}]\cdot\linebreak[3]\alpha_i - \text{tr}(\alpha_i)$ $(i=1,\dots,n)$ and
$g(x):=\prod(x-\beta_i)(\in\mathbb{Z}[x])$.
Here $\text{tr}$ denotes the trace from $\mathbb{Q}(f)$ to $\mathbb{Q}$.
Since  a linear relation among roots $\sum_i l_i\alpha_i =l_{n+1}$ $(l_i\in\mathbb{Q})$ 
is equivalent to $\sum_i l_i\beta_i=0$,  
the vector space $LR_0$ is the same for polynomials $f(x),g(x)$ by $tr(\beta_i)=0$,
and the constants $m_i$ for the polynomial $g(x)$ are $0$.
Take a prime ideal $\mathfrak{p}$ of degree $1$ of $\mathbb{Q}(f)=\mathbb{Q}(g)$ and
let their local roots be $r_i,R_i$:  Equations $\alpha_i\equiv r_{\sigma(i)}\bmod\mathfrak{p},
\beta_i\equiv R_{\mu(i)}\bmod\mathfrak{p}$ for some permutations $\sigma,\mu$ imply
$[\mathbb{Q}(f):\mathbb{Q}]r_{\sigma(i)}- tr(\alpha_i) \equiv R_{\mu(i)} \bmod p$.
The relation between $\sigma,\mu$ is irregular.
Even if $f(x)$ has no multiple roots, $g(x)$ may have multiple roots.
Can one reduce the proof of conjectures for $f(x)$  to that of $g(x)$ nevertheless?

\vspace{2mm}

It may be interesting to  consider another criterial numbering instead of the numbering 
\eqref{eq5}.
As an example, we give the order $ 0\le \{qx_1\} \le\dots\le 
\{qx_n\}<1$ $(\bm{x}\in[0,1)^n)$ for a non-zero
integer $q$, where $\{x\}$ means the decimal part of $x$, i.e. $0\le\{x\}<1, x-\{x\}\in\mathbb{Z}$.
This gives the numbering of roots $\{r\mid f(r)\equiv 0 \bmod p,0\le r<p\}$ if $p$ does not divide $q$.
The order \eqref{eq5} corresponds to $q=1$.
There is no reason to consider that the similar conjectures do not exit.

\vspace{2mm}

We explain the above in detail after the next section.
Let us brief about subsequent sections.
In \S$2$, we 
give a remark on the  purely periodic decimal expansion of rational numbers:
 Let $p\,(>5)$ be a prime  number
and let the  purely periodic decimal expansion of
$1/p$ be
$$
0.\dot{c_1}\cdots\dot{c_e} =  0.{c_1}\cdots{c_e}{c_1}\cdots{c_e} \cdots,
\quad(0\le c_i \le 9)
$$
where  $e$ is the minimal length of periods, i.e.
$e = $ the order of $10 $ in $(\mathbb{Z}/p\mathbb{Z})^\times$.
Suppose $e = ln$ for natural numbers $n\;(>1),l$;
we divide the period to $n$ parts of equal length $l$, and add them.
Then we have
$$
c_1\cdots c_l +c_{l+1}\cdots c_{2l} + \cdots + c_{(n-1)l+1}\cdots c_{ln} 
=k \cdot9\cdots9
$$
for some integer $k$ with $1\le k\le n-1$.
Here we see that $kp$ is the sum of all roots $r_i$ of $x^n-1\equiv 0\bmod p$  with $0\le r_i<p$.
If we restrict the factor $n$ to a fixed prime number,  the density of prime numbers $p$ corresponding to the integer $k$ is closely related to that given by \eqref{eq16}.
This is our starting observation.

In \S$3$, we discuss   linear relations among roots.
If the degree of an irreducible polynomial is prime, then there is no non-trivial linear relation. 
In case that (i) the degree of an irreducible polynomial is $4$, (ii) 
 the degree is $6$ and  the polynomial is  Galois, or (iv) the polynomial is abelian,  
we see how the structure of linear relations is with the help of representation theory of a finite group.

In \S\ref{sec4}, we see how  $\frak{D}(f,\sigma)$ and $Pr(f,\sigma)$ depend on
$\sigma$ involving the groups $\hat {\bm G},{\bm G}$.
We introduce an assistant subset
\begin{align*}
M(f,\mu):=\{ p\in Spl(f) \mid   \alpha_i\equiv r_{\mu(i)}\bmod\mathfrak p\,\,\,(1\le{}^\forall i\le n) 
\text{ for }{}^\exists\mathfrak p\,|\,p  \},
\end{align*}
where $\mathfrak p$ denotes a prime ideal of $\mathbb{Q}(f)$.

In \S\ref{sec5}, we show that the conjecture \eqref{eq11} implies the one-dimensional equi-distribution of
$r_i/p$ $(i=1,\dots,n)$ for $p\in Spl(f)$ as a corollary of evaluation of $\lim_{X\to\infty}\#\{p\in Spl_X(f)\mid r_i/p<a\}/\#Spl_X(f)$
in the case that (i) $f(x)$ has no non-trivial  linear relation among roots, and 
(ii)  $f(x)$ is an irreducible polynomial of degree $4$.
Without invoking forcible calculation, we show that Conjectures \ref{conj1},\ref{conj2} imply the equi-distribution of
 $r_i/p$ for  local roots $r_i$ of a more general polynomial in \S{\ref{subsec5.3}}.

We discuss the Weyl criterion in \S\ref{sec6},
give numerical data  in \S\ref{sec7}, 
and   in \S\ref{sec8}
conjectures on $M(f,\mu)$ and as a question  derived from it
the density of the set
$$
\left\{
p\in Spl(f)\left|
\left\{\frac{g_1(r_i)}{p}\right\} < \left\{\frac{g_2(r_i)}{p}\right\} ({}^\forall i)
\right.
\right\}
$$
in $Spl(f)$,
where $g_1(x),g_2(x)$ are not necessarily monic polynomials  over $\mathbb{Z}$
and $r_i$ runs over local roots of $f(x)$ at $p$ and $\{a\}$ is the decimal part of $a$.
Data lead us to one more conjecture: 
\begin{conj}\label{conj5}
Let $f(x)$ be  an irreducible polynomial of degree $n\,(\ge2)$.
Then the sequence of vectors 
 $$
\left (\left\{\frac{r}{p}\right\}, \dots,\left\{\frac{r^{n-1}}{p}\right\}\right)
 $$
 is uniformly distributed in $[0,1)^{n-1}$,
   where $r$ runs over $n$ roots of $f(x)\bmod p$ for $p\in Spl(f)$.
\end{conj}
 Lastly, in \S\ref{sec9} we refer to other directions.

\section{Decimal expansion of  a rational number}\label{sec2}
Let us begin with the example $1/7=0.\dot14285\dot7=0.142857142857\dots$: We see 
\begin{align*}
&142\hspace{8.5mm}+\hspace{8.5mm}857=1\cdot999,
\\
&14\hspace{3.2mm}+\hspace{3.2mm}28\hspace{3.2mm}+\hspace{3.2mm}57=1\cdot99,
\\
&1+4+2+8+5+7=3\cdot9.
\end{align*}
The question is what numbers $1,1,3$ are.
One answer is
\begin{thm}\label{thm2.1}
Let $a,b$ be natural numbers satisfying $(10a,b)=1$ and $a<b$,
and let the  purely periodic decimal expansion of
$a/b$ be
$$
0.\dot{c_1}\cdots\dot{c_e} =  0.{c_1}\cdots{c_e}{c_1}\cdots{c_e} \cdots,
\quad(0\le c_i \le 9)
$$
where  $e$ is the minimal length of periods, i.e.
$e = $ the order of $10 \bmod b$.
Suppose $e = ln$ for natural numbers $n\,(>1),l$, and define natural numbers $B,L$
by $L=(10^l-1,b),\,b=BL$.
We divide the period to $n$ parts of equal length $l$ , and add them;
then we have
$$
c_1\cdots c_l +c_{l+1}\cdots c_{2l} + \cdots + c_{(n-1)l+1}\cdots c_{ln} 
=k (10^l - 1)/L,
$$
where $k:=\sum_{i=0}^{n-1}s_i/B$ is an integer with  the natural number $s_i$ satisfying $s_i\equiv 10^{li}a\bmod b$ and $0\le s_i<b$.
In case of $a=1,\,n=2$, $k $ is uniquely determined by
$$
Bk\equiv2\bmod L,\,1\le k\le L.
$$
\end{thm}
Before the proof, let us give remarks.
Suppose $L=1$; then $B=b$ and 
$$
c_1\cdots c_l +c_{l+1}\cdots c_{2l} + \cdots + c_{(n-1)l+1}\cdots c_{ln} 
=k\cdot 9\cdots9.
$$
In case of $a=1,n=2$ moreover,  $k=1$ holds, which is classical.
Next, consider the case that $a=1$ and $ b$ is a prime number $p$; then integers $s_i$ give
 all roots of $x^n-1\equiv0\bmod p$ and $\sum s_i=k p$.
Hence the sum of local roots $r_i\,(0\le r_i<p)$ of $x^{n-1}+x^{n-2}+\dots+1\equiv0\bmod p$ is $-1+k p$
 (cf. \eqref{eq7}).
 
Next suppose $L>1$;
letting $E$ be the order of $10\bmod L$, we decompose $l$ as $KE$ by $10^l\equiv 1\bmod L$.
Then the equation
$$
(10^l-1)/L=(10^E-1)/L\cdot((10^E)^{K-1}+(10^E)^{K-2}+\dots+1)
$$
means that $(10^l-1)/L$ is the $K$ times iteration of the minimal period
$(10^E-1)/L$ of $1/L$.

We were concerned with the case  that $a=1$ and $b$ is prime.
\vspace{2mm}

\noindent
{\bf Proof of Theorem \ref{thm2.1}}

\vspace{2mm}
\noindent
The equation  $ 10^{li}a \equiv s_i\bmod b$ implies $10^{li}a/b=\lfloor10^{li}a/b\rfloor+s_i/b$,
where $\lfloor x \rfloor$ is the integer satisfying $x-1< \lfloor x \rfloor\le x$.
We see 
\begin{align*}
&c_{li+1}\dots c_{l(i+1)}
\\=\,
&c_1\dots c_{l(i+1)}-c_1\dots c_{li}\times10^l
\\=\,
&\lfloor10^{l(i+1)}a/b\rfloor -\lfloor10^{li}a/b\rfloor\times10^l
\\=\,
&10^{l(i+1)}a/b -s_{i+1}/b - (10^{li}a/b -s_{i}/b)10^l
\\=\,
&(10^ls_i-s_{i+1})/b,
\end{align*}
hence
\begin{align*}
&c_1\cdots c_l +c_{l+1}\cdots c_{2l} + \cdots + c_{(n-1)l+1}\cdots c_{ln} 
\\=\,&
\sum_{i=0}^{n-1}(10^ls_i-s_{i+1})/b
\\=\,&
\frac{10^l-1}{L}\cdot\frac{\sum_{i=0}^{n-1}s_i}{B},
\end{align*}
noting $s_n=s_0$.
On the other hand,
by the definition of $e,n,l$, we have
$$
(10^l-1)(10^{l(n-1)}+10^{l(n-2)}+\cdots+1 )=10^{ln}-1 =10^e-1\equiv 0 \bmod b
$$
and then  $(10^{l}-1,b)=L$ implies
$$
10^{l(n-1)}+10^{l(n-2)}+\cdots+1\equiv 0 \bmod B.
$$
Hence we have $\sum_{i=0}^{n-1} s_i\equiv(\sum_{i=0}^{n-1}10^{li})a\equiv0\bmod B$,
that is $k$ is an integer.

Suppose that $a=1,\,n=2$;   
then we find $s_0+s_1=Bk$ and $s_0=1$, and moreover $s_1\equiv10^l\bmod b$ implies
$s_1\equiv1\bmod L$, hence $Bk\equiv2\bmod L$.
They imply $Bk=1+s_1\le1+(b-1)=BL$, i.e. $1\le k\le L$.   \qed
\vspace{2mm}

\noindent
Let us apply Theorem \ref{thm2.1} to $a=1,b=p $ for a prime number $p\,(>5)$, 
and let $e,l,n,B,L,k, s_i$ be those in the theorem.
Then the condition $n>1$ implies $B=b=p$, and $s_i$ satisfies $s_i^n\equiv1\bmod p$ with $s_0=1$.
For the polynomial  $f(x):=(x^n-1)/(x-1)=x^{n-1}+\dots+1$,  $s_i$ $(i=1,\dots,n-1)$ are all roots of 
$f(x)\equiv0\bmod p$, that is $p\in Spl(f)$ and
 the set $\{s_1,\dots,s_{n-1}\}$ is equal to the set of local roots of 
$f(x)\bmod p$.
And the integer $k$ in the theorem is equal to $(\sum_{i=0}^{n-1} s_i+1)/p$,
which is given as $\frak{s}(p)$ in \cite{HKKN}.
Let us start conversely from a natural number $n$ and the polynomial $f(x)$ above.
By supposing that $n$ is a prime number,  $f(x)$ has no non-trivial linear relation among roots.
Under Conjecture \ref{conj2},
we have
\begin{alignat*}{3}
&\lim_{X\to\infty}\frac{\#\{p\in {Spl}_X(f)\mid(r_1+\dots+r_{n-1}+1)/p=k\}}{\#{Spl}_X(f)} &&
\\[1.5ex]
 =&\, \lim_{X\to\infty}\frac{\#\{p\in {Spl}_X(f)\mid
| (r_1/p+\dots+r_{n-1}/p) - k|\le1/3\}}{\#{Spl}_X(f)} &&
\\[1.5ex]
=&\, \lim_{X\to\infty}\frac{\#\{p\in {Spl}_X(f)\mid (r_1/p,\dots,r_{n-1}/p)\in D_k\}}{\#{Spl}_X(f)}             \qquad\qquad (={Pr}_{D_k}(f,id))&&
\\ \intertext{ 
where $D_k:=\{\bm{x}\in[0,1)^{n-1}\mid |(x_1+\dots+x_{n-1})-k|\le1/3\}$
}
& =\,\frac{{vol}\,(D_k\cap\hat{\mathfrak{D}}_{n-1})}
{{vol}\,({\hat{\mathfrak{D}}_{n-1}})} \qquad\qquad\;\;\;\text{(by \eqref{eq11})}&&
\\[1.5ex]
& =\,\frac{{vol}\,(\{(x_1,\dots,x_{n-1})\in[0,1)^{n-1}\mid x_1\le\dots\le x_{n-1},\sum_{i=1}^{n-1}x_i=k\})} {{vol}\,({\hat{\mathfrak{D}}_{n-1}})}                &&
\\[1.5ex]
&=\,\frac{{vol}\,(\{(x_1,\dots,x_{n-1})\in[0,1)^{n-1}\mid \sum_{i=1}^{n-1}x_i=k\})}
{
{vol}\,(\{(x_1,\dots,x_{n-1})\in[0,1)^{n-1}\mid \sum_{i=1}^{n-1} x_i\in\mathbb{Z}\})}                                                      &&
\\[1.5ex]
&=\,{vol}\left(\! \left\{ (x_1,\dots,x_{n-2}) \in [0,1)^{n-2}
\mid
(\lceil \sum_{i=1}^{n-2}x_i \rceil=k \right\}\!\right)
                                (\text{projected to }\mathbb{R}^{n-2})&&
\\[1ex]
&=\,E_{n-1}(k).&&
\end{alignat*}
In \cite{HKKN}, we took as a population,  instead of $Spl_X(f)$, the set of primes $p\,(\le X)$ such that the order of $10$ in $(\mathbb{Z}/p\mathbb{Z})^\times$ is 
divisible by $n$, which is a subset of $Spl_X(f)$, and we got the similar numerical data.
So, two events $\lceil( r_1+\dots+r_{n-1})/p\rceil=k$ and the order of $10$ in $(\mathbb{Z}/p\mathbb{Z})^\times$ being divisible by $n$ seem to be independent in $Spl(f)$.

The proof of the equation \eqref{eq16} is quite similar.
\section{Linear relation among roots}\label{sec3}
A typical  example of a polynomial which has a non-trivial linear relation among roots is a reducible polynomial and a decomposable one.
Let $f(x)=g(x)h(x)$ $(g(x),h(x)\in\mathbb{Z}[x])$ be a reducible polynomial;
then the sum of roots of $g(x)$ is an integer, which is a non-trivial linear relation.
Let $f(x)=g(h(x))$ be a decomposable polynomial, that is $1<\deg g(x),\deg h(x)<\deg f(x)$ with $g(x),h(x)\in\mathbb{Q}[x]$, and let $g(x)=\prod(x-\beta_i)$.
Then the sum of roots of $h(x)-\beta_1=0$ is rational, hence we get a non-trivial linear relation among roots. 

Let us see how to get the linear relations when we know the generator $w$ of $\mathbb{Q}(f)$.
Since $\{1,w,\dots,w^{d-1}\}$ $(d=\dim_{\mathbb{Q}}\mathbb{Q}(f))$ is a basis of $ \mathbb{Q}(f)$
over $\mathbb{Q}$, we define a rational matrix $A^{(n,d)}=(\bm{a}_1,\dots,\bm{a}_d)$ by
$$
\left(
\begin{array}{c}
\alpha_1
\\
\vdots
\\
\alpha_n
\end{array}
\right)
=
A\left(
\begin{array}{c}
1
\\
w
\\
\vdots
\\
w^{d-1}
\end{array}
\right).
$$
Then, for $(l_1,\dots,l_n)\in\mathbb{Q}^n$, the condition  $\sum_i l_i\alpha_i\in\mathbb{Q}$ holds if and only if
$(l_1,\dots,l_n)A=(*,0,\dots,0)$ holds.
So, we see that 
\begin{gather*}
LR_0=\{ \bm{l}\in\mathbb{Q}^n\mid \bm{l}(\bm{a}_2,\dots,\bm{a}_d)=0^{(n,d-1)}\},
\\
\dim LR_0= n-\text{rank}(\bm{a}_2,\dots,\bm{a}_d).
\end{gather*}
The condition $LR_0=\{ \bm{l}\in\mathbb{Q}^n\mid \bm{l}A=0\}$ holds
if and only if $\bm{a}_1$ is contained in $\mathbb{Q}[\bm{a}_2,\dots,\bm{a}_d]$. 
This is useful when we consider examples, but theoretically seems to be useless.
\subsection{Non-trivial linear relation}
Let us give a sufficient condition for a polynomial  to have no non-trivial linear relation among roots. 
\begin{prop}\label{prop3.1}
Let $f(x)$ be an irreducible polynomial of  degree $n$.
If $n$ is a prime number $p$, 
or if the  Galois group $Gal(\mathbb{Q}(f)/\mathbb{Q}) $ is isomorphic to $ S_n$ or $A_n$ $(n\ge4)$ as a permutation group of roots of $f(x)$,
then $f(x)$ has no non-trivial linear relation among roots. 
\end{prop}
\proof{
First, suppose that the degree of a polynomial $f(x)$ is a prime $p$,
and let $\alpha_1,\dots,\alpha_p$ be roots of $f(x)$,  and suppose  a  linear relation
\begin{equation}\label{eq18}
\sum_{i=1}^p m_i\alpha_i=m\quad(m_i,m\in \mathbb Z).
\end{equation}
Adding a trivial relation $\sum \alpha_i=tr(f)$ to it if necessary,
we may assume that $\sum m_i\ne0$.
The Galois group $Gal(\mathbb{Q}(f)/\mathbb{Q})$
acts faithfully on the set of all roots and contains an element $\sigma$ of order $p$,
since $p=[\mathbb{Q}(\alpha_1):\mathbb{Q}]$ divides $[\mathbb{Q}(f):\mathbb{Q}]$,
hence we may assume that $(\sigma(\alpha_1),\dots,\sigma(\alpha_p))=(\alpha_2,\dots,\alpha_p,\alpha_1)$ by renumbering roots.
Then from the assumption \eqref{eq18} follows 
$$
\begin{pmatrix}
m_1&m_2&\dots&m_p 
\\
m_p&m_1&\dots&m_{p-1} 
\\
\vdots
\\
m_2&m_3&\dots&m_1
\end{pmatrix}
\begin{pmatrix}
\alpha_1
\\
\alpha_2
\\
\vdots
\\
\alpha_p
\end{pmatrix}
=
\begin{pmatrix}
m
\\
m\\
\vdots
\\
m
\end{pmatrix}.
$$
Since $\alpha_i$'s are not rational,
the determinant of the coefficient matrix of entries $m_i$ vanishes, 
hence we have
$\prod_{i=0}^{p-1}(m_1+\zeta^im_2+\zeta^{2i}m_3+\dots+\zeta^{(p-1)i}m_p)=0$
for a primitive $p$th root $\zeta:=\zeta_p$ of unity, using a formula for cyclic determinant.
By the assumption $\sum m_i\ne0$,
we have $m_1+\zeta^im_2+\zeta^{2i}m_3+\dots+\zeta^{(p-1)i}m_p=0$
for some $i$ $(0<i<p)$,
which implies  $m_1=\dots=m_p$, 
that is \eqref{eq18} is trivial,
since $\zeta^{i}$ is still  a primitive $p$th root of unity.

Next, suppose that the Galois group $Gal(\mathbb{Q}(f)/\mathbb{Q})$ is isomorphic to  the symmetric group $S_n$.
For any $1\le i <j\le n$, there is an automorphism $\sigma$ which induces a transposition of
$\alpha_i$ and $\alpha_j$.
Hence we have
$$
m=(\sum_{k\ne i,j} m_k\alpha_k)+m_i\alpha_i+m_j\alpha_j=
(\sum_{k\ne i,j} m_k\alpha_k)+m_i\alpha_j+m_j\alpha_i,
$$
which implies $m_i(\alpha_i-\alpha_j) = m_j(\alpha_i-\alpha_j) $.
By $\alpha_i\ne \alpha_j$, we have $m_i=m_j$,
thus \eqref{eq18} is trivial.

Finally, suppose that $Gal(\mathbb{Q}(f)/\mathbb{Q})$ is the alternative group $A_n$
and that \eqref{eq18} is non-trivial.
Let us show that coefficients $m_1,\dots,m_n$ are mutually distinct, first.
Suppose that $m_1=m_2$\,; acting an even permutation $\alpha_1\to\alpha_2\to\alpha_3
\to\alpha_1(=(2,3)(1,3))$ on \eqref{eq18},
we have
\begin{align*}
&m_1\alpha_1 +m_2\alpha_2 +m_3\alpha_3=m-\sum_{i>3}m_i\alpha_i,
\\
&m_3\alpha_1 +m_1\alpha_2 +m_2\alpha_3=m-\sum_{i>3}m_i\alpha_i,
\end{align*}
which imply $(m_1-m_3)(\alpha_1-\alpha_3)=0$, hence $m_2=m_1=m_3$.
Considering other $\alpha_i$ $(i>3)$ instead of $\alpha_3$, we get $m_1=m_2=\dots=m_n$,
which contradicts the non-triviality of \eqref{eq18}.
Thus coefficients $m_i$ are mutually distinct.
Next, considering even permutations $\{\alpha_1\leftrightarrow\alpha_2,\alpha_3\leftrightarrow\alpha_4\}$,$\{\alpha_1\leftrightarrow\alpha_3,\alpha_2\leftrightarrow\alpha_4\}$,$\{\alpha_1\leftrightarrow\alpha_4,\alpha_2\leftrightarrow\alpha_3\}$,
we get
\begin{align*}
&
\left\{
\begin{array}{l}
m_1\alpha_1 +m_2\alpha_2 +m_3\alpha_3+m_4\alpha_4=m-\sum_{i>4}m_i\alpha_i,
\\
m_2\alpha_1 +m_1\alpha_2 +m_4\alpha_3+m_3\alpha_4=m-\sum_{i>4}m_i\alpha_i,
\end{array}\right.
\\
&
\left\{
\begin{array}{l}
m_3\alpha_1 +m_4\alpha_2 +m_1\alpha_3+m_2\alpha_4=m-\sum_{i>4}m_i\alpha_i,
\\
m_4\alpha_1 +m_3\alpha_2 +m_2\alpha_3+m_1\alpha_4=m-\sum_{i>4}m_i\alpha_i,\end{array}\right.
\end{align*}
which imply
\begin{align*}
&(\alpha_1 -\alpha_2)(m_1-m_2) +(\alpha_3-\alpha_4)(m_3-m_4)=0,
\\
& (\alpha_3-\alpha_4)(m_1-m_2)+(\alpha_1 -\alpha_2)(m_3-m_4)=0,
\end{align*}
hence $(\alpha_1-\alpha_2)^2=(\alpha_3-\alpha_4)^2$.
Similarly we have $(\alpha_1-\alpha_3)^2=(\alpha_2-\alpha_4)^2$,
hence taking the difference between them, we have $(2\alpha_1-\alpha_2-\alpha_3)(-\alpha_2+\alpha_3)=(\alpha_2+\alpha_3-2\alpha_4)(\alpha_3-\alpha_2)$, i.e.
$\alpha_1+\alpha_4=\alpha_2+\alpha_3$.
Similarly we find $\alpha_1+\alpha_3=\alpha_2+\alpha_4$, hence a contradiction $\alpha_3=\alpha_4$.
\qed
}
\subsection{Polynomial of $\deg=4$}
In case that the degree of a polynomial is $4$, we can determine polynomials with non-trivial linear relation
among roots.
\begin{prop}\label{prop3.2}
Let $f=x^4+a_3x^3+a_2x^2+a_1x+a_0$ be an irreducible polynomial.
If there is a non-trivial linear relation  among roots of $f(x)$,
then $f(x)$\,is decomposable, that is $f(x)=g(h(x))$ for quadratic polynomials $g(x),h(x)$.
\end{prop}
\proof{
Let  $\alpha_1,\dots,\alpha_4$ be the roots of $f(x)$.
Let $G:=Gal\,(\mathbb{Q}(f)/\mathbb{Q})$ be the
Galois group; then it operates faithfully on a set
$\{\alpha_1,\dots,\alpha_4\}$ and there is a subgroup $H$ of  order  $4$ in $G$,
because $4=[\mathbb{Q}(\alpha_1):\mathbb{Q}]$ divides $|G|$.
Noting that for permutations $\sigma,\mu$,
\begin{alignat*}{5}
&\sigma =(1,2),  &\quad&\mu=(1,3)   &\quad \Rightarrow    &\quad&\sigma\mu\ne\mu\sigma,
\\
&\sigma =(1,2)(3,4),  &&\mu=(2,3)    & \Rightarrow   &&\sigma\mu\ne\mu\sigma,
\\
&\sigma =(1,2)(3,4),  &&\mu=(1,3)(2,4)& \Rightarrow  &&\sigma\mu = \mu\sigma,\\[-20pt]
\end{alignat*}
we see that (i) $H$ is cyclic, (ii) $H$ is generated by $(1,2)(3,4),(1,3)(2,4)$ or (iii) $(1,2),(3,4)$
up to conjugate.
Renumbering roots, we have 
\begin{itemize}
\item[(i)] $G$ is generated by the cyclic permutation $\sigma$ :
  $$ \sigma:(\alpha_1,\alpha_2,\alpha_3,\alpha_4)
     \to
    (\alpha_2,\alpha_3,\alpha_4,\alpha_1),
   $$

\item[(ii)] $G$ is generated by  permutations $\sigma_1,\sigma_2$ :
\begin{align*}
&\quad  \sigma_1:(\alpha_1,\alpha_2,\alpha_3,\alpha_4)\to(\alpha_2,\alpha_1,\alpha_4,\alpha_3),
\\
&\quad\sigma_2:(\alpha_1,\alpha_2,\alpha_3,\alpha_4)\to(\alpha_3,\alpha_4,\alpha_1,\alpha_2) ,
\end{align*}

\item[(iii)] $G$ is generated by transpositions $\sigma_1,\sigma_2$ :
  $\sigma_1(\alpha_1)= \alpha_2$ and $\sigma_2(\alpha_3)=\alpha_4$.
\end{itemize}

Suppose that  $\sum_{i=1}^4m_i\alpha_i=m$ is non-trivial, that is ${}^\exists m_i\ne{}^\exists m_j$,
and if the coefficient $a_3$ of the polynomial vanishes, then considering $f(x-1)$ instead of $f(x)$,
we  may assume that $a_3\ne0$ and furthermore $\sum m_i \ne0$,
adding a trivial relation.

\medskip
\par
First, let us consider
case (i):

By  linear equations $\sum_i m_i\sigma^j(\alpha_i)=m$ $(j=0,1,2,3)$,
we have
$$
\begin{pmatrix}
m_1&m_2&m_3&m_4
\\
m_4&m_1&m_2&m_3
\\
m_3&m_4&m_1&m_2
\\
m_2&m_3&m_4&m_1
\end{pmatrix}
\begin{pmatrix}
\alpha_1
\\
\alpha_2
\\
\alpha_3
\\
\alpha_4
\end{pmatrix}
=
\begin{pmatrix}
\alpha_1&\alpha_2&\alpha_3&\alpha_4
\\
\alpha_2&\alpha_3&\alpha_4&\alpha_1
\\
\alpha_3&\alpha_4&\alpha_1&\alpha_2
\\
\alpha_4&\alpha_1&\alpha_2&\alpha_{3}
\end{pmatrix}
\begin{pmatrix}
m_1
\\
m_2
\\
m_3
\\
m_4
\end{pmatrix}
=
\begin{pmatrix}
m
\\
m\\
m
\\
m
\end{pmatrix}.
$$
Since $\alpha_i$'s are irrational, the determinant of coefficient matrix on $m_i$ vanishes,
i.e.
$\prod_{i=0}^3(m_1+\zeta^i m_2+\zeta^{2i}m_3 + \zeta^{3i}m_4)=0$ for a primitive fourth root $\zeta:=\zeta_4$ of unity.
By the assumption $\sum m_i\ne0$,
we have
$$
m_1+\zeta^i m_2+\zeta^{2i}m_3 + \zeta^{3i}m_4=0
\quad\text{for some } i=1,\,2,\,3,\, \\[-5pt]
$$
hence
\begin{itemize}
\item[(i.1)]
$ m_1-m_3=m_2-m_4=0$ in the case of $i=1,3$
or
\item[(i.2)]
$m_1-m_2 +\,m_3-m_4=0$ in the case of $i=2$.
\end{itemize}

\par\noindent
Case of (i.1),
i.e. $m_1=m_3,\,m_2=m_4$\,:  \\
The difference of the first row and the second row gives
$$
(m_1-m_2)(\alpha_1-\alpha_2+\alpha_3-\alpha_4)=0.
$$
If $m_1=m_2$ holds, we have a contradiction $m_1=\dots=m_4$.
Hence we find $\alpha_1+\alpha_3=\alpha_2+\alpha_4=-a_3/2$, hence $f(x)=(x^2+a_3x/2+\alpha_1\alpha_3)(x^2+a_3x/2+\alpha_2\alpha_4)$~\linebreak
is a polynomial in $x^2+a_3x/2$, that is $f(x)$ is decomposable.

\smallskip

\par\noindent
Case of (i.2),
hence $m_1+m_3=m_2+m_4$\,:  \\
It is easy to see that
$$
\begin{pmatrix}
\alpha_1&\alpha_2&\alpha_3&\alpha_4
\\
\alpha_2&\alpha_3&\alpha_4&\alpha_1
\\
\alpha_3&\alpha_4&\alpha_1&\alpha_2
\\
\alpha_4&\alpha_1&\alpha_2&\alpha_{3}
\end{pmatrix}
\begin{pmatrix}
m_1-m_2
\\
m_2-m_3
\\
m_3-m_4
\\
m_4-m_1
\end{pmatrix}
=
\begin{pmatrix}
0
\\
0
\\
0
\\
0
\end{pmatrix}.
$$
By non-triviality $(m_1-m_2,\dots,m_4-m_1)\ne(0,\dots,0)$,
the cyclic determinant of coefficients matrix vanishes,
i.e.
$$
a_3(\alpha_1+\zeta \alpha_2-\alpha_3-\zeta\alpha_4)(\alpha_1-\alpha_2+\alpha_3-\alpha_4)(\alpha_1-\zeta\alpha_2-\alpha_3+\zeta\alpha_4)=0.
$$

\par\noindent  (i.2.1)\quad
Suppose $\alpha_1+\zeta\alpha_2-\alpha_3-\zeta\alpha_4=0$,
i.e.
$\alpha_1-\alpha_3=-\zeta(\alpha_2-\alpha_4)$.
By equations
$\sum m_i\alpha_i=m$ and by acting $\sigma^2$ on it, $m_1\alpha_3+m_2\alpha_4+m_3\alpha_1+m_4\alpha_2=m$,
we have
$(m_1-m_3)(\alpha_1-\alpha_3)+(m_2-m_4)(\alpha_2-\alpha_4)=0$,
hence
$((m_1-m_3)(-\zeta)+m_2-m_4)(\alpha_2-\alpha_4)=0.$
Therefore we get
$m_1=m_3,\,m_2=m_4$
and so a contradiction
$m_1=m_2=m_3=m_4$ by the assumption $m_1+m_3=m_2+m_4$.

\par\noindent  (i.2.2)\quad
Suppose that $\alpha_1-\alpha_2+\alpha_3-\alpha_4=0$; it  implies $\alpha_1+\alpha_3=\alpha_2+\alpha_4$,
which implies that $f(x)$ is  decomposable as above.

\par\noindent  (i.2.3)\quad
The case of $\alpha_1-\zeta\alpha_2-\alpha_3+\zeta\alpha_4=0$ is similar to (i.2.1).

\medskip
\par
Thus we have shown that in the case of (i), $f(x)$ is decomposable.

\noindent
Case (ii)

The second case gives the following equations:
\begin{align}\label{eq19}
&m_1\alpha_1+m_2\alpha_2+m_3\alpha_3+m_4\alpha_4=m,
\\\label{eq20}
&m_2\alpha_1+m_1\alpha_2+m_4\alpha_3+m_3\alpha_4=m,
\\\label{eq21}
&m_3\alpha_1+m_4\alpha_2+m_1\alpha_3+m_2\alpha_4=m,
\\\label{eq22}
&m_4\alpha_1+m_3\alpha_2+m_2\alpha_3+m_1\alpha_4=m.
\end{align}
The sum of \eqref{eq19}, \eqref{eq20} (resp.  \eqref{eq21}, \eqref{eq22}) gives
\begin{gather*}
(m_1+m_2)(\alpha_1+\alpha_2)+(m_3+m_4)(\alpha_3+\alpha_4)=2m,
\\
(m_3+m_4)(\alpha_1+\alpha_2)+(m_1+m_2)(\alpha_3+\alpha_4)=2m,
\end{gather*}
hence if $m_1+m_2\ne m_3+m_4$ holds, then $\alpha_1+\alpha_2=\alpha_3+\alpha_4$ follows,~\linebreak
i.e.
$f(x)$ is decomposable.
Hence we may assume that  $m_1+m_2= m_3+m_4$.
Similarly, using \eqref{eq19}, \eqref{eq22} (resp.  \eqref{eq20}, \eqref{eq21}) , we may suppose $m_1+m_4=m_2+m_3$,
and $m_1+m_3=m_2+m_4$, using  \eqref{eq19}, \eqref{eq21} (resp.  \eqref{eq20}, \eqref{eq22}).
These give a contradiction $m_1=m_2=m_3=m_4$.

\noindent
Case (iii).

Acting $\sigma_1,\sigma_2$ on $\sum m_i\alpha_i=m$,
we have
\begin{align*}
&m_1\alpha_1+m_2\alpha_2+m_3\alpha_3+m_4\alpha_4=m,
\\
&m_1\alpha_2+m_2\alpha_1+m_3\alpha_3+m_4\alpha_4=m,
\\
&m_1\alpha_1+m_2\alpha_2+m_3\alpha_4+m_4\alpha_3=m,
\end{align*}
which imply
$$
(m_1-m_2)(\alpha_1-\alpha_2)=(m_3-m_4)(\alpha_3-\alpha_4)=0.
$$
Since
$\alpha_i$'s are distinct,
we have
$$ 
m_1=m_2, \quad m_3=m_4.
$$
Hence, the equations
$$
(\alpha_1+\alpha_2)+(\alpha_3+\alpha_4)
= -a_3, 
\quad 
m_1(\alpha_1+\alpha_2)+m_3(\alpha_3+\alpha_4)
= m
$$
imply
$$
b_1:=\alpha_1+\alpha_2\in\mathbb{Q},
\quad
b_2:=\alpha_3+\alpha_4\in\mathbb{Q}
$$
by $m_1\ne m_3$.  
Therefore
$$
f(x)=(x^2-b_1x+\alpha_1\alpha_2)(x^2-b_2x+\alpha_3\alpha_4)
$$
is equal to
$$
x^4+a_3x^3+(\alpha_3\alpha_4
+b_1b_2+\alpha_1\alpha_2)x^2
-(b_1\alpha_3\alpha_4+b_2\alpha_1\alpha_2)x+f(0),
$$
which implies 
$$
\alpha_1\alpha_2+\alpha_3\alpha_4=a_2-b_1b_2,
\quad
b_2\alpha_1\alpha_2+b_1\alpha_3\alpha_4=-a_1.
$$
If $b_1\ne b_2$ holds, then solving them, we have $\alpha_1\alpha_2,\,\alpha_3\alpha_4\in\mathbb{Q}$,
which implies that $f(x)$ is reducible.
Thus we have $b_1=b_2$ and then $f(x)$ is a polynomial in $x^2-b_1x$,
that is decomposable.
\qed
}
\begin{prop}\label{prop3.3}
Let a polynomial $f(x)=(x^2+ax)^2+b(x^2+ax)+c$ $(a,b,c\in\mathbb{Z})$ be  irreducible, and
put
$$
x^2+bx+c=(x-\beta_1)(x-\beta_2),\quad x^2+ax-\beta_i=(x-\alpha_{i,1})(x-\alpha_{i,2}).
$$
Then equations $\alpha_{i,1}+\alpha_{i,2} = -a$ $(i = 1,2)$ give a $\mathbb{Z}$-basis of $LR\cap\mathbb{Z}^5$.
\end{prop}
\begin{proof}
We have only to show that they are a basis of $LR$.
Let
$$
 m_{1,1}\alpha_{1,1}+ m_{1,2}\alpha_{1,2}+ m_{2,1}\alpha_{2,1}+ m_{2,2}\alpha_{2,2} =m
\quad(m_{i,j},m\in\mathbb{Q})
$$
be a linear relation.
Using $\alpha_{i,1}+\alpha_{i,2}=-a$, we may  suppose
$$
m_{1,2}\alpha_{1,2}+m_{2,2}\alpha_{2,2}=m.
$$
We have only to show $m_{2,2}=0$, which implies $m_{1,2}=m=0$,
hence we complete the proof.
Suppose that $m_{2,2}\ne0$, and dividing   $m_{2,2}$, we may assume
$$
\alpha_{2,2}=m_1\alpha_{1,2}+m_2\quad(m_1,m_2\in\mathbb{Q}).
$$
Hence $\alpha_{1,2}$ is a root of both $g(x):=x^2+ax-\beta_1$ and $h(x):=(m_1x+m_2)^2 + ~\linebreak
a(m_1x+m_2) - \beta_2= (m_1x+m_2)^2 + a(m_1x+m_2) +b+\beta_1$,
which are  polynomials over a quadratic field $\mathbb{Q}(\beta_1)$.
Since $g(x)$ is irreducible in $\mathbb{Q}(\beta_1)[x]$,
we have $h(x) =m_1^2g(x)$, hence comparing constant terms
$ m_2^2+am_2+b+\beta_1=-m_1^2\beta_1$.
Thus we find a contradiction that $\beta_1$ is rational.
\end{proof}

In the case of $n=6$, the polynomial $f(x)=x^6+2x^5+4x^4+x^3+2x^2-3x+1$ 
 is not decomposable but has a non-trivial linear relation : For a root $\alpha$, we find
\begin{align*}
f(x)=&(x-\alpha)(x+1/2\cdot \alpha^5 + 3/2\cdot \alpha^4 + 7/2\cdot \alpha^3 + 3\alpha^2 + 3\alpha - 1/2)
\\&\times(x -1/2\cdot \alpha^5 - 3/2\cdot \alpha^4 - 7/2\cdot \alpha^3 - 3\alpha^2 - 2\alpha + 3/2)
\\
&\times(x +3/4\cdot \alpha^5+ 9/4\cdot \alpha^4 + 17/4\cdot \alpha^3 + 3\alpha^2 + 3/2\cdot \alpha - 3/4)
\\
&\times(x+1/4\cdot \alpha^5 + 1/4\cdot \alpha^4 + 3/4\cdot \alpha^3 - 1/2\cdot \alpha^2 + \alpha - 3/4)
\\&\times(x- \alpha^5 - 5/2\cdot \alpha^4 - 5\alpha^3 - 5/2\cdot \alpha^2 - 5/2\cdot \alpha + 5/2)
\end{align*}
and the sum of first three roots is $-1$.
\vspace{2mm}

In  case that  $n=6$ and $\mathbb{Q}(\alpha_1)$ is  a
Galois extension of the rational number field $\mathbb{Q}$, we can see how linear relations among roots are in  subsequent subsections.
 To explain it, we study the problem group-theoretically. 
\subsection{Abelian polynomial}

 In this subsection, we assume that $f(x)$ is a monic irreducible polynomial of any degree over the rational number field $\mathbb{Q}$ with a complex root $\alpha$, and  
$\mathbb{Q}(\alpha)/\mathbb{Q}$ is a Galois extension with Galois group
 $G:=Gal(\mathbb{Q}(\alpha)/\mathbb{Q})$.
 With the help of the representation theory, we can see how the linear relations among roots
 are in the   case that  every absolutely irreducible representation of degree $>1$ 
   is defined over $\mathbb{Q}$.

We consider the group algebra $\mathbb{C}[G]$ and generalize the action of $G$ on the root $\alpha$
 to $\mathbb{C}[G]$ as $M[\alpha]:=\sum_{g\in G} m_gg(\alpha)
\in\mathbb{C}$ 
for $M=\sum_{g\in G}m_gg\in \mathbb{C}[G]$.
In case of $M\in G$, $M[\alpha]$ is the ordinary action of the element $M\in G$ on $\alpha$,
and if $M\in \mathbb{Q}[G]$, then $M[\alpha]\in\mathbb{Q}(\alpha)$ is obvious.
We define the vector space $GLR$ of linear relations among roots  by
\begin{align}\label{eq23}
GLR &:=\{M\in\mathbb{Q}[G]\mid M[\alpha]\in\mathbb{Q}\}.
\end{align}
For the numbering $G=\{g_1,\dots,g_n\}$ and $\alpha_i=g_i(\alpha)$, two conditions $\sum_i l_ig_i\in GLR$
and $(l_1,\dots,l_n)\in LR_0$ are equivalent.

It is known (\cite{L}) that  by making $G$ act on $\mathbb{C}[G]$ from the left,
$\mathbb{C}[G]$ contains the trivial representation with multiplicity $1$ and every  irreducible representation of $G$, that is
by putting 
$$
\delta:=|G|^{-1}\sum_{g\in G}g,
$$
the $G$-stable subspace $V$ spanned by $g-\delta$ $(g\in G)$ contains all non-trivial irreducible representations   and does not contain the trivial representation.
 The equations $\delta^2=\delta$, $g\delta=\delta g=\delta$  $({}^\forall g\in G)$, $1-\delta=-\sum_{g\not=1}(g-\delta)$, and $\delta V=V\delta=0$
 are clear where $1$ is the unit of the group $G$.
It is also obvious that $\mathbb{C}[G]=\mathbb{C}\delta\oplus V$ by
the decomposition $\sum_g m_gg=(\sum_g m_g)\delta+\sum_g m_g(g-\delta)$,
which implies with $\delta\in GLR$
\begin{align*}
GLR=\mathbb{Q}\delta\oplus (V\cap GLR).
\end{align*}
We say that an element $M\in \mathbb{C}[G]$ is trivial if and only if $M\in\mathbb{C}\delta$.
It is obvious that
there is a non-trivial linear relation among roots if and only if $\dim_{\mathbb{Q}} GLR>1$.
%
\begin{prop}\label{prop3.4}
For   $M\in\mathbb{C}[G]$,   $M$ is trivial 
 if and only if   $M$ annihilates $V$.
 \end{prop}
\proof
{
The ``only if''-part follows from $\delta V=0$.
Conversely, we suppose that  $MV=0$ with $M=\sum m_gg$, in particular $M(1-\delta)=0$ for the unit $1\in G$; then
\begin{align*}
M(1-\delta)&=\sum_g m_{g}(g - \delta)
\\
&=\sum_{g\ne1} (m_{g}-m_{1})(g - \delta).
\end{align*}
Thus we have $m_{g} = m_{1}$ for every $g\in G$, since the set of $g-\delta$ $(g\ne1)$ is a basis of $V$.
Therefore $M\in\mathbb{C}\delta$ follows.
\qed
}

Let us  note the following fundamental equation for $M\in\mathbb{Q}[G]$
\begin{align}\label{eq24}
\sum_{g\in G}g(M[\alpha])g^{-1}=M(\sum_{g\in G}g(\alpha)g^{-1}),
\end{align}
where $g(M[\alpha])$ is the action of $g\in G=Gal(\mathbb{Q}(\alpha)/\mathbb{Q})$ on
$M[\alpha]\in\mathbb{Q}(\alpha)$  and  the right-hand side is the product in the group ring 
$\mathbb{C}[G]$.
Because, putting $M=\sum_{h\in G}m_hh$ with $m_h \in \mathbb{Q}$,
we see
\begin{align}\nonumber
\sum_{g\in G}g(M[\alpha])g^{-1}&=
\sum_{g\in G}(\sum_{h\in G}m_hgh(\alpha))(ghh^{-1})^{-1}
\\\nonumber
&=(\sum_{h\in G}m_hh)(\sum_{g\in G}g(\alpha)g^{-1}).
\end{align}
%
\begin{prop}\label{prop3.5}
Let $M\in\mathbb{Q}[G]$; then
 $M\in GLR$ holds  if and only if
 \begin{equation}\label{eq25}
M(\sum_{g\in G}g(\alpha)g^{-1})V=0.
 \end{equation}
\end{prop}
\proof{
The equation  \eqref{eq25} is equivalent to $(\sum_{g\in G}g(M[\alpha])g^{-1})V=0$ by \eqref{eq24}.
Suppose that  $M\in GLR$; then we have $m:=M[\alpha]\in\mathbb{Q}$ by the definition and 
 it is easy to see $\sum_{g\in G}g(M[\alpha])g^{-1}\linebreak[3]
=\sum_{g\in G}g(m)g^{-1}=m|G|\delta$,
and it annihilates $V$.
Conversely, suppose that \eqref{eq25}; then  Proposition \ref{prop3.4} implies that
 $\sum_{g\in G} g(M[\alpha])g^{-1}$ is trivial, that is $M[\alpha]\in\mathbb{Q}(\alpha)$ is
 fixed by $G=Gal(\mathbb{Q}(\alpha)/\mathbb{Q})$, i.e. $M[\alpha]$ is a rational number, 
 hence $M\in GLR$.
\qed
}

\vspace{1mm}
Let $\chi$ be an irreducible character of $G$ and 
denote the central idempotent in $\mathbb{C}[G]$ to $\chi$ by
$$
c_\chi:=\chi(1)|G|^{-1} \sum_{g\in G}\overline{\chi(g)}g.
$$
They satisfy
$$
\sum_\chi c_\chi=1,\,\,c_{\chi_1}c_{\chi_2}=\delta_{\chi_1,\chi_2}c_{\chi_1},
$$
where $\delta_{i,j}$ is Kronecker's delta function.
If $\chi$ is the trivial character, then $c_\chi=\delta$. 
The subspace $V$ is the direct sum of $\mathbb{C}[G]c_\chi$ with non-trivial irreducible characters $\chi$.
If $A$ is a matrix representation corresponding to $\chi$, then $\mathbb{C}[G]c_\chi$ is a direct sum of minimal ideals  giving the representation $A$.

If $\chi$ is an irreducible character, then the mapping $g\mapsto \sigma(\chi(g))$ is also an irreducible character for every $ \sigma\in Gal(\mathbb{C}/\mathbb{Q})$.
We denote it by $\sigma(\chi)$,
and let $\Psi$ be a minimal  set of irreducible characters such that every irreducible character is of the form
$\sigma(\chi)$ for ${}^\exists \sigma\in Gal(\mathbb{C}/\mathbb{Q}),{}^\exists \chi\in \Psi$. 

For an irreducible character $\chi$, we put
 $$
\mathbb{Q}(\chi):=\mathbb{Q}(\{\chi(g)\mid g \in G\}),
$$
which is an abelian extension over $\mathbb{Q}$,
and  denote the Galois group $Gal(\mathbb{Q}(\chi)/\mathbb{Q})$ by $Gal(\chi)$. 
Let $K$ be  a splitting field  of  $G$ such that  it  is  a Galois extension of $\mathbb{Q}$
and contains $\mathbb{Q}(\chi)$.
For example, the field $\mathbb{Q}(\zeta_e)$ with exponent $e$ of $G$ is such a  field.
It is easy to see that $\sigma\in Gal(\chi)$ is the identity if and only if $\sigma(\chi)=\chi$.
We put
$$
C(\chi) := \sum_{\sigma \in Gal(\chi)} c_{\sigma(\chi)}
=\chi(1)|G|^{-1}\sum_{g\in G} tr_{\mathbb{Q}(\chi)/\mathbb{Q}}(\chi(g))g\,\,(\,\in\mathbb{Q}[G]\,),
$$
which is a central idempotent satisfying $\sum_{\chi\in\Psi}C(\chi)=1$.
We note that $C(1)=\delta$, hence $\sum_{\chi(\ne1)\in\Psi}C(\chi)=1-\delta$, 
and  \eqref{eq25} is equivalent to
$$
M(\sum_{g\in G}g(\alpha)g^{-1})C(\chi)=0
$$
for every $\chi(\ne1)\in\Psi$.

For an irreducible character $\chi$, we put
\begin{equation}\label{eq26}
LR(\chi):=\{M\in  \mathbb{Q}[G]   \mid M=MC(\chi),
 M(\sum_{g\in G}g(\alpha)g^{-1})=0
\}.
 \end{equation}
For  the trivial character $1$, we have
$$
LR(1)=\{m\delta\mid m\in\mathbb{Q}, m\,tr_{\mathbb{Q}(\alpha)/\mathbb{Q}}(\alpha)=0\}.
$$
 We note that 
 for $M\in\mathbb{Q}[G]$ with  $MC(\chi)=M$,
the condition $ M(\sum_gg(\alpha)g^{-1})\linebreak[3]=0$ is equivalent to 
$$ 
M(\sum_gg(\alpha)g^{-1})C(\chi)\mathbb{Q}[G]=0.
$$
 \begin{prop}\label{prop3.6}
For an irreducible non-trivial character $\chi$, we have
\begin{equation}\label{eq27}
LR(\chi)=\{MC(\chi)\mid M\in GLR\}\,(=GLR\cdot C(\chi)),
\end{equation}
and we have a linear relation
\begin{equation}\label{eq28}
M[\alpha]=0\,\,\text{ for }M\in LR(\chi).
\end{equation}
 \end{prop}
\proof
{
Let $M\in LR(\chi)$; then $M=MC(\chi)$ holds by definition and
the fundamental equation  \eqref{eq24} implies  $\sum_gg(M[\alpha])g^{-1} =0$, 
that is $g(M[\alpha]) =0$ for every $g\in G$,
hence we have
$M[\alpha]=0$, i.e. \eqref{eq28} and hence $M \in GLR$, i.e.
$M=MC(\chi)\in GLR\cdot C(\chi)$.
Conversely, let $M\in GLR$; then $MC(\chi)\in\mathbb{Q}[G]$
and $(MC(\chi))C(\chi) = MC(\chi)$ are  obvious.
Proposition \ref{prop3.5} implies $M(\sum g(\alpha)g^{-1})V\linebreak[3]=0$,
which, with the condition   $C(\chi)\delta=0$ implies $MC(\chi)(\sum g(\alpha)g^{-1})=0$, i.e. $MC(\chi) \in LR(\chi)$.
\qed
}
%
\begin{cor}\label{cor3.1}
We see that  {\em(i)}
\begin{align}\label{eq29}
V\cap  GLR =\oplus_{\chi(\ne1)\in\Psi} LR(\chi),
\end{align}
{\em(ii)}
$LR_0$ is spanned by $(1,\dots,1),c(M)$ $(M\in V\cap GLR)$,
where $c(M)$  denotes the vector $(\dots,m_g,\dots)$  for $M=\sum_{g\in G} m_gg$,

\noindent
{\em(iii)} $M[\alpha]=0$ holds for $M\in V\cap GLR$.
\end{cor}
\proof
{
First, let us show \eqref{eq29}.
Let $M\in V\cap  GLR$;
then  we see that $M\delta=0$  implies $M=M(1-\delta)=\sum_{\chi(\ne1)\in\Psi}MC(\chi)$.
Proposition \ref{prop3.6} implies $MC(\chi)\in LR(\chi)$.
Hence the left-hand side is contained in the right-hand side of \eqref{eq29}.
Conversely, suppose that $M\in LR(\chi)$ $(\chi\ne1)$; then we see that  $M=MC(\chi)\in V$ and   \eqref{eq28} implies  $M\in GLR$.
Thus the right-hand side of \eqref{eq29} is contained in the left-hand side. 
%

Let $\sum_g m_gg(\alpha) = m$ $(m_g,m\in\mathbb{Q})$ be a linear relation.
Putting $M:= \sum m_g g$,
we have $M\in GLR$. 
Decompose as $M=\sum m_gg=(\sum m_g)\delta+M'$ with $M':=\sum m_g(g-\delta)$.
Then $M'\in V \cap  GLR$ is clear.
Hence there are elements $M_\chi\in LR(\chi)$ for $\chi(\ne1)\in \Psi $ such that
$M'=\sum M_\chi$ by Corollary \ref{cor3.1}.
Hence we have the decomposition
\begin{equation}\label{eq30}
\sum_{g\in G} m_gg=(\sum_{g\in G} m_g)\delta+\sum_{\chi (\ne1)\in\Psi}M_\chi,
\end{equation}
where $m=\sum_g m_gg(\alpha)=(\sum m_g)|G|^{-1}\sum g(\alpha)$ with $M_\chi[\alpha]=0$.
\qed
}\vspace{2mm}

\noindent
In case that  $\deg\chi=1$ or $\chi$ is defined over $\mathbb{Q}$, the structure of $LR(\chi)$ is
rather easy to find as we see from now on.
\begin{thm}\label{thm3.1}
Suppose that $\deg\chi=1$ and let $l$ be the order  of $\chi$;
then  $LR(\chi)\not=\{0\}$ holds if and only if  $\sum_{g\in G} \chi(g)^dg(\alpha)=0 $ for every integer $d$ relatively prime to $l$, i.e. 
$c_{\sigma(\chi)}[\alpha]=0$ for ${}^\forall\sigma\in Gal(\chi)$, and then
we have
\begin{align*}
LR(\chi)
 &=\{\sum_g tr_{\mathbb{Q}(\zeta_l)/\mathbb{Q}}(\,m\,\overline{\chi(g)}\,) g \mid m \in \mathbb{Q}(\zeta_l)\}
\end{align*}
with $\dim_{\mathbb{Q}} LR(\chi)=[\mathbb{Q}(\zeta_l):\mathbb{Q}]$.
\end{thm}
\proof
{
We note that $\mathbb{Q}(\chi)=\mathbb{Q}(\zeta_l)$ and the equation 
$hc_\chi=\chi(h)c_\chi$ holds for each element $h\in G$,
hence for an element  $M=MC(\chi)\in\mathbb{Q}[G]$, there is  an element 
$m\in \mathbb{Q}(\chi)$ such that  $Mc_\chi=mc_\chi$.
Under this remark, let us see that the mapping $M\mapsto m$ from the set   $\{M=MC(\chi)\in\mathbb{Q}[G]\}$
to $\mathbb{Q}(\chi)$ is isomorphic.
Suppose that $m=0$, i.e. $Mc_\chi=0$; then $Mc_{\sigma(\chi)}=0$ holds, acting $\sigma$ on coefficients all together,
hence $M=MC(\chi)=M\sum c_{\sigma(\chi)}=0$.
Next, let $m\in \mathbb{Q}(\chi)$; then there is an element $M'c_\chi=mc_\chi$
with $M'\in\mathbb{Q}[G]$.
It is easy to see that $M:=\sum_\sigma M'c_{\sigma(\chi)}=M'C(\chi)$ satisfies $M=MC(\chi)\in
\mathbb{Q}[G]$ and $Mc_\chi=M'c_\chi=mc_\chi$.
Therefore, the mapping $M\mapsto m$ is isomorphic,
hence $M=0$ and $m=0$ are equivalent.
Note that $Mc_\chi=mc_\chi$ implies $Mc_{\sigma(\chi)}=\sigma(m)c_{\sigma(\chi)}$.
Take an element  $M=MC(\chi)\in\mathbb{Q}[G]$.
Using the above notation,
we see the following two equations
\begin{align*}
M&=\sum_{\sigma\in Gal(\chi)}Mc_{\sigma(\chi)}
\\
&=\sum_\sigma   \sigma(m)c_{\sigma(\chi)}
\\
&=|G|^{-1}\sum_gtr_{\mathbb{Q}(\chi)/\mathbb{Q}}
(m\overline{\chi(g)})g,
\end{align*}
and 
\begin{align*}
Mc_{\sigma(\chi)}(\sum_g g(\alpha)g^{-1})
&=
\sigma(m)|G|^{-1}\sum_{h,g}\sigma(\overline{\chi(h)})
 g(\alpha)hg^{-1}
\\
&=|G|^{-1}\sigma(m)\sum_{h,g}\sigma(\overline{\chi(hg)})g(\alpha)h
\\
&=|G|^{-1}\sigma(m)\sum_h\sigma(\overline{\chi(h)})\{\sum_{g}\sigma({\chi(g)})g^{-1}(\alpha)\}h.
 \end{align*}
 The condition $M=MC(\chi)\in LR(\chi)$ holds if and only if 
$M(\sum g(\alpha)g^{-1})\linebreak[3]=MC(\chi)(\sum g(\alpha)g^{-1})=0$,
which means $Mc_{\sigma(\chi)}(\sum g(\alpha)g^{-1})=0$ for every $\sigma\in Gal(\chi)$.
By the second equation above, it is equivalent to
$$
\sigma(m)\sum_g \sigma(\chi(g))g^{-1}(\alpha)=0
$$ 
for every $\sigma\in Gal(\chi)$,
which means either $m=0$ or $\sum_g \sigma(\chi(g))g^{-1}(\alpha)=0$.
Thus we have $\sum_g \sigma(\chi(g))g^{-1}(\alpha)=0$ 
for every $\sigma\in Gal(\chi)$ if $LR(\chi)\ne0$.
If, conversely  $\sum_g \sigma(\chi(g))g^{-1}(\alpha)=0$ holds
for every $\sigma\in Gal(\chi)$,
then $M\sum g(\alpha)g^{-1}=0$, i.e. $M\in LR(\chi)$ follows for every 
$M=MC(\chi)\in\mathbb{Q}[G]$.
Thus the first statement of the theorem has been proven.
Suppose $LR(\chi)\not=0$; then the above argument implies $LR(\chi)=\{MC(\chi)\mid M\in\mathbb{Q}[G]\}=\{\sum_g tr_{\mathbb{Q}(\zeta_l)/\mathbb{Q}}(\,m\,\overline{\chi(g)}\,) g \mid m \in \mathbb{Q}(\zeta_l)\}$ by the first equation.
Since $\sum_g tr_{\mathbb{Q}(\zeta_l)/\mathbb{Q}}(\,m\,\overline{\chi(g)}\,) g=0$ if and only if $m =0$, we find $\dim_{\mathbb{Q}} LR(\chi)=[\mathbb{Q}(\zeta_l):\mathbb{Q}]$.
\qed
}

\vspace{2mm}

\noindent
By this theorem,  Corollary \ref{cor3.1} and \eqref{eq30}, we can understand the structure 
of linear relations among roots if
$\mathbb{Q}(\alpha)$ is an abelian extension of $\mathbb{Q}$.

\noindent
{\bf Example 3.1}\label{ex3.1}
Let us consider  a polynomial $f(x)$ of degree $6$ with a root $\alpha$ 
 such that $\mathbb{Q}(\alpha)$ is an abelian
extension of $\mathbb{Q}$.
Let $G$ be the Galois group $Gal(\mathbb{Q}(\alpha)/\mathbb{Q})$ with a generator 
$ g_0$, and let $\omega$ be a third root of unity, i.e.  $\omega^2+\omega+1=0$.
We number roots by $\alpha_i = g_0^i(\alpha)$.
Let $\chi$ be a non-trivial irreducible character of order $l$ satisfying $LR(\chi)\ne 0$.
Then the relations corresponding to $LR(\chi)$ are $\sum_g tr_{\mathbb{Q}(\chi(g_0))/\mathbb{Q}}(m\overline{\chi(g)})g(\alpha)=0$ for every $m\in\mathbb{Q}(\chi(g_0))$.
\vspace{2mm}

\noindent
Case of $l=2$: We have $\chi(g_0)=-1$.
Relations are $\alpha_1+\alpha_3+\alpha_5=\alpha_2+\alpha_4+\alpha_6$.
\vspace{2mm}

\noindent
Case of $l=3$: We may assume $\chi(g_0)=\omega$.
Then relations of $LR(\chi)$ are $\alpha_1+\alpha_2+\alpha_4+\alpha_5=2\alpha_3+2\alpha_6$ 
for $m=1$,
$\alpha_2+\alpha_3+\alpha_5+\alpha_6=2\alpha_1+2\alpha_4$ for $m=\omega$,
i.e. $\alpha_1+\alpha_4=\alpha_2+\alpha_5=\alpha_3+\alpha_6$.
\vspace{2mm}

\noindent
Case of $l=6$: We assume $\chi(g_0)=-\omega$.
Equations $\sum_g tr_{\mathbb{Q}(\omega)/\mathbb{Q}}(m\overline{\chi(g)})g(\alpha)\linebreak[3]=0$ for $m=1,\omega$ imply $\alpha_1+\alpha_5+2\alpha_6=\alpha_2+2\alpha_3+\alpha_4$ and
$2\alpha_1+\alpha_2+\alpha_6=\alpha_3+2\alpha_4+\alpha_5$, respectively,
hence the basis of $LR(\chi)$ is $\alpha_1-\alpha_4=-\alpha_2+\alpha_5=\alpha_3-\alpha_6$.
\vspace{2mm}

The case of $l=6$ is compatible with neither the case of $l=2$ nor $l=3$.
\vspace{2mm}

The following polynomials define the same field $\mathbb{Q}(\zeta_7)$ and 
let $g_0 : \zeta_7\mapsto\zeta_7^3$. We simply write $\zeta=\zeta_7$.
\begin{enumerate}
\item 
$f(x)=x^6+x^5+\dots+1$ has only a trivial linear relation $\sum\alpha_i=-1$.
As stated, $G=S_6$, $\det((\bm{m}_i,\bm{m}_j))=6$ and $Pr(f,\sigma)=1$ for every $\sigma\in S_6$.

\item($l=2$)
$f(x)=x^6 - 4x^5 + 9x^4 - 15x^3 + 18x^2 - 16x + 8$ has roots
$[\alpha_1,\dots,\alpha_6] = [\zeta^5 + \zeta^4 + 1, \zeta^5 + \zeta + 1, \zeta^3 + \zeta + 1, \zeta^3 + \zeta^2 + 1, -\zeta^5 - \zeta^4 - \zeta^3 - \zeta, -\zeta^5 - \zeta^3 - \zeta^2 - \zeta]$,
and the basis of linear relations is $\alpha_1+\alpha_3+\alpha_5=\alpha_2+\alpha_4+\alpha_6=2$.
$\#G =72$, $\det((\bm{m}_i,\bm{m}_j))=9$ and calculation by  computer suggests that $Pr(f,\sigma)>0$ holds if and only if $\sigma \not\in G$. 

\item($l=3$)
$f(x) = x^6 - 3x^5 + 9x^4 - 13x^3 + 11x^2 - 5x + 1$ has roots
$[\alpha_1,\dots,\alpha_6]  = [-\zeta^3 - \zeta^2 - \zeta, \zeta^5 + \zeta^4 + \zeta + 1, \zeta^5 + \zeta^3 + \zeta + 1, \zeta^3 + \zeta^2 + \zeta + 1, -\zeta^5 - \zeta^4 - \zeta, -\zeta^5 - \zeta^3 - \zeta]$,
and the basis of linear relations is $\sum\alpha_i=3$,\,$\alpha_1+\alpha_4=\alpha_2+\alpha_5=\alpha_3+\alpha_6(=1)$, i.e. $\alpha_1+\alpha_4=\alpha_2+\alpha_5=\alpha_3+\alpha_6=1$.
$\#G=48$, $\det((\bm{m}_i,\bm{m}_j))=8$, and $Pr(f,\sigma)>0$ holds if and only if $\sigma \in (4,6)G$,
and then $Pr(f,\sigma)=1$.

\item($l=2,3$)
$f(x) = x^6 + 14x^4 + 49x^2 + 7$ has roots
$[\alpha_1,\dots,\alpha_6]   = [-\zeta^5 - 2\zeta^3 - \zeta^2 - 2\zeta - 1, \zeta^5 + \zeta^4 - \zeta^3 - \zeta^2, 2\zeta^5 + \zeta^4 + \zeta^3 + 2\zeta + 1, \zeta^5 + 2\zeta^3 + \zeta^2 + 2\zeta + 1,
 -\zeta^5 - \zeta^4 + \zeta^3 + \zeta^2, -2\zeta^5 - \zeta^4 - \zeta^3 - 2\zeta - 1]$
and the basis of linear relations is $\sum\alpha_i=0$, $\alpha_1+\alpha_3+\alpha_5=\alpha_2+\alpha_4+\alpha_6(=0)$,
$\alpha_1+\alpha_4=\alpha_2+\alpha_5=\alpha_3+\alpha_6(=0)$, i.e.
$\alpha_4=-\alpha_1,\alpha_5=-\alpha_2,\alpha_6=-\alpha_3,\alpha_1+\alpha_3+\alpha_5=0$.
$\#G=12$, $\det((\bm{m}_i,\bm{m}_j))=12$
 and $Pr(f,\sigma)>0$ holds if and only if $\sigma\in  (1, 3, 6)G \cup (2,5)(4,6)G$.
$Pr(f,(1,3,6))=3/4,Pr(f,(2,5)(4,6))=1/4$.
It is easy to see that
\begin{align*}
\mathfrak{D}(f,&(1,3,6))
\\
=&\left\{(x_1,x_2,x_1+x_2,1-x_1-x_2,1-x_2,1-x_1)\left|
\begin{array}{l}
0\le x_1\le x_2,
\\
x_1+x_2 \le1/2
\end{array} 
\right.
\right\}
\\
=\,&
\{
C+y_1\delta_1+y_2\delta_2\mid y_2/\sqrt{3}\le y_1\le\sqrt{3}y_2,\sqrt{3}y_1+y_2\le\sqrt{3}
\},
\end{align*}
where $C=(0,0,0,1,1,1)$ and orthnormal vectors $\delta_1,\delta_2$ are given by $2\delta_1=(1,0,1,-1,0,-1)$,
$2\sqrt{3}\delta_2=(-1,2,1,-1,-2,1)$ and the volume is $\sqrt{3}/8$,
and
\begin{gather*}
\hspace{-7cm}\mathfrak{D}(f,(2,5)(4,6))
\\
=\left\{  (x_1,x_2,1-x_1-x_2,x_1+x_2,1-x_2,1-x_1)\left|
\begin{array}{l}
0\le x_1\le x_2,
\\
x_1+2x_2\le1,
\\
x_1+x_2\ge1/2
\end{array} 
\right.
\right\}
\\
=
\!\{
C\!+y_1\delta_1+y_2\delta_2\mid y_2/\sqrt{3}\le y_1\le\!\sqrt{3}y_2,y_1+\!\sqrt{3}y_2\le2,y_1+y_2/\sqrt{3}\ge1
\},
\end{gather*}
where $C=(0,0,1,0,1,1)$ and orthnormal vectors $\delta_1,\delta_2$ are given by $2\delta_1 =(1,0,-1,1,0,-1)$, $2\sqrt{3}\delta_2=(-1,2,-1,1,-2,1)$ and the volume is $\sqrt{3}/24$,
hence the ratio  $\mathfrak{D}(f,\sigma)/Pr(f,\sigma)=\sqrt{3}/6=\sqrt{|\det((\bm{m}_i,\bm{m}_j))|}/\#\bm{G}$ is independent of $\sigma$.

\item($l=6$)
$f(x) = x^6 - x^5 + x^4 - x^3 + 15x^2 - x + 29$ has roots
$[\alpha_1,\dots,\alpha_6] = [\zeta^3 - \zeta^2 - \zeta, \zeta^5 + \zeta^4 + 2\zeta^2 + \zeta + 1, -\zeta^5 - 2\zeta^4 - \zeta^3 - 2\zeta^2 - \zeta - 1, 2\zeta^4 + \zeta^3 + \zeta^2 + \zeta + 1, \zeta^5 - \zeta^4 - \zeta, -\zeta^5 - \zeta^3 + \zeta]$
and the basis of linear relations is $\sum\alpha_i=1$
and $\alpha_1-\alpha_4=-\alpha_2+\alpha_5=\alpha_3-\alpha_6$.
$\#G=12$, $\det((\bm{m}_i,\bm{m}_j))=72$ and $Pr(f,\sigma)>0$ holds if and only if $\sigma$ is in one of 18 cosets of $G$ represented
by $id,(2,3),(2,5),(3,4),(3,6),(4,5), (5,6), $ $(2,3)(4,5), (2,3)(5,6), (2,4)(3,5),$ $ (2,5)(3,4),
(3,4)(5,6),(3,5)(4,6),(4,5)(3,6),  (3,5,4,6), (3,6,4,5),$ $\linebreak (2,4,3,5),(2,5,3,4).$

\end{enumerate}
\subsection{$S_3$}
To see the case of $S_3$-extension, we need more,
and from now on, we assume that the absolutely  irreducible character $\chi$ is associated with a rational matrix representation $A$,
hence there is a minimal ideal $\mathfrak{m}$ of a simple algebra $\mathbb{Q}[G]c_\chi (\subset\mathbb{Q}[G])$ with
a basis  $\{w_1,\dots,w_{\chi(1)}\}$ of $\mathfrak{m}$ over $\mathbb{Q}$ such that
$$
g(w_1,\dots,w_{\chi(1)})=(w_1,\dots,w_{\chi(1)})A(g).
$$
The correspondence $g\mapsto A(g)$ is extended to  an isomorphism from 
$\mathbb{C}[G]c_\chi$ to the matrix ring $M_{\chi(1)}(\mathbb{C})$  with 
$A(\mathbb{Q}[G]c_\chi)=M_{\chi(1)}(\mathbb{Q})$
because of $\mathbb{Q}$ being a splitting field for the representation $A$.
Through this correspondence, take $v_{i,j}\in \mathbb{Q}[G]c_\chi $ satisfying
 \begin{equation*}
v_{i,j}(w_1,\dots,w_{\chi(1)})=(w_1,\dots,w_{\chi(1)})E_{j,i}, \text{ i.e. }v_{i,j}w_k=\delta_{i,k}w_j,
\end{equation*}
where $E_{j,i}$ is the matrix whose $(j,i) $-entry is $1$, otherwise $0$.
Therefore the set $\{v_{i,j}\mid 1\le i,j\le \chi(1)\} $ is a basis of $\mathbb{Q}[G]c_\chi$ over $\mathbb{Q}$.
The identity $C(\chi)=c_\chi$ is obvious  by the assumption that $\chi$ is associated with  a rational matrix representation $A$.

%
\begin{lem}\label{lem3.1}
We have
\begin{gather}\label{eq31}
g(v_{i,1},\dots,v_{i,\chi(1)})=(v_{i,1},\dots,v_{i,\chi(1)})A(g)\text{ for } {}^\forall g\in G,
\\\label{eq32}
v_{i,j}v_{k,l}=\delta_{i,l}v_{k,j}\text{ for } 1\le i,j,k,l\le\chi(1). 
\end{gather}
\end{lem}
\proof
{
Write $gv_{i,j}=\sum_{1\le k,l\le\chi(1)}c_{k,l}v_{k,l}$ with $c_{k,l}\in \mathbb{Q}$;
then we see that $gv_{i,j}w_m$ is $\sum_{k,l}c_{k,l}v_{k,l}w_m=
\sum_{k,l}c_{k,l}\delta_{k,m}w_l=\sum_{1\le l\le\chi(1)}c_{m,l}w_l  =\sum_{1\le k\le\chi(1)}c_{m,k}w_k $.
On the other hand,  we  find that $gv_{i,j}w_m\linebreak[3]=g\delta_{i,m}w_j=\delta_{i,m}\sum_{1\le k\le\chi(1)}w_kA(g)_{k,j}$,
where $A(g)_{k,j}$ is the $(k,j)$-entry of $A(g)$.
Comparing the coefficients of $w_k$ of two equations, we have $c_{m,k}=\delta_{i,m}A(g)_{k,j}$,
i.e. $c_{k,l}=\delta_{i,k}A(g)_{l,j}$.
Hence  we have 
\begin{align*}
gv_{i,j}&=\sum_{1\le k,l\le\chi(1)}\delta_{i,k}A(g)_{l,j}v_{k,l}
\\
&= \sum_{1\le l\le\chi(1)}A(g)_{l,j}v_{i,l} ,
\end{align*}
 i.e. \eqref{eq31}.
The second equation \eqref{eq32} follows from
\begin{align*}
v_{i,j}v_{k,l}(w_1,\dots,w_{\chi(1)})&=v_{i,j}(w_1,\dots,w_{\chi(1)})E_{l,k}
\\
&=(w_1,\dots,w_{\chi(1)})E_{j,i}E_{l,k}
\\
&=\delta_{i,l}(w_1,\dots,w_{\chi(1)})E_{j,k}
\\
&=\delta_{i,l}v_{k,j}(w_1,\dots,w_{\chi(1)}).
\end{align*}

\qed
}

\begin{lem}\label{lem3.2}
For $M=\sum_{1\le k,l\le\chi(1)}m_{k,l}v_{l,k}\in \mathbb{Q}[G]c_\chi$ with $m_{k,l}\in \mathbb{Q}$, 
we have   
\begin{equation}\label{eq33}
A(M(\sum_{g\in G}g^{-1}(\alpha)g))=(m_{i,j})(\sum_{g\in G}g^{-1}(\alpha)A(g)).
\end{equation}
\end{lem}
\proof
{

It is easy to see that
\begin{align*}
M(w_1,\dots,w_{\chi(1)})&=\sum_{k,l}m_{k,l}v_{l,k}(w_1,\dots,w_{\chi(1)})
\\
&=\sum_{k,l}m_{k,l}(w_1,\dots,w_{\chi(1)})E_{k,l}
\\
&=
(w_1,\dots,w_{\chi(1)})(m_{i,j}),
\end{align*}
and
$$
\sum_{g\in G}g^{-1}(\alpha)g(w_1,\dots,w_{\chi(1)})=
(w_1,\dots,w_{\chi(1)})(\sum_gg^{-1}(\alpha)A(g)),
$$
hence
$$
M(\sum_{g\in G}g^{-1}(\alpha)g)(w_1,\dots,w_{\chi(1)})=
(w_1,\dots,w_{\chi(1)})(m_{i,j})(\sum_gg^{-1}(\alpha)A(g)),
$$
i.e. \eqref{eq33}.
\qed
}

\begin{prop}\label{prop3.7}Each column of the matrix  $B:=\sum_{g\in G}g^{-1}(\alpha)A(g)$
is  equal to a linear combination of the first column of $h(B)$ $(h\in G)$ over $\mathbb{Q}$.
\end{prop}
\proof
{
Taking elements $c_{j,g}\in \mathbb{Q}$ such that $w_j = \sum_{h\in G} c_{j,h}hw_1$,
we have $gw_j=\sum_{h\in G}c_{j,h}ghw_1$, hence
\begin{align*}
gw_j&=(w_1,\dots,w_{\chi(1)})\times(\text{\,the $j$th column of } A(g))
\\
&=\sum_{h\in G} c_{j,h}(w_1,\dots,w_{\chi(1)})\times(\text{\,the first column of } A(gh))),
\end{align*}
therefore
\begin{align*}
\text{\,the $j$th column of } A(g)
=\sum_h c_{j,h}\times(\text{\,the first column of } A(gh)).
\end{align*}
Thus the $j$-th column of $B$ is
\begin{align*}
&\sum_{g\in G} g^{-1}(\alpha)\times(\text{ the $j$th column of }A(g))
\\
=\,& \sum_{g,h\in G} g^{-1}(\alpha)c_{j,h}\times(\text{ the first column of }A(gh)) 
\\
=\,& \sum_{g,h\in G} c_{j,h}hg^{-1}(\alpha)\times(\text{ the first column of }A(g)) 
\\
=\,& \sum_{h\in G} c_{j,h}h\left(\sum_{g\in G}g^{-1}(\alpha)\times(\text{ the first column of }A(g)) \right).
\\
=\,& \sum_{h\in G} c_{j,h}\times\text{ the first column of }h(B) .
\end{align*}
\qed
}
\vspace{2mm}

%
%
%
\begin{thm}\label{thm3.2}For $M=Mc_\chi\in\mathbb{Q}[G]$, write 
$$
Mc_\chi=\sum_{1\le k,l\le\chi(1)}m_{k,l}v_{l,k}\in \mathbb{Q}[G]c_\chi\quad (m_{k,l}\in \mathbb{Q}).
$$
Then $M\in LR(\chi)$ holds  if and only if
\begin{equation}\label{eq34}
(m_{i,j})\left(\sum_{g\in G} g^{-1}(\alpha)A(g)\right)=0^{(\chi(1))}.
\end{equation}
Moreover it is equivalent to 
\begin{equation}\label{eq35}
(m_{i,j})\times(\text{ the first column of }\sum_{g\in G} g^{-1}(\alpha)A(g))
=0^{(\chi(1),1)}.
\end{equation}
\end{thm}
\proof{For $M=Mc_\chi\in \mathbb{Q}[G]$, the condition $M\in LR(\chi)$ is, by definition
 equivalent to   $M(\sum_gg^{-1}(\alpha)g)=0$, which is $(m_{i,j})(\sum_g g^{-1}(\alpha)A(g))=0^{(\chi(1))}$,
i.e. \eqref{eq34} by Lemma \ref{lem3.2}.
Then \eqref{eq34} follows from Proposition \ref{prop3.7}. 
\qed
}
\vspace{1mm}

%
%
%
%
%

Now let us  take an irreducible polynomial $f(x)$ of degree $6$ with root $\alpha$
such that $\mathbb{Q}(\alpha)$ is a Galois extension of $\mathbb{Q}$ with Galois group 
isomorphic to $S_3$.
Denote the Galois group $Gal(\mathbb{Q}(\alpha)/\mathbb{Q})$ by $G$, and
let $G=\langle\sigma,\mu\rangle$ with $\sigma^3=\mu^2=1,\mu\sigma\mu=\sigma^2$.
The characters of $G$ are
\begin{enumerate}
\item\label{c'1}
the trivial character $\chi_1$ with $c_{\chi_1}=\frac{1}{6}\sum_gg$,
\item\label{c'2}
the character $\chi_2$ of degree $1$ defined by $\chi(\sigma)=1,\chi(\mu)=-1$
with $c_{\chi_2}=\frac{1}{6}(\sum_{i=0}^2\sigma^i -\sum_{i=0}^2\sigma^i\mu)$,
\item\label{c'3}
the character of degree $2$ corresponding to the representation 
$$
A(\sigma)=\left(\begin{array}{cc}0&-1\\1&-1\end{array}\right),\,
A(\mu)=\left(\begin{array}{cc}0&1\\1&0\end{array}\right)
$$
with $c_{\chi_3}=\frac{1}{3}(2-\sigma-\sigma^2) \in \mathbb{Q}[G]$
and $\mathbb{Q}[G]c_{\chi_3}=\langle 1-\sigma,\sigma-\sigma^2,\mu-\sigma\mu,\sigma\mu-\sigma^2\mu\rangle_{\mathbb{Q}}$.
\end{enumerate}
For simplicity, we number roots of $f(x)$ as follows:
\begin{equation*}
\begin{array}{lll}
\alpha_1 := \alpha,&\alpha_2:=\sigma(\alpha),&\alpha_3:=\sigma^2(\alpha),
\\
\alpha_4:=\mu(\alpha),&\alpha_5:=\sigma(\alpha_4),&\alpha_6:=\sigma^2(\alpha_4),
\\
(\mu(\alpha_2)=\alpha_6,&\mu(\alpha_3)=\alpha_5)&
\end{array}
\end{equation*}
and abbreviate as
\begin{equation}\nonumber
m_1=m_{\sigma^0},\,
m_2=m_{\sigma},\,m_3=m_{\sigma^2},\,m_4=m_{\mu},\,m_5=m_{\sigma\mu}
,\,m_6=m_{\sigma^2\mu}.
\end{equation}

\noindent
In case of $\chi_2$,  Theorem \ref{thm3.1} implies that if $LR(\chi_2)\ne0$, then
$$
LR(\chi_2)=\{m(1+\sigma+\sigma^2-\mu-\sigma\mu-\sigma^2\mu)\mid m\in\mathbb{Q}\}.
$$
Let us consider the case $\chi:=\chi_3$.

Put
\begin{align*}
v_1&:=(1-\sigma-\sigma\mu+\sigma^2\mu)/3\,(=v_1c_\chi),
\\
v_2&:=-(\sigma-\sigma^2-\mu+\sigma\mu)/3\,(=v_2c_\chi=-v_1\sigma),
\end{align*}
then it is easy to see that
 for $i=1,2$
 \begin{align*}
\sigma(v_i,\sigma v_i)=(v_i,\sigma v_i)A(\sigma),\mu(v_i,\sigma v_i)=(v_i,\sigma v_i)A(\mu),
\end{align*}
and
\begin{align*}
v_1(v_i,\sigma v_i)=(v_i,\sigma v_i)\left(\begin{array}{rr}1&0\\0&0\end{array}\right),
\sigma v_1(v_i,\sigma v_i)=(v_i,\sigma v_i)\left(\begin{array}{rr}0&0\\1&0\end{array}\right),
\\
v_2(v_i,\sigma v_i)=(v_i,\sigma v_i)\left(\begin{array}{rr}0&1\\0&0\end{array}\right),
\sigma v_2(v_i,\sigma v_i)=(v_i,\sigma v_i)\left(\begin{array}{rr}0&0\\0&1\end{array}\right).
\end{align*}
Thus $v_{1,1}:=v_1,v_{1,2}:=\sigma v_1,v_{2,1}:=v_2,v_{2,2}:=\sigma v_2$ satisfy \eqref{eq31}
and \eqref{eq32}, and $v_{i,j}$ are in $\mathbb{Q}[G]$ for  $i,j=1,2$.
These  imply, for $M=c_1v_1+c_2\sigma v_1+c_3v_2+c_4\sigma v_2$
\begin{equation}\label{eq36}
M(v_i,\sigma v_i)=(v_i,\sigma v_i)\left(
\begin{array}{rr}c_1&c_3\\c_2&c_4\end{array}
\right)\,(=(v_i,\sigma v_i)C, \text{ say}).
\end{equation}
Moreover we see, for $i=1,2$
\begin{align*}
&(\sum_g g^{-1}(\alpha)g)(v_i,\sigma v_i)
\\=\,&(v_i,\sigma v_i)
\left(\begin{array}{cc}
\alpha_1-\alpha_2-\alpha_5+\alpha_6&\alpha_2-\alpha_3+\alpha_4-\alpha_6
\\
-\alpha_2+\alpha_3+\alpha_4-\alpha_5&\alpha_1-\alpha_3+\alpha_5-\alpha_6
\end{array}\right)
\\
=&\,(v_i,\sigma v_i){\bf{A}}, \text{ say},
\end{align*}
hence the condition $M\in LR(\chi)$ is equivalent to $C{\bf{A}}=0.$
The second column of ${\bf{A}}$ is the image of the first column by $\sigma$,
hence $C{\bf{A}}=0$ is certainly equivalent to $C\times(\text{the first column of }{\bf{A}})=0$.
Let us see  $({\bf{A}}_{1,1},{\bf{A}}_{2,1})\ne(0,0)$.
If ${\bf{A}}_{1,1}={\bf{A}}_{2,1}=0$ holds,
then we get ${\bf{A}}_{1,1}+\sigma({\bf{A}}_{2,1})=2\alpha_1-\alpha_2-\alpha_3=0$ 
and $2\alpha_2-\alpha_3-\alpha_1=0$, acting $\sigma$. Their  difference is $3(\alpha_1-\alpha_2)=0$,  which is a contradiction.

\noindent
Let us see $LR(\chi)$ explicitly.
\begin{prop}\label{prop3.8}
Continue the above and suppose that  $LR(\chi)\ne\{0\}$;
then there are rational numbers $a,b$ with $(a,b)\ne(0,0)$ such that 
\begin{equation}\label{eq37}
a(\alpha_1-\alpha_2-\alpha_5+\alpha_6)=
b(\alpha_2-\alpha_3-\alpha_4+\alpha_5),
\end{equation}
and the basis of $LR(\chi)$ is given by $av_1+bv_2,\sigma(av_1+bv_2)$,
and corresponding linear relations are spanned by
$a(\alpha_1-\alpha_2-\alpha_5+\alpha_6)+b(-\alpha_2+\alpha_3+\alpha_4-\alpha_5)=0$
and $a(\alpha_2-\alpha_3+\alpha_4-\alpha_6)+b(\alpha_1-\alpha_3+\alpha_5-\alpha_6)=0$.
\end{prop}
\proof
{
Suppose $LR(\chi)\ne0$ and take a non-zero element $M=
c_1v_1+c_2\sigma v_1+c_3v_2+c_4\sigma v_2\in \mathbb{Q}[G]C(\chi)$.
 Then at least one of $c_i$ is not zero,
 hence there are rational numbers $a,b$ with $(a,b)\ne(0,0)$ satisfying $a{\bf{A}}_{1,1}+b{\bf{A}}_{2,1}=0$,
 i.e. \eqref{eq37}.
 The we have $\left|\begin{array}{rr}{\bf{A}}_{1,1}&b\\-{\bf{A}}_{2,1}&a\end{array}\right|
 =\left|\begin{array}{rr}{\bf{A}}_{1,1}&c_3\\-{\bf{A}}_{2,1}&c_1\end{array}\right|=
 \left|\begin{array}{rr}{\bf{A}}_{1,1}&c_4\\-{\bf{A}}_{2,1}&c_2\end{array}\right|=0$
by $C{\bf{A}}=0$.
Therefore,   there are rational numbers $\kappa_1,\kappa_2$
 such that $(c_3,c_1)=\kappa_1(b,a),(c_4,c_2)=\kappa_2(b,a)$, thus we have
 $M= \kappa_1(av_1+bv_2)+\kappa_2(a\sigma v_1+b\sigma v_2)$.

 \qed
}
\vspace{2mm}

\noindent
{\bf Example 3.2}\label{ex3.2}
Let $f_0(x):=x^6+3$ with a root  $\zeta$.
For a primitive third root $\omega$ of unity, define the automorphisms $\sigma$, $\mu$ by $\sigma(\zeta)=\omega\zeta$ and $\mu(\zeta)=-\zeta$.
Roots are $[\alpha_1,\dots,\alpha_6]=[\zeta,-\zeta^4/2-\zeta/2, \zeta^4/2-\zeta/2,
                                                                            -\zeta,\zeta^4/2+\zeta/2, -\zeta^4/2+\zeta/2]$,
which satisfy $\alpha_1+\alpha_2+\alpha_3=\alpha_4+\alpha_5+\alpha_6$, and
$\alpha_1-\alpha_2-\alpha_5+\alpha_6=\alpha_2-\alpha_3-\alpha_4+\alpha_5$ and
$\alpha_2-\alpha_3+\alpha_4-\alpha_6=-\alpha_1+\alpha_3-\alpha_5+\alpha_6$,
that is $LR(\chi_2)\ne0,LR(\chi_3)\ne0$.

The following three polynomials give the same field $\mathbb{Q}(\zeta)$ as $f_0(x)$.
\vspace{1mm}

\noindent
A polynomial $f_1(x):=x^6 + 6x^5 + 24x^4 + 14x^3 + 15x^2 - 12x + 16$ with roots 
$
-\zeta^3 -\zeta^2 - \zeta - 1,\,
-\zeta^5/2 + \zeta^4/2 - \zeta^3 + \zeta^2/2 + \zeta/2 - 1,\,
\zeta^5/2 - \zeta^4/2 - \zeta^3 + \zeta^2/2 + \zeta/2 - 1,\,
\zeta^3 - \zeta^2 + \zeta - 1,\,
-\zeta^5/2 - \zeta^4/2 + \zeta^3 + \zeta^2/2 - \zeta/2 - 1,\,
\zeta^5/2 + \zeta^4/2 + \zeta^3 + \zeta^2/2 - \zeta/2 - 1
$ 
has no non-trivial linear relation.
\vspace{1mm}

\noindent
A polynomial $f_2(x) :=x^6 + 6x^5 + 15x^4 + 14x^3 + 24x^2 + 24x + 16$ with roots
$
-\zeta^2 -\zeta - 1,\,
-\zeta^5/2 + \zeta^4/2 + \zeta^2/2 + \zeta/2 - 1,\,
\zeta^5/2 - \zeta^4/2 + \zeta^2/2 + \zeta/2 - 1,\,
-\zeta^2 + \zeta - 1,\,
-\zeta^5/2 - \zeta^4/2 + \zeta^2/2 - \zeta/2 - 1,\,
\zeta^5/2 + \zeta^4/2 + \zeta^2/2 - \zeta/2 - 1
$ has linear relations $\sum\alpha_i=-6$ and $\alpha_1+\alpha_2+\alpha_3-\alpha_4-\alpha_5-\alpha_6=0$ 
corresponding to $\chi_2$ only.
\vspace{1mm}

\noindent
A polynomial $f_3(x) :=x^6 + 6x^5 + 24x^4 + 56x^3 + 114x^2 + 132x + 67$ with roots
$
-\zeta^3 - \zeta - 1,\,
\zeta^4/2 - \zeta^3 + \zeta/2 - 1,\,
-\zeta^4/2 - \zeta^3 + \zeta/2 - 1,\,
\zeta^3 + \zeta - 1,\,
-\zeta^4/2 + \zeta^3 - \zeta/2 - 1,\,
\zeta^4/2 + \zeta^3 - \zeta/2 - 1
$
has the  linear relation $\sum\alpha_i=-6$ and the one corresponding to $\chi_3$ with $a=b= 1$ at \eqref{eq37} only.

The author does not know what the set of values $a/b$ is.
Numbers $2/7,37/17,\linebreak[3]\dots$ are such examples.

\noindent
{\bf Remark }
Let $g_1(x)$  be a polynomial with roots $\alpha_1,\dots,\alpha_n$ and  $g_2(x)$   a polynomial 
with roots $\beta_1,\dots,\beta_m$ and   $f(x):=g_1(x)g_2(x)$.
Suppose that $\mathbb{Q}(g_1) \cap \mathbb{Q}(g_2)=\mathbb{Q}$;
then a linear relation $\sum l_i\alpha_i+\sum m_i\beta_i\in \mathbb{Q}$ with all
$ l_i,m_j\in\mathbb{Q}$  implies  obviously
$\sum l_i\alpha_i,\sum m_i\beta_i\in \mathbb{Q}$.
However, in case of $\mathbb{Q}(g_1) \cap \mathbb{Q}(g_2)\not=\mathbb{Q}$
a linear relation among root of $f(x)$ is not necessarily reduced to those of $g_1(x),g_2(x)$ (cf. {\bf{\ref{subsec6.4}}}).
\section{Dependency of $Pr(f,\sigma),\frak{D}(f,\sigma)$ on $\sigma$}\label{sec4}
It is easy to see that for $\nu \in S_n$, the condition $\nu \in \hat {\bm G}$ is equivalent to that 
there is a unimodular matrix $A_\nu$  of degree $t$ such that
\begin{equation}\label{eq38}
\begin{pmatrix}
\nu(\hat{\bm m}_1)
\\
\vdots
\\
\nu(\hat{\bm m}_t)
\end{pmatrix}
=A_\nu
\begin{pmatrix}
\hat{\bm m}_1
\\
\vdots
\\
\hat{\bm m}_t
\end{pmatrix},
\end{equation}
and $\nu \in  {\bm G}$ is equivalent to that 
there is a unimodular  matrix $A'_\nu$ of degree $t$  such that
\begin{equation}\label{eq39}
\begin{pmatrix}
\nu({\bm m}_1)
\\
\vdots
\\
\nu({\bm m}_t)
\end{pmatrix}
=A'_\nu
\begin{pmatrix}
{\bm m}_1
\\
\vdots
\\
{\bm m}_t
\end{pmatrix},
\end{equation}
and \eqref{eq38} implies \eqref{eq39} for $A'_\nu=A_\nu$,
hence $\hat {\bm G}$ is a subgroup of ${\bm G}$.
Here the unimodularity of  matrices $A_\nu,A'_\nu$ is replaced by the  integrality,
since the elementary divisors of the matrix whose rows are $\bm{m_i}$ are $1$.

\begin{prop}\label{prop4.1}
We have
\begin{align*}
\hat{{\bm G}}&=\{\sigma\in S_n\mid\sum_i m_{j,i}\alpha_{\sigma(i)}=m_j  \,({}^\forall j)\}
\\
&=\{\sigma\in S_n\mid\sum_i  l_i\alpha_{\sigma(i)}=l_{n+1} \,\text{ for }{}^\forall (l_1,\dots,l_n,
l_{n+1})\in LR\},
\\
{\bm G}&=\{\sigma\in S_n\mid\sum_i m_{j,i}\alpha_{\sigma(i)}\in\mathbb{Z}  \,({}^\forall j)\}
\\
&=\{\sigma\in S_n\mid\sum_i l_{i}\alpha_{\sigma(i)}\in\mathbb{Q}  \, \text{ for }{}^\forall
 (l_1,\dots,l_{n})\in LR_0 \}.
\end{align*}
\end{prop}
\proof{
Let $\sigma\in S_n$. 
Suppose the condition $\sum_i m_{j,i}\alpha_{\sigma(i)}=m_j  $, which  is clearly equivalent to
 $\sum_i m_{j,\sigma^{-1}(i)}\alpha_i=m_j  $, hence there are integers $a_{i,j}$ such that
$$
\sigma(\hat{\bm{m}}_j)=(m_{j,\sigma^{-1}(1)},\dots,m_{j,\sigma^{-1}(n)},m_j)=\sum_k a_{j,k}\hat{\bm{m}}_k,
$$
which means $\sigma\in \hat {\bm G}$.
If, conversely $\sigma$ is in $\hat {\bm G}$, then there are integers $a_{i,j}$ such that
$\sigma(\hat{\bm{m}_j}) = \sum_k a_{j,k}\hat{\bm{m}}_k$, i.e. 
$$
(m_{j,\sigma^{-1}(1)},\dots,m_{j,\sigma^{-1}(n)},m_j)=\sum_k
a_{j,k}(m_{k,1},\dots,m_{k,n},m_k).
$$
Hence, we see $\sum_i m_{j,i}\alpha_{\sigma(i)}=
\sum_i m_{j,\sigma^{-1}(i)}\alpha_i=\sum_{i,k}a_{j,k}m_{k,i}\alpha_i = \sum_{k}a_{j,k}m_k=m_j$.
Thus we have shown the first equality. 
The second equality follows from that $(m_{j,1},\dots,m_{j,n},m_n)$ $(j=1,\dots,t)$ are a basis of 
$LR\cap\mathbb{Z}^{n+1}$ over $\mathbb{Z}$.
The statement for  $G$ is proved similarly,
noting that the vectors $\bm{m}_1,\dots,\bm{m}_t$ are a basis of 
$LR_0\cap\mathbb{Z}^n=\{(l_1,\dots,l_n)\in\mathbb{Z}^n
\mid \sum l_i\alpha_i\in\mathbb{Z}\}$.
\qed
}
\begin{cor}\label{cor4.1}
If the polynomial $f(x)$ is irreducible, then  $\hat {\bm G}={\bm G}$ holds.
\end{cor}
\proof
{
Let a permutation $\sigma$ be in $\bm{G}$.
Take $(l_1,\dots,l_{n+1}) \in LR$; then we have only to show that
$l:=\sum_{i=1}^n l_{i}\alpha_{\sigma(i)}\in\mathbb{Q} $ is equal to $l_{n+1}$.
Denoting the trace from $\mathbb{Q}(f)$ to $\mathbb{Q}$ by $tr$, we have 
$tr(\alpha_1)=\dots=tr(\alpha_n)$ and
$[\mathbb{Q}(f):\mathbb{Q}]l  
=\sum_i l_{i}tr(\alpha_{\sigma(i)}) =\sum_i l_{i}tr(\alpha_{i})=tr(\sum_i l_{i}\alpha_{i})=tr(l_{n+1})=
[\mathbb{Q}(f):\mathbb{Q}]l_{n+1}$, hence $l=l_{n+1}$.
\qed
}

We remark that
the corollary is not necessarily true for a reducible polynomial (cf. {\bf{6.4.1}}, {\bf{6.4.2}}),
that is $l= l_{n+1}$ fails in the above proof.

\subsection{$Pr(f,\sigma)$}\label{subsect4.1}
To prove the next aim, we introduce  more notations:
\begin{align*}
\bm{G}_0&:=\{ \mu\in S_n\mid \alpha_{\mu(i)}=\tilde{\mu}(\alpha_i)\,\,(1\le{}^\forall i\le n)\text{ for }
{}^\exists \tilde{\mu}\in Gal(\mathbb{Q}(f)/\mathbb{Q}) \},
\\
M(f,\mu)&:=\{ p\in Spl(f) \mid   \alpha_i\equiv r_{\mu(i)}\bmod\mathfrak p\,\,\,(1\le{}^\forall i\le n) 
\text{ for }{}^\exists\mathfrak p\,|\,p  \},
\end{align*}
where $\mathfrak p$ denotes a prime ideal of $\mathbb{Q}(f)$ over $p$,
and then the prime ideal $\mathfrak p$ is unique if $p$ is relatively prime to 
$\prod_{\alpha_i\ne\alpha_j}(\alpha_i - \alpha_j)$, that is $p$ does not divide the integer 
$\prod_{\alpha_i\ne\alpha_j}(\alpha_i - \alpha_j)^2$.
$\bm{G}_0$ is a  subgroup of $\hat{\bm{G}}$ and is isomorphic to $Gal(\mathbb{Q}(f)/\mathbb{Q})$
if $f(x)$ has no multiple root.

It is clear that $Spl(f) = \cup_{\mu\in S_n} M(f,\mu)$, and the condition $Gal(\mathbb{Q}(f)/\mathbb{Q})= S_n$ implies $M(f,\mu)= Spl(f)$.
The notion $M(f,\mu)$ here is better than $M_\mu$ in \cite{K11}.

The set $M(f,\mu)$ depends on the numbering of roots $\alpha_i$'s.
Let $\mathbb{Q}(f) =\mathbb{Q}(\alpha)$ for an algebraic integer $\alpha$ with monic minimal 
  polynomial $F(x)$ and write  $\alpha_i=g_i(\alpha)$ $(i=1,\dots,n)$ for polynomials
$g_i(x)\in\mathbb{Q}[x]$.
Then we see that except finitely many primes $p$
\begin{align*}
M(f,\mu)&=\{
p\in Spl(f)  \mid g_i(\alpha)\equiv r_{\mu(i)}\bmod \mathfrak{p}\,\,\,(1\le{}^\forall i\le n) 
\text{ for }^\exists\mathfrak{p}\,|\,p 
\}
\\
&=\left\{ p\in Spl(f) \left|   
\begin{array}{l}
\text{there is an integer $r$ such that }F( r) \equiv 0 \bmod p 
\\
\text{and }g_i(r)\equiv r_{\mu(i)}\bmod p\,\,\,(1\le{}^\forall i\le n) 
\end{array}
\right.
 \right\}
\\
&=\left\{ p\in Spl(f) \left|   
\begin{array}{l}
\text{there is an integer $r$ such that }F( r) \equiv 0 \bmod p
\\
\text{and }
\{g_{\mu^{-1}(1)}(r)/p\}\le\dots\le\{g_{\mu^{-1}(n)}(r)/p\}
\end{array}\hspace{-2mm}
\right.
\right\},
\end{align*}
where $\{x\}$ means the decimal part of $x$, i.e. $0\le\{x\}<1$ and  $x-\{x\}\in\mathbb{Z}$.
We note that $\mathfrak{p}=(\alpha-r,p)$ is a prime ideal except finitely many primes $p$
 and the ideal $\mathfrak{p}$ and $r\bmod p$ above  are unique if the prime $p$
is sufficiently large.
For an integer $r$ satisfying $F(r)\equiv0\bmod p$ define a prime ideal $\mathfrak{p}$ by 
$(\alpha - r,p)$; then the equation $\alpha_i=g_i(\alpha)\equiv g_i(r)\bmod \mathfrak{p}$ implies
that there is a permutation $\eta\in S_n$ such that $g_i(r)\equiv r_{\eta(i)}\bmod p$
$(i=1,\dots,n)$, 
and it is easy to see that the coset $\eta \bm{G}_0$ is independent of the choice of the root $r$.
If an integer $r'$ has the same properties, then for the prime ideal $(\alpha-r',p)=\mathfrak{p}'=\sigma(\mathfrak{p})$ with some $\sigma\in Gal(\mathbb{Q}(f)/\mathbb{Q})$,
we have  $\alpha_i=g_i(\alpha)\equiv g_i(r')\equiv r_{\eta(i)}\bmod \mathfrak{p}'$,
hence $\sigma^{-1}(\alpha_i)\equiv r_{\eta(i)} \equiv \alpha_i\bmod \mathfrak{p}$.
Thus $\sigma$ is the identity on $\mathbb{Q}(f)=\mathbb{Q}(\alpha_1,\dots,\alpha_n)$,
hence we have $\mathfrak{p}=\mathfrak{p}'$, i.e. $r\equiv r'\bmod p$.

The formulation above leads us to a new question.

{\bf Problem}\label{prob1}
Let $f(x)$ be  an irreducible polynomial
and let $g_l(x),h_l(x)$, \newline $G_l(x_1,\dots,x_n),H_l(x_1,\dots,x_n)$ be polynomials in 
$\mathbb{Q}[x],\mathbb{Q}[x_1,\dots,x_n]$ $(1\le l\le m)$, respectively.
 What is the density of the following sets in  $Spl(f) $ ?
\begin{enumerate}
\item 
$$
\left\{ p\in Spl(f) \left|
\left\{\frac{g_{l}(r_i)}{p}\right\}\le\left\{\frac{h_l(r_i)}{p}\right\}\,(1\le{}^\forall l\le m)\text{ for }{}^\exists i \right.
\right\},
$$
\item
$$
\left\{
p\in Spl(f)\left|
\left\{\frac{G_l(r_1,\dots,r_n)}{p}\right\} \le 
\left\{\frac{H_{l}(r_1,\dots,r_n)}{p}\right\}
(1\le {}^\forall l\le m)\right.
\right\}.
$$
\end{enumerate}
Here the numbers $g_l(r_i), h_l(_ir),G_l(r_1,\dots,r_n)$  and $H_l(r_1,\dots,r_n)$ are supposed to be integers which correspond to elements in $\mathbb{Z}/p\mathbb{Z}$ for primes $p$ not
dividing denominators of $g_l(x),h_l(x),G_l(x_1,\dots,x_n)$ and $H_l(x_1,\dots,x_n)$, respectively.

Let the polynomial  $F(x)$ be the minimal polynomial of an algebraic integer $\alpha$ satisfying $\mathbb{Q}(f)=\mathbb{Q}(\alpha)$, and if we consider local roots of $F(x)$ instead of $f(x)$ in the problems above,
then it is equivalent to considering the problems for $F(x)$ from the beginning instead of $f(x)$ because of $Spl(f)=Spl(F)$.
%


These are discussed in \S \ref{sec8} again.
\begin{lem}\label{lem4.1}
Suppose $\#M(f,\sigma)=\infty$;
then for $\eta\in S_n$ 
the following three conditions are equivalent : \text{\em(i)} $M(f,\eta)=M(f,\sigma)$,
\text{\em(ii)} $\sigma^{-1}\eta\in\bm{G}_0$,  and
\text{\em(iii)} $\#( M(f,\eta) \cap M(f,\sigma))=\infty$.
\end{lem}
\proof{
Suppose the condition (iii), i.e. $ \alpha_i\equiv r_{\eta(i)}\bmod\mathfrak p$ and 
$ \alpha_i\equiv r_{\sigma(i)}\bmod\tilde{\mu}(\mathfrak p)$ for some $\tilde{\mu}\in Gal(\mathbb{Q}(f)/\mathbb{Q})$ for a sufficiently large $p\in Spl(f)$.
They imply  $\tilde\mu(\alpha_i)\equiv \tilde\mu(r_{\eta(i)}) \equiv
r_{\eta(i)}  \bmod \tilde\mu(\mathfrak{p})$
and $\alpha_{\sigma^{-1}\eta(i)}\equiv r_{\eta(i)}\bmod \tilde\mu(\mathfrak{p})$, hence $\tilde\mu(\alpha_i)\equiv\alpha_{\sigma^{-1}\eta(i)}\bmod \tilde\mu(\mathfrak{p})$, which imply $\tilde\mu(\alpha_i)=\alpha_{\sigma^{-1}\eta(i)}$.
Thus the definition implies $\sigma^{-1}\eta\in\bm{G}_0$, hence (ii).
Suppose (ii); then $\sigma=\eta\mu^{-1}$ for some $\mu\in \bm{G}_0$ with $\tilde{\mu}(\alpha_i)
=\alpha_{\mu(i)}$ for some $\tilde{\mu}\in Gal(\mathbb{Q}(f)/\mathbb{Q})$.
For a prime $p\in Spl(f)$, we see that $p\in M(f,\sigma)\leftrightarrow \alpha_i\equiv 
r_{\sigma(i)}\bmod\mathfrak{p}\leftrightarrow \alpha_i\equiv r_{\eta\mu^{-1}(i)}\bmod \mathfrak{p}
\leftrightarrow\alpha_{\mu(i)}\equiv r_{\eta(i)}\bmod\mathfrak{p}\leftrightarrow
\tilde{\mu}(\alpha_i)\equiv r_{\eta(i)}\bmod\mathfrak{p}\leftrightarrow\alpha_i\equiv r_{\eta(i)}\bmod \tilde{\mu}^{-1}(\mathfrak{p})
\leftrightarrow p\in M(f,\eta)$, hence we get (i).
The  condition (i) implies obviously (iii).
\qed
}

The following is fundamental.
\begin{prop}\label{prop4.2}
We have
\begin{equation}\label{eq40}
Spl(f,\sigma)=(\cup_{\mu} M(f,\mu))\cup T_\sigma
\end{equation}
where $\mu$ runs over  the set of permutations satisfying 
$\mu\in \sigma\hat {\bm G}$ with $\#M(f,{\mu})=\infty$,
and $T_\sigma$ is a finite set.
Suppose  $\#Spl(f,\sigma)=\infty$; then for $\nu\in S_n$ we see that
 $Spl(f,\sigma)\cap Spl(f,\nu)$ is an infinite set if and only if 
 $\sigma\hat{\bm{G}} =\nu\hat{\bm{G}} $, and then the difference 
between $Spl(f,\sigma)$ and $ Spl(f,\nu)$ is a finite set.
\end{prop}
\proof{
By  the identity $Spl(f,\sigma) = \cup_{\mu\in S_n}( Spl(f,\sigma)\cap M(f,\mu))$,
we have only to show that $\#(Spl(f,\sigma)\cap M(f,\mu))=\infty$ holds if and only if both 
$\mu\hat G = \sigma\hat G$ and $\#M(f,{\mu})=\infty$ hold, and then $M(f,\mu)\subset Spl(f,\sigma)$.
Suppose that $\#(Spl(f,\sigma)\cap M(f,{\mu}))=\infty$.
The property $\#M(f,\mu)=\infty$ is clear.
For $p\in Spl(f,\sigma)\cap M(f,{\mu})$, we have
$$
\sum_i m_{j,i}r_{\sigma(i)}\equiv m_j\bmod p ,\,\, r_i\equiv \alpha_{\mu^{-1}(i)}\bmod 
{}^\exists\mathfrak p,
$$
which implies $\sum_i m_{j,i}\alpha_{\mu^{-1}\sigma(i)}\equiv m_j\bmod  \mathfrak p$ for infinitely many primes in $p\in Spl(f,\sigma)\cap M(f,{\mu})$,
which implies $\sum_i m_{j,i}\alpha_{\mu^{-1}\sigma(i)}= m_j$,
hence we have $\mu^{-1}\sigma\in \hat {\bm G}$, i.e. $\mu\hat {\bm G} =\sigma \hat {\bm G}$.

\noindent
Conversely, suppose that $\mu\hat {\bm G}= \sigma \hat {\bm G}$ and $\#M(f,{\mu})=\infty$ hold;
then we have the equation $\sum_i m_{j,i}\alpha_{\mu^{-1}\sigma(i)}\linebreak[3]= m_j$.
Hence, for $p\in M(f,\mu)$, we see $\sum_i m_{j,i}r_{\sigma(i)}\equiv m_j \bmod \mathfrak p$,
that is $p\in Spl(f,\sigma)$ and so $M(f,\mu)\subset Spl(f,\sigma)$,
thus $\#(Spl(f,\sigma)\cap M(f,\mu))=\infty$.

\noindent
Hence the condition $\#(Spl(f,\sigma)\cap M(f,{\mu}))=\infty$ is equivalent to 
$\#M(f,{\mu})=\infty$
and $\mu\hat {\bm G} =\sigma \hat {\bm G}$.
And then, we have  $M(f,\mu)\subset Spl(f,\sigma)$ as above.
Therefore \eqref{eq40} has been proved.
If $Spl(f,\sigma)\cap Spl(f,\nu)$ is an infinite set, then there are  permutations $\mu,\mu'$ such that 
$\sigma \hat{\bm{G}}=\mu\hat{\bm{G}},\nu\hat{\bm{G}}=
 \mu'\hat{\bm{G}}$ and $\#(M(f,\mu) \cap M(f,\mu'))=\infty$, hence $\mu\bm{G}_0 = {\mu'}\bm{G}_0$
by Lemma \ref{lem4.1} and then $\sigma \hat{\bm{G}}=\nu\hat{\bm{G}}$.
Hence the difference between $Spl(f,\sigma)$ and $ Spl(f,\nu)$ is that between $T_\sigma$ and $T_\nu$.
If, conversely $\sigma \hat{\bm{G}}=\nu\hat{\bm{G}}$, then from \eqref{eq40} follows that    
the difference between $Spl(f,\sigma)$ and $ Spl(f,\nu)$ is finite.
This completes the proof.
\qed
}
\vspace{1mm}

Thus  wee see that
\begin{align}\nonumber
Spl(f) &=\cup_{\sigma\in S_n} Spl(f,\sigma) = \cup_{\sigma\in S_n/\hat{G}\atop \#Spl(f,\sigma)=\infty} Spl(f,\sigma)\cup
(\text{ a finite set})
\\  \label{eq40.5}
&= \cup_{\sigma\in S_n/\hat{G}\atop \#Spl(f,\sigma)=\infty} \cup_{\mu\in\sigma \hat{G}/G_0} M(f,\mu)\cup(\text{a finite set}).
\end{align}
It is likely that the density  of $M(f,\mu)$ in $Spl(f,\sigma)\,(=Spl(f,\mu))$  at \eqref{eq40} equal to  $1/[\hat{\bm{G}}: \bm{G}_0]$.

\begin{cor}\label{cor4.2}
If  a  prime $p\in Spl(f)$ is sufficiently large, then there are exactly $\#\hat {\bm G}$ permutations $\sigma$
satisfying $p\in Spl(f,\sigma)$, i.e. \eqref{eq6}.
\end{cor}
\proof{
Let $p\in Spl(f)$ be large so that $p\not\in T_\sigma$  and $p\not\in
M(f,\sigma)\cap M(f,\mu)$ if $\#(M(f,\sigma)\cap M(f,\mu))<\infty$ for $\sigma,\mu \in S_n$.
If $p\in Spl(f)$ is contained in $Spl(f,\sigma) $, then $p\in Spl(f,\sigma\nu) $ 
$({}^\forall\nu\in\hat{\bm{G}})$
is clear by Proposition \ref{prop4.2}.
If $p\in Spl(f)$ is contained in $Spl(f,\sigma) \cap Spl(f,\mu)$, hence
in $M(f,{\sigma'})\cap M(f,{\mu'})$ for ${}^\exists\sigma' \in \sigma\hat{\bm{G}}, 
{}^\exists\mu' \in \mu\hat{\bm{G}}$,
 then  $M(f,{\sigma'})\cap M(f,{\mu'})$ is an infinite set,
hence  we see ${\sigma'}^{-1}\mu' \in \bm{G}_0$ and then $\sigma^{-1}\mu\in\hat{\bm{G}}$.
Therefore $p\in Spl(f,\sigma)$ holds for a permutation $\sigma$ in the only one coset by
$\hat {\bm G}$.
\qed
}
\vspace{1mm}

\begin{cor}\label{cor4.3}
We have
\begin{equation}\label{eq41}
\sum_{\sigma\in S_n} Pr(f,\sigma)
=\#\hat{{\bm G}}.
\end{equation}
\end{cor}
\proof{
By Corollary \ref{cor4.2},
we see
\begin{align*}
\#\hat {\bm G}=\#\hat{{\bm G}}\lim_{X\to\infty}\frac{\#Spl_X(f)}{\#Spl_X(f)}
=\lim_{X\to\infty}\sum_{\sigma\in S_n}\frac{\#Spl_X(f,\sigma)}{\#Spl_X(f)}
=\sum_{\sigma\in S_n} Pr(f,\sigma).
\end{align*}
\qed
}
\begin{cor}
Let $\sigma\in S_n$.
If a sufficiently large prime $p\in Spl(f)$ satisfies $\sum_i m_{j,i}\alpha_{\sigma(i)} \equiv m_j\bmod
\frak{p}$  $(j=1,\dots,t)$ for some prime ideal $\frak{p}$ of $\mathbb{Q}(f)$ lying over $p$,
then $\sum_i m_{j,i}\alpha_{\sigma(i)} =m_j$ $(j=1,\dots,t)$  holds.
\end{cor}
\proof
{
Take a permutation  $\kappa\in S_n$ such that $\alpha_{\sigma(i)}\equiv r_{\kappa(i)}\bmod\frak{p}$ $(i=1,\dots,n)$.
Then we see $\sum_i m_{j,i}r_{\kappa(i)}\equiv m_j\bmod p$ $(j=1,\dots,t)$, i.e. $p\in Spl(f,\kappa)$,
hence $p\in M(f,{\kappa\eta})$ for some $\eta\in\hat{G}$ with $\#M(f,{\kappa\eta})=\infty$.
Therefore we have $\alpha_i\equiv r_{\kappa\eta(i)}\bmod \frak{p}'$ $(i=1,\dots,n)$ for some prime ideal $\frak{p}'$ lying over $p$.
For an automorphism $\iota\in Gal(\mathbb{Q}(f)/\mathbb{Q})$ satisfying $\iota(\frak{p}')=\frak{p} $,
we have $\iota(\alpha_i)\equiv r_{\kappa\eta(i)}\bmod \frak{p}$, hence $\iota(\alpha_{\eta^{-1}(i)})
\equiv r_{\kappa(i)}\equiv\alpha_{\sigma(i)}\bmod\frak{p}$, hence $\iota(\alpha_{\eta^{-1}(i)})=
\alpha_{\sigma(i)}$.
Since $\hat{G}$ is a group, we see $\eta^{-1}\in\hat{G}$, hence $\sum_i m_{j,i}\alpha_{\eta^{-1}(i)}=m_j$ 
$(j=1,\dots,t)$, which implies $\sum_i m_{j,i}\alpha_{\sigma(i)}=m_j$ 
$(j=1,\dots,t)$.

\qed
}

By keeping notations in the proof above, the proof shows : Starting from $\sum_i m_{j,i}r_{\kappa(i)}\equiv m_j\bmod p$ and taking $\sigma\in S_n$ satisfying 
$r_{\kappa(i)}\equiv \alpha_{\sigma(i)}\bmod\frak{p}$,  we have 
$\sum_i m_{j,i}\alpha_{\sigma(i)}= m_j$. 
Incidentally  the completion $\mathbb{Q}(f)_{\frak{p}}$ is isomorphic to $\mathbb{Q}_p$ and 
let $\hat{r}_i$ be roots of $f(x)=0$ in $\mathbb{Q}_p$ satisfying 
$\hat{r}_i\equiv r_i\bmod p\mathbb{Z}_p$.
Making use of  the isomorphism from 
$\mathbb{Q}(f)$ into $\mathbb{Q}(f)_{\frak{p}}\cong\mathbb{Q}_p$ and $\alpha_i\equiv r_{\kappa\sigma^{-1}}(i)\bmod\frak{p}$,
we see that the root $\alpha_i$ in $\mathbb{Q}(f)$ is mapped to the root 
$\hat{r}_{\kappa\sigma^{-1}(i)}$.
Thus, we see that for $l_1,\dots,l_{n+1}\in\mathbb{Q}$, the identity 
$\sum_{i=1}^n l_i\alpha_i=l_{n+1}$ is
equivalent to $\sum_{i=1}^n l_i\hat{r}_{\kappa\sigma^{-1}(i)}=l_{n+1}$,
and combining the above,  $\sum_i m_{j,i}r_{\kappa(i)}\equiv m_j\bmod p$ implies
$\sum_i m_{j,\sigma^{-1}(i)}\hat{r}_{\kappa\sigma^{-1}(i)}= m_j$, i.e.
$\sum_i m_{j,i}\hat{r}_{\kappa(i)}= m_j$ in $\mathbb{Q}_p$ .

\begin{prop}\label{prop4.3}
Let $\sigma,\nu$ be permutations, and suppose that $\nu\in\hat {\bm G}$.
Then we have, neglecting a finite set of primes
\begin{gather}\nonumber
Spl(f,\sigma)=Spl(f,\sigma\nu),\,
\\\label{eq42}
Spl(f,\sigma,\{k_j\})=Spl(f,\sigma\nu,\{k_j'\}),\,
\\\label{eq43}
Spl(f,\sigma,\{k_j\},L,\{R_j\})=Spl(f,\sigma\nu,\{k_j'\},L,\{R_j\}),
\end{gather}
where ${}^t(k_1',\dots,k_t'):=A_\nu\cdot{}^t(k_1,\dots,k_t)$ for the integral matrix $A_\nu=(a_{ij})\in GL_t(\mathbb{Z})$ given at \eqref{eq38}.
In particular,
we  have
\begin{gather*}
Pr(f,\sigma)=Pr(f,\sigma\nu),
\\
Pr_D(f,\sigma)=Pr_D(f,\sigma\nu),
\\
Pr(f,\sigma,\{k_j\})=Pr(f,\sigma\nu,\{k_j'\}),
\\
Pr(f,\sigma,\{k_j\},L,\{R_i\})=Pr(f,\sigma\nu,\{k_j'\},L,\{R_i\}).
\end{gather*} 
\end{prop}
\proof{
The first  equation  follows from Proposition \ref{prop4.2}.
Let $p$ be a prime in $Spl(f,\sigma,\{k_j\})$; then we see 
$\sum_i m_{j,i}r_{\sigma(i)}=m_j+k_jp$ and so $\sum_j a_{l,j}\sum_i m_{j,i}r_{\sigma(i)}
\linebreak[3]=\sum_j a_{l,j}m_j+\sum_j a_{l,j}k_jp$, that is  $\sum_i m_{l,\nu^{-1}(i)}r_{\sigma(i)}\linebreak[3]=m_l+k_l'p$,
which implies $p\in Spl(f,\sigma\nu,\{k_j'\})$, that is  $ Spl(f,\sigma,\{k_j\})$ is included in 
$ Spl(f,\sigma\nu,\{k'_j\})$.
Since $A_\nu^{-1}$ is also integral, we have the converse inclusion 
$ Spl(f,\sigma\nu,\{k_j'\})\subset Spl(f,\sigma,\{k_j\})$
similarly,  hence \eqref{eq42}, \eqref{eq43}.
\qed
}

\subsection{$\mathfrak{D}(f,\sigma)$}
By definition, we see  
\begin{align}\nonumber
&{\mathfrak{D}}(f,\sigma)
\\
=&\label{eq44}
\left\{
\bm{x}=(x_1\dots,x_n)\in [0,1]^n\left|
\begin{array}{l}
0\le x_1\le\dots\le x_n\le1,
\\
(\bm{m}_j,\sigma^{-1}(\bm{x}))\in\mathbb{Z}\text{ for }1\le{}^\forall j\le t
\end{array}
\right\}
\right.
.
\end{align}
Another  aim in this section is 
\begin{prop}\label{prop4.4}
Suppose that $vol({\mathfrak{D}}(f,\sigma))>0$;
then 
we have following equivalences.
\begin{align*}
&vol({\mathfrak{D}}(f,\sigma)\cap {\mathfrak{D}}(f,\mu))>0
\\
\iff
&
\langle\sigma(\bm{m}_1),\dots,\sigma(\bm{m}_t)\rangle_{\mathbb{Z}}=
\langle\mu(\bm{m}_1),\dots,\mu(\bm{m}_t)\rangle_{\mathbb{Z}}\,
\\
\iff
&\mu^{-1}\sigma\in {\bm G}
\\
\iff
&\mathfrak{D}(f,\sigma)=\mathfrak{D}(f,\mu).
\end{align*}
In particular, $ {\mathfrak{D}}(f,\sigma)= {\mathfrak{D}}(f,\sigma\nu)$ holds if
and only if  $\nu\in {\bm G}$.
\end{prop}
\proof{
For an integral vector $\bm{k}=(k_1,\dots,k_t) \in \mathbb{Z}^t$, we take a vector  $\bm{x}_{\bm{k}} \in \mathbb{R}^n$
such that $(\bm{m}_j,\bm{x}_{\bm{k}})=k_j$ $(j=1,\dots,t)$.
For $\bm{x}\in\mathbb{R}^n$, the condition $(\bm{m}_j,\sigma^{-1}(\bm{x}))\linebreak=k_j$  $(j=1,\dots,t)$ is equivalent  to
$\sigma^{-1}(\bm{x}) -\bm{x}_{\bm{k}}\in\langle\bm{m}_1,\dots,\bm{m}_t\rangle^\perp_{\mathbb{R}}$.
Therefore we have
\begin{equation}\label{eq45}
{\mathfrak{D}}(f,\sigma)=
\hat{\mathfrak{D}}_n\cap\{
\cup_{\bm{k}\in\mathbb{Z}^t}
\left[\sigma(\bm{x}_{\bm{k}})+\langle\sigma(\bm{m}_1),\dots,\sigma(\bm{m}_t)\rangle^{\perp}_{\mathbb{R}}\right]
\}.
\end{equation}
%
%
Suppose that $vol({\mathfrak{D}}(f,\sigma))>0$; 
if the property $vol({\mathfrak{D}}(f,\sigma)\cap {\mathfrak{D}}(f,\mu))>0$ holds,
then \eqref{eq45} implies $\langle\sigma(\bm{m}_1),\dots,\sigma(\bm{m}_t)\rangle^{\perp}_{\mathbb{R}}=
\langle\mu(\bm{m}_1),\dots,\mu(\bm{m}_t)\rangle^{\perp}_{\mathbb{R}}$, i.e.
\begin{equation*}
\langle\sigma(\bm{m}_1),\dots,\sigma(\bm{m}_t)\rangle_{\mathbb{R}}=
\langle\mu(\bm{m}_1),\dots,\mu(\bm{m}_t)\rangle_{\mathbb{R}}.
\end{equation*}
Since the matrix whose $j$th row is $\bm{m}_j$ is integral with every elementary divisor being $1$,
the above is equivalent to
\begin{equation*}
\langle\sigma(\bm{m}_1),\dots,\sigma(\bm{m}_t)\rangle_{\mathbb{Z}}=
\langle\mu(\bm{m}_1),\dots,\mu(\bm{m}_t)\rangle_{\mathbb{Z}},
\end{equation*}
hence $\mu^{-1}\sigma\in {\bm G}$ holds
and the above identity implies  $\bm{x}\in \mathfrak{D}(f,\sigma)\iff \bm{x}\in \mathfrak{D}(f,\mu)$
by \eqref{eq44}, hence $\mathfrak{D}(f,\sigma)=\mathfrak{D}(f,\mu)$.

\qed
}
\begin{cor}\label{cor4.4}
We have
\begin{equation}\label{eq46}
\sum_{\mu\in S_n}vol(\mathfrak{D}(f,\mu))=\#{\bm G}\cdot vol(\cup_{\mu\in S_n}\mathfrak{D}(f,\mu)).
\end{equation}
\end{cor}
\proof{
Put
$$
S':=\{\sigma\in S_n\mid vol(\mathfrak{D}(f,\sigma))>0\}.
$$
Then we have
\begin{align*}
&\sum_{\sigma\in S_n} vol(\mathfrak{D}(f,\sigma))
\\
=&\sum_{\sigma\in S'} vol(\mathfrak{D}(f,\sigma))
\\
=&
\sum_{\mu\in S'/{\bm G}}\hspace{1mm}\sum_{\sigma\in \mu {\bm G}}vol(\mathfrak{D}(f,\sigma))
\\
=&\,\#{\bm G}\sum_{\mu\in S'/{\bm G}}vol(\mathfrak{D}(f,\mu))
\\
\\=&\,\#{\bm G} \cdot vol(\cup_{\mu\in S'/{\bm G}}\mathfrak{D}(f,\mu))
\\
\\=&\,\#{\bm G}\cdot  vol(\cup_{\mu\in S_n}\mathfrak{D}(f,\mu)).
\end{align*}
\qed
}

In Proposition \ref{prop5.2}, we show that \eqref{eq46} is equal to $\sqrt{\det((\bm{m}_i,\bm{m}_j))}$
under some conditions.

\begin{lem}\label{lem4.2}
Suppose that $f(x)$ has no rational root and $\alpha_i-\alpha_j\notin \mathbb{Q}$ for any distinct $i,j$ and  $\#Spl(f,\sigma,\{k_j\})=\infty$.
For a prime $p\in Spl(f,\sigma,\{k_j\})$ with finite exceptions,
we have
$$
\sum_i m_{j,i}(r_{\sigma(i)}-tr(\alpha_i)/d)/p=k_j\,\,\,( j=1,\dots,t)
$$
and
$$
0<(r_{1}-tr(\alpha_{\sigma^{-1}(1)})/d)/p<\dots<(r_{n}-tr(\alpha_{\sigma^{-1}(n)})/d)/p<1,
$$
where  $tr$ denotes the trace from $\mathbb{Q}(f)$ to $\mathbb{Q}$, and $d :=[\mathbb{Q}(f):\mathbb{Q}]$.
\end{lem}
\proof
{
Since the set
$$
Spl(f,\sigma,\{k_j\})=\{p\in Spl(f)\mid\sum_i m_{j,i}r_{\sigma(i)}=m_j+k_jp\,(1\le j\le t)\}
$$ 
is an infinite set,
the original equation  $\sum_i m_{j,i}\alpha_i = m_j$ imply  $\sum_i m_{j,i}tr(\alpha_i) = dm_j$,
hence for every $ p\in Spl(f,\sigma,\{k_j\})$, we have
$$
\sum_i m_{j,i}(r_{\sigma(i)}-tr(\alpha_i)/d)=k_jp,
$$
which is the first statement.
Next, let us show the second, i.e.
$$
0<(r_{1}-tr(\alpha_{\sigma^{-1}(1)})/d)<\dots<(r_{n}-tr(\alpha_{\sigma^{-1}(n)})/d)<p.
$$
To prove it, it is enough to show that for  $X:=\sum_i |tr(\alpha_i)|/d$, 
all of conditions $r_1>X, p-r_n>X,r_{i+1}-r_i>X$ $(i=1,\dots,n-1)$ hold for a sufficiently large prime $p$ 
in $Spl(f,\sigma,\{k_j\})$.
If the inequality $(0\le)\,r_1\le X$ occurs for infinitely many primes, then there is an integer $r\in[0,X]$ such that $f(r)\equiv0\bmod p$ for infinitely many $p$, which means $f(r)=0$.
If the inequality $(0<)\,p-r_n\le X$ occurs for infinitely many primes, 
then there is an integer $r$ such that 
$r=p-r_n\le X$, hence $f(-r)\equiv0\bmod p$ for infinitely many $p$, which means $f(-r)=0$.
If $r_{i+1}-r_i\le X$ occurs for infinitely many primes, then the equality $r_{i+1}-r_i=r(\le X)$ occurs
for some integer $r$ for infinitely many primes $p$,
that is polynomials $f(x),f(x+r) \bmod p$ have a common root $r_i$ for infinitely many primes, 
which implies that $f(x),f(x+r)$ are not relatively prime, and they are divisible by some irreducible  polynomial $h(x)$, 
hence $f(x)$ is divisible by $h(x),h(x-r)$, so for a root $\alpha$ of $h(x)$, $\alpha,\alpha+r$ are roots of
$f(x)$.
This contradicts the assumption.
\qed
}
\begin{prop}\label{prop4.5}
Suppose that $f(x)$ has no rational root and $\alpha_i-\alpha_j\notin \mathbb{Q}$ for any distinct $i,j$. 
Then we see that the condition $\#Spl(f,\sigma,\{k_j\})=\infty$ implies $vol(\mathfrak{D}(f,\sigma,\{k_j\}))>0$.
\end{prop}
\proof
{
We note that
$\mathfrak{D}(f,\sigma,\{k_j\})$ is the intersection of $(n-t)$-dimensional set $ \{\bm{x}\in\mathbb{R}^n \mid \sum_i m_{j,i}x_{\sigma(i)}=k_j\,({}^\forall j)\}$ and $\{\bm{x}\mid 0\le  x_1\le\dots\le x_n\le1\}$.
And its
$(n-t)$-dimensional open subset $\{\bm{x}\in\mathbb{R}^n \mid \sum_i m_{j,i}x_{\sigma(i)}
=k_j\,({}^\forall j)\} \cap\{\bm{x}\mid 0< x_1<\dots< x_n<1\}$ is not empty,
since it contains $$((r_{1}-tr(\alpha_{\sigma^{-1}(1)})/d)/p,\dots,(r_{n}-tr(\alpha_{\sigma^{-1}(n)})/d)/p)$$  by Lemma \ref{lem4.2}.
Therefore we have $vol(\mathfrak{D}(f,\sigma,\{k_j\}))>0$.
\qed
}

%
%
Lastly, let us give the following 
\begin{prop}\label{prop4.6}
We see that
$$
vol(\hat{\mathfrak D}_n)=\frac{\sqrt{n}}{n!}
$$
and, for an integer $k$ $(1\le k \le n-1)$,
$$
vol(\{{\bm x}\in \hat{\mathfrak D}_n\mid \sum x_i=k\})/vol(\hat{\mathfrak D}_n)=E_n(k).
$$
\end{prop}
\proof{
We note that $vol$ is the $(n-1)$-dimensional volume here.
Let $\theta$ be the angle between two hyperplanes defined by $x_n=0$ and $\sum x_i=0$,
that is  the angle between vectors $(0,\dots,0,1)$ and $(1,\dots,1)$, hence $\cos\theta=1/\sqrt{n}$.
Since a permutation acts on $\mathbb{R}^n$ and on the set $\{\bm{x}\in\mathbb{R}^n\mid \sum x_i=k\}$
as an orthogonal transformation and
 the dimension of the set defined by $\sum x_i\in\mathbb{Z}$ and $x_i=x_{i+1}$ is $n-2$, 
$vol(\hat{\mathfrak D}_n)$ is equal to the volume of the set
$$
\{\bm{x}\mid 0<x_1<\dots<x_n<1,\sum x_i\in\mathbb{Z}\},
$$
hence
we see that
\begin{align*}
vol(\hat{\mathfrak D}_n)
=\frac{1}{n!}vol(\{(x_1,\dots,x_n)\in[0,1)^n \mid \sum x_i\in\mathbb{Z}\})
\end{align*}
and the projection $(x_1,\dots,x_n)\mapsto(x_1,\dots,x_{n-1})$ from 
$\{(x_1,\dots,x_n)\in[0,1)^n \mid \sum x_i\in\mathbb{Z}\}$ to
$[0,1)^{n-1}$   is bijective by $x_1+\dots+x_{n-1}+x_n=\lceil x_1+\dots+x_{n-1}\rceil$,
we have $n!vol(\hat{\mathfrak D}_n)\cos\theta=vol([0,1)^{n-1})=1.$
Next,   we have
\begin{align*}
&\,vol(\{{\bm x}\in \hat{\mathfrak D}_n\mid \sum x_i=k\})/vol(\hat{\mathfrak D}_n)
\\
=&\,
n!\cos\theta\cdot vol(\{{\bm x}\in\hat{\mathfrak D}_n\mid \sum x_i=k\})
\\
=&\,
\cos\theta\cdot vol(\{{\bm x}\in[0,1)^n\mid \sum x_i=k\})
\\
=&\cos\theta\cdot vol(\{{\bm x}\in[0,1)^n\mid \left\lceil\sum_{i=1}^{n-1} x_i\right\rceil=k\})
\\
=&\,
vol(\{\bm x\in[0,1)^{n-1}\mid\left\lceil\sum_{i=1}^{n-1} x_i\right\rceil=k\})
\\
=&\,
E_n(k).
\end{align*}
\qed
}

\noindent
{\bf Remark }
The set $\hat{\frak{D}}_n$ is  in a union of parallel transformations of the subspace  
$S:=\{\bm{x}\in\mathbb{R}^n\mid \sum_{i=1}^n x_i=0\}$.
Write $\bm{e}=(1,\dots,1),\bm{f}=(0,\dots,0,\sqrt{n})$ and let $\tau$ be the symmetry defined by
$\bm{x}\mapsto\bm{x}-\frac{2(\bm{e}-\bm{f},\bm{x})}{(\bm{e}-\bm{f},\bm{e}-\bm{f})}(\bm{e}-\bm{f})$; then by $S=\bm{e}^\perp=\tau(\bm{f})^\perp$, 
we see that
$\tau(\bm{e}_i)
=
\bm{e}_i -\frac{1}{n-\sqrt{n}}(\bm{e}-\bm{f}) $ $(i=1,\dots,n-1)$ are
orthnormal basis of $S$, where $ \bm{e}_1=(1,0,\dots,0),\bm{e}_2=(0,1,0,\dots,0)$ and so on.
Another basis of $S$ is $\bm{u}_i:=\frac{1}{\sqrt{(n-i)(n-i+1)}}(0,\dots,0,n-i,-1,\dots,-1)$,
where the number $n-i$ is on the $i$th entry.

Let us give a geometric remark here.
For linearly independent vectors $\bm{u}_1,\dots,\linebreak[3]\bm{u}_T\in\mathbb{R}^n$ $(1\le T<n)$,
write $$S=\{\sum_{i=1}^T x_i\bm{u}_i \mid x_i\in\mathbb{R}\},$$
and we add vectors $\bm{u}_{T+1},\dots,\bm{u}_n$  so that the matrix $U=\left(\begin{array}{cc}U_1^{(T)}&U_2\\U_3&U_4\end{array}\right)$ with $i$th row being $\bm{u}_i$
is regular.
Moreover, let vectors $\bm{f}_i\in \mathbb{R}^n$ $(i=1,\dots,T)$ be orthonormal basis of $S$ and define the matrix $C=(c_{ij})$
by $\bm{u}_i=\sum_{j=1}^Tc_{ij}\bm{f}_j$.

\noindent
Defining the projection $pr$ from $\mathbb{R}^n$ to $V:=\{(x_1,\dots,x_T,0,\dots,0)\mid x_i\in\mathbb{R}\}$
by
$$
pr(x_1,\dots,x_n)=(x_1,\dots,x_T,0,\dots,0),
$$
let us show that there is a constant $c_1$ such that
${vol(pr(Y))}=c_1{ vol(Y)}$ for any  measurable set $Y$ on $S$ under the condition $\det U_1\ne0$.

First, let us find the condition to  $pr(S)=V$. 
Suppose that  $\bm{y}=(y_1,\dots,y_n)\in\mathbb{R}^n$ and  
$\bm{x}=(x_1, \dots,x_T)\in\mathbb{R}^T$ satisfy the relation  $\bm{y}=\sum_{i=1}^T x_i\bm{u}_i$,
which is equivalent to
$$
(y_1,\dots,y_T)=(x_1,\dots,x_T)U_1,\,\,(y_{T+1},\dots,y_n)=(x_1,\dots,x_T)U_2
$$
by $\bm{y}=(y_1,\dots,y_n)=(x_1,\dots,x_T,0,\dots,0)U$. 
Hence  the conditions $pr(S)=V$ and  $\det U_1\ne0$ are equivalent, and then the mapping $pr$ is 
the isomorphism from $S$ to $V$ and
we see 
\begin{gather*}
pr^{-1}(y_1,\dots,y_T,0,\dots,0)=((y_1,\dots,y_T),(y_1,\dots,y_T)U_1^{-1}U_2)
=\sum_{i=1}^T x_i\bm{u}_i,
\\
\frac{\partial(y_1,\dots,y_T)}{\partial(x_1,\dots,x_T)}=\det U_1.
\end{gather*}
Write 
 $Y:=\{\sum_{i=1}^T x_i\bm{u}_i\mid (x_1,\dots,x_T)\in X\}$ for  a measurable set $X$ in $\mathbb{R}^T$.
Then it is easy to see
$$
Y=\{\sum_{j=1}^T y_j\bm{f}_j\mid (y_1,\dots,y_T):=(x_1,\dots,x_T)C,(x_1,\dots,x_T)\in X\},
$$
hence $vol(Y)=vol(\{(x_1,\dots,x_T)C\mid(x_1,\dots,x_T)\in X\}$ by our definition and it is equal to
$|\det C|vol(X)$.
Therefore we see 
\begin{align*}
vol(pr(Y))&=\int_{(y_1,\dots,y_T)\in pr(Y) }1dy_1\dots dy_T
\\
&=
|\det U_1|\int_{(x_1,\dots,x_T)\in X}1dx_1\dots dx_T
\\
&=|\det U_1|vol(X)=|\det U_1||\det C|^{-1}vol(Y).
\end{align*}
Thus the constant  $c_1=|\det U_1||\det C|^{-1}$ which is independent of $Y$
satisfies $vol(pr(Y))=c_1vol(Y)$.

In case of $Y=\{\sum_{i=1}^T x_i\bm{u}_i\mid 0\le x_i\le 1 \,(i=1,\dots,T)\}$, we see $vol(Y)=\sqrt{\det((\bm{u}_i,\bm{u}_j))}=|\det{C}|$ and $vol(pr(Y))=|\det{U_1}|$.

In case of  $Y=\{\sum_{i=1}^T x_i\bm{f}_i\mid 0\le x_i\le 1 (i=1,\dots,T)\}$, we see
$vol(Y)=1$ and $vol(pr(Y)) = $ the absolute value of the determinant of the left $(T,T)$-submatrix of  $\left(\begin{array}{c}\bm{f}_1\\\vdots\\\bm{f}_T\end{array}\right)$.

As an example, let us see the subspaces $S =\{\bm{x}\in\mathbb{R}^n \mid \sum_{i=1}^n x_i=0\}$,
and $V=\{\bm{x}\in \mathbb{R}^n\mid x_n=0\}$.
For $S$, we take the orthonormal basis $\bm{u}_i=\bm{f}_i:=\frac{1}{\sqrt{(n-i)(n-i+1)}}(0,\dots,0,n-i,-1,\dots,-1)$ 
$(i=1,\dots,n-1)$ given in the remark after Proposition \ref{prop4.6}.
Then, writing  the $(i,j)$-entry of the matrix $U_1$ by $u_{i,j}$, we see that $u_{i,j}=0$ if $i>j$ and 
$u_{i,i}=\sqrt{\frac{n-i}{n-i+1}}$, hence $c_1=\det U_1 =\frac{1}{\sqrt{n}}$.

\section{One-dimensional distribution}\label{sec5}
In this section, we discuss	 the one-dimensional 
equi-distribution of  $r_i/p$ $(1\le i \le n, p \in Spl(f))$ for local roots $r_i$, 
that is whether, for $0\le a<1$
$$
\lim_{X\to\infty}\frac{\sum_{p\in Spl_X(f)}\#\{i\mid r_i/p\le a , 1\le i \le n\}}
{n\cdot\#Spl_X(f)}=a
$$  
is true or not.
We note that if $f(x)$ has no rational integral root, then the number of primes $p\in Spl(f)$ satisfying $r_1<X$ or
$r_n>p-X$ is finite for any number $X$, and
if $f(x)$ has a rational integral root $r_0$, then $r=r_0$ or $r=p+r_0$ is a local root according to
$r_0\ge0$ or $r_0<0$ if $p\in Spl(f)$ is large, hence $r/p$ tends to $0$ or $1$
and   the one-dimensional equi-distribution is false.
So, we exclude such polynomials.
In {\bf5.1}, we treat the case that  a polynomial $f(x)$ 
has no non-trivial linear relation among roots, and in {\bf5.2}  the case of  a  decomposable and 
irreducible polynomial of degree 4.
In {\bf{5.3}}, for a more general polynomial, e.g. an irreducible polynomial of degree larger than $1$,
we show that Conjectures \ref{conj1}, \ref{conj2} imply the  one-dimensional equi-distribution of roots $r_i/p$ not based on hard calculation.

Supposing $f(0)\ne0$, we see easily that for a rational number $a$, there are only finitely many primes 
$p\in Spl(f)$ such that $a-r/p=O(\frac{1}{ph(p)})$ where $r$ is a local root and $h(x)$ is any function satisfying $h(x)\to\infty$ $(x\to\infty)$.

\subsection{Case without non-trivial linear relation among roots}
The aim in this subsection is to prove the following two theorems.
\begin{thm}\label{thm5.1}
Let $f(x)$ be a  polynomial  of degree $n\,(>1)$,
which has  no non-trivial linear relation  among roots.
Then the conjecture \eqref{eq11} implies for $1\le i \le n$, $0\le a <1$
\begin{align*}
&(n-1)!\lim_{X\to\infty}\frac{
\#\{p\in Spl_X(f)\mid r_i/p<a\}}
{\#Spl_X(f)}
\\
=&\sum_{0\le h \le n\atop 1\le l \le n-1}
\sum_{k =i}^n(-1)^{h+k+n}{n\choose k}\sum_{m=1}^{n-1}{k \choose n-h-m+l}
{n-k \choose m-l}M(l-ha)^{n-1},
\end{align*}
where the binomial coefficient $A\choose B$ is supposed to vanish  unless $0\le B \le A$,
and $M(x):=\max(x,0).$
\end{thm}
When $i=1$, a simpler formula is given in Proposition \ref{prop5.1}.

\noindent
{\bf Example}
Denoting $\lim_{X\to\infty}\frac{\#\{p\in Spl_X(f)\mid r_i/p<a\}}{\#Spl_X(f)}$ by $D(a,i;n)$,
we see
\begin{align*}
D(a,1;2)&=
\begin{cases}
2a&  \hspace{8.5mm} \text{ if } 0<a\le 1/2,
\\
1 &   \hspace{8.5mm}  \text{ if }  1/2< a \le 1,
\end{cases}
\\
D(a,2;2)&=
\begin{cases}
0&\text{ if } 0<a\le 1/2,
\\
-1+2a &\text{ if }  1/2< a \le 1,
\end{cases}
\\
D(a,1;3)&=
\begin{cases}
-3a^2+3a&\text{ if } 0<a\le 1/3,
\\
3a^2/2+1/2 &\text{ if }  1/3< a \le 1/2,
\\
-9a^2/2+6a-1&\text{ if }  1/2< a \le 2/3,
\\
1&\text{ if }  2/3< a \le 1,
\end{cases}
\\
D(a,2;3)&=
\begin{cases}
3a^2&\hspace{3.5mm}\text{ if } 0<a\le 1/3,
\\
-6a^2+6a-1 &\hspace{3.5mm}\text{ if }  1/3< a \le 1/2,
\\
6a^2-6a+2&\hspace{3.5mm}\text{ if }  1/2< a \le 2/3,
\\
-3a^2+6a-2&\hspace{3.5mm}\text{ if }  2/3< a \le 1,
\end{cases}
\\
D(a,3;3)&=
\begin{cases}
0&\hspace{-0.7mm}   \text{ if } 0<a\le 1/3,
\\
9a^2/2-3a+1/2 &\hspace{-0.7mm}   \text{ if }  1/3< a \le 1/2,
\\
-3a^2/2+3a-1&\hspace{-0.7mm}   \text{ if }  1/2< a \le 2/3,
\\
3a^2-3a+1&\hspace{-0.7mm}   \text{ if }  2/3< a \le 1.
\end{cases}
\end{align*}

\noindent
Using Theorem \ref{thm5.1}, we show the one-dimensional equi-distribution of $r_i/p$.
\begin{thm}\label{thm5.2}
Let $f(x)$ be a  polynomial of degree $n\,(>1)$ which has no non-trivial linear relation   among roots;
then \eqref{eq11} implies 
\begin{equation}\label{eq47}
\lim_{X\to\infty} \frac{\sum_{p\in Spl_X(f)}\#\{i\mid r_i/p\le a,1\le i \le n\}}{n\cdot\#Spl_X(f)}
=a
\end{equation}
for a real number $a\in[0,1)$.
\end{thm}
In case of $n=2$, the equation \eqref{eq47} is proved by \cite{DFI} and \cite{T}.
\vspace{2mm}

\noindent
The rest of this subsection is devoted to their proof and 
let us start from  the following easy lemma. 
\begin{lem}\label{lem5.1}
For $t,a\in\mathbb{R}$, we have
$$
\int_a^1M(t-w)^mdw =\frac{1}{m+1}\{M(t-a)^{m+1}-M(t-1)^{m+1}\}.
$$
\end{lem}
\proof{
The left-hand side is equal to
\begin{align*}
&\int_a^1\max(t-w,0)^mdw
\\
=&\int_{t-a}^{t-1}\max(W,0)^m(-dW)
\\
=&-\int_{-\infty}^{t-1}\max(W,0)^mdW+\int_{-\infty}^{t-a}\max(W,0)^mdW
\\
=&
-\frac{1}{m+1}M(t-1)^{m+1} +\frac{1}{m+1}M(t-a)^{m+1}.
\end{align*}
\qed
}

\noindent
For the sake of convenience, we recall the proof of  the following \cite{Fe}.
\begin{lem}\label{F}
For a natural number $k$, the volume $U_k(x)$ $(x\in\mathbb{R})$ of a subset of the unit cube $[0,1)^k$ defined by
$\{(x_1,\dots,x_k)\in[0,1)^k\mid x_1+\dots +x_k\le x\}$ is given by
\begin{equation}\label{eq48}
U_k(x)=\frac{1}{k!}\sum_{i=0}^k(-1)^i{k\choose i}M(x-i)^k.
\end{equation}
\end{lem}
\proof{
In case of $k=1$, $U_1(x)=M(x)-M(x-1)$  is easy, hence  the lemma is  true.
Suppose that the lemma is true for $k\ge1$;
then we see that
\begin{align*}
&vol(\{(x_1,\dots,x_{k+1})\in[0,1)^{k+1}\mid x_1+\dots +x_{k+1}\le x\})
\\=\,&
\int_{0}^1 U_k(x-x_{k+1})dx_{k+1}
\\=\,&
\int_0^1\frac{1}{k!}\sum_{i=0}^k(-1)^i{k\choose i}M(x-y-i)^kdy
\\=\,&
\frac{1}{k!}\sum_{i=0}^k(-1)^i{k\choose i}
\frac{1}{k+1}\{M(x-i)^{k+1}-M(x-i-1)^{k+1}\}
\\=\,&
\frac{1}{(k+1)!}\left\{\sum_{i=0}^k(-1)^i{k\choose i}M(x-i)^{k+1}
-\sum_{i=0}^k(-1)^i{k\choose i}M(x-i-1)^{k+1}\right\}
\\=\,&
\frac{1}{(k+1)!}\left\{\sum_{i=0}^k(-1)^i{k\choose i}M(x-i)^{k+1}
-\sum_{i=1}^{k+1}(-1)^{i-1}{k\choose {i-1}}M(x-i)^{k+1}\right\}
\\=\,&
\frac{1}{(k+1)!}\sum_{i=0}^{k+1}(-1)^i{{k+1}\choose i}M(x-i)^{k+1},
\end{align*}
which completes the induction. 
\qed
}
\begin{cor}\label{cor5.1}
Let $n\,(\ge1)$ be  an integer.
Then we have, 
for   a polynomial $P(x)=c_nx^n+\dots+c_0$,
we have
\begin{equation}\label{eq49}
\sum_{k=0}^n (-1)^kP(k){n\choose k}=c_n(-1)^n\,n!.
\end{equation} 
\end{cor}
\proof{
The case of $n=1$ is obvious.
The equation \eqref{eq49} means
\begin{equation*}
\sum_{k=0}^n (-1)^k k^j{n\choose k}=
\begin{cases}
0&\text{ if }j<n,
\\
(-1)^nn!&\text{ if } j = n.
\end{cases}
\end{equation*}
The equation \eqref{eq48} implies that  
$(U_n(x)=)\frac{1}{n!}\sum_{k=0}^n(-1)^k{n\choose k}M(x-k)^n=1$ if $x>n$, 
hence we have the identity $\sum_{k=0}^n(-1)^k{n\choose k}(x-k)^n=n!$.
We find that  the coefficient  of $x^{n-j}$   of the  identity    
is $\sum_{k=0}^n(-1)^k{n\choose k}{n\choose j}(-k)^j$. 
If, hence $j<n$, then we have $\sum_{k=0}^n(-1)^k{n\choose k}k^j=0$,
i.e. \eqref{eq49} for a polynomial $P$ with $\deg P<n$.
If $j=n$, then $\sum_{k=0}^n(-1)^k{n\choose k}(-k)^n=n!$
which implies \eqref{eq49} for $P(x)=x^n$.
\qed
}
\begin{cor}\label{cor5.2}
 For  natural numbers   $k,n$  with $1\le k \le n$ we have
\begin{equation}\label{eq50}
U_n(k)-U_n(k-1)=\frac{1}{n!}\sum_{i=0}^{k-1} (-1)^i{n+1\choose i}(k-i)^n
=\frac{A(n,k)}{n!},
\end{equation}
where $A(n,k)$ denotes the Eulerian number in the introduction.
\end{cor}
\proof{
The first equality follows from
\begin{align*}
&U_n(k)-U_n(k-1)
\\=\,&
\frac{1}{n!}\left\{\sum_{i=0}^n(-1)^i{n\choose i}M(k-i)^n 
-\sum_{i=0}^n(-1)^i{n\choose i}M(k-1-i)^n\right\}
\\=\,&
\frac{1}{n!}\left\{k^n+\sum_{i=0}^{n-1}(-1)^{i+1}{n\choose {i+1}}M(k-i-1)^n 
-\sum_{i=0}^{n-1}(-1)^i{n\choose {i}}M(k-1-i)^n\right\}
\\=\,&
\frac{1}{n!}\left\{k^n+\sum_{i=0}^{n-1}(-1)^{i+1}{{n+1}\choose{ i+1}}M(k-1-i)^n \right\} 
\\=\,&
\frac{1}{n!}\left\{k^n+\sum_{i=1}^{n}(-1)^{i}{{n+1}\choose{ i}}M(k-i)^n \right\} 
\\=\,&
\frac{1}{n!}\sum_{i=0}^{k-1}(-1)^{i}{{n+1}\choose{ i}}(k-i)^n  .
\end{align*}
Putting
$B(n,k):=n!(U_n(k)-U_n(k-1)   )$, we have only to confirm
$$
B(1,1)=1,\,B(n,k)=(n-k+1)B(n-1,k-1)+kB(n-1,k).
$$
The equation $B(1,1)=1$ is clear, and we see that 
\begin{align*}
&(n-k+1)B(n-1,k-1)+kB(n-1,k)
\\=\,&
(n-k+1)\sum_{i=0}^{k-2}(-1)^{i}{{n}\choose{ i}}(k-1-i)^{n-1} 
+k\sum_{i=0}^{k-1}(-1)^{i}{{n}\choose{ i}}(k-i)^{n-1} 
\\=\,&
(n-k+1)\sum_{i=1}^{k-1}(-1)^{i-1}{{n}\choose{ i-1}}(k-i)^{n-1}
+k\sum_{i=0}^{k-1}(-1)^{i}{{n}\choose{ i}}(k-i)^{n-1} 
\\=\,&
\sum_{i=0}^{k-1}(-1)^i(k-i)^{n-1}\left\{(k-n-1){n\choose{i-1}}+k{n\choose i}\right\}
\\=\,&
\sum_{i=0}^{k-1}(-1)^i(k-i)^{n}{n+1\choose{i}}
\\=\,&
B(n,k).
\end{align*}
\qed
}

Let $f(x) $ be a polynomial of degree $n\,(>1)$ without non-trivial linear relation among roots.
Given number $a\in[0,1]$,
putting 
\begin{equation}\label{eq51}
D_{i,a}:=\{ (x_1,\dots,x_n)\in[0,1)^n\mid x_i\le a \},
\end{equation}
 we see that Conjecture \ref{conj2} implies
\begin{align}\nonumber
&\lim_{X\to\infty} \frac{\sum_{p\in Spl_X(f)}\#\{i\mid r_i/p\le a,1\le i \le n\}}{n\cdot\#Spl_X(f)}
\\\nonumber
=&\lim_{X\to\infty} \frac{\sum_{p\in Spl_X(f)}
\#\{i\mid(r_1/p,\dots,r_n/p)\in D_{i,a}\}}
{n\cdot\#Spl_X(f)}
\\\nonumber
=&\sum_{i=1}^n Pr_{D_{i,a}}(f)/n
\\ \label{eq52}
=&\sum_{i=1}^n\displaystyle\frac{ vol(  D_{i,a}\cap{\hat{\mathfrak{D}}_n})} 
{n\cdot vol(\hat{\mathfrak{D}}_n)},
\end{align}
hence to prove  Theorem \ref{thm5.2},
 we have only to show that  the above value is equal to $a$.
Hereafter, a real number $a$ with $0\le a <1$ is fixed.
\begin{lem}\label{lem5.3}
For an integer $k$ with $1\le k \le n$, let
\begin{equation}\label{eq53}
V(k):=vol(\{\bm{x}\in[0,1)^n \mid   
x_1,\dots,x_k\le a<x_{k+1},\dots,x_n,
\sum_{j=1}^nx_j\in\mathbb{Z}
\})\cdot\cos\theta,
\end{equation}
for  the angle $\theta$ of two hyper-planes defined by $x_j=0$ and by $x_1+\dots+x_n=0$
in $\mathbb R^n$. 
Then we have
\begin{equation}
\frac{ vol(  D_{i,a}\cap{\hat{\mathfrak{D}}_n})}{vol( \hat{\mathfrak{D}}_n)}=
\sum_{k=i}^n{n\choose k}V(k).
\end{equation}
\end{lem}
\proof{
It is easy to see
\begin{align*}
&vol(  D_{i,a}\cap{\hat{\mathfrak{D}}_n}) 
\\
=&\sum_{ k=i}^n vol\{\bm{x}\mid 0\le x_1\le\dots\le x_k \le a <x_{k+1}\le \dots\le x_n<1, 
\sum x_j\in\mathbb{Z}\} 
\\
=&
\sum_{ k=i}^n\frac{1}{k!(n-k)!}vol\{\bm{x}\mid 0\le x_1,\dots, x_k \le a <x_{k+1},\dots, x_n<1, 
\sum x_j\in\mathbb{Z}\} 
\\
=&\,
\frac{1}{n!}\sum_{ k=i}^n{n\choose k}vol\{
\bm{x}\mid
 0\le x_1,\dots, x_k \le a <x_{k+1},\dots, x_n<1, 
\sum_j x_j\in\mathbb{Z}
\}
\\
=&\, vol(\hat{\mathfrak{D}}_n)\sum_{k=i}^n{n\choose k}V(k),
\end{align*}
using $vol(\hat{\mathfrak{D}}_n)=\frac{1}{n!\cos\theta}$.
\qed
}

\begin{lem}\label{lem5.4}
For $k=n$, we have
\begin{equation}\label{eq55}
V(n) =\frac{1}{(n-1)!}\sum_{0\le i \le n,\atop1\le k\le n-1}(-1)^i{n\choose i}M(k-ia)^{n-1}.
\end{equation}
\end{lem}
\proof{
It is easy to see that 
\begin{align*}
V(n)
=&\,vol(\{\bm{x}\in \mathbb{R}^n\mid0\le x_1,\dots,x_n\le a,\sum_{i=1}^nx_i\in\mathbb{Z}\})\cdot\cos\theta
\\
=&\sum_{k=1}^{n-1}vol(\{\bm{x}\in\mathbb{R}^{n}\mid 0\le x_1,\dots,x_n\le a,\sum_{i=1}^nx_i=k\})
\cdot\cos\theta
\\
=&\sum_{k=1}^{n-1}vol(\{\bm{x}\in\mathbb{R}^{n-1}\mid 0\le x_1,\dots,x_{n-1}\le a,0\le k-\sum_{i=1}^{n-1}x_i\le a\})
\\
=&\sum_{k=1}^{n-1}vol(\{\bm{x}\in\mathbb{R}^{n-1}\mid 0\le x_1,\dots,x_{n-1}\le a,\sum_{i=1}^{n-1}x_i\le k\})
\\
&\hspace{1cm}-\sum_{k=1}^{n-1}vol(\{\bm{x}\in\mathbb{R}^{n-1}\mid 0\le x_1,\dots,x_{n-1}\le a,\sum_{i=1}^{n-1}x_i\le k-a\}).
\end{align*} 
The volume of the set $\{\bm{x}\in\mathbb{R}^{n-1}\mid 0\le x_1,\dots,x_{n-1}\le a,\sum_{i=1}^{n-1}x_i\le K\}$ is equal to
\begin{align*}
&a^{n-1}vol(\{\bm{x}\in\mathbb{R}^{n-1}\mid 0\le t_1,\dots,t_{n-1}\le 1,\sum_{i=1}^{n-1}t_i\le K/a\})
\\
=\,&a^{n-1}U_{n-1}(K/a)
\\
=\,&
\frac{a^{n-1}}{(n-1)!}\sum_{i=0}^{n-1}(-1)^i{n-1\choose i}M(K/a-i)^{n-1}
\\
=\,&
\frac{1}{(n-1)!}\sum_{i=0}^{n-1}(-1)^i{n-1\choose i}M(K-ia)^{n-1}.
\end{align*}
Therefore we have
\begin{align*}
V(n)
=&\frac{1}{(n-1)!}\sum_{k=1}^{n-1}\sum_{i=0}^{n-1}(-1)^i{n-1\choose i}\{M(k-ia)^{n-1}-M(k-(i+1)a)^{n-1}
\}
\\
=&\frac{1}{(n-1)!}\sum_{k=1}^{n-1}\sum_{i=0}^{n-1}(-1)^i{n-1\choose i}M(k-ia)^{n-1}
\\
+&\frac{1}{(n-1)!}\sum_{k=1}^{n-1}\sum_{i=1}^{n}(-1)^i{n-1\choose i-1}M(k-ia)^{n-1}
\\
=&\frac{1}{(n-1)!}\sum_{k=1}^{n-1}\sum_{i=0}^{n}(-1)^i{n\choose i}M(k-ia)^{n-1}.
\end{align*}
\qed}
\begin{lem}\label{lem5.5}
In case of  $1 \le k \le n-1$, we have
\begin{equation}\label{eq56}
V(k)=\sum_{m=1}^{n-1}(U_{k,n-1}(m-a)-U_{k,n-1}(m-1)),
\end{equation}
where
\begin{equation}\label{eq57}
U_{k,r}(t):=vol\{\bm{x}\in[0,1)^r \mid x_1,\dots,x_k\le a < x_{k+1},\dots,x_r,\sum_{j=1}^r x_j<t\}.
\end{equation}
\end{lem}
\proof{
We see that
\begin{align*}
&V(k)
\\
=\,&\sum_{m=1}^{n-1}vol\{\bm{x}\in[0,1)^n \mid    
x_1,\dots,x_k\le a<x_{k+1},\dots,x_n,
\sum_{j=1}^nx_j=m
\}\cdot\cos\theta
\\
=\,&
\sum_{m=1}^{n-1}vol\{\bm{x}\in[0,1)^{n-1} \mid    
\begin{array}{c}
x_1,\dots,x_k\le a<x_{k+1},\dots,x_{n-1},
\\
a<m-\sum_{j=1}^{n-1}x_j<1
\end{array}
\}
\\
=\,&\sum_{m=1}^{n-1}(U_{k,n-1}(m-a)-U_{k,n-1}(m-1)).
\end{align*} 
\qed
}
\begin{lem}\label{lem5.6}
For integers $r,k$ with $1\le k \le r$, we have
\begin{equation*}
U_{k,r+1}(t)=\int_a^1U_{k,r}(t-w)dw.
\end{equation*}
\end{lem}
\proof{
This follows from the equation
\begin{align*}
U_{k,r+1}(t)=\int_{x_{r+1}=a}^1 (\int_{D}dx_1,\dots dx_r)dx_{r+1}
\end{align*}
where the domain $D$ is given by the conditions
$$ 0\le x_1,\dots,x_k\le a<x_{k+1},\dots,x_r,\,\,\sum_{j=1}^r x_j< t-x_{r+1}.$$
\qed
}
\begin{lem}\label{lem5.7}
For integers $j,k$ with $j\ge0,k\ge1$, $U_{k,k+j}(t)$ is equal to
\begin{equation}\label{eq58}
\frac{1}{(k+j)!}\sum_{i=0}^k(-1)^i{k\choose i}
\sum_{h=0}^j (-1)^{j+h}{j\choose h} M(t+h-j-(i+h)a)^{k+j}.
\end{equation}
\end{lem}
\proof{
Suppose that $j=0$; then $U_{k,k}(t)$ equals
\begin{align*}
&vol\{\bm{x}\in[0,1)^k \mid x_1,\dots,x_k\le a ,\sum_{j=1}^k x_j<t\}
\\
=\,&
a^kU_k(t/a)
\\
=\,&
\frac{1}{k!}\sum_{i=0}^k (-1)^i{k\choose i}M(t-ia)^k.
\end{align*}
Second, suppose that the equation \eqref{eq58} is true; then we see that
$U_{k,k+j+1}(t)$ equals
\begin{align*}
&
\int_a^1U_{k,k+j}(t-w)dw
\\
=&\frac{1}{(k+j)!}\sum_{i=0}^k(-1)^i{k\choose i}
\sum_{h=0}^j (-1)^{j+h}{j\choose h} \int_a^1M(t+h-j-(i+h)a  -w)^{k+j}dw
\\
=&\frac{1}{(k+j)!}\sum_{i=0}^k(-1)^i{k\choose i}
\sum_{h=0}^j (-1)^{j+h}{j\choose h}\times
\\
&\hspace{2cm}
\frac{1}{k+j+1}\Bigl\{M( t+h-j-(i+h+1)a)^{k+j+1} 
\\
&\hspace{5cm}
-  M( t+h-j-1-(i+h)a)^{k+j+1} \Bigr\}
\\
=&\frac{1}{(k+j+1)!}\sum_{i=0}^k(-1)^i{k\choose i}\times
\\
&\left\{\sum_{h=1}^{j+1}(-1)^{j+h+1}{j\choose h-1}M(t+h-j-1-(i+h)a)^{k+j+1}\right.
\\
&\left.
\hspace{3cm}-\sum_{h=0}^{j}(-1)^{j+h}{j\choose h}M(t+h-j-1-(i+h)a)^{k+j+1}
\right\}
\\
=&\frac{1}{(k+j+1)!}\sum_{i=0}^k(-1)^i{k\choose i}\times
\\
&
\left\{\sum_{h=1}^{j}(-1)^{j+h+1}\left({j\choose h-1}+{j\choose h}\right)M(t+h-j-1-(i+h)a)^{k+j+1}
\right.
\\
&  \hspace{2cm}
+M(t-(i+j+1)a)^{k+j+1} - (-1)^j  M(t-j-1-ia)^{k+j+1}\biggr\} 
\\
=&\frac{1}{(k+j+1)!}\sum_{i=0}^k(-1)^i{k\choose i}\times
\\
&\left\{\sum_{h=1}^{j}(-1)^{j+h+1}{j+1\choose h}M(t+h-j-1-(i+h)a)^{k+j+1}
\right.
\\
&
\hspace{2cm}+M(t-(i+j+1)a)^{k+j+1} - (-1)^j  M(t-j-1-ia)^{k+j+1}\biggr\} ,
\end{align*}
which completes the induction.
\qed
}
\begin{lem}\label{lem5.8}
For $1\le k \le n-1$, we have
\begin{align}\nonumber
&(n-1)!
V(k)
\\ \nonumber
=&
\sum_{0\le h\le n\atop1\le l\le n-1}
(-1)^{n+k+h}  \sum_{m=1}^{n-1}    {k\choose n-h-m+l}{n-k \choose m-l}M(l-ha)^{n-1}.
\end{align}
\end{lem}
\proof{
For $1\le k , m\le n-1$,  Lemma \ref{lem5.7} implies
\begin{align*}
&(n-1)!\{U_{k,n-1}(m-a)-U_{k,n-1}(m-1)\}
\\
=&
\sum_{i,h\in\mathbb Z}(-1)^i{k\choose i}(-1)^{n-1-k+h}{n-1-k\choose h} 
\\& \hspace{1cm}\times \{
M(m-a+h-(n-1-k)-(i+h)a)^{n-1}
\\
&\hspace{1.5cm}-
M(m-1+h-(n-1-k)-(i+h)a)^{n-1}\},
\end{align*}
noting the supposition that the binomial coefficient $A\choose B$ vanishes unless $0\le B\le A$.
Continuing the calculation, it is 
\begin{align*}
&\sum_{r,s\in\mathbb Z}(-1)^{n+k+s}{k\choose n+r-s-m}
\!\!\left\{\!{n-1-k\choose m-r}\!+\!{n-1-k\choose m-r-1}\right\} 
\!M(r-sa)^{n-1}
\\
&=\!\!\sum_{r,s\in\mathbb{Z}}(-1)^{n+k+s}{k\choose n+r-s-m}
{n-k\choose m-r}M(r-sa)^{n-1},
\end{align*}
hence we have
\begin{align*}
&(n-1)!
V(k)
\\ 
=&
\sum_{r,s\in\mathbb{Z}}
(-1)^{n+k+s}  \sum_{m=1}^{n-1}    {k\choose n-s-m+r}{n-k \choose m-r}M(r-sa)^{n-1}
\\ 
=&
\sum_{0\le s\le n\atop1\le r\le n-1}
(-1)^{n+k+s}  \sum_{m=1}^{n-1}    {k\choose n-s-m+r}{n-k \choose m-r}M(r-sa)^{n-1}
\end{align*}
where the restrictions on $r,s$ follow from conditions $1\le k,m\le n-1,0\le n+r-s-m\le k,0\le m-r\le n-k$.
Lemma \ref{lem5.5} completes the proof.
\qed
}
\begin{lem}\label{lem5.9}
Let $m,n$ be  integers satisfying $0\le m\le n-1$.
Then we have
\begin{equation}\label{eq59}
\sum_{k=0}^m (-1)^k{n\choose k}=(-1)^m{n-1\choose m}.
\end{equation} 
\end{lem}
This is well-known and we omit the proof.
\begin{prop}\label{prop5.1}
For an integer $i$ with $1\le i \le n$ and a real number $a\in[0,1)$, we have
\begin{align*}
&(n-1)!\,vol(D_{i,a}\cap\hat{\mathfrak D}_n)/vol({\hat{\mathfrak D}}_n)
\\
=
&\sum_{0\le h \le n\atop 1\le l \le n-1}
\left\{
\sum_{k =i}^n(-1)^{h+k+n}{n\choose k}
\left[{n\choose h}-\sum_{h\le q\le \max(l,h-1)}{k\choose q-h}{n-k\choose n-q}\right]
\right\}
\\
&\hspace{8cm}\times
M(l-ha)^{n-1}.
\end{align*}
In particular, we have for $i=1$
\begin{equation}\label{eq60}
vol(D_{1,a}\cap\hat{\mathfrak D}_n)/vol({\hat{\mathfrak D}}_n)
=\sum_{0\le h \le n,\atop1\le l\le n-1}C_1(l,h)M(l-ha)^{n-1},
\end{equation}
where
$$
C_1(l,h)=\frac{1}{(n-1)!}
\begin{cases}
(-1)^{n+h+1}{n\choose h}&\text{if }h\ge l+1,
\\
0&\text{if }1\le h \le l,
\\
(-1)^{n+l+1}{n-1\choose l}&\text{if } h=0.
\end{cases}
$$
\end{prop}
\proof{
By Lemmas \ref{lem5.3}, \ref{lem5.4}, \ref{lem5.8}, we have
\begin{align*}
&(n-1)!\,vol(D_{i,a}\cap\hat{\mathfrak D}_n)/vol({\hat{\mathfrak D}}_n)
\\
=&(n-1)!
\sum_{k=i}^n{n\choose k}V(k)
\\
=&(n-1)!\,V(n)
+(n-1)!\sum_{k=i}^{n-1}{n\choose k}V(k)
\\
=&
\sum_{0\le h\le n\atop1\le l\le n-1}(-1)^h{n\choose h}M(l-ha)^{n-1}
\\
+&
\sum_{k=i}^{n-1}{n\choose k}
\sum_{0\le h\le n\atop1\le l\le n-1}(-1)^{n+k+h}   \sum_{m=1}^{n-1}   {k\choose n-h-m+l}{n-k \choose m-l}M(l-ha)^{n-1}
\\
=&
\sum_{k=i}^{n}{n\choose k}
\sum_{0\le h\le n\atop1\le l\le n-1}(-1)^{n+k+h}  \sum_{m=1}^{n-1} {k\choose n-h-m+l}{n-k \choose m-l}M(l-ha)^{n-1},
\end{align*}
since the binomial coefficient ${0\choose m-l}$ vanishes unless $m=l$.
The partial sum $ \sum_{m=1}^{n-1} {k\choose n-h-m+l}{n-k \choose m-l}$ is equal to
\begin{align}\nonumber
&\sum_{1-l\le q\le n-1-l}{k\choose n-h-q}{n-k\choose q}
\\\nonumber
=\,&\sum_{0\le q\le \min(n-1-l,n-h)}{k\choose n-h-q}{n-k\choose q}
\\\nonumber
=\,&{n\choose n-h}-\sum_{\min(n-1-l,n-h)+1\le q \le n-h}{k\choose n-h-q}{n-k\choose q}
\\\label{eq61}
=\,&{n\choose h}-\sum_{h\le q\le \max(l,h-1)}{k\choose q-h}{n-k\choose n-q},
\end{align}
where we use for the second equality, the formula $\sum_{r=0}^p{n\choose r}{m\choose p-r}={n+m\choose p}$ if $p\le n+m$.

Next, let us assume  $i=1$ to show \eqref{eq60}.
Putting
\begin{align*}
T(l,h) := &\sum_{k =1}^n(-1)^{h+k+n}{n\choose k}
\left({n\choose h}-\sum_{h\le q\le \max(l,h-1)}{k\choose q-h}{n-k\choose n-q}\right)
\\
=&-(-1)^{h+n}{n\choose h}- \sum_{k =1}^n(-1)^{h+k+n}{n\choose k}\sum_{q=h}^{\max(l,h-1)}{k\choose q-h}{n-k\choose n-q},
\end{align*}
we have only to prove $T(l,h)=(n-1)!\,C_1(l,h)$.
It is obviously true if $h\ge l+1$,
since the partial sum on $q$ is empty. 
In case of $h=0$,
we see that $T(l,0)$ is equal to 
\begin{align*}
&-(-1)^{n}- \sum_{k =1}^n(-1)^{k+n}{n\choose k}\sum_{q=0}^{l}{k\choose q}{n-k\choose n-q}
\\
=&
-(-1)^{n}- \sum_{k =1}^n(-1)^{k+n}{n\choose k}\sum_{q=0}^{l}\delta_{k,q}
\\
=&
-(-1)^{n}- \sum_{k =1}^l (-1)^{k+n}{n\choose k}
\\
=&
- \sum_{k =0}^l (-1)^{k+n}{n\choose k}
\\
=&-(-1)^{n+l}{n-1\choose l},
\end{align*}
by Lemma \ref{lem5.9}.
Lastly assume that $1\le h \le l$.
The sum $T(l,h)+  (-1)^{h+n}{n\choose h}$  is equal to
\begin{align*}
&- \sum_{k =1}^n(-1)^{h+k+n}{n\choose k}\sum_{q=h}^l {k\choose q-h}{n-k\choose n-q}
\\
=\,&- \sum_{q =h}^l (-1)^{h+n}{n\choose h} {n-h\choose q-h}\sum_{k=1}^n  (-1)^k{h\choose q-k}
\\
=\,&- \sum_{q =h}^l (-1)^{h+n}{n\choose h} {n-h\choose q-h}(-1)^q\sum_{K=0}^{q-1}  (-1)^K{h\choose K}
\\
=\,&- \sum_{q =h}^l (-1)^{h+n}{n\choose h} {n-h\choose q-h}(-1)^q(-1)^{q-1}{h-1\choose q-1}
\\
=\,& \sum_{q =h}^l (-1)^{h+n}{n\choose h} {n-h\choose q-h}{h-1\choose q-1}
\\
=\,& (-1)^{h+n}{n\choose h} ,
\end{align*}
which implies $T(l,h)=0$.
\qed
}

The proposition gives Theorem \ref{thm5.1} by \eqref{eq61},
and we see that   \eqref{eq52} is the sum of $C(l,h)M(l-ha)^{n-1}$ over integers $l,h$ satisfying
\begin{equation}\label{eq62}
1\le l \le n-1, 0\le h \le n,
\end{equation}
where
\begin{align}\nonumber
C(l,h):=&\frac{1}{n!}
\sum_{1\le i \le k \le n}(-1)^{h+k+n}{n\choose k}
\!\!\left( \!{n\choose h}-\sum_{h\le q\le \max(l,h-1)}{k\choose q-h}{n-k\choose n-q}\!\right)
\\ \nonumber
=&\frac{1}{n!}\sum_{0\le k \le n}(-1)^{h+k+n}\,k\,{n\choose k}
\!\!\left(\!{n\choose h}-\sum_{h\le q\le \max(l,h-1)}{k\choose q-h}{n-k\choose n-q}\!\right)
\\ \label{eq63}
=&\frac{-1}{n!}
\sum_{0\le k \le n}(-1)^{h+k+n}\,k\,{n\choose k}
\sum_{h\le q\le \max(l,h-1)}{k\choose q-h}{n-k\choose n-q},
\end{align}
using Corollary \ref{cor5.1} with $P(x)=x$ and $n>1$.
To prove \eqref{eq52} being to equal to $a$,
we will show
\begin{equation}\label{eq64}
C(l,h)=
\begin{cases}
\displaystyle\frac{-(-1)^{n-l}}{(n-1)!}{n-2 \choose l-1} &\text{ if } h=0,
\\[4mm]
\displaystyle\frac{(-1)^{n-l}}{(n-1)!}{n-2 \choose l-1} &\text{ if } h=1,
\\[4mm]
0&\text{ if }h\ge2.
\end{cases}
\end{equation}
Under the equations \eqref{eq64}, \eqref{eq52} being $a$, i.e. Theorem \ref{thm5.1} is proved as follows:
The expression \eqref{eq52} is equal to
\begin{align*}
&\sum_{l=1}^{n-1} \frac{-(-1)^{n-l}}{(n-1)!}{n-2\choose l-1}M(l)^{n-1}
+
\sum_{l=1}^{n-1} \frac{(-1)^{n-l}}{(n-1)!}{n-2\choose l-1}M(l-a)^{n-1}
\\
=&
\sum_{l=1}^{n-1} \frac{-(-1)^{n-l}}{(n-1)!}{n-2\choose l-1}(l^{n-1}-(l-a)^{n-1})
\\
=&
\sum_{l=0}^{n-2} \frac{(-1)^{n+l}}{(n-1)!}{n-2\choose l}((n-1)al^{n-2}+O(l^{n-3}))
\\
=&\hspace{1mm} a,
\end{align*}
by Corollary \ref{cor5.1}.

Let us show \eqref{eq64}. Suppose $h=0$; 
 we see that
\begin{align*}
C(l,0)
=&
\frac{-1}{n!}
\sum_{0\le k \le n}(-1)^{k+n}\,k\,{n\choose k}
\sum_{0\le q\le l}{k\choose q}{n-k\choose n-q}
\\
=&
\frac{-1}{n!}
\sum_{0\le k \le n}(-1)^{k+n}\,k\,{n\choose k}
\sum_{0\le q\le l}\delta_{k,q}
\\
=&
\frac{-1}{n!}
\sum_{0\le k \le l}(-1)^{k+n}\,k\,{n\choose k}
\\
=&
\frac{-1}{(n-1)!}
\sum_{0\le k \le l}(-1)^{k+n}{n-1\choose k-1}
\\
=&
\frac{-1}{(n-1)!}
\sum_{0\le k \le l-1}(-1)^{k+n+1}{n-1\choose k}
\\
=&
\frac{(-1)^{n+l+1}}{(n-1)!}
{n-2\choose l-1},
\end{align*}
which is \eqref{eq64}.

Second we see that   
\begin{align*}
C(l,1)=
&\frac{-1}{n!}
\sum_{0\le k \le n}(-1)^{1+k+n}\,k\,{n\choose k}
\sum_{1\le q\le l}{k\choose q-1}{n-k\choose n-q}.
\end{align*}
Unless $q-1\le k$ and $n-q\le n-k$, binomial coefficients vanish,
hence we may assume that $q=k$ or $q=k+1$, and we see
\begin{align*}
C(l,1)
=&
\frac{-1}{n!}
\sum_{0\le k \le n}(-1)^{1+k+n}\,k\,{n\choose k}
\sum_{1\le q\le l}(k\delta_{q,k}+(n-k)\delta_{q,k+1})
\\
=&\frac{-1}{n!}
\sum_{0\le k \le l}(-1)^{1+k+n}\,k^2\,{n\choose k}
+\frac{-1}{n!}
\sum_{0\le k \le l-1}(-1)^{1+k+n}\,k(n-k){n\choose k}
\\
=&\frac{(-1)^n}{(n-1)!}
\sum_{0\le k \le l-1}(-1)^{k}k{n\choose k}
+\frac{1}{n!}(-1)^{l+n} l^2{n \choose  l}
\\
=&\frac{(-1)^n n}{(n-1)!}
\sum_{0\le k \le l-1}(-1)^{k}{n-1\choose k-1}
+\frac{1}{n!}(-1)^{l+n} l^2{n \choose  l}
\\
=&\frac{(-1)^{n+1} n}{(n-1)!}
(-1)^{l-2}{n-2\choose l-2}
+\frac{1}{n!}(-1)^{l+n} l^2{n \choose  l}
\\
=&\frac{(-1)^{n+l}}{(n-1)!}{n-2 \choose l-1}.
\end{align*}
Finally, assume that $h\ge2$; hence $1\le l \le n-1,2\le h \le n$ are supposed.
By \eqref{eq63}, we have
\begin{align*}
&-n!C(l,h)
\\
=&\sum_{ h\le 	q \le\max(l,h-1)}(-1)^{h+n}{n\choose h}{n-h\choose n-q}
\sum_{0\le k \le  n}(-1)^kk{h\choose q-k}
\\
=&\sum_{ h\le 	q \le\max(l,h-1)}(-1)^{h+n}{n\choose h}{n-h\choose n-q}
\sum_{0\le K \le  q}(-1)^{q+K}(q-K){h\choose K}
\\
=&\,0,
\end{align*}
since
\begin{align*}
\sum_{0\le K \le  q}(-1)^{q+K}(q-K){h\choose K}
=(-1)^q\sum_{0\le K \le  h}(-1)^{K}(q-K){h\choose K}=0,
\end{align*}
using Corollary \ref{cor5.1} to $P(x)=q-x$ with $h\ge2$.
Thus we have completed the proof of Theorem \ref{thm5.2}.
\subsection{Case of degree $4$ with non-trivial linear relation among roots}
Even if an irreducible polynomial has non-trivial linear relations among roots, the one-dimensional equi-distribution of  $r_i/p$ for a local root is likely.
The following is an example.
\begin{thm}\label{thm5.3}
For an irreducible polynomial $f(x)$ of degree $4$ with non-trivial linear relation among roots,
Conjecture \ref{conj2} implies  the equi-distribution of $r_i/p$ $(1\le i \le 4, p\in Spl(f))$. 
\end{thm}
\proof{
Let an irreducible polynomial $f(x)$ of degree $4$ with non-trivial linear relation among roots be
$f(x) = (x^2+ax)^2+b(x^2+ax)+c$ by Proposition \ref{prop3.2}; 
then by defining numbers $\beta_i,\alpha_{i,j} $ by
$$
x^2+bx+c=(x-\beta_1)(x-\beta_2),\, x^2+ax-\beta_i=(x-\alpha_{i,1})(x-\alpha_{i,2})
$$
as Proposition \ref{prop3.3}, a basis of linear relations among roots is $\alpha_{i,1}+\alpha_{i,2}=-a$
$(i=1,2)$.  Letting $\alpha_1:=\alpha_{1,1},\alpha_2:=\alpha_{2,1},\alpha_3:=\alpha_{2,2},\alpha_4:=\alpha_{1,2}$, we have $\alpha_1+\alpha_4=\alpha_2+\alpha_3=-a$,
hence $t=2$ and $\hat{\bm{m}}_1=(1,0,0,1,-a),\hat{\bm{m}}_2=(0,1,1,0,-a)$. 
For local solutions $r_i$ for a large prime $p\in Spl(f)$, induced linear relations in \eqref{eq7} are
$r_1+r_4=r_2+r_3=-a+p$, i.e.  $k_1=k_2=1$, and $Spl_X(f,\sigma)\ne\emptyset$ if and only if 
$\{\sigma(1),\sigma(4)\}=\{1,4\}$ or $\{2,3\}$.
Thus, we see $\#\hat{G}=8,[\mathbb{Q}(f):\mathbb{Q}]=4,8$.
For such $\sigma,k_j$, we have $Spl_X(f)=Spl_X(f,\sigma)$ with possibly finitely many exceptional
primes.
Only for such permutations, $vol(\mathfrak{D}(f,\sigma))>0$ and $\mathfrak{D}(f,\sigma)=\mathfrak{D}(f,id)$ are easy,
hence we have, by \eqref{eq11}
\begin{equation}\label{eq66}
Pr_D(f,id) = \frac{vol(D\cap\mathfrak{D}(f,id))}{vol(\mathfrak D(f,id))}.
\end{equation} 
Hence we have only to show
\begin{equation}\label{eq67}
\sum_{i=1}^4\frac{1}{4}\frac{vol(D_i)}{vol(\mathfrak{D}(f,id))}=A\quad(0\le A<1),
\end{equation}
by putting $D_i:=\{(x_1,\dots,x_4)\mid x_i\le A\}\cap \,\mathfrak{D}(f,id)$.
By $\mathfrak{D}(f,id)=\{(x_1,\dots,x_4)\mid 0\le x_1\le\dots\le x_4<1, x_1+x_4=1,x_2+x_3=1\}$
we have
\begin{align*}
\mathfrak{D}(f,id)&=\{(x_1,x_2,1-x_2,1-x_1)\mid 0\le x_1\le x_2<1/2\},
\\
D_1&=\{(x_1,x_2,1-x_2,1-x_1)\mid 0\le x_1\le x_2<1/2,x_1\le A\},
\\
D_2&=\{(x_1,x_2,1-x_2,1-x_1)\mid 0\le x_1\le x_2<\min(1/2, A)\},
\\
D_3&=\{(x_1,x_2,1-x_2,1-x_1)\mid 0\le x_1\le x_2<1/2,1-x_2\le A\},
\\
D_4&=\{(x_1,x_2,1-x_2,1-x_1)\mid 0\le x_1\le x_2<1/2,1-x_1\le A\},
\end{align*}
and projecting them on the $(x_1,x_2)$-plane, we see
\begin{align*}
vol(pr(\mathfrak{D}(f,id)))&=1/8,
\\
vol(pr(D_1))&=
\begin{cases}
A/2-A^2/2&\text{if }A\le1/2,
\\
1/8&\text{if }A\ge1/2,
\end{cases}
\\
vol(pr(D_2))&=
\begin{cases}
A^2/2&\text{if }A\le1/2,
\\
1/8&\text{if }A\ge1/2,
\end{cases}
\\
vol(pr(D_3))&=
\begin{cases}
0&\text{if }A\le1/2,
\\
(A-1/2)/2-(A-1/2)^2/2&\text{if }A\ge1/2,
\end{cases}
\\
vol(pr(D_4))&=
\begin{cases}
0&\text{if }A\le1/2,
\\
(A-1/2)^2/2&\text{if }A\ge1/2.
\end{cases}
\end{align*}
Thus we see that  \eqref{eq67} is true.
\qed
}

\subsection{More general case}\label{subsec5.3}

In this subsection, assume Conjectures \ref{conj1}, \ref{conj2} and that 
$f(x)$ has no rational root and $\alpha_i-\alpha_j\not\in\mathbb{Q}$ if $i\ne j$, and
the conditions $Pr(f,\sigma)>0,vol(\mathfrak{D}(f,\sigma))>0$
are equivalent.
Using Corollary \ref{cor4.2}, we have a similar expression to \eqref{eq52}
with the definition \eqref{eq51} of $D_{i,a}$ for
a more general polynomial $f(x)$:
\begin{align*}
&\lim_{X\to\infty}\frac{\sum_{p\in Spl_X(f)}\#\{i\mid r_i/p\le a, 1\le i \le n\}}
{n\#Spl_X(f)}
\\=&\,
\lim_{X\to\infty}\frac{\sum_{p\in Spl_X(f)}\#\{i\mid (r_1/p,\dots,r_n/p)\in D_{i,a}\}}
{n\#Spl_X(f)}
\\=&\,
\lim_{X\to\infty}\sum_{i=1}^n
\frac{\#\{p\in Spl_X(f)\mid (r_1/p,\dots,r_n/p)\in D_{i,a}\}}
{n\cdot\#Spl_X(f)}
\\
=&\,
\lim_{X\to\infty}\sum_{i=1}^n
\frac{\sum_{\sigma\in S_n}\#\{p\in Spl_X(f,\sigma)\mid (r_1/p,\dots,r_n/p)\in D_{i,a}\}}
     {n\cdot\#\hat{{\bm G}}\cdot\#Spl_X(f)}
\\
=&\,
\lim_{X\to\infty}\sum_{i=1}^n
\frac{\sum_{\sigma\in S'_n}\#\{p\in Spl_X(f,\sigma)\mid (r_1/p,\dots,r_n/p)\in D_{i,a}\}}
     {\#Spl_X(f,\sigma)}\times
\\     
&\hspace{5cm}     \frac{\#Spl_X(f,\sigma)}
{n\cdot\#\hat{{\bm G}}\cdot\#Spl_X(f)}
\\
\intertext{where  $S_n'$ is a subset of $S_n$ consisting of permutations satisfying  $Pr(f,\sigma)>0$}
=&\,
\frac{1}{n\#\hat{{\bm G}}}\sum_{i=1}^n\sum_{\sigma\in S'_n}
Pr_{D_{i,a}}(f,\sigma)Pr(f,\sigma)
\\
=&\,
\frac{1}{nc\#{\hat{\bm G}} }
\sum_{i=1}^n\sum_{\sigma\in S'_n}
Pr_{D_{i,a}}(f,\sigma)
{{vol}({\mathfrak{D}}(f,\sigma))}\quad\text{ (by Conjecture \ref{conj1})}
\\\nonumber
=&\,
\frac{1}{nc\#{\hat{\bm G}}}\sum_{i=1}^n\sum_{\sigma\in S'_n}
{{vol}(D_{i,a} \cap {\mathfrak{D}}(f,\sigma))}\quad\text{ (by Conjecture \ref{conj2})}.
\end{align*}
We note that the dimension of the set
\begin{align*}
&D_{i,0} \cap {\mathfrak{D}}(f,\sigma)
\\
=&\{(x_1,\dots,x_n)\mid0\le x_1\le\dots\le x_n<1,x_i=0,\sum_{k=1}^n m_{j,k}x_{\sigma(k)}\in\mathbb{Z}
\,({}^\forall j)\}
\end{align*}
is less than $n-t$, 
since  the condition $x_i=0$ is independent of conditions $\sum_{k=1}^n m_{j,k}x_{\sigma(k)}\in\mathbb{Z}\,({}^\forall j)$ by the assumption that $f(x)$ has no rational root.
Therefore the above is equal to $0$ at $a=0$, and to $1$ at $a=1$ by
$$
c = \frac{\sum_{Pr(f,\sigma)>0} vol(\mathfrak{D}(f,\sigma))}{\sum Pr(f,\sigma)}
=\sum_{Pr(f,\sigma)>0} vol(\mathfrak{D}(f,\sigma))/\#\hat{\bm{G}}.
$$
Hence,
once we have shown that it is a linear form in $a$,  it is equal to $a$ on $[0,1)$,  that is the distribution of $r_i/p$ is uniform under the above conditions. 
Till here, we do not use the assumption that $vol(\mathfrak{D}(f,\sigma))>0$ implies 
$Pr(f,\sigma)>0$ and the difference between the  sum restricted on $S_n'$ above and the full sum
\begin{equation}\label{eq68}
\sum_{i=1}^n\sum_{\sigma\in S_n}
{{vol}(D_{i,a} \cap {\mathfrak{D}}(f,\sigma))}
\end{equation}
is
\begin{equation}\label{eq69}
\sum_{i=1}^n\sum_{Pr(f,\sigma)=0,vol(\mathfrak{D}(f,\sigma))>0}
{{vol}(D_{i,a} \cap {\mathfrak{D}}(f,\sigma))}.
\end{equation}

From now on, the aim  is to show that \eqref{eq68}
is a linear form in $a$,
which means that  the distribution of $r_i/p$ is
uniform.
If, moreover
\eqref{eq65}
is also a linear form in $a$, then the distribution of $r_i/p$ is
uniform 
without the assumption that ${vol}(D_{i,a} \cap {\mathfrak{D}}(f,\sigma))>0$ implies $Pr(f,\sigma)>0$.
(cf. {\bf {6.4.2}}).

Now we assume that a polynomial $f(x)$ has no rational roots, and
put
\begin{align*}
  \beta_i&:=\alpha_i - tr_{\mathbb{Q}(f)/\mathbb{Q}}(\alpha_i)/[\mathbb{Q}(f):\mathbb{Q}]\quad(\ne0),
  \\
LR_0&=\{(l_1,\dots,l_n)\in\mathbb{Q}^n \mid\sum_{i=1}^nl_i\alpha_i\in\,\mathbb{Q}\}
  \\
  &\hspace{1mm}=\{(l_1,\dots,l_n)\in\mathbb{Q}^n \mid\sum_{i=1}^nl_i\beta_i=0\},
  \\
  \mathfrak{D}_f&:=\{\bm{x}\in\mathbb{R}^n\mid (\bm{x},\bm{l})\in\mathbb{Z}
  \text{ for }{}^\forall \bm{l}\in LR_0\cap \mathbb{Z}^n\}.
\end{align*}
Recall that vectors $\bm{m}_1,\dots,\bm{m}_t$ are a basis of $LR_0\cap\mathbb{Z}^n$,
hence $ \mathfrak{D}_f$ is $G$-stable.
\begin{lem}\label{lem5.10}
Suppose that $\alpha_i - \alpha_j\not\in\mathbb{Q}$ for any distinct $i,j$, moreover.
  We have
  \begin{equation}
\sum_{i=1}^n\sum_{\sigma\in S_n}
{{vol}(D_{i,a} \cap {\mathfrak{D}}(f,\sigma))}=
\sum_{i=1}^n{vol}(\{\bm{x}\in[0,1)^n\mid x_i\le a\} \cap {\mathfrak{D}}_f),
  \end{equation}
in particular, 
   \begin{equation}\label{eq71}
\sum_{\sigma\in S_n}
{{vol}( {\mathfrak{D}}(f,\sigma))}=
{vol}([0,1)^n \cap {\mathfrak{D}}_f). 
  \end{equation} 
\end{lem}  
\proof
{
The equations   
  \begin{align*}
    &D_{i,a}\cap\mathfrak{D}(f,\sigma)
    \\
    &=\{\bm{x}\in[0,1)^n\mid x_1\le\dots\le x_n,x_i\le a,(\bm{l},\sigma^{-1}(\bm{x}))\in\mathbb{Z}\text{ for }{}^\forall\bm{l}\in LR_0\cap\mathbb{Z}^n \}
      \\
      &=\{\bm{x}\in[0,1)^n\mid x_1\le\dots\le x_n,x_i\le a,\sigma^{-1}(\bm{x})\in\mathfrak{D}_f\}
        \\
       & =\sigma(\{\bm{y}\in[0,1)^n\mid y_{\sigma^{-1}(1)}\le\dots\le y_{\sigma^{-1}(n)},  y_{\sigma^{-1}(i)}\le a\}\cap \mathfrak{D}_f  )
    \end{align*}
imply
\begin{align*}
&vol(D_{i,a}\cap\mathfrak{D}(f,\sigma))
  \\
  =\,&vol(\{\bm{y}\in\mathbb{R}^n\mid 0<y_{\sigma^{-1}(1)}<\dots< y_{\sigma^{-1}(n)}<1,
  y_{\sigma^{-1}(i)}\le a\}\cap \mathfrak{D}_f  ).
\end{align*}
Here, we note that $\mathfrak{D}_f$ is defined by $t$ linearly independent vectors
$\bm{m}_j$ and the volume $vol$  is $(n-t)$-dimensional one, hence it is not necessary to care
an additional equation $x_i = x_j$ for $i\ne j$ by the assumption that the vector space $LR_0$ spanned by
$\bm{m}_j$ does not contain vectors $(0,\dots,0,1,0,\dots,0,-1,0,\dots,0)$.
Let us define a mapping $\phi$ from
$$
X(i):=\{(\bm{x},i)\mid \bm{x}\in(0,1)^n,x_i\le a,x_j\ne x_l\text{ if }j\ne l\}\quad(1\le i\le n)
$$
to the union of
$$
Y(\sigma,k):=\{(\bm{y},\sigma,k)\mid\bm{y}\in(0,1)^n,
0<y_{\sigma(1)}<\dots< y_{\sigma(n)}<1,
  y_{\sigma(k)}\le a\}
$$
  by $\phi((\bm{x},i))=(\bm{x},\sigma,k)$, where $\sigma,k$ are defined by
  $$
x_{\sigma(1)}<\dots<x_{\sigma(n)},\,
  k=\sigma^{-1}(i),
  $$
  hence $\phi(X(i))=\cup_{\sigma,k\text{:}\sigma(k)=i}Y(\sigma,k)$.
  The mapping is clearly bijective, 
  and if $(\bm{y},\sigma,k)\in Y(\sigma,k)$ and $(\bm{y},\sigma',k')\in Y(\sigma',k')$ with $\sigma(k)=\sigma'(k')$ occur, then $\sigma=\sigma',k=k'$ hold,  
  hence, defining $pr$ by  $pr(\bm{y},\sigma,k)=\bm{y}$, we have
  $pr(Y(\sigma,k))\cap pr(Y(\sigma',k'))=\emptyset$ if $\sigma(k)=\sigma'(k')$ and either $\sigma\ne\sigma'$ or $k\ne k'$, and so
  \begin{align*}
    &\sum_{i=1}^n \sum_\sigma vol(D_{i,a}\cap\mathfrak{D}(f,\sigma))
    \\ =\,
    &\sum_{i=1}^n \sum_\sigma vol(pr(Y(\sigma^{-1},i))\cap\mathfrak{D}_f)
    \\ =\,
    &\sum_{k=1}^n \sum_\sigma vol(pr(Y(\sigma,k))\cap\mathfrak{D}_f)
    \\    
    =\,&\sum_i\sum_{k,\sigma\text{:}\sigma(k)=i} vol(pr(Y(\sigma,k))\cap\mathfrak{D}_f  )
    \\
    =\,&\sum_i vol(\cup_{ k,\sigma\text{:}\sigma(k)=i }pr(Y(\sigma,k))\cap \mathfrak{D}_f  )
     \\
    =\,&\sum_i vol(pr(\cup_{ k,\sigma\text{:}\sigma(k)=i }Y(\sigma,k))\cap \mathfrak{D}_f  )
     \\
    =\,&\,\sum_i vol(pr(X(i))\cap\mathfrak{D}_f)
    \\
    =\,&\,\sum_i vol(\{\bm{x}\in[0,1)^n \mid x_i\le a\}\cap\mathfrak{D}_f).
    \end{align*}
\qed
    }

We note that 
$$
\sigma^{-1}(\mathfrak{D}(f,\sigma))=
\{\bm{y}\in[0,1)^n\mid y_{\sigma^{-1}(1)}\le \dots\le y_{\sigma^{-1}(n)}\}\cap\mathfrak{D}_f,
$$    
and putting
\begin{align*}
{\mathfrak{D}}(f,\sigma)^\circ&:=
\left\{
(x_1\dots,x_n)\in (0,1)^n\left|
\begin{array}{l}
0< x_1<\dots< x_n<1,
\\[2pt]
\sum_{i=1}^n m_{j,i}\,x_{\sigma(i)}\in\mathbb{Z}\:\:(1\le{}^\forall j\le t)
\end{array}
\right\}
\right.,
  \\
  \mathfrak{D}_f^\circ&:=\{\bm{x}\in\mathbb{R}^n\mid 
  x_i\ne x_j\text{ if } i\ne j,(\bm{x},\bm{l})\in\mathbb{Z}
  \text{ for }{}^\forall \bm{l}\in LR_0\cap \mathbb{Z}^n\},
\end{align*}
it is easy to see that the mapping $\phi$ from the disjoint union $\sqcup_\sigma{\mathfrak{D}}(f,\sigma)^\circ$
to $(0,1)^n\cap\mathfrak{D}_f^\circ$ defined by $\phi(\bm{x}) = (x_{\sigma(1)},\dots,x_{\sigma(n)})$
for $\bm{x} \in{\mathfrak{D}}(f,\sigma)^\circ$ is bijective. (cf. Proposition \ref{prop4.4})

To show the equi-distribution, we have only to see that each factor
    ${vol}(\{\bm{x}\in[0,1)^n\mid x_i\le a\} \cap {\mathfrak{D}}_f)$ is a linear form in $a$.

\begin{lem}\label{lem5.11}
  Suppose that $\beta_1,\dots,\beta_r$ are linearly independent over $\mathbb{Q}$ and
  $$
(\beta_{r+1},\dots,\beta_n)=(\beta_1,\dots,\beta_r)T\quad\text{ for }{}^\exists T\in M_{r,n-r}(\mathbb{Q}).
  $$
  Then we have
  $r=n-t$,
      \begin{align*}
        LR_0
        &= \left\{(l_1,\dots,l_n)\in\mathbb{Q}^n \mid
         \left(
  \begin{array}{c}
    {l}_1\\
    \vdots\\
    {l}_r
  \end{array}
  \right)
  =- T
  \left(
  \begin{array}{c}
    {l}_{r+1}\\
    \vdots\\
    {l}_n
  \end{array}
  \right)
\right\}
        \end{align*}

  and
  $$
  \left(
  \begin{array}{c}
    \bm{m}_1\\
    \vdots\\
    \bm{m}_t
  \end{array}
  \right)
  =(S{}^tT,-S)
  \quad\text{ for }{}^\exists S\in GL_{n-r}(\mathbb{Q}).$$
\end{lem}
\proof
    {
      By
      \begin{align*}
\sum_{i=1}^n x_i\beta_i 
  &= (\beta_1,\dots,\beta_r)\left\{ \left(
  \begin{array}{c}
    {x}_1\\
    \vdots\\
    {x}_r
  \end{array}
  \right)
  + T
  \left(
  \begin{array}{c}
    {x}_{r+1}\\
    \vdots\\
    {x}_n
  \end{array}
  \right)
  \right\}\quad (x_i\in\mathbb{R}),
      \end{align*}
      we see
      \begin{align*}
        LR_0 &= \{(l_1,\dots,l_n)\in\mathbb{Q}^n \mid \sum_{i=1}^nl_i\beta_i=0\}
        \\
        &= \left\{(l_1,\dots,l_n)\in\mathbb{Q}^n \mid
         \left(
  \begin{array}{c}
    {l}_1\\
    \vdots\\
    {l}_r
  \end{array}
  \right)
  =- T
  \left(
  \begin{array}{c}
    {l}_{r+1}\\
    \vdots\\
    {l}_n
  \end{array}
  \right)
\right\},
        \end{align*}
      hence $t = \dim{LR_0}=n-r$.
      On the other hand, the definition of $T$ implies
      $$
({}^t T, -1_{n-r})       \left(
  \begin{array}{c}
    {\beta}_1\\
    \vdots\\
    {\beta}_n
  \end{array}
  \right)=0^{(n-r,1)},
      $$
  hence row vectors of the matrix $({}^t T, -1_{n-r}) $ span $LR_0$ by $
  \text{rank}({}^t T, -1_{n-r})=\dim LR_0 $ 
   and so
  there is a matrix $S\in GL_{n-r}(\mathbb{Q})$ such that
  $$
   \left(
  \begin{array}{c}
    \bm{m}_1\\
    \vdots\\
    \bm{m}_t
  \end{array}
  \right)
=S({}^t T, -1_{n-r})=(S{}^t T, -S).
  $$
      \qed
    }

\begin{lem}\label{lem5.12}
Supposing the assumption on $\beta_i$ in Lemma \ref{lem5.11},
we have
\begin{align*}
  \mathfrak{D}_f
  &=
  \left\{
\bm{x}\in\mathbb{R}^n \mid
(S{}^t T, -S){}^t\bm{x}\in M_{t,1}(\mathbb{Z})
\right\}
\\
&=
\left\{
\bm{x}\in\mathbb{R}^n \mid
(x_{r+1},\dots,x_n)=(x_1,\dots,x_r)T+\bm{k}{}^tS^{-1}\,({}^\exists\bm{k}\in M_{1,t}(\mathbb{Z}))
\right\}.
\end{align*}
For $(x_1,\dots,x_r)\in[0,1)^r$, there are exactly $|\det(S)|$ vectors $(x_{r+1},\dots,x_n)\in
[0,1)^{n-r}$ such that $(x_1,\dots,x_n)\in\mathfrak{D}_f$.
\end{lem}
\proof
    {
We see that
\begin{align*}
   \mathfrak{D}_f &=
   \left\{\bm{x}\in\mathbb{R}^n \mid
   \left(
   \begin{array}{c}
     \bm{m}_1\\
     \vdots\\
     \bm{m}_t
   \end{array}
   \right)   {}^t  \bm{x}\in M_{t,1}(\mathbb{Z})
   \right\}
\\
&=
\left\{
\bm{x}\in\mathbb{R}^n \mid
(S\,{}^t T, -S){}^t\bm{x}\in M_{t,1}(\mathbb{Z})
\right\}
\\
&=
\left\{
\bm{x}\in\mathbb{R}^n \mid
((x_{r+1},\dots,x_n)-(x_1,\dots,x_r)T){}^tS\in M_{1,t}(\mathbb{Z})
\right\}
\\
&=
\left\{
\bm{x}\in\mathbb{R}^n \mid
(x_{r+1},\dots,x_n)=(x_1,\dots,x_r)T+\bm{k}{}^tS^{-1}\,({}^\exists\bm{k}\in M_{1,t}(\mathbb{Z}))
\right\}.
\end{align*}
Since $S\in GL_{n-r}(\mathbb{Q})$ is an integral regular matrix,
the inclusion $\mathbb{Z}^{n-r}\,{}^t\!S^{-1}\supset \mathbb{Z}^{n-r}$ is clear.
Hence, for  $(x_1,\dots,x_r)\in\mathbb{R}^r$, there is a vector $(x_{r+1},\dots,x_n):=(x_1,\dots,x_r)T +\bm{k}{}^tS^{-1}\in
[0,1)^{n-r}$ for some integral vector $\bm{k}$.
The number of such integral vectors is $|\det(S)|$ since $S$ is integral.
\qed
    }

    Therefore we see that, assuming that $\beta_1,\dots,\beta_r$ is linearly independent over $\mathbb{Q}$ and $1\le i\le r$ as above
\begin{align*}
  & \{\bm{x}\in[0,1)^n\mid x_i\le a\}\cap\mathfrak{D}_f  =
    \\&
\left\{
\bm{x}\in[0,1)^n \mid x_i\le a,
(x_{r+1},\dots,x_n)=(x_1,\dots,x_r)T+\bm{k}{}^tS^{-1}\,({}^\exists\bm{k}\in M_{1,t}(\mathbb{Z}))
\right\}, 
\end{align*}
which is included in the union of sets parallel to the subspace
$$
\left\{
\bm{x}\in \mathbb{R}^n \mid 
(x_{r+1},\dots,x_n)=(x_1,\dots,x_r)T
\right\},
$$
and the projection to  $\{\bm{x}\in[0,1)^r\mid x_i\le a\}$ is $|\det(S)|$-fold on
 by the lemma above.
Hence  the volume of  $\{\bm{x}\in[0,1)^n\mid x_i\le a\}\cap\mathfrak{D}_f $ as an $r\,(=n-t)$-dimensional set  is proportional to $a$.
In general, for given $i$, we have only to take a subset $j_1,\dots,j_r$ such that
$\beta_{j_1},\dots,\beta_{j_r}$ are linearly independent over $\mathbb{Q}$ and $j_k=i$ for some $k$. 
Thus,
we have proved
\begin{thm}\label{thm5.4}
  If a polynomial $f(x)$ has no rational root,
 $\alpha_i - \alpha_j\not\in\mathbb{Q}$ holds for any distinct $i,j$, 
  and  the conditions $Pr(f,\sigma)>0$ and ${vol}({\mathfrak{D}}(f,\sigma))>0$ are equivalent,
  then
  Conjectures \ref{conj1}, \ref{conj2} imply the equi-distribution of $r_i/p$  for local roots $r_i$ of the polynomial. 
  \end{thm}
So, the equi-distribution of $r_i/p$ for local roots $r_i$ is likely for an irreducible polynomial $f(x)$ 
of $\deg f>1$.
Although the polynomial $f(x)=(x^2-2)((x-1)^2-2)$ does not satisfy the assumption of the theorem,  the equi-distribution of local roots $r_i/p$ is true (cf. {\bf{6.4.2}}).

Put, for an algebraic number field $F$ containing $\mathbb{Q}(f)$ 
$$
Spl(f,F):=\{p\in Spl(f) \mid p\text{ is fully splitting at }F\}.
$$
If a subsequence of $r_i/p$ for local roots $r_i$ of $f$ on $p\in Spl(f,F)$ distributes uniformly for every 
irreducible  polynomial $f$ of $\deg f>1$ and every number field $F$, then the one-dimensional distribution
of $r_i/p$ for any reducible polynomial without rational root is true.

Let us give  more remarks.
\begin{prop}\label{prop5.2}
Suppose that 
$f(x)$ is not a product of linear forms and
$\alpha_i - \alpha_j\not\in\mathbb{Q}$ for any distinct $i,j$;
 then we have
\begin{align*}
&\#\bm{G}\cdot vol(\cup_{\sigma\in S_n}\mathfrak{D}(f,\sigma))=
\sum_{\sigma\in S_n}vol(\mathfrak{D}(f,\sigma))
\\
=\,
&vol([0,1)^n\cap \mathfrak{D}_f)=\sqrt{\det((\bm{m}_i,\bm{m}_j))_{1\le i,j\le t}}.
\end{align*}
\end{prop}
\proof{
The first two equalities  are nothing but \eqref{eq46} and \eqref{eq71}.
The matrix whose rows are $\bm{m}_1,\dots,\bm{m}_t$ is primitive,
hence there are integral vectors $\bm{m}_{t+1},\dots,\bm{m}_n$ such that 
the matrix $M$ with the $i$th row being $\bm{m}_i$ is unimodular, that is $M$ is an integral matrix with determinant $1$.
Denote the $i$-th row of the matrix ${}^t\!M^{-1}$ by $\bm{m}'_i$, hence $(\bm{m}'_i,\bm{m}_j) =\delta_{i,j}$.
We see $\bm{x}+\bm{m}\in\mathfrak{D}_f$ for
 $\bm{x}\in\mathfrak{D}_f$ and $\bm{m}\in\mathbb{Z}^n$, hence $\mathbb{Z}^n$ acts on $\mathfrak{D}_f$ 
as a volume-preserving  discontinuous group.
The set $[0,1)^n\cap\mathfrak{D}_f$ is a fundamental domain 
( $=$ a realization of $\mathfrak{D}_f/\mathbb{Z}^n$).
Another is $\{\sum_{i=t+1}^nx_i\bm{m}'_i\mid {}^\forall x_i\in[0,1)\}$,
since  for a vector $\bm{x}= \sum x_i\bm{m}'_i $ in
$$
\mathfrak{D}_f=\{\sum_{i=1}^n y_i\bm{m}'_i\mid y_1,\dots,y_t\in\mathbb{Z},y_{t+1},\dots,y_n\in\mathbb{R}\},
$$
the condition $\bm{x}\in\mathbb{Z}^n$ holds if and only if $x_{t+1},\dots,x_n\in\mathbb{Z}$ holds.
Hence we have $vol([0,1)^n\cap\mathfrak{D}_f)=vol(\{\sum_{i=t+1}^n y_i\bm{m}'_i\mid 0\le 
{}^\forall y_i<1\})=\sqrt{\det((\bm{m}'_,\bm{m}'_j))_{t<i,j\le n}}$.
Decompose the matrix $M{}^tM$ as
$$
M{}^tM=\begin{pmatrix}A^{(t)}&B\\{}^tB&C\end{pmatrix}
=\begin{pmatrix}1_t& 0\\{}^t(A^{-1}B)&1_{n-t}\end{pmatrix}
\begin{pmatrix}A&0\\0&C-{}^tBA^{-1}B\end{pmatrix}
\begin{pmatrix}1_t &A^{-1}B\\0&1_{n-t}\end{pmatrix},
$$
which implies $\det A\cdot\det(C-{}^tBA^{-1}B)=1$.
Here we note that $M{}^t\!M$ is positive definite, and
$A=((\bm{m}_i,\bm{m}_j))_{1\le i,j\le t}$.
Similarly, the right lower $(n-t)$-square matrix of ${}^t\!M^{-1}{}^t({}^t\!M^{-1})=(M{}^t\!M)^{-1}$
is $((\bm{m}'_i,\bm{m}'_j))_{t+1\le i,j\le n}$.
By the above decomposition, we see that $(M{}^t\!M)^{-1}$ is
\begin{gather*}
\begin{pmatrix}1_t& -A^{-1}B)\\0&1_{n-t}\end{pmatrix}
\begin{pmatrix}A^{-1}&0\\0&(C-{}^tBA^{-1}B)^{-1}\end{pmatrix}
\begin{pmatrix}1_t &0\\-{}^t(A^{-1}B)&1_{n-t}\end{pmatrix}
\\
=\begin{pmatrix}
\text{\,*}&\text{\,*}\\ \text{*}&(C-{}^tBA^{-1}B)^{-1}
\end{pmatrix},
\end{gather*}
hence we see that $\det((\bm{m}'_i,\bm{m}'_j))_{t+1\le i,j\le n}= \det(C-{}^tBA^{-1}B)^{-1}=\det A
=\det((\bm{m}_i,\bm{m}_j))_{1\le i,j\le t}$,
which completes the proof.
\qed
}

The equation above is a  little bit curious, because $vol$ is, by definition
 the $(n-t)$-dimensional volume and
$\sqrt{\det((\bm{m}_i,\bm{m}_j))}$ is the volume of the $t$-dimensional  paralleltope spanned by $\bm{m}_1,\dots,
\bm{m}_t$.
\begin{prop}\label{prop5.3}
Suppose that 
$f(x)$ is not a product of linear forms,
$\alpha_i - \alpha_j\not\in\mathbb{Q}$ for any distinct $i,j$
and
$vol(\mathfrak{D}(f,\sigma))=cPr(f,\sigma)$ for every $\sigma\in S_n$ 
;
then we have
$$
c=\sqrt{\det((\bm{m}_i,\bm{m}_j))_{1\le i,j\le t}}/\#\hat{\bm{G}}.
$$
\end{prop}
\proof
{
 We see easily
\begin{align*}
&\sqrt{\det((\bm{m}_i,\bm{m}_j))}
\\
=&\,\sum_{\sigma}vol(\mathfrak{D}(f,\sigma))
\\
=&\,c\sum_{\sigma}Pr(f,\sigma)
\\
=&\,c\#\hat{\bm{G}}\quad\text{ by }\eqref{eq41},
\end{align*}
hence we have
$$
c=\sqrt{\det((\bm{m}_i,\bm{m}_j))}/\#\hat{\bm{G}}.
$$
\qed
}

\noindent 
{\bf Example }
 The case that the polynomial $f(x)$ has no non-trivial linear relation 
 among roots :
It is easy to see $Pr(f,\sigma)=1$ and $vol(\frak{D}(f,\sigma))=vol(\hat{\frak{D}}_n)=
\frac{\sqrt{n}}{n!}$ by Proposition \ref{prop4.6}, hence $c =\frac{\sqrt{n}}{n!}$.
On the other hand, the base of $LR_0\cap\mathbb{Z}^n$ is $(1,\dots,1)$ and $\hat{G}=S_n$
is clear, hence the identity in Proposition \ref{prop5.3} is checked.

The case that  the irreducible polynomial $f(x)$ is of degree $4$ and has non-trivial linear relations
among roots : 
Suppose that  $f(x)$, roots $\alpha_i$, the basis $\bm{m}_i$ $(i=1,2)$ are those in the proof of Theorem \ref{thm5.3};
then, only for permutations $\sigma$ satisfying $\{\sigma(1),\sigma(4)\}=\{1,4\}$ or $\{2,3\}$,
$Pr(f,\sigma)>0, vol(\frak{D}(f,\sigma))>0$ hold, and then $Pr(f,\sigma)=1$ holds and
$\frak{D}(f,\sigma)$ is equal to
\begin{align*}
&\{(x_1,x_2,1-x_2,1-x_1)\mid 0\le x_1\le x_2<\frac{1}{2}\}
\\
 =&\{(0,0,1,1)+y_1(\frac{1}{\sqrt{2}},0,0,-\frac{1}{\sqrt{2}})
+y_2(0,\frac{1}{\sqrt{2}},-\frac{1}{\sqrt{2}},0) \mid 0\le y_1\le y_2<\frac{1}{\sqrt{2}}\}, 
\end{align*}
hence $vol(\frak{D}(f,\sigma))=1/4$.
On the other hand, it is easy to see $\det((\bm{m}_i,\bm{m}_j))=4,\#\hat{G}=8$,
hence the identity on $c$ is checked.
\section{The Weyl criterion}\label{sec6}
The equation \eqref{eq13} in Conjecture \ref{conj3} is equivalent to
\begin{align}\nonumber
&\lim_{X\to\infty}\frac{\sum_{p\in Spl_X(f,\sigma,\{k_j\})}F(r_1/p,\dots,r_n/p)} {\#Spl_X(f,\sigma,\{k_j\})} 
\\\label{eq72}
=&
\frac{1}{vol(\mathfrak{D}(f,\sigma,\{k_j\}))}\int_{\mathfrak{D}(f,\sigma,\{k_j\})}F(\bm{x})d\bm{x},
\end{align}
where $F=\chi_D$ is the characteristic function of the set $D$ there.
This is a kind of uniform distribution.
Here we remark that for a prime $p\in Spl_X(f,\sigma,\{k_j\})$, the point $(\frac{r_1}{p},\dots,\frac{r_n}{p})$ is not  on $\mathfrak{D}(f,\sigma,\{k_j\})$ unless $m_1=\dots=m_t=0$ by the condition 
$\sum_i m_{j,i}r_{\sigma(i)}=m_j+k_jp$,
hence the domain of the definition of the function $F$ should be  paid attention to.
As the original case, we want to approximate the characteristic function by manageable functions.
The author does not know what they are,
although trigonometric functions are expected.

Let us recall the original Weyl criterion for a uniformly distributed sequence modulo $1$ 
(see \cite{K-N}):
The Weyl criterion states that  a sequence $\{x_n\}$, $x_n \in[0,1)$, $n=1,2,\dots$ is uniformly distributed modulo $1$ if and only if the 
following equation
\begin{equation}
\lim_{N\to\infty}\frac{1}{N}\sum_{n=1}^N F(x_n) = 
\int_0^1 F(x)dx 
\tag{W}
\end{equation}
holds for $F(x)= \exp(2\pi i hx)$ $({}^\forall h\in\mathbb{Z})$. 
Usually,  the right-hand side is replaced by `` $0$ for $h\ne0$ '',
omitting the trivial case $h=0$.
The proof proceeds as follows:
\begin{enumerate}
\item
The sequence $\{x_n\}$, $x_n \in[0,1)$, $n=1,2,\dots$ is uniformly distributed modulo $1$ if and only if the 
 equation (W) holds for the characteristic function of every interval in $[0,1]$,
 compactifing the interval $[0,1)$ to $[0,1]$  with the identification of endpoints $0,1$.
\item
Approximating a continuous function $F(x)$ $(F(0)=F(1))$ on $[0,1]$ by step functions from above and below
by the supremum norm,
one proves that the equation (W) is true for step functions if and only if it is true for
continuous functions .
\item
Approximating a continuous function $F(x)$ $(F(0)=F(1))$ on $[0,1]$ by trigonometric functions
$ \exp(2\pi i hx)$ $( h\in\mathbb{Z})$, one shows that (W) is true for  continuous functions $F(x)$
if and only if (W) is true for  trigonometric functions $ \exp(2\pi i hx)$ $( {}^\forall h\in\mathbb{Z})$.
\end{enumerate}
To apply the idea to our case, the following   Stone-Weierstrass theorem may be helpful.
\begin{thm}\label{SW}
Let $T$ be a compact set and  let $C(T)$ 
be  the algebra of $\mathbb{R}$-valued continuous functions on $T$.
Suppose that $\mathfrak{A}$ is a subalgebra of  $C(T)$  such that 
(i) $\mathfrak{A}$  contains every constant function, 
 (ii)   for distinct points $x,y\in T$, there is a function $g\in\mathfrak{A}$ such that $g(x)\ne g(y)$.
Then a continuous function $f$ on $T$  is
approximated by functions in $\mathfrak{A}$ by the supremum norm $||f||:=\sup_{x\in T} |f(x)|$.
\end{thm}
The  proof  in \cite{BD} is simple.
Just to make sure, we give the following variation.
\begin{cor}
Let ${T}$ be a compact set and let $E_{\gamma}\,(\gamma\in \Gamma)$ be mutually disjoint
 subsets of  ${T}$, and let $C({T})$ 
be  the algebra of $\mathbb{R}$-valued continuous functions on ${T}$.
Suppose that $\mathfrak{A}$ is a subalgebra of  $C({T})$  such that 
(i) $\mathfrak{A}$
 contains every constant function, 
 (ii)  every function in $\mathfrak{A}$  takes the same value on each $E_{\gamma}$, 
and (iii) for distinct points $x,y\in {T}$, 
there is a function $g\in\mathfrak{A}$ such that $g(x)\ne g(y)$
if there is no $\gamma\in\Gamma$ such that $x, y\in E_{\gamma}$.
Then a continuous function $f$ on ${T}$ such that $f(x)$ takes the same value on each $E_{\gamma}$ is
approximated by functions in $\mathfrak{A}$ by the supremum norm $||f||:=\sup_{x\in {T}} |f(x)|$.
\end{cor}
The proof is easy. We have only to apply the Stone-Weierstrass theorem to the quotient space 
of $T$ by the equivalence shrinking $E_\gamma$ to one point for each $\gamma\in\Gamma$.
%
%
%
%
%
%
\subsection{The Weyl criterion}
Hereafter we  suppose that the polynomial $f(x)$ has no rational root and is not of the form 
$g(x)g(x-c)h(x)$ $(g(x),h(x)\in\mathbb{Z}[x],c\in \mathbb{Z})$, and
write simply
$$
\mathfrak{D}:=\mathfrak{D}(f,\sigma,\{k_j\})
 =\{(x_1,\dots,x_n)\in\Delta\mid\sum_{l=1}^n m_{i,l}x_{\sigma(l)}=k_i\:(1\le{}^\forall i\le t)\}.
$$

First, 
we supplement the bases $\bm{m}_1,\dots,\bm{m}_t$ of linear relations among roots by
integral vectors $\bm{m}_{t+1},\dots,\bm{m}_n$ so that the $(n,n)$-matrix $M=(m_{i,j})$ 
with the $i$-th row $ \bm{m}_{i}$ is in $SL_n(\mathbb{Z})$
and $\sigma(M):=(m_{i,\sigma^{-1}(j)})$,
that is its $i$-th row  is $\sigma(\bm{m}_i)$.
Then,  writing
$$
\mathcal{S}:=\{\bm{x} \in\mathbb{R}^n\mid (\sigma(\bm{m}_i),\bm{x})=k_i\:(1\le{}^\forall i\le t)\},
$$
we have
\begin{align*}
\mathfrak{D}=\Delta\cap {\mathcal{S}},
\end{align*}
and
via the transformation  $\bm{x}\mapsto\bm{y}:=\bm{x}{}^t\!(\sigma(M))$ we see that,
noting $(\sigma(\bm{m}_i),\bm{x})=(\sigma(\bm{m}_i)\sigma(M)^{-1},\bm{x}{}^t\!(\sigma(M)))=y_i$
\begin{align*}
{\mathcal{S}}\,{}^t\!(\sigma(M))&=\{(y_1,\dots,y_n)\in\mathbb{R}^n\mid y_i=k_i\,(1\le{}^\forall i\le t)\}.
\end{align*}
The dimension of the set $\mathfrak{D}$ is at most $n-t$, which is the dimension of the 
supporting space $\mathcal{S}$.

Noting that, via
$$
(x_1,\dots,x_n)=(u_1,\dots,u_n)
\left(
\begin{array}{cccc}
1&1&\dots&1
\\
0&1&\dots&1
\\
\vdots&\vdots&\ddots&\vdots
\\
0&0&\dots&1
\end{array}
\right),
$$
\begin{align*}
\Delta=\left\{
\bm{x}\mid0\le x_1\le\dots\le x_n\le1\}=\{\sum u_i\bm{w}_i\mid {}^\forall u_i\ge0,
\sum u_i\le 1
\right\},
\end{align*}
with $\bm{w}_1:=(1,\dots,1),
\bm{w}_2:=(0,1,\dots ,1),\dots,\bm{w}_n:=(0,\dots,0,1)$, 
the set $\Delta$
 is  the convex hull with vertexes $\bm{w}_1,\dots,\bm{w}_n$ and the origin $\bm{w}_{n+1}:=\bm{O}$,
that is
$$
\Delta=\left\{
\sum_{i=1}^{n+1} u_i\bm{w}_i\mid {}^\forall u_i\ge0,
\sum_{i=1}^{n+1} u_i= 1
\right\},
$$
hence we see that $\mathfrak{D} \,(= \mathfrak{D}(f,\sigma,\{k_j\})) $ is also a convex hull, so a finite union of simplexes.

Denote the $i$-th row of the matrix ${}^t\!(\sigma(M))^{-1}$ by $\bm{m}'_i$, 
that is
\begin{equation}\label{eq73}
\begin{pmatrix}
\bm{m}'_1\\\vdots\\\bm{m}'_n    
\end{pmatrix}:={}^t\hspace{-0.3mm}\sigma(M)^{-1},
\quad i.e.
\quad(\bm{m}'_j,\sigma(\bm{m}_i))=\delta_{i,j},
\end{equation}
and put
$$
\bm{c}:=\sum_{i=1}^tm_i\bm{m}'_i.
$$ 
\begin{lem}\label{lem6.1}
 Write $\bm{r}:=(r_1,\dots,r_n)$ for local roots $r_i$ for $p\in Spl(f,\sigma,\{k_j\})$; 
 then we have
\begin{gather*}
 {\mathcal{S}}{}^t\!(\sigma(M))=\{(k_1,\dots,k_t,y_{t+1},\dots,y_n) \mid y_{t+1},\dots,y_n\in\mathbb{R}\},
 \\
 \bm{c}{}^t\!(\sigma(M))=(m_1,\dots,m_t,0,\dots,0),
 \\
 (\bm{r}-\bm{c})/p \in \mathcal{S},
 \end{gather*}
and 
$$(\bm{r}-\bm{c})/p\in\mathfrak{D}(=\Delta\cap \mathcal{S})
$$ 
holds except a finite number of primes under the assumption at the beginning of this subsection
 on the polynomial $f$.
 \end{lem}
 \proof
 {
The first is already given above,
the second  identity follows from \eqref{eq73} and the definition of $\bm{c}$, and the third does from $(\bm{c},\sigma(\bm{m}_j))=m_j$ and
 $(\bm{r},\sigma(\bm{m}_j))=m_j+k_jp$ $(j=1,\dots,t)$ by definition of the local root vector $\bm{r}$.
 To see the last inclusion, we have only to show that
 the number of primes $p$ satisfying $(\bm{r}-\bm{c})/p\not\in\Delta$
is finite.
If $(\bm{r}-\bm{c})/p\not\in\Delta$ happens  for infinitely many primes, then one of conditions
$0\le r_1< c_1$, $0\le r_{i+1}-r_i < c_{i+1}-c_i$ or $0<p-r_n<-c_n$ occurs for infinitely many primes,
hence one of $r_1=c$, $r_{i+1}-r_i=c$ or $r_n -p = c$   for some constant $c$ 
occurs for infinitely many primes.
This means $f(c)\equiv0\bmod p$, $f(r_i)\equiv f(r_i+c)\bmod p$, $f(c)\equiv0\bmod p$, respectively.
The first or the third condition implies easily that the polynomial $f(x)$ has a rational root $c$.
The second implies that  the polynomial $f$ is  of the form $g(x)g(x-c)h(x)$. 
Because, if the identity $r_{i+1}-r_i=c$ holds for infinitely many primes, then $f(r_i+c)\equiv
f(r_i)\equiv0\bmod p$ holds for infinitely many primes, which implies  $(f(x),f(x+c))\ne1$ 
(resp. $(f(x),f'(x))\ne1$) for $c\ne0$ (resp.  $c=0$),
Hence, in the case of $c\ne0$, there is a decomposition $f(x)=g(x)h(x),f(x+c)=g(x)k(x)$ 
for an irreducible polynomial $g(x)$,
i.e. $f(x)$ is divisible by $g(x)g(x-c)$. If $c=0$, then $f(x)$ is of the form $g(x)^2h(x)$.   
 
 \qed
 }

Suppose  the following equality
\begin{align}\nonumber	
&\#\{p \in Spl_X(f,\sigma,\{k_j\})\mid \bm{r}/p \in D\}
\\ \label{eq74}
=&\#\{p \in Spl_X(f,\sigma,\{k_j\})\mid( \bm{r}-\bm{c})/p \in D\}
+o(\# Spl_X(f,\sigma,\{k_j\})),
\end{align}
which implies
\begin{align}\nonumber	
&\lim_{X\to\infty}\frac{\#\{p \in Spl_X(f,\sigma,\{k_j\})\mid \bm{r}/p \in D\}}
{\# Spl_X(f,\sigma,\{k_j\})}
\\ \nonumber
=&\lim_{X\to\infty}\frac{\#\{p \in Spl_X(f,\sigma,\{k_j\})\mid( \bm{r}-\bm{c})/p \in 
D\}}
{\# Spl_X(f,\sigma,\{k_j\})}.
\end{align}

Thus, under the assumption \eqref{eq74}  the equation \eqref{eq13} in Conjecture \ref{conj3} is equivalent
by Lemma \ref{lem6.1} to
\begin{align}\label{eq75}
\lim_{X\to\infty}\frac{\#\{Spl_X(f,\sigma,\{k_j\}) \mid (\bm{r}-\bm{c})/p\in D\cap \mathfrak{D}\}} 
{    \#Spl_X(f,\sigma,\{k_j\})    }
=
\frac{vol(D\cap \mathfrak{D})}{vol(\mathfrak{D})},
\end{align}
and the equation for the Weyl  criterion   is
 \begin{align}\nonumber
&\lim_{X\to\infty}\frac{1} {\#Spl_X(f,\sigma,\{k_j\})} \sum_{p\in Spl_X(f,\sigma,\{k_j\})}F((\bm{r}-\bm{c})/p)
\\ \label{eq76}
=&
\frac{1}{vol(\mathfrak{D})}\int_{\mathfrak{D}}F(\bm{x})d\bm{x},
\end{align}
for a continuous function $F$ on $\mathfrak{D}$.

Last, let us give a sufficient condition to the supposition  \eqref{eq74}, i.e.
\begin{align*}
&\#\left\{ p\in Spl_X(f,\sigma,\{k_j\}) 
\left|
\begin{array}{l}
\bm{r}/p \in D\text{ and }   (\bm{r}-\bm{c})/p \not\in D,\text{ or }   
\\
 \bm{r}/p \not\in D \text{ and }   (\bm{r}-\bm{c})/p \in D
 \end{array}
\right.
\right\}
\\
=&\,o(\# Spl_X(f,\sigma,\{k_j\})).
\end{align*}
\begin{prop}\label{prop6.1}
Let the set $D$ be  of the form $D= \{\bm{x}\mid \sum_{j=1}^n d_{i,j}x_j\le d_i\,( i=1,\dots,l)\}$
with all $d_{i,j},d_i$ being rational.
If $\sum_{j=1}^n d_{i,\sigma(j)} \alpha_j\not\in\mathbb{Q}$
holds for $1\le{}^\forall i\le  l$ and any permutation $\sigma$,
then the equation  \eqref{eq74} is true for the set $D$.
\end{prop}
\proof{
Let $d$ be a positive integer such that all $dd_{i,j},dd_i$ are integers. 
Let us show the contradiction under the supposition that there are infinitely many primes $p\in Spl(f)$ such that $\bm{r}/p\in D$ and $(\bm{r}-\bm{c})/p\not\in D$; then we have $\sum_jd_{i,j}r_j\le pd_i$ for all $i$ and $\sum_jd_{k,j}(r_j-c_j)>pd_k$ for some $k$, hence
$\sum_j dd_{k,j}c_j< \sum_j dd_{k,j}r_j -pdd_k \le0$,
where $ \sum_j dd_{k,j}c_j$ is a constant independent of $p$ and $ \sum_j dd_{k,j}r_j - pdd_k$ is an integer.
Therefore  there are infinitely many primes $p\in Spl(f)$ such that $ \sum_j dd_{k,j}r_j -pdd_k =\kappa$ for some integer $\kappa$ 
satisfying $\sum_j dd_{k,j}c_j<\kappa \le0$.
Taking a prime ideal $\frak{p}$ of $\mathbb{Q}(f)$ over $p$, we see that
there are infinitely many prime ideals $\frak{p}$ and a permutation $\sigma$
such that $\sum_jdd_{k,j}\sigma(\alpha_j)\equiv \kappa\bmod \frak{p}$, which implies
 $\sum_jdd_{k,j}\sigma(\alpha_j)= \kappa$, i.e. $\sum_j dd_{k,\sigma^{-1}(j)}\alpha_j= \kappa$, which implies the contradiction $(d_{k,\sigma^{-1}(1)},\dots,d_{k,\sigma^{-1}(n)})\in LR_0$.
The case of  $\bm{r}/p \not\in D$  and $   (\bm{r}-\bm{c})/p \in D$ is similarly proved.
\qed
}
\begin{cor}\label{cor6.2new}
Suppose that the polynomial $f(x)$ has no rational root.
Then \eqref{eq74} is true for the set of the form $D= \{\bm{x}\mid a_i\le x_i\le b_i\,(^\forall i)\}$
with all $a_i,b_i$ being rational.
\end{cor}
\proof{
Vectors $(d_{i,1},\dots,d_{i,n})$ in Proposition \ref{prop6.1} are $(0,\dots,0,\pm1,0,\dots,0)$. Such a vector is in $LR_0$ if and only if $f(x)$ has a rational root.
\qed
}

How about the case of  $d_{i,j}$ being irrational or more general boundaries?

\subsection{}
Let us give vertexes  of $\mathfrak{D}=\mathfrak{D}(f,\sigma,\{k_j\})$ explicitly in the special case. 
Since $\Delta$ is a  convex hull,  the set $\mathfrak{D}=\Delta\cap\mathcal{S}$ is also a convex hull, hence a finite union of simplexes.

First of all, we note that a point $\bm{x}\in\mathfrak{D}$ is not a vertex if and only if  there is a 
non-zero 
vector $\bm{v}=(v_1,\dots,v_n)$ 
such that  $\bm{x}+\eta\bm{v}\in\mathfrak{D}$ holds
for any $\eta$ in a short open interval containing $0$.
Hence a point $\bm{x}\in\mathfrak{D}$ is not a vertex
 if and only if 
there is a non-zero vector $\bm{v}$ satisfying that
\begin{itemize}
\item
$\sum_im_{j,\sigma^{-1}(i)}v_i=0$ for ${}^\forall j=1,\dots,t$, and
\item
$v_i=0,v_i=v_{i+1},v_i=0$ hold according to $x_i=0,x_i=x_{i+1},x_i=1$, respectively.
\end{itemize} 

Another approach is as follows: Since the set $\Delta$ is the convex hull of $\bm{w}_1:=(1,\dots,1),
\bm{w}_2:=(0,1,\dots ,1),\dots,\bm{w}_n:=(0,\dots,0,1)$, and $\bm{w}_{n+1}=(0,\dots,0)$, vertexes of $\mathfrak{D}=\Delta\cap S$
are the set of the intersection of the line segment $\overline{\bm{w}_i\bm{w}_j}$ connecting $\bm{w}_i$ and $\bm{w}_j$ and $(n-t)$-dimensional plane $S$.
Hence, for $1\le i <j\le n+1$, a vector $\bm{v}=u\bm{w}_i+(1-u)\bm{w}_j$ $(0\le{}^\exists u\le 1)$
 on $\overline{\bm{w}_i\bm{w}_j}$ is on $\mathcal{S}$ if and only if $\bm{v}{}^t\sigma(M)=(k_1,\dots,k_t,*,\dots,*)$.
Note that
\begin{align*}
u\bm{w}_i+(1-u)\bm{w}_j=
\left\{
\begin{array}{ll}
(\underbrace{0,\dots,0}_{i-1},\underbrace{u,\dots,u}_{n-i+1})     & \text{ if }j=n+1,
\\
(\underbrace{0,\dots,0}_{i-1},\underbrace{u,\dots,u}_{j-i},\underbrace{1,\dots,1}_{n-j+1})     & \text{ if }j\le n.
\end{array}
\right.
\end{align*}

\noindent
{\bf{Example 1}}
Suppose that the polynomial $f(x)$ has no non-trivial linear relations among roots, i.e. $t=1$, and 
\begin{align*}
\Delta_k:= \mathfrak{D}(f,\sigma,k)
=\{\bm{x}\in\Delta\mid \sum_{i=1}^n x_i =k\}
\quad(1\le k \le n-1).
\end{align*}
The vertexes of $\Delta_k$ are following $k(n-k)+1$ vectors:
\begin{align*}
P(p,q,r):=(\underbrace{0,\dots,0}_{p},\underbrace{\epsilon,\dots,\epsilon}_{q},
\underbrace{1,\dots,1}_{r})
\end{align*}
with non-negative integers $p,q,r$ satisfying $p+q+r=n$ and
$$
\left\{
\begin{array}{ll}
p=n-k,q=0&\text{ if } r=k,
\\
k-r+1\le q\le n-r,
\epsilon=\frac{k-r}{q}&\text{ if }0\le r \le k-1.
\end{array}
\right.
$$
Because, if $\bm{x}\in\Delta_k$ satisfies three strict inequalities $x_{i}<x_{i+1}=\dots=x_{i+a}<x_{i+a+1}
=\dots=x_{i+a+b}<x_{i+a+b+1}$, then for any  $\eta$ sufficiently close to $0$,
the vector $\bm{x}(\eta)$, replacing $x_{i+1}=\dots=x_{i+a},
x_{i+a+1}=\dots=x_{i+a+b}$ by $x_{i+a}+\eta/a,x_{i+a+1}-\eta/b$, respectively is still in 
$\Delta_k$,
that is a short segment containing $\bm{x}$ is in $\Delta_k$.

$\Delta_k$ is a simplex if and only if $k=1,n-1$, and
vertexes of $\Delta_1$ are 
$$
(\underbrace{0,\dots,0}_{n-q},\underbrace{\frac{1}{q},\dots,\frac{1}{q}}_{q})\quad(1\le q\le n),
$$
and those of $\Delta_{n-1}$ are 
$$(0,1,\dots,1),\,
(\underbrace{\frac{q-1}{q},\dots,\frac{q-1}{q}}_{q},\underbrace{1,\dots,1}_{n-q})\quad (2\le q\le n).$$

\noindent
{\bf{Example 2}}
Suppose that the polynomial $f(x)$ is of degree $4$ and has a non-trivial linear relations among roots; Then  we may assume that their basis are $\alpha_1+\alpha_4,\alpha_2+\alpha_3\in\mathbb{Z}$ by Proposition \ref{prop3.3},
i.e.  $\bm{m}_1=(1,0,0,1),\bm{m}_2=(0,1,1,0)$.
Hence we may take $\bm{m}_3:=(0,0,1,0),\bm{m}_4:=(0,0,0,1)$; then we have
$$
\mathfrak{D}(f,id,{k_1,k_2})=
\{
(x_1,\dots,x_4)\mid 0\le x_1\le \dots\le x_4,x_1+x_4=k_1,x_2+x_3=k_2
\}.
$$
It is easy to see that $\dim\mathfrak{D}=2\Leftrightarrow k_1=k_2=1$, and then the vertexes are $(0,0,1,1),(0,\frac{1}{2},\frac{1}{2},1),(\frac{1}{2},\frac{1}{2},\frac{1}{2},\frac{1}{2})$.

\subsection{}
Let us see the definition of the integral on the set $\mathfrak{D}$  explicitly.
Let us recall definitions :
\begin{align*}
\Delta:&=\{\bm{x}=(x_1,\dots,x_n)\in\mathbb{R}^n\mid0\le x_1\le\dots\le x_n\le1\},
  \\
\mathcal{S}:&=\{\bm{x} \in\mathbb{R}^n\mid (\sigma(\bm{m}_j),\bm{x})=k_j\,(j=1,\dots,t)\},
 \\
 \mathfrak{D}:= \mathfrak{D}(f,\sigma,\{k_j\}):&=\{(x_1,\dots,x_n)\in\Delta\mid\sum_{i=1}^n m_{j,i}x_{\sigma(i)}=k_j\:(1\le{}^\forall j\le t)\}
\\
&=\Delta \cap \mathcal{S},
\end{align*}
and 
write
\begin{align*}
\mathcal{S}_0:&=\{\bm{x} \in\mathbb{R}^n\mid (\sigma(\bm{m}_j),\bm{x})=0\,(1\le{}^\forall j\le t)\}
\\
&=\mathbb{R}[\bm{m}'_{t+1},\dots,\bm{m}'_n]
\quad\text{by \eqref{eq73}},
\end{align*}
hence we have 
\begin{equation}\label{eq77}
\mathfrak{D}\subset\mathcal{S}=\bm{x}_0+\mathcal{S}_0=\bm{x}_0+\mathbb{R}[\bm{m}'_{t+1},\dots,\bm{m}'_n]
\end{equation} 
for $\bm{x}_0\in\mathcal{S}\setminus\mathcal{S}_0$. 
Thus, by $(\sigma(\bm{m}_j),\bm{x})=
(\sigma(\bm{m}_j)\sigma(M)^{-1},\bm{x}\,{}^t\sigma(M))$ we see
\begin{align*}
\mathcal{S}_0{}^t\hspace{-0.3mm}\sigma(M)& =\{\bm{x}\in\mathbb{R}^n\mid x_1=\dots=x_t=0\},
\\
\mathcal{S}\,{}^t\hspace{-0.3mm}\sigma(M) &=\{\bm{x}\in\mathbb{R}^n\mid x_j=k_j\,\,(1\le {}^\forall j \le t)\}
,
\end{align*}
hence
\begin{equation}\label{eq78}
\mathfrak{D}\,{}^t\sigma(M)
\subset
\bm{x}_0{}^t\sigma(M)+\{\bm{x}\in\mathbb{R}^n \mid x_1=\dots=x_t=0\}.
\end{equation}

Next, let the vectors $\bm{f}_i\in\mathbb{R}^n$ $(i=t+1,\dots,n)$ be  an orthonormal basis $\bm{f}_i$ of $\mathcal{S}_0$,
hence there is a matrix $C^{(n-t,n-t)}$ such that
\begin{gather*}
\begin{pmatrix}
\bm{m}'_{t+1}\\\vdots\\\bm{m}'_n    
\end{pmatrix}
=C
\begin{pmatrix}
 \bm{f}_{t+1}\\\vdots\\\bm{f}_n  
\end{pmatrix},
\end{gather*}
therefore
$$
((\bm{m}'_i,\bm{m}'_j))_{t+1\le i,j\le n}
=C\,{}^tC.
$$
Let us see  $$|\det C|=\sqrt{\det M_1}.$$
By \eqref{eq73},  we see that the above matrix is the submatrix of
$$
((\bm{m}'_i,\bm{m}'_j))_{1\le i,j\le n}
={}^t\sigma(M)^{-1}\sigma(M)^{-1}=
(\sigma(M){}^t\sigma(M))^{-1},
$$
and writing
\begin{align*}
\tilde{M}:&=\sigma(M)\,\,{}^t(\sigma(M))
=M\,\,{}^tM
=
\left(
\begin{array}{cc}
    M_1^{(t)} & M_2^{(t,n-t)}
    \\
    {}^tM_2&M_4
    \end{array}
\right)
\\
&=
\left(
\begin{array}{cc}
    M_1^{(t)} & 0^{(t,n-t)}
    \\
    0&M_4-M_1^{-1}[M_2]
\end{array}
\right)
\left[\left(
\begin{array}{cc}
1_t &M_1^{-1}M_2
    \\
    0^{(n-t,t)}&1_{n-t}
\end{array}
\right)\right],   
\end{align*}
where $A[B]$ means ${}^tBAB$,
we have
\begin{align*}
\tilde{M}^{-1}&=
\left(
\begin{array}{cc}
    M_1^{-1} & 0
    \\
    0&(M_4-M_1^{-1}[M_2])^{-1}
\end{array}
\right)
\left[\left(
\begin{array}{cc}
1_t &0
    \\
  -{}^t(M_1^{-1}M_2)&1_{n-t}
\end{array}
\right)\right]
\\
&=
\left(
\begin{array}{cc}
    * & *
    \\
    *&(M_4-M_1^{-1}[M_2])^{-1}
\end{array}
\right),
\end{align*}
which implies
$C\,{}^tC=(M_4-M_1^{-1}[M_2])^{-1}$,
hence $|\det C|=\sqrt{\det M_1}$.

Next,
for the vector
\begin{align}\nonumber
 \bm{x}_0&:=(k_1,\dots,k_t)\begin{pmatrix}
     \bm{m}'_1\\\vdots\\\bm{m}'_t
 \end{pmatrix}   
 +
 (k_{1},\dots,k_t)M_1^{-1}M_2\begin{pmatrix}
     \bm{m}'_{t+1}\\\vdots\\\bm{m}'_n
 \end{pmatrix} 
\\\nonumber
\,&=(k_1,\dots,k_t)(
(1_t,0^{(t,n-t)})
+(0^{(t,t)},M_1^{-1}M_2))
){}^t\sigma(M)^{-1}
 \\\nonumber
 &=(k_1,\dots,k_t)M_1^{-1}(M_1,M_2)\,\,{}^t\sigma(M)^{-1}
 \\\nonumber
 &=(k_1,\dots,k_t)M_1^{-1}(1_t,0^{(t,n-t)})\tilde{M}\,\,{}^t\sigma(M)^{-1}
 \\ \label{eq79}
 &=(k_1,\dots,k_t)(M_1^{-1},0^{(t,n-t)})\sigma(M) 
 ,
\end{align}
let us see that it satisfies conditions
$$
(\bm{x}_0,\mathcal{S}_0)=0\text{ and  }
(\bm{x}_0,\sigma(\bm{m}_j))=k_j\quad(j=1,\dots,t).
$$
The second property follows from the first line of the definition of $\bm{x}_0$.
Let us see $(x_0,\mathcal{S}_0)=0$,  which is equivalent to
$$
\bm{x}_0(   {}^t  \bm{m}'_{t+1},\dots,{}^t\bm{m}'_n)=0^{(1,n-t)}, 
$$
which follows from 
\begin{align*}
&\bm{x}_0(   {}^t  \bm{m}'_{t+1},\dots,{}^t\bm{m}'_n)
\\
=\,
&(k_1,\dots,k_t)(M_1^{-1},0^{(t,n-t)})\sigma(M)
(   {}^t  \bm{m}'_{t+1},\dots,{}^t\bm{m}'_n)
\\
=\,&
(k_1,\dots,k_t)(M_1^{-1},0^{(t,n-t)})
\begin{pmatrix}
0^{(t,n-t)} \\ 1_{n-t}    
\end{pmatrix}
\\
=\,\,&0^{(1,n-t)},
\end{align*}

Let vectors $\{\bm{u}_i\}$ be vertexes of $\mathfrak{D}$, hence $\mathfrak{D} =\{\sum c_i\bm{u}_i\mid c_i\ge0,\sum c_i=1\}$, and
in view of \eqref{eq77} we define the bijective mapping $\psi$ from the set
$$
X:=\{(x_{t+1},\dots,x_n)\in\mathbb{R}^{n-t}\mid  \bm{x}_0+\sum_{j=t+1}^n x_j\bm{m}'_j\in\mathfrak{D}\}
$$
 to $\mathfrak{D}$  by
$$
\psi:\bm{x}=(x_{t+1},\dots,x_n)\mapsto\bm{u}=\bm{x}_0+\sum_{j=t+1}^n x_j\bm{m}'_j.
$$
Then we see that for $\bm{x} \in  X$
\begin{align}\label{eq80}
\psi(\bm{x})=\bm{u}
\Leftrightarrow
(\bm{u} - \bm{x}_0)\,{}^t\sigma(M)=(0^{(1,t)},\bm{x}),
\end{align}
and the vector $\bm{x}_i\in X$ satisfying $\psi(\bm{x}_i)=\bm{u}_i$ is the vertex of $X$. 
Writing 
\begin{align*}
\bm{k}:&=\bm{x}_0{}^t\sigma(M)=(k_1,\dots,k_t)(1_t,M_1^{-1}M_2)\quad \text{(by \eqref{eq79})},
\end{align*}
we have, for a vector $\bm{u}\in\mathfrak{D}$
$$
\bm{u}\,{}^t\sigma(M) - \bm{k} 
=(\bm{u} - \bm{x}_0)\,{}^t\sigma(M)
=(\underbrace{0,\dots,0}_t,*,\dots,*).
$$

\noindent
Thus for  a function  $F$ on  $\mathfrak{D}=\Delta\cap\mathcal{S}$,  we see 
\begin{align*}
&\int_{\mathfrak{D} }F(\bm{x})d\bm{x}
\\
=\,&\int_{\bm{x}_0+\sum_{j=t+1}^n y_j\bm{f}_j\in\mathfrak{D} }F(\bm{x}_0+\sum_{j=t+1}^n y_j\bm{f}_j)dy_{t+1}\dots dy_n
\\
=\,&
\int F(\bm{x}_0+(y_{t+1},\dots,y_n)C^{-1}
\begin{pmatrix}
\bm{m}'_{t+1}\\\vdots\\\bm{m}'_n    
\end{pmatrix})dy_{t+1}\dots dy_n
\\
=\,&\sqrt{\det(M_1)}
\int_{(x_{t+1},\dots,x_n)\in X} F(\bm{x}_0+\sum_{j=t+1}^n x_j\bm{m}'_j)dx_{t+1}\dots dx_n,
\end{align*}
where  $\bm{m}'_j$ denotes $j$th row of the matrix ${}^t\sigma(M)^{-1}$ and $M_1=((\bm{m}_i,\bm{m}_j))_{1\le i,j\le t}$ as before.

\noindent
Since the set $X$ is the convex hull, it is divided to the union of simplexes with vertexes being ones of $X$.

Next, let us take a trigonometric function $e((\bm{x},\bm{a}))$ for $\bm{a}\in\mathbb{R}^n$ as $F(\bm{x})$ in the above,
where
$$
e(x):=\exp(2\pi ix),
$$
then we see
\begin{align*}
&e((\bm{x}_0+\sum_{j=t+1}^nx_j\bm{m}'_j,\bm{a}))
\\
=\,&e((\bm{x}_0,\bm{a}))e(\sum_{j=t+1}^nx_j(\bm{m}'_j,\bm{a}))
\\
=\,&e((\bm{x}_0,\bm{a}))e((\tilde{x},\tilde{a})),
\end{align*}
where $\tilde{x}:=(x_{t+1},\dots,x_n),\tilde{a}=((\bm{m}'_{t+1},\bm{a}),\dots,(\bm{m}'_n,\bm{a}))$.
Thus we see
\begin{align} \nonumber
&\int_{\mathfrak{D}} e((\bm{x},\bm{a}))d\bm{x}
\\ \label{eq81}
=\,&\sqrt{\det(M_1)}\,\,e((\bm{x}_0,\bm{a}))\int_{\tilde{x}\in X} e((\tilde{x},\tilde{a}))dx_{t+1}\dots dx_n.
\end{align}
The integral is calculated in the next subsection.

\subsection{Calculation of the integral}
We want to evaluate the integral \eqref{eq81}.
So, in the following, it may be supposed that $m=n-t$ and 
\begin{align*}
(\#)\left\{
\begin{array}{ll}
\text{    
$\mathcal{D}$ is $X$, which is not a simplex in general,}
\\
\text{ and the vector $\bm{a}$ is $\tilde{a}=((\bm{m}'_{t+1},\bm{a}),\dots,(\bm{m}'_n,\bm{a}))$.
}
\end{array}
\right.
\end{align*}

Let 
\begin{align*}
\mathcal{D}:&=\left\{\sum_{i=1}^m x_i(\bm{v}_i-\bm{v}_{m+1}) + \bm{v}_{m+1}\mid {}^\forall x_i\ge0,
\sum_{i=1}^m x_i\le1\right\}
\\
&=\left\{
\sum_{i=1}^{m+1}y_i\bm{v}_i\mid {}^\forall y_i\ge0,\sum_{i=1}^{m+1} y_i =1
\right\}
\end{align*}
be an $m$-dimensional  simplex,
where  $m+1$ vectors $\bm{v}_1,\dots,\bm{v}_{m+1}$ are in $\mathbb{R}^m$ and $m$ vectors 
$\bm{v}_1-\bm{v}_{m+1},\dots,\bm{v}_m-\bm{v}_{m+1}$ are linearly independent.
Then, writing  $V:=\left(\begin{array}{c}\bm{v}_1-\bm{v}_{m+1}\\\vdots\\\bm{v}_m-\bm{v}_{m+1}\end{array}\right)$,
 we have
\begin{equation}\label{eq82}
\int_{(y_1,\dots,y_m)\in\mathcal{D}} 1\,dy_1\dots dy_m
=\frac{|\det(V)|}{m!},    
\end{equation}
because
putting $(y_1,\dots,y_m)=(x_1,\dots,x_m)V+\bm{v}_{m+1}$, we see that the integral is
equal to
\begin{align*}
&|\det(V)|\int_{x_i\ge0,\sum x_i\le1}dx_1\dots dx_m
\\
\intertext{
transforming from $x_i$ to $t_i:=x_1+\dots+x_t$
\,$(0\le t_1\le\dots\le t_m\le1)$
}
=&|\det(V)|vol(\Delta)
\\
=&\frac{|\det(V)|}{m!} 
\sum_{\sigma\in S_m}vol(\sigma(\Delta))
\\
=&\frac{|\det(V)|}{m!}
\end{align*}
and  writing $\bm{a}:=(a_1,\dots,a_m)$ 
\begin{align}\nonumber
&\int_{\mathcal{D}} e(\sum_{i=1}^m a_iy_i)dy_1\dots dy_m
\\ \label{eq83}
=\,&e((\bm{a},\bm{v}_{m+1}))|\det(V)|\int_{x_i\ge0\atop \sum x_i\le1}e((x_1,\dots,x_m)V\,{}^t\!\bm{a})dx_1\dots dx_m.
\end{align}
To calculate \eqref{eq83}, we have only to put $\alpha_i:=(\bm{v}_i,\bm{a})$ in the following lemma.
It is clear that
$$
\det(V)=\det
\left(
\begin{array}{ll}
\bm{v}_1&1
\\
\vdots&\vdots
\\
\bm{v}_m&1
\\
\bm{v}_{m+1}&1
\end{array}
\right).
$$


\begin{lem}\label{lem6.2}
Let $m$ be a positive integer.
For mutually distinct real numbers $\alpha_1,\dots,\alpha_{m+1}$, write $A_i:=\alpha_i-\alpha_{m+1}$ $(i=1,\dots,m)$, we have
\begin{align*}
&\int_{x_i\ge0,\sum x_i\le1} e(\sum_{i=1}^m A_i x_i)dx_1\dots dx_m
\\
=\,&\frac{1}{(2\pi i)^m}\left\{\sum_{i=1}^{m}
\frac{e(A_i)}{A_i\prod_{j=1\atop j\ne i}^{m}(A_i-A_j)}+\frac{1}{\prod_{j=1}^m(-A_j)}\right\},
\end{align*}
which is equal to
\begin{align*}
&\frac{1}{(2\pi i)^m}\sum_{i=1}^{m+1}
\frac{e(\alpha_i-\alpha_{m+1})}{\prod_{j=1\atop j\ne i}^{m+1}(\alpha_i-\alpha_j)}
\\
=&
\frac{(-1)^{m}e(-\alpha_{m+1})}{(2\pi i)^m}
\frac{
\left|
\begin{array}{cccc}
e(\alpha_1)&e(\alpha_2)&\dots&e(\alpha_{m+1})   \\
1         &    1       &\dots& 1                      \\
\alpha_1   &    \alpha_2  & \dots &\alpha_{m+1}     \\
\vdots&\vdots&\vdots&\vdots   \\
\alpha_1^{m-1}   &    \alpha_2^{m-1}  & \dots &\alpha_{m+1}^{m-1}  
\end{array}
\right|}
{
\left|
\begin{array}{cccc}
1         &    1       &\dots& 1                      \\
\alpha_1   &    \alpha_2  & \dots &\alpha_{m+1}     \\
\vdots&\vdots&\vdots&\vdots   \\
\alpha_1^{m}   &    \alpha_2^{m}  & \dots &\alpha_{m+1}^{m}  
\end{array}
\right|}
\end{align*}
and  the integral vanishes if all $A_i$ are integers.
Moreover, we have
\begin{equation*}
\sum_{i=1}^{m}
\frac{1}{A_i\prod_{j=1\atop j\ne i}^{m}(A_i-A_j)}+\frac{1}{\prod_{j=1}^m(-A_j)}=0,
\end{equation*}
i.e.
\begin{equation*}
\sum_{i=1}^{m+1}\frac{1}{\prod_{j=1\atop j\ne i}^{m+1}(\alpha_i-\alpha_j)}=0,
\end{equation*}
and for a polynomial $g(x) =c_{m-1}x^{m-1}+c_{m-2}x^{m-2}+\dots$
\begin{equation}\label{eq84}
\sum_{i=1}^{m}\frac{g(A_i)}{\prod_{j=1\atop j\ne i}^{m}(A_i-A_j)}=c_{m-1}.
\end{equation}
\end{lem}
\proof
{
Denote the integral by $S(A_1,\dots,A_m)$.
We use  the induction on $m$. 
It is easy to see that $S(A_1) = \frac{e(A_1)-1}{2\pi iA_1}$
and  $S(A_1,\dots,A_m)$ is equal to
\begin{align*} 
&\int\limits_{x_1,\dots,x_{m-1}\ge0,\atop\sum_{i=1}^{m-1}x_i\le1}
\left(\int\limits_{0\le x_m\le1-x_1-\dots-x_{m-1}}e(\sum A_ix_i)dx_m\right)
dx_1\dots dx_{m-1}
\\
=\,
&\int\limits_{x_1,\dots,x_{m-1}\ge0,\atop\sum_{i=1}^{m-1}x_i\le1}
\frac{e(\sum_{i=1}^{m-1}A_ix_i)}{2\pi i A_m}
\left(
e(A_m(1-\sum_{i=1}^{m-1}x_i))-1
\right)dx_1\dots dx_{m-1}
\\
=\,&
\frac{e(A_m)}{2\pi iA_m}S(A_1-A_m,\dots,A_{m-1}-A_m)-
\frac{S(A_1,\dots,A_{m-1})}{2\pi iA_m}.
\end{align*}
The assumption of the induction completes easily the proof of the first equality.
W.r.t. the transform to $\alpha_i$, we have only to use the Vandermonde determinant:
$$
\det\left(
\begin{array}{cccc}
1         &    1       &\dots& 1                      \\
\alpha_1   &    \alpha_2  & \dots &\alpha_{m}     \\
\vdots&\vdots&\vdots&\vdots   \\
\alpha_1^{m-1}   &    \alpha_2^{m-1}  & \dots &\alpha_{m}^{m-1}  
\end{array}
\right)
=(-1)^{m(m-1)/2}\prod_{i<j}(\alpha_i-\alpha_j).
$$
For the second identities, replace $A_i$ by $\epsilon A_i$;
then we have 
\begin{align*}
&\int_{x_i\ge0,\sum x_i\le1} e(\sum_{i=1}^m \epsilon A_i x_i)dx_1\dots dx_m
\\
=\,&\frac{1}{(2\pi i\epsilon )^m}\left\{\sum_{i=1}^{m}
\frac{e(\epsilon A_i)}{A_i\prod_{j=1\atop j\ne i}^{m}(A_i-A_j)}+\frac{1}{\prod_{j=1}^m(-A_j)}\right\}
\\
=\,&\frac{1}{(2\pi i\epsilon )^m}\left\{
\sum_{i=1}^{m}
\frac{1}{A_i\prod_{j=1\atop j\ne i}^{m}(A_i-A_j)}+\frac{1}{\prod_{j=1}^m(-A_j)}\right.
\\
&+\left.\sum_{k=1}^\infty\frac{(2\pi i\epsilon)^k}{k!}
\sum_{i=1}^{m}
\frac{ A_i^{k-1}}{\prod_{j=1\atop j\ne i}^{m}(A_i-A_j)}
\right\}
\\
&
\to\int_{x_i\ge0,\sum x_i\le1} 1dx_1\dots dx_m
=\frac{1}{m!}\quad(\epsilon\to0).
\end{align*}
Therefore the coefficient of $\epsilon^{k-m}$ vanishes if $k-m<0$ and is $\frac{1}{m!}$ if $k-m=0$.

\qed
}

\begin{cor}
Let $\mathcal{D}=\{\sum_{i=1}^py_i\bm{v}_i\mid {}^\forall y_i\ge0,\sum_{i=1}^py_i=1\}$ be the convex hull of vectors $\bm{v}_i\in\mathbb{R}^m$. Then for $\bm{a}\in\mathbb{R}^m$, we have 
$$
\int_{\bm{y}\in \mathcal{D}} e((\bm{y},\bm{a}))dy_1\dots dy_m=0
$$
if all $(\bm{v}_i,\bm{a})$ are mutually distinct and differences
$(\bm{v}_i,\bm{a})-(\bm{v}_j,\bm{a})$ are integers.
\end{cor}
\proof{
We have only to divide the convex hull $\mathcal{D}$ to the union of simplexes with vertexes being $\bm{v}_i$.
\qed
}

Let us apply the corollary to the case $(\#)$; then the set of $\bm{v}_i$ is the set of the vertex $\bm{x}_i$ of $X$, and $\bm{a}$ is 
$\tilde{\bm{a}}=((\bm{m}'_{t+1},\bm{a}),\dots,(\bm{m}'_n,\bm{a}))$.
Hence the sufficient condition to the vanishment of the integral \eqref{eq81} is equivalent
to
\begin{quote}
all $(\bm{x}_i,\tilde{\bm{a}})$ are mutually distinct  and  
all $(\bm{x}_i -\bm{x}_j,\tilde{\bm{a}})$ are integers.
\end{quote}
Do functions $e((\bm{x},\bm{a}))$ for vectors $\bm{a}$ satisfying the above conditions
approximate the continuous functions on $\mathfrak{D}(f,\sigma,\{k_i\})$ ?

We note that for $\bm{u}=\psi(\bm{x})$ $(\bm{x}=(x_{t+1},\dots,x_n)\in X)$ and $\bm{a}=\sum_{i=1}^n c_i\sigma(\bm{m}_i)
\linebreak[3]\in\mathbb{R}^n$ $(c_i=(\bm{a},\bm{m}'_i))$
\begin{align*}
(\bm{u},\bm{a})&=(\bm{x}_0+\sum_{j=t+1}^n x_j\bm{m}'_j, \sum_{i=1}^n c_i\sigma(\bm{m}_i))
\\
&=(\bm{x}_0,\bm{a})
+\sum_{j=t+1}^nx_jc_j
\\
&=(\bm{x}_0,\bm{a})+(\bm{x},\tilde{\bm{a}}),
\end{align*}
where $\tilde{\bm{a}}=((\bm{m}'_{t+1},\bm{a}),\dots,(\bm{m}'_n,\bm{a}))$ as before.

\vspace{3mm}

\noindent
{\bf Example 1}
Suppose that the polynomial $f(x)$ has no non-trivial linear relations among roots, i.e. $t=1$ and moreover $k=1$.
Then we  can take $M$ as
\begin{gather*}
\sigma(M)=
\begin{pmatrix}
1&1&1&\dots&1
\\
0&1&0&\dots&0
\\
\vdots&\vdots&\ddots&\vdots&\vdots
\\
\vdots&\vdots&\ddots&1&0
\\0&\dots&\dots&0&1
\end{pmatrix},
\\
\begin{pmatrix}
\bm{m}'_1\\ \vdots \\ \bm{m}'_n    
\end{pmatrix}    
={}^t\sigma(M)^{-1}=
\begin{pmatrix}
1&0&0&\dots&0
\\
-1&1&0&\dots&0
\\
\vdots&\vdots&\ddots&\vdots&\vdots
\\
\vdots&\vdots&\ddots&1&0
\\-1&\dots&\dots&0&1
\end{pmatrix}.
\end{gather*}
Hence we see
\begin{align*}
\tilde{M}=\sigma(M)\,{}^t\sigma(M)=
\left(
\begin{array}{cc}
    M_1^{(1)} & M_2^{(1,n-1)}
    \\
    {}^tM_2&M_4
    \end{array}
\right)
\end{align*}
with
$$
M_1=n,M_2=(1,\dots,1),M_4=1_{n-1},
$$
therefore $\bm{x}_0(:=(k_1,\dots,k_t)(M_1^{-1},0^{(t,n-t)})\sigma(M)) =(\frac{1}{n},\dots,\frac{1}{n})$.
Then vertexes of $\mathfrak{D}$ are 
$$
\bm{u}_q:=(\underbrace{0,\dots,0}_{n-q},\underbrace{\frac{1}{q},\dots,\frac{1}{q}}_{q})\quad(1\le q\le n),
$$
and through the equation 
$$
\bm{u}=\bm{x}_0+\sum_{i=2}^n x_{i}\bm{m}'_i
$$
by writing
$$
\bm{u}_q=\bm{x}_0+\sum_{i=2}^n x_{q,i}\bm{m}'_i
=\bm{x}_0+(-\sum_{i=2}^n x_{q,i},x_{q,2},\dots,x_{q,n}),
$$
the vertex $\bm{x}_q$ of $X$ corresponding to $\bm{u}_q$
is $(x_{q,2},\dots,x_{q,n})$.
Hence, we see
\begin{align*}
\bm{u}_q-\bm{x}_0&=
(\underbrace{0,\dots,0}_{n-q},\underbrace{\frac{1}{q},\dots,\frac{1}{q}}_{q})
-
(\frac{1}{n},\dots,\frac{1}{n})
\\
&=
(-\sum_{i=2}^n x_{q,i},x_{q,2},\dots,x_{q,n})
\end{align*}
and
\begin{align*}
V^{(n-1,n-1)}:&=\begin{pmatrix}
\bm{x}_{n-1}-\bm{x}_n\\\vdots\\\bm{x}_1-\bm{x}_n    
\end{pmatrix}
\\
&
=
\begin{pmatrix}
    \frac{1}{n-1} &\dots&\dots&\frac{1}{n-1}
   \\
    0&\frac{1}{n-2}\dots&\dots&\frac{1}{n-2}
    \\
    \vdots&\vdots&\ddots&\vdots
    \\
    0&\dots&0&1
\end{pmatrix}
-
\begin{pmatrix}
    \frac{1}{n}&\dots&\frac{1}{n}
    \\
    \vdots&\vdots&\vdots
    \\
    \frac{1}{n}&\dots&\frac{1}{n}    
\end{pmatrix},
\end{align*}

whose inverse matrix $V^{-1}$ is
$$
\begin{pmatrix}
 2(n-1) &  -(n-2)&    0& \dots &\dots&\dots& 0
 \\
 n-1    &  n-2 &-(n-3)&0&\dots&\dots& 0
  \\
 \vdots&  0    &n-3 & -(n-4)&0&\dots &0
 \\
  \vdots& \vdots& \vdots& \vdots& \vdots& \vdots&\vdots
 \\
\vdots &0&\dots&\dots& 0&2&-1
 \\
 n-1&0&\dots&\dots&\dots&0&1
\end{pmatrix}.
$$
Thus a vector $\tilde{\bm{a}}$ satisfying $(\bm{x}_i-\bm{x}_j,\tilde{\bm{a}})\in\mathbb{Z}$ is a linear combination of column vectors of $V^{-1}$
over $\mathbb{Z}$.
Hence, writing 
$\bm{w}_{n-1}:=(n-1)\cdot(2,1,\dots,1)$, $\bm{w}_{n-2}:=(n-2)\cdot(-1,1,0,\dots,0)$,
$\bm{w}_{n-3}:=(n-3)\cdot(0,-1,1,0,\dots,0)$, $\dots$, $\bm{w}_2:=2\cdot(0,\dots,0,-1,1,0)$, $\bm{w}_1:=1\cdot(0,\dots,0,-1,1)$ over $\mathbb{Z}$,
we have the equivalence
$$
(\bm{x}_i-\bm{x}_j,\tilde{\bm{a}})\in\mathbb{Z}\text{ for }({}^\forall i,j)
\Leftrightarrow
\tilde{\bm{a}}=\sum_{i=1}^{n-1}c_i\bm{w}_i \text{ for }c_i\in\mathbb{Z}.
$$

\noindent
{\bfseries{Example\,2}} Suppose that the polynomial $f(x)$ is of degree $4$ and has a non-trivial linear relations among roots; Then  we may assume that their basis are $\alpha_1+\alpha_4,\alpha_2+\alpha_3\in\mathbb{Z}$ by Proposition \ref{prop3.3},
i.e.  $\bm{m}_1=(1,0,0,1),\bm{m}_2=(0,1,1,0)$.
Hence we may take $\bm{m}_3:=(0,0,1,0),\bm{m}_4:=(0,0,0,1)$; then we have
$$
M^{-1}=\begin{pmatrix}
1&0&0&-1\\
0&1&-1&0\\
0&0&1&0\\
0&0&0&1
\end{pmatrix}
=({}^t\bm{m}'_1,\dots,{}^t\bm{m}'_4).
$$
Thus, we see that
\begin{align*}
\mathcal{S}&=\{ \bm{x}\in\mathbb{R}^4\mid x_1+x_4=k_1,x_2+x_3=k_2\} ,
\\
\bm{x}_0&=k_1\bm{m}'_1+k_2\bm{m}'_2=(\frac{k_1}{2},\frac{k_2}{2},\frac{k_2}{2},\frac{k_1}{2}),
\\
\mathcal{S}_0&=\mathbb{R}[\bm{m}'_3,\bm{m}'_4]
=\{\bm{x}\in\mathbb{R}^4\mid x_1+x_4=x_2+x_3=0\}.
\end{align*}
The dimension of $\mathfrak{D}=2$ is $2$ if and only if $k_1 = k_2 = 1,$ and then the vertexes are $\bm{u}_1:=(0,0,1,1),\bm{u}_2:=(0,\frac{1}{2},\frac{1}{2},1),\bm{u}_3:=(\frac{1}{2},\frac{1}{2},\frac{1}{2},\frac{1}{2})$.
With respect to the mapping $\psi$,
we see $\psi((x_{3},x_4))=(\frac{1}{2},\frac{1}{2},\frac{1}{2},\frac{1}{2})+(-x_4,-x_3,x_3,x_4)$,
hence
$\bm{u}_1=\psi((\frac{1}{2},\frac{1}{2})),
\bm{u}_2=\psi((0,\frac{1}{2})),\bm{u}_3=\psi((0,0))$.
Thus the vertexes of $X$ are 
$$
\bm{x}_1=(\frac{1}{2},\frac{1}{2}),\,
\bm{x}_2=(0,\frac{1}{2}),\,\bm{x}_3=(0,0).
$$
It is clear that 
$$
(\bm{x}_i - \bm{x}_j,\tilde{\bm{a}})\in\mathbb{Z} \text{  for }
{}^\forall i,j\Leftrightarrow 
\tilde{\bm{a}}\in 2\mathbb{Z}.
$$
For $\bm{a}=(a_1,\dots,a_4)$, we see that
$\tilde{\bm{a}}=((\bm{m}'_{3},\bm{a}),(\bm{m}'_4,\bm{a}))=(-a_2+a_3,-a_1+a_4) =(\tilde{a}_1,\tilde{a}_2)$, say, 
$(\bm{x}_1,\tilde{\bm{a}})=\frac{1}{2}(\tilde{a}_1+\tilde{a}_2),
(\bm{x}_2,\tilde{\bm{a}})=\frac{1}{2}\tilde{a}_2,
(\bm{x}_3,\tilde{\bm{a}})=0$, and they are mutually distinct if and only if 
$$
\tilde{a}_1\ne 0,\,\tilde{a}_2\ne 0\text{ and }
\tilde{a}_1+\tilde{a}_2\ne 0.
$$
Is any continuous function on $X$ approximated by $e((\bm{x},\tilde{\bm{a}}))$ satisfying
the above two conditions?
\begin{prop}\label{prop6.4}
Let $\alpha_1,\dots\alpha_{m+1}\in\mathbb{R}$ and let $a_1,\dots,a_u$ be their distinct numbers,
i.e. $\{\alpha_i\mid1\le i\le{m+1}\}=\{ a_i\mid 1\le i \le u \}$
and $C_J:=\{j\mid \alpha_j=a_J\,(1\le j\le m+1)\}$.
Then we have
\begin{align*}
&\int_{x_i\ge0,\sum x_i\le1} e(\sum_{i=1}^m(\alpha_i-\alpha_{m+1})x_i)dx_1\dots dx_m
\\
=&\frac{1}{(2\pi i )^m}\sum_{J=1}^u \frac{e(a_J-\alpha_{m+1})}{\prod\limits_{ l\notin C_J}
(a_J-\alpha_l)}
\Bigl\{
\sum_{n_0,n_d\ge0\,(d\not\in C_J),\atop n_0+\sum\limits_{d\not\in C_J}n_d=\#C_J-1}
\frac
{
(2\pi i)^{n_0}
}
{
\prod_{d\not\in C_J}{(\alpha_d-a_J)^{n_d}}
}
\Bigr\},
\end{align*} 
where the indexes $l,d$ run over $\{1,\dots,m+1\}\setminus C_J$.
\end{prop}
\proof
{
Let $\delta_i$ be mutually distinct numbers and put $\beta_i := \alpha_i+\epsilon\delta_i$;
then we see that $\beta_i-\beta_j=\alpha_i-\alpha_j+(\delta_i-\delta_j)\epsilon$ is not $0$ 
if $i\ne j$ and the non-zero number $\epsilon$ is sufficiently close to $0$.
By Lemma \ref{lem6.2}, we have
\begin{align*}
&\int_{x_i\ge0,\sum x_i\le1} e(\sum_{i=1}^m(\alpha_i-\alpha_{m+1}))dx_1\dots dx_m
\\=&\lim_{\epsilon\to0}\int_{x_i\ge0,\sum x_i\le1} e(\sum_{i=1}^m(\beta_i-\beta_{m+1}))dx_1\dots dx_m
\\
&=\frac{(-1)^m}{(2\pi i)^m}\lim_{\epsilon\to0}\frac{N(\epsilon)}{D(\epsilon)},
\end{align*} 
where
$$
N(\epsilon)=\left|
\begin{array}{cccc}
e(\beta_1-\beta_{m+1})&e(\beta_2-\beta_{m+1})&\dots&e(\beta_{m+1}-\beta_{m+1})   \\
1         &    1       &\dots& 1                      \\
\beta_1   &    \beta_2  & \dots &\beta_{m+1}     \\
\vdots&\vdots&\vdots&\vdots   \\
\beta_1^{m-1}   &    \beta_2^{m-1}  & \dots &\beta_{m+1}^{m-1}  
\end{array}
\right|
$$
and
$$
D(\epsilon)=
\left|
\begin{array}{cccc}
1         &    1       &\dots& 1                      \\
\beta_1   &    \beta_2  & \dots &\beta_{m+1}     \\
\vdots&\vdots&\vdots&\vdots   \\
\beta_1^{m}   &    \beta_2^{m}  & \dots &\beta_{m+1}^{m}  
\end{array}
\right|.
$$
Hereafter the notations $\dot{\sum}_{k,l},\dot{\prod}_{k,l}$ mean the sum, the product on $k,l$ under the condition $1\le k < l \le m+1$. 
First, we see  that
\begin{align*}
D(\epsilon)&=(-1)^{(m+1)m/2}\underset{k,l}{\dot{\prod}}(\beta_k-\beta_l)
\\
&=(-1)^{(m+1)m/2}\underset{k,l,\alpha_k\ne\alpha_l}{\dot{\prod}}\{(\alpha_k-\alpha_l)
(1+\frac{\delta_k-\delta_l}{\alpha_k-\alpha_l}\epsilon)\}
\underset{k,l,\alpha_k=\alpha_l}{\dot{\prod}}\{(\delta_k-\delta_l)\epsilon\}
\\
&=\Bigl\{
(-1)^{(m+1)m/2}\underset{k,l,\alpha_k\ne\alpha_l}{\dot{\prod}}(\alpha_k-\alpha_l)\cdot
\underset{k,l,\alpha_k=\alpha_l}{\dot{\prod}}(\delta_k-\delta_l)
\Bigr\}
\epsilon^\kappa+O(\epsilon^{\kappa+1}),
\end{align*}
where $\kappa:=\#\{(k,l)\mid1\le k<l\le m+1,\alpha_k=\alpha_l\}$ and 
$O(\epsilon^{\kappa+1})$ means a power series whose coefficient of $\epsilon^{q}$ is 
$0$ if $q<\kappa+1$.

Next, we see that
\begin{align*}
N(\epsilon)&=\sum_{j=1}^{m+1}(-1)^{j-1}e(\beta_j-\beta_{m+1})\cdot(-1)^{m(m-1)/2}
\underset{k,l\ne j}{\dot{\prod}}(\beta_k-\beta_l)
\\
=&
(-1)^{m(m-1)/2}\sum_{j=1}^{m+1}(-1)^{j-1}e(\alpha_j-\alpha_{m+1})\times
\\
&\underset{k,l\ne j,\alpha_k\ne\alpha_l}{\dot{\prod}}(\alpha_k-\alpha_l)\cdot
\underset{k,l\ne j,\alpha_k=\alpha_l}{\dot{\prod}}(\delta_k-\delta_l)
\times
\\
&
e((\delta_j-\delta_{m+1})\epsilon)
\cdot
\underset{k,l\ne j,\alpha_k\ne\alpha_l}{\dot{\prod}}(1+\frac{\delta_k-\delta_l}{\alpha_k-\alpha_l}\epsilon)
\cdot
\epsilon^{\kappa_j}
\end{align*}
where $\kappa_j:=\#\{(k,l)\mid1\le k<l\le m+1,k,l\ne j,\alpha_k=\alpha_l\}$.
Here we note that $\kappa_j$ depends only on $C_J$ containing $j$  by the equation 
$\kappa-\kappa_j=\#\{(k,l)\mid 1\le k <l\le m+1,\alpha_k=\alpha_l,k\text{ or }l=j\}
=\#C_J-1$.
By
\begin{align*}
&(-1)^{j-1}\underset{k,l\ne j,\alpha_k\ne\alpha_l}{\dot{\prod}}(\alpha_k-\alpha_l)\cdot
\underset{k,l\ne j,\alpha_k=\alpha_l}{\dot{\prod}}(\delta_k-\delta_l)
\\
=&\frac{\underset{k,l,\alpha_k\ne\alpha_l}{\dot{\prod}}(\alpha_k-\alpha_l)\cdot
\underset{k,l,\alpha_k=\alpha_l}{\dot{\prod}}(\delta_k-\delta_l)
}
{
\underset{l,\alpha_l\ne\alpha_j}{{\prod}}(\alpha_j-\alpha_l)\cdot
\underset{l\ne j,\alpha_l=\alpha_j}{{\prod}}(\delta_j-\delta_l)
},
\end{align*}
we have
\begin{align*}
&N(\epsilon)
\\
=& (-1)^{m(m-1)/2}\underset{k,l,\alpha_k\ne\alpha_l}{\dot{\prod}}(\alpha_k-\alpha_l)\cdot
\underset{k,l,\alpha_k=\alpha_l}{\dot{\prod}}(\delta_k-\delta_l)\times
\\
&
\sum_{J=1}^u e(a_J-\alpha_{m+1}) \sum_{j\in C_J}
\frac{1}
{
\underset{l,\alpha_l\ne\alpha_j}{{\prod}}(\alpha_j-\alpha_l)\cdot
\underset{l\ne j,\alpha_l=\alpha_j}{{\prod}}(\delta_j-\delta_l)
}\times
\\
&
e((\delta_j-\delta_{m+1})\epsilon)
\underset{k,l\ne j,\alpha_k\ne\alpha_l}{\dot{\prod}}(1+\frac{\delta_k-\delta_l}{\alpha_k-\alpha_l}\epsilon)
\cdot
\epsilon^{\kappa_j}
\\
=&(-1)^{m(m-1)/2}\underset{k,l,\alpha_k\ne\alpha_l}{\dot{\prod}}(\alpha_k-\alpha_l)\cdot
\underset{k,l,\alpha_k=\alpha_l}{\dot{\prod}}(\delta_k-\delta_l)\times
\\
&\sum_{J=1}^u \frac{e(a_J-\alpha_{m+1})}{\underset{l\not\in C_J}{{\prod}}(a_J-\alpha_l)}
\sum_{j\in C_J}
\frac{
e((\delta_j-\delta_{m+1})\epsilon)
}
{
\underset{l\in C_J,l\ne j}{{\prod}}(\delta_j-\delta_l)
}
\underset{k,l\ne j,\alpha_k\ne\alpha_l}{\dot{\prod}}(1+\frac{\delta_k-\delta_l}{\alpha_k-\alpha_l}\epsilon)
\cdot
\epsilon^{\kappa_j}.
\end{align*}
Here,  by the equality
\begin{align*}
&\{(k,l) \mid k,l\ne j,\alpha_k\ne\alpha_l\}
\\
=&\{(k,l) \mid k,l\not\in C_J,\alpha_k\ne\alpha_l\}
\\
&\cup\{(k,l) \mid k\not\in C_J,l\in C_J,l\ne j\}
\\
&\cup\{(k,l) \mid k\in C_J,k\ne j,l\not\in C_J\},
\end{align*}
the product $\underset{k,l\ne j,\alpha_k\ne\alpha_l}{\dot{\prod}}(1+\frac{\delta_k-\delta_l}{\alpha_k-\alpha_l}\epsilon)$ is equal to
\begin{align*}
&
\underset{k,l\not\in C_J,\alpha_k\ne\alpha_l}{\dot{\prod}}(1+\frac{\delta_k-\delta_l}{\alpha_k-\alpha_l}\epsilon)
\underset{k\not\in C_J,l\in C_J,l\ne j}{\dot{\prod}}(1+\frac{\delta_k-\delta_l}{\alpha_k-\alpha_l}\epsilon)
\times
\\
&\underset{k\in C_J,k\ne j,l\not\in C_J}{\dot{\prod}}(1+\frac{\delta_k-\delta_l}{\alpha_k-\alpha_l}\epsilon)
\\
=&
\underset{k,l\not\in C_J,\alpha_k\ne\alpha_l}{\dot{\prod}}(1+\frac{\delta_k-\delta_l}{\alpha_k-\alpha_l}\epsilon)
\underset{1\le c,d\le m+1, \atop c\in C_J,c\ne j,d\not\in C_J}{{\prod}}(1+\frac{\delta_c-\delta_d}{\alpha_c-\alpha_d}\epsilon)
\\
=&
\underset{k,l\not\in C_J,\alpha_k\ne\alpha_l}{\dot{\prod}}(1+\frac{\delta_k-\delta_l}{\alpha_k-\alpha_l}\epsilon)
\frac{
\underset{1\le c,d\le m+1, \atop c\in C_J,d\not\in C_J}{{\prod}}(1+\frac{\delta_c-\delta_d}{\alpha_c-\alpha_d}\epsilon)
}
{
\underset{d\not\in C_J}{{\prod}}(1+\frac{\delta_j-\delta_d}{\alpha_j-\alpha_d}\epsilon)
}
\end{align*}
hence
\begin{align*}
N(\epsilon)=& (-1)^{m(m-1)/2}\underset{k,l,\alpha_k\ne\alpha_l}{\dot{\prod}}(\alpha_k-\alpha_l)\cdot
\underset{k,l,\alpha_k=\alpha_l}{\dot{\prod}}(\delta_k-\delta_l)\times
\\
&
\sum_{J=1}^u 
\frac{e(a_J-\alpha_{m+1})}{\underset{l\not\in C_J}{{\prod}}(a_J-\alpha_l)}
\underset{k,l\not\in C_J,\alpha_k\ne\alpha_l}{\dot{\prod}}(1+\frac{\delta_k-\delta_l}{\alpha_k-\alpha_l}\epsilon)
\underset{1\le c,d\le m+1, \atop c\in C_J,d\not\in C_J}{{\prod}}(1+\frac{\delta_c-\delta_d}{\alpha_c-\alpha_d}\epsilon)
\times
\\
&\epsilon^{\kappa-\#C_J+1}
\sum_{j\in C_J}
\Bigl\{
\frac{
e((\delta_j-\delta_{m+1})\epsilon)
}
{
\underset{l\in C_J,l\ne j}{{\prod}}(\delta_j-\delta_l)
}
\frac{1}
{
\underset{1\le d\le m+1,\atop d\not\in C_J}{{\prod}}(1+\frac{\delta_j-\delta_d}{a_J-\alpha_d}\epsilon)
}
\Bigr\}.
\end{align*}
The  partial sum on $j\in C_J$  is
\begin{align*}
&\sum_{j\in C_J}
\frac{
e((\delta_j-\delta_{m+1})\epsilon)
}
{
\underset{l\in C_J,l\ne j}{{\prod}}(\delta_j-\delta_l)
}
\frac{1}
{
\underset{1\le d\le m+1,\atop d\not\in C_J}{{\prod}}(1+\frac{\delta_j-\delta_d}{a_J-\alpha_d}\epsilon)
}
\\
=&
\sum_{j\in C_J}
\frac{
e((\delta_j-\delta_{m+1})\epsilon)
}
{
\underset{l\in C_J,l\ne j}{{\prod}}(\delta_j-\delta_l)
}
\underset{1\le d\le m+1,\atop d\not\in C_J}{\prod} 
\Bigl\{
\sum_{q\ge0}(-\frac{\delta_j-\delta_d}{a_J-\alpha_d}\epsilon)^q
\Bigr\}
\\
=&
\underset{ n_0,n_d\ge0 \atop(1\le d\le m+1,d\not\in C_J)}{\sum}
\Bigl\{
\sum_{j\in C_J}
\frac{ (2\pi i(\delta_j-\delta_{m+1}))^{n_0}\underset{d\not\in C_J}{\prod}
\Bigl (\frac{\delta_j-\delta_d}{\alpha_d-a_J}\Bigr)^{n_d}}{ \underset{l\in C_J,l\ne j}{{\prod}}(\delta_j-\delta_l) }
\Bigr\}
 \epsilon^{n_0+\sum_{d\not\in C_J}n_d},
\end{align*}
applying \eqref{eq84} to the sum on $j\in C_J$ with $\delta_j-\delta_{m+1}$ for $A_i$ there,
which is equal to 
\begin{align*}
&\underset{n_0,n_d\ge0, d\not\in C_J ,\atop n_0+\sum_{d\not\in C_J}n_d =\#C_J-1}{\sum}
\frac{
 (2\pi i)^{n_0}
 }
 {
\underset{d\not\in C_J}{\prod}
{(\alpha_d-a_J)^{n_d}}
}\epsilon^{\#C_J-1} +O(\epsilon^{\#C_J}).
\end{align*}
Thus we have
\begin{align*}
N(\epsilon)=& (-1)^{m(m-1)/2}\underset{k,l,\alpha_k\ne\alpha_l}{\dot{\prod}}(\alpha_k-\alpha_l)\cdot
\underset{k,l,\alpha_k=\alpha_l}{\dot{\prod}}(\delta_k-\delta_l)\times
\\
&
\sum_{J=1}^u 
\frac{e(a_J-\alpha_{m+1})}{\underset{l\not\in C_J}{{\prod}}(a_J-\alpha_l)}
\Bigl\{
\underset{n_0,n_d\ge0, d\not\in C_J ,\atop n_0+\sum_{d\not\in C_J}n_d =\#C_J-1}{\sum}
\frac{
 (2\pi i)^{n_0}
 }
 {
\underset{d\not\in C_J}{\prod}
{(\alpha_d-a_J)^{n_d}}
}
\Bigr\}\epsilon^{\kappa}+
\\
&O(\epsilon^{\kappa+1}),
\end{align*}
which  completes the proof.
\qed}

\noindent
{\bf Example } In case of $m=2,\alpha_1\ne0,\alpha_2=\alpha_3=0$, the integral is
$\frac{1}{(2\pi i)^2}\{\frac{e(\alpha_1)-1}{\alpha_1^2} - \frac{2\pi i}{\alpha_1}\}$.

\section{Numerical data}\label{sec7}
In this section, the integer $X_m=X_m(f)$ denotes the  smallest prime number $p\in Spl(f)$ with $p>10^m$.
\subsection{Case without non-trivial linear relation among roots}
Let $f(x)$ be a  polynomial of degree $n$ without non-trivial linear relation among roots.
First  we check \eqref{eq14}, i.e.  \eqref{eq11} for a special domain $D_i=\{(x_1,\dots,x_n)\in[0,1)^n\mid x_i\le a\}$
 at $a=k/(10n)$.
 ``diff'' in the table means the maximum of the absolute value of the difference between $1$ 
and  the left-hand side divided by the right-hand side in \eqref{eq14} at
$a=k/(10n)$ $(k=1,\dots,10n)$, $i=1,\dots,n$ and $X=X_m$ for $m=5,\dots,11$.
``Diff'' means the  maximum of the absolute value of the difference
$$
\frac{\#Spl_{X_m}(f,\sigma,k,L,\{R_i\})}{\#Spl_{X_m}(f,\sigma,k)}\times\#\frak{R}(f,\sigma,k,L) -1
$$
for $\{R_i\}\in\frak{R}(f,\sigma,k,L) $.
\vspace{2mm}

\noindent
(1) $f(x)=x^2+1$
$$
\begin{array}{|c|c|c|c|c|c|c|c|c|}
\hline
m&5&6&7&8&9&10&11 \\
\hline
\text{diff}&0.04473& 0.01189&  0.00623&0.00120 & 0.00056& 0.00008& 0.00007
\\
\hline
\end{array}
$$
\begin{equation*}k=1,L=2\,:\,
\begin{array}{|c|c|c|c|c|c|c|}
\hline
m&5&6&7&8&9 \\
\hline
\text{Diff}&0.00501& 0.00336&  0.0004&0.00035 & 0.00027
\\
\hline
\end{array}
\end{equation*}
\begin{equation*}k=1,L=3\,:\,
\begin{array}{|c|c|c|c|c|c|c|}
\hline
m&5&6&7&8&9 \\
\hline
\text{Diff}&0.04724&0.01633& 0.00344&0.00143&0.00044
\\
\hline
\end{array}
\end{equation*}
\begin{equation*}k=1,L=4\,:\,
\begin{array}{|c|c|c|c|c|c|c|}
\hline
m&5&6&7&8&9 \\
\hline
\text{Diff}& 0.01505 &  0.00888&  0.00325 &   0.00111 &0.00054
\\
\hline
\end{array}
\end{equation*}
\begin{equation*}k=1,L=5\,:\,
\begin{array}{|c|c|c|c|c|c|c|}
\hline
m&5&6&7&8&9 \\
\hline
\text{Diff}& 0.08444& 0.03359 &  0.02004 & 0.00412  &  0.00230 
\\
\hline
\end{array}
\end{equation*}
\begin{equation*}k=1,L=6\,:\,
\begin{array}{|c|c|c|c|c|c|c|}
\hline
m&5&6&7&8&9 \\
\hline
\text{Diff}&0.10619  &  0.02705 &  0.00743 & 0.00263  &0.00148
\\
\hline
\end{array}
\end{equation*}
\begin{equation*}k=1,L=7\,:\,
\begin{array}{|c|c|c|c|c|c|c|}
\hline
m&5&6&7&8&9 \\
\hline
\text{Diff}&  0.15008&   0.05978&  0.02224 &   0.00725&0.00312
\\
\hline
\end{array}
\end{equation*}
\begin{equation*}k=1,L=8\,:\,
\begin{array}{|c|c|c|c|c|c|c|}
\hline
m&5&6&7&8&9 \\
\hline
\text{Diff}&  0.05351 &0.02266  &    0.01289&  0.00355  &0.00091
\\
\hline
\end{array}
\end{equation*}

\noindent
(2) $f(x)= x^3+2$
$$
\begin{array}{|c|c|c|c|c|c|c|c|c|}
\hline
m&5&6&7&8&9&10&11 \\
\hline
\text{diff}&  0.66667 &0.19435  &  0.07787 &  0.00455 &   0.00511&   0.00116& 0.00079
\\
\hline
\end{array}
$$
\begin{equation*}k=1,L=2\,:\,
\begin{array}{|c|c|c|c|c|c|c|}
\hline
m&5&6&7&8&9 \\
\hline
\text{Diff}&  0.15325&   0.02518   & 0.01088  &0.00153   &0.00085
\\
\hline
\end{array}
\end{equation*}

\begin{equation*}k=1,L=3\,:\,
\begin{array}{|c|c|c|c|c|c|c|}
\hline
m&5&6&7&8&9 \\
\hline
\text{Diff}& 0.12338 &  0.03931&  0.01912 & 0.00822  & 0.00174
\\
\hline
\end{array}
\end{equation*}

\begin{equation*}k=1,L=4\,:\,
\begin{array}{|c|c|c|c|c|c|c|}
\hline
m&5&6&7&8&9 \\
\hline
\text{Diff}&0.45974  & 0.17936 &   0.05044 &  0.01699 & 0.00745
\\
\hline
\end{array}
\end{equation*}

\begin{equation*}k=1,L=5\,:\,
\begin{array}{|c|c|c|c|c|c|c|}
\hline
m&5&6&7&8&9 \\
\hline
\text{Diff}& 0.94805 & 0.39742 & 0.12869  &     0.04340     &0.01448
\\
\hline
\end{array}
\end{equation*}

\begin{equation*}k=1,L=6\,:\,
\begin{array}{|c|c|c|c|c|c|c|}
\hline
m&5&6&7&8&9 \\
\hline
\text{Diff}&0.48571  &  0.15541 &   0.05472 &  0.01866 &0.00413
\\
\hline
\end{array}
\end{equation*}

\begin{equation*}k=1,L=7\,:\,
\begin{array}{|c|c|c|c|c|c|c|}
\hline
m&5&6&7&8&9 \\
\hline
\text{Diff}&  2.4364 &  0.67045 &   0.22705 & 0.06783  &0.02385
\\
\hline
\end{array}
\end{equation*}

\begin{equation*}k=1,L=8\,:\,
\begin{array}{|c|c|c|c|c|c|c|}
\hline
m&5&6&7&8&9 \\
\hline
\text{Diff}& 1.6597&  0.57248& 0.19170  &  0.06838 &0.02304
\\
\hline
\end{array}
\end{equation*}

\begin{equation*}k=2,L=2\,:\,
\begin{array}{|c|c|c|c|c|c|c|}
\hline
m&5&6&7&8&9 \\
\hline
\text{Diff}& 0.05941 & 0.03219 &  0.01761 &  0.00594 &0.00103
\\
\hline
\end{array}
\end{equation*}

\begin{equation*}k=2,L=3\,:\,
\begin{array}{|c|c|c|c|c|c|c|}
\hline
m&5&6&7&8&9 \\
\hline
\text{Diff}& 0.13527 & 0.03679 &  0.02912  & 0.00816  &0.00259
\\
\hline
\end{array}
\end{equation*}

\begin{equation*}k=2,L=4\,:\,
\begin{array}{|c|c|c|c|c|c|c|}
\hline
m&5&6&7&8&9 \\
\hline
\text{Diff}& 0.27434 & 0.08249 & 0.03130  &  0.01076 &0.00376
\\
\hline
\end{array}
\end{equation*}

\begin{equation*}k=2,L=5\,:\,
\begin{array}{|c|c|c|c|c|c|c|}
\hline
m&5&6&7&8&9 \\
\hline
\text{Diff}& 0.89633 &  0.39528 &  0.11930  &0.03852   &0.01388
\\
\hline
\end{array}
\end{equation*}

\begin{equation*}k=2,L=6\,:\,
\begin{array}{|c|c|c|c|c|c|c|}
\hline
m&5&6&7&8&9 \\
\hline
\text{Diff}& 0.50190 &  0.12052 & 0.06497  &  0.02453  &0.00598
\\
\hline
\end{array}
\end{equation*}

\begin{equation*}k=2,L=7\,:\,
\begin{array}{|c|c|c|c|c|c|c|}
\hline
m&5&6&7&8&9 \\
\hline
\text{Diff}&  1.9735 &  0.59430&  0.18924 & 0.07223  &0.02245
\\
\hline
\end{array}
\end{equation*}

\begin{equation*}k=2,L=8\,:\,
\begin{array}{|c|c|c|c|c|c|c|}
\hline
m&5&6&7&8&9 \\
\hline
\text{Diff}& 1.4273 &  0.33456&   0.10510 &  0.04694 &0.01919
\\
\hline
\end{array}
\end{equation*}
Hereafter, we give cases $(1\le k\le \deg f -1, 2\le L \le8)$ when $\text{Diff}>0.1$ at  $m=9$.
Comparing with $\#\mathfrak{R}(f,\sigma,k,L)(\ge L^{n-1})$, $m=9$ seems to be small.

\noindent
(3) $f(x)= x^4+x^3+x^2+x+1$
$$
\begin{array}{|c|c|c|c|c|c|c|c|c|}
\hline
m&5&6&7&8&9&10&11 \\
\hline
\text{diff}& 1.5126  &  1.7526 & 0.51711 &  0.24419 &0.0699  & 0.01081 &  0.00454
\\
\hline
\end{array}
$$

\begin{equation*}k=1,L=7\,:\,
\begin{array}{|c|c|c|c|c|c|c|}
\hline
m&5&6&7&8&9 \\
\hline
\text{Diff}& 15.915 &  3.4587  &     1.1594 &    0.30438 & 0.11569
\\
\hline
\end{array}
\end{equation*}

\begin{equation*}k=3,L=7\,:\,
\begin{array}{|c|c|c|c|c|c|c|}
\hline
m&5&6&7&8&9 \\
\hline
\text{Diff}&  14.170 & 4.1020  & 1.2984  &  0.31371  &0.12101
\\
\hline
\end{array}
\end{equation*}

\begin{equation*}k=3,L=8\,:\,
\begin{array}{|c|c|c|c|c|c|c|}
\hline
m&5&6&7&8&9 \\
\hline
\text{Diff}&14.096  & 3.4425 &  1.5085 & 0.3744  & 0.12044
\\
\hline
\end{array}
\end{equation*}

\noindent
(4) $f(x)= x^5+2$
$$
\begin{array}{|c|c|c|c|c|c|c|c|c|}
\hline
m&5&6&7&8&9&10&11 \\
\hline
\text{diff}&   32.259&  18.965 &  1.3509 & 1 &   0.50828&  0.20874 & 0.01738
\\
\hline
\end{array}
$$

\begin{equation*}k=4,L=4\,:\,
\begin{array}{|c|c|c|c|c|c|c|}
\hline
m&5&6&7&8&9 \\
\hline
\text{Diff}& 8.4815 & 6.0621 &   1.3874&   0.40524 &0.14844
\\
\hline
\end{array}
\end{equation*}

\begin{equation*}k=4,L=5\,:\,
\begin{array}{|c|c|c|c|c|c|c|}
\hline
m&5&6&7&8&9 \\
\hline
\text{Diff}& 45.296  & 11.931  &  2.1385  &  0.81505 & 0.27803
\\
\hline
\end{array}
\end{equation*}

\begin{equation*}k=2,L=6\,:\,
\begin{array}{|c|c|c|c|c|c|c|}
\hline
m&5&6&7&8&9 \\
\hline
\text{Diff}&   24.288&    7.6065 &   1.7536 &   0.51365&    0.17542
\\
\hline
\end{array}
\end{equation*}

\begin{equation*}k=3,L=6\,:\,
\begin{array}{|c|c|c|c|c|c|c|}
\hline
m&5&6&7&8&9 \\
\hline
\text{Diff}&  12.091 &   3.2938&  1.1890  &   0.32440& 0.11074
\\
\hline
\end{array}
\end{equation*}

\begin{equation*}k=4,L=6\,:\,
\begin{array}{|c|c|c|c|c|c|c|}
\hline
m&5&6&7&8&9 \\
\hline
\text{Diff}& 95.000 &  34.752 &  6.4376  & 1.7959  &0.51436
\\
\hline
\end{array}
\end{equation*}

\begin{equation*}k=2,L=7\,:\,
\begin{array}{|c|c|c|c|c|c|c|}
\hline
m&5&6&7&8&9 \\
\hline
\text{Diff}& 139.55 & 22.917 &  4.7391 & 1.5129   & 0.43525
\\
\hline
\end{array}
\end{equation*}

\begin{equation*}k=3,L=7\,:\,
\begin{array}{|c|c|c|c|c|c|c|}
\hline
m&5&6&7&8&9 \\
\hline
\text{Diff}& 71.758  & 22.864 & 5.5512   &1.3991   & 0.47073
\\
\hline
\end{array}
\end{equation*}

\begin{equation*}k=4,L=7\,:\,
\begin{array}{|c|c|c|c|c|c|c|}
\hline
m&5&6&7&8&9 \\
\hline
\text{Diff}& 532.56  &98.352  & 30.003  & 6.1719  & 1.7150
\\
\hline
\end{array}
\end{equation*}

\begin{equation*}k=1,L=8\,:\,
\begin{array}{|c|c|c|c|c|c|c|}
\hline
m&5&6&7&8&9 \\
\hline
\text{Diff}& 681.67  &   106.08&   34.438 &  5.9500 & 1.7811
\\
\hline
\end{array}
\end{equation*}

\begin{equation*}k=2,L=8\,:\,
\begin{array}{|c|c|c|c|c|c|c|}
\hline
m&5&6&7&8&9 \\
\hline
\text{Diff}& 78.922 &17.134  &   3.3514 &  1.1124 & 0.35088
\\
\hline
\end{array}
\end{equation*}

\begin{equation*}k=3,L=8\,:\,
\begin{array}{|c|c|c|c|c|c|c|}
\hline
m&5&6&7&8&9 \\
\hline
\text{Diff}& 164.49  &  26.141& 5.3863  &1.4805   &0.47993
\\
\hline
\end{array}
\end{equation*}

\begin{equation*}k=4,L=8\,:\,
\begin{array}{|c|c|c|c|c|c|c|}
\hline
m&5&6&7&8&9 \\
\hline
\text{Diff}&150.70  & 55.497 &  10.753 &  2.3986  &0.85270
\\
\hline
\end{array}
\end{equation*}
For the following, $m=9$ is  too small to get $\text{Diff}<0.1$.

\noindent
(5) $f(x)= x^6+x^5+x^4+x^3+x^2+x+1$
$$
\begin{array}{|c|c|c|c|c|c|c|c|c|}
\hline
m&5&6&7&8&9&10&11 \\
\hline
\text{diff}& 70.571 &  12.099 &15.044  &  6.3971 &1.0950  &  1&  1.6436
\\
\hline
\end{array}
$$

\noindent
(6) $f(x)= x^7+2$
$$
\begin{array}{|c|c|c|c|c|c|c|c|c|}
\hline
m&5&6&7&8&9&10&11 \\
\hline
\text{diff}& 1 &   7.2083&  24.018 &  4.1811 &  1&  1& 1
\\
\hline
\end{array}
$$
\subsection{Case of degree $6$ with non-trivial linear relation among roots}
Let $f(x) =x^6+14x^4+49x^2+7$, which is the fourth polynomial in  Example 3.1.
Then taking roots there, we see that $t=4$ and the basis of linear relations 
\begin{align*}
&\hat{\bm{m}_1} =(1,0,0,1,0,0,0),
\\
&\hat{\bm{m}_2} =(0,1,0,0,1,0,0),
\\
&\hat{\bm{m}_3} =(0,0,1,0,0,1,0),
\\
&\hat{\bm{m}_4} =(1,0,1,0,1,0,0).
\end{align*}
Let us determine permutations $\sigma$ satisfying $\dim\frak{D}(f,\sigma)=2$.
To do it, suppose that $0<y_1<\dots<y_6<1,y_1+y_6=y_2+y_5=y_3+y_4=1$ and 
the corresponding equation to $\hat{\bm{m}_4} $
$$
\left\{\begin{array}{l}y_1 \\y_6 \end{array}\right.+
\left\{\begin{array}{l}y_2 \\y_5 \end{array}\right.+
\left\{\begin{array}{l}y_3 \\y_4 \end{array}\right. \in\mathbb{Z}.
$$
Then possible combinations above are $y_1+y_2+y_3=1$ or $y_1+y_2+y_4=1$.
Identifying a permutation $\sigma\in S_6$ with the image $[\sigma(1),\dots,\sigma(6)]$,
the set of  representatives of cosets of $\sigma$ with $\dim\frak{D}(f,\sigma)=2$ 
by the group $\hat {\bm G}$  is $\sigma_1=[1, 2,4, 6, 5,3],\sigma_2=[1, 2, 3, 6, 5, 4]$,
supposing that $\sigma(1)=1,\sigma(4)=6,\sigma(2)=2,\sigma(5)=5,\{\sigma(3),\sigma(6)\}=\{3,4\}$
by $\{\{\sigma(1),\sigma(4)\}$, $\{\sigma(2),\sigma(5)\}$, $\{\sigma(3),\sigma(6)\}\}=$ $\{\{1,6\}$,
$\{2,5\}$, $\{3,4\}$$\}$.
\begin{align*}
&\frak{D}(f,\sigma_1)
\\
=&\left\{(x_1,x_2,x_3,x_4,x_5,x_6)\left|
\begin{array}{l}
0\le x_1\le x_2,x_1+x_2\le1/2,\\
x_3=x_1+x_2, x_4=1-(x_1+x_2),\\
x_5=1-x_2,  x_6=1-x_1
\end{array}
\right.
\right\}
\\
=&\{x_1\bm{u}_1+x_2\bm{u}_2+(0,0,0,1,1,1)\mid 0\le x_1\le x_2,x_1+x_2\le1/2\}
\\
&\text{ where }\bm{u}_1:=(1,0,1,-1,0,-1),\bm{u}_2:=(0,1,1,-1,-1,0),
\\
&\frak{D}(f,\sigma_2)
\\
=&\left\{(x_1,x_2,x_3,x_4,x_5,x_6)\left|
\begin{array}{l}
0\le x_1\le x_2,x_1+x_2\ge1/2, x_1+2x_2\le1,
\\
x_3=1-x_1-x_2,x_4=x_1+x_2,
\\
x_5=1-x_2,  x_6=1-x_1
\end{array}
\right.
\right\}
\\
=&\left\{x_1\bm{v}_1+x_2\bm{v}_2+(0,0,1,0,1,1)\left|
\begin{array}{l}
0\le x_1\le x_2,x_1+x_2\ge1/2,
\\ x_1+2x_2\le1
\end{array}
\right.\right\}
\\
&\text{ where }\bm{v}_1:=(1,0,-1,1,0,-1),\bm{v}_2:= (0,1,-1,1,-1,0) .
\end{align*}
Since the length of $\bm{u}_i,\bm{v}_i$ $(i=1,2)$ is $2$ and the angle of $\bm{u}_1,\bm{u}_2$
(resp.   $\bm{v}_1,\bm{v}_2$) is $\pi/3$ and
the volumes of their projection to $(x_1,x_2)$-space are $1/16,1/48$,
we see $vol(\frak{D}(f,\sigma_1))=3vol(\frak{D}(f,\sigma_2))$.
Therefore Conjecture \ref{conj1} is numerically confirmed as follows :
In the following,
``diff$_1$'' (resp. ``diff$_2$'') means the difference  $\#Spl_{X_m}(f,\sigma_1)/\#Spl_{X_m}(f) - 3/4$
(resp. $\#Spl_{X_m}(f,\sigma_2)/\#Spl_{X_m}(f) - 1/4$).
$$
\begin{array}{|c|c|c|c|c|c|c|}
\hline
m&5&6&7&8 &9&10
\\
\hline
\text{diff}_1&0.00094&0.00436&0.00190&-0.00009&-0.00015&-0.00004
\\
\hline
\text{diff}_2&-0.00094&-0.00436&-0.00190&0.00009&0.00015&0.00004
\\
\hline
\end{array}
$$
Let us check Conjecture \ref{conj4}.
We assume that $\sigma=\sigma_1$ or $\sigma_2$ above and $k_1=\dots=k_4=1$
because of $\#Spl_\infty(f,\sigma,\{k_j\})<\infty$ otherwise.
Denoting $\#\frak{R}(f,\sigma,\{k_j\},L)$ by $R$, and by ``er'' the maximum
of
$$|\#Spl_{X_m}(f,\sigma,\{k_j\},L,\{R_i\})/\#Spl_{X_m}(f,\sigma,\{k_i\})\times R-1|
$$ 
with  $\{R_i\}\in\frak{R}(f,\sigma,\{k_j\},L)$,
 we see

\noindent
$\sigma=\sigma_1, L=2,(R=4 )$
$$
\begin{array}{|c|c|c|c|c|c|}
\hline
m&5&6&7&8 &9
\\
\hline
er&0.09106& 0.01897& 0.00979& 0.00303&0.0005
\\
\hline
\end{array}
$$

\noindent
$\sigma=\sigma_1, L=3,(R= 144)$
$$
\begin{array}{|c|c|c|c|c|c|}
\hline
m&5&6&7&8 &9
\\
\hline
er&0.15789   &  0.08493  & 0.04991  &    0.01013 &   0.00353
\\
\hline
\end{array}
$$

\noindent
$\sigma=\sigma_1, L=4,(R= 256)$
$$
\begin{array}{|c|c|c|c|c|c|}
\hline
m&5&6&7&8 &9
\\
\hline
er& 0.28321  &  0.10076  &   0.03344  & 0.02033  &  0.00492
\\
\hline
\end{array}
$$

\noindent
$\sigma=\sigma_1, L=5,(R=6400 )$
$$
\begin{array}{|c|c|c|c|c|c|}
\hline
m&5&6&7&8 &9
\\
\hline
er& 0.75439  &  0.24810  & 0.09133   &  0.03212  &    0.01131
\\
\hline
\end{array}
$$

\noindent
$\sigma=\sigma_1, L=6,(R= 576)$
$$
\begin{array}{|c|c|c|c|c|c|}
\hline
m&5&6&7&8 &9
\\
\hline
er& 0.57895 &    0.20365  &   0.09654 &    0.03440 &   0.00863
\\
\hline
\end{array}
$$

\noindent
$\sigma=\sigma_1, L=7,(R=63504 )$
$$
\begin{array}{|c|c|c|c|c|c|}
\hline
m&5&6&7&8 &9
\\
\hline
er&   0.51462 &  0.06361  &   0.21441 &  0.02060  &  0.00569
\\
\hline
\end{array}
$$

\noindent
$\sigma=\sigma_2, L=2,(R=4 )$
$$
\begin{array}{|c|c|c|c|c|c|}
\hline
m&5&6&7&8 &9
\\
\hline
er&  0.12343   & 0.02835    &   0.00950  &  0.00731  &  0.00230
\\
\hline
\end{array}
$$

\noindent
$\sigma=\sigma_2, L=3,(R= 144)$
$$
\begin{array}{|c|c|c|c|c|c|}
\hline
m&5&6&7&8 &9
\\
\hline
er&  0.40554  &   0.13867 &  0.05190  &     0.02128&  0.00587
\\
\hline
\end{array}
$$

\noindent
$\sigma=\sigma_2, L=4,(R=256 )$
$$
\begin{array}{|c|c|c|c|c|c|}
\hline
m&5&6&7&8 &9
\\
\hline
er&0.53149   &    0.24213  &  0.05839  &   0.03716   &  0.00668
\\
\hline
\end{array}
$$

\noindent
$\sigma=\sigma_2, L=5,(R= 6400)$
$$
\begin{array}{|c|c|c|c|c|c|}
\hline
m&5&6&7&8 &9
\\
\hline
er&  1.2670  & 0.55812  & 0.15834    &  0.05164  & 0.01688
\\
\hline
\end{array}
$$

\noindent
$\sigma=\sigma_2, L=6,(R= 576)$
$$
\begin{array}{|c|c|c|c|c|c|}
\hline
m&5&6&7&8 &9
\\
\hline
er&   1.5390  &  0.45840   &    0.14348  &  0.04519   &     0.01725
\\
\hline
\end{array}
$$

\noindent
$\sigma=\sigma_2, L=7,(R=63504 )$
$$
\begin{array}{|c|c|c|c|c|c|}
\hline
m&5&6&7&8 &9
\\
\hline
er&   0.60453  &    0.29791   & 0.10483    &    0.03857& 0.01230
\\
\hline
\end{array}
$$

\subsection{Case of degree $4$ with non-trivial linear relation among roots}
Let us consider the polynomial $f(x)=x^4+1$ in this subsection.
First  we check \eqref{eq11} numerically for $D_i=\{(x_1,\dots,x_4)\mid x_i\le A\}$.
To do it, as in the subsection 4.2 we have only to check
$$\displaystyle
\frac{\#\{p\in {Spl}_X(f) \mid (r_1/p,\dots,r_n/p)\in D_i\}}
{\#{Spl}_X(f)}
\to
\displaystyle\frac{ {vol}(  {D_i}\cap{\mathfrak{D}}(f,id))}
{{vol}({\mathfrak{D}}(f,id))}
$$
and the right-hand side is given by $8vol(pr(D_i\cap\frak{D}(f,id))$ as in that subsection.

 ``diff'' in the table means the maximum of the absolute value of the difference between $1$ 
and  the left-hand side divided by the right-hand side at
$A=k/40$ $(k=1,\dots,40)$, $i=1,\dots,4$ and $X=X_m$ for $m=5,\dots,11$. 
$$
\begin{array}{|c|c|c|c|c|c|c|c|c|}
\hline
m&5&6&7&8&9&10&11 \\
\hline
\text{diff}& 0.18691 &  0.16125& 0.04833 & 0.02029& 0.00323& 0.00303& 0.00058
\\
\hline
\end{array}
$$
Next, let us check Conjecture \ref{conj4}.
As in the subsection 4.2 again, we have only to consider the case of $\{\sigma(1),\sigma(4)\}=\{1,4\}$
or $\{2,3\}$ and $k_1=k_2=1$.
Then
\begin{align*}
&\mathfrak{R}(f,\sigma,\{k_j\},L)=\\
&
\left\{\{R_i\}\in[0,L-1]^n\left|
\begin{array}{l}
 R_1+R_4\equiv R_2+R_3\bmod L
\text{ with }(R_1+R_4,L)=1,
\\
\text{[[}R_1+R_4\text{]]}=\text{[[}1\text{]] }
\text{ on }\mathbb{Q}(\zeta_8)\cap\mathbb{Q}(\zeta_L)
\end{array}
\right.
\right\}.
\end{align*}
Therefore, we see
\begin{equation*}
\#\mathfrak{R}(f,\sigma,\{k_j\},L)=L^2\varphi(L)/[\mathbb{Q}(\zeta_8)\cap\mathbb{Q}(\zeta_L):\mathbb{Q}].
\end{equation*}
The following are tables of the maximum of
$$
\left|\frac{Spl_{X_m}(f,\sigma,\{k_j\},L,\{R_i\})}{   Spl_{X_m}(f,\sigma,\{k_j\})  }\times
\#\frak{R}(f,\sigma,\{k_j\},L)
-1\right|
$$
at $\{R_i\} \in \frak{R}(f,\sigma,\{k_j\},L)$ for $m=5,\dots,10$ and $L=2,\dots,10$.
\vspace{2mm}

\noindent
$L=2$
$$
\begin{array}{|c|c|c|c|c|c|c|c|}
\hline
m&5&6&7&8&9&10 \\
\hline
\text{diff}&   0.05073 &  0.00721  &   0.00718 &  0.00172  &  0.00014  & 0.00014  
\\
\hline
\end{array}
$$

\noindent
$L=3$
$$
\begin{array}{|c|c|c|c|c|c|c|c|}
\hline
m&5&6&7&8&9&10\\
\hline
\text{diff}&   0.22264 &   0.07574 &  0.01518  &   0.00460  &   0.00273 &  0.00057 
\\
\hline
\end{array}
$$

\noindent
$L=4$
$$
\begin{array}{|c|c|c|c|c|c|c|c|}
\hline
m&5&6&7&8&9&10\\
\hline
\text{diff}  & 0.16730  &  0.04659 &  0.01596  &  0.00635   &  0.0025  & 0.00073  
\\
\hline
\end{array}
$$

\noindent
$L=5$
$$
\begin{array}{|c|c|c|c|c|c|c|c|}
\hline
m&5&6&7&8&9&10 \\
\hline
\text{diff}&  0.59329 &   0.18683 &   0.05858 &   0.02349 &   0.00701 &  0.00258 
\\
\hline
\end{array}
$$

\noindent
$L=6$
$$
\begin{array}{|c|c|c|c|c|c|c|c|}
\hline
m&5&6&7&8&9&10 \\
\hline
\text{diff}  &  0.50943 & 0.17148  & 0.05412  &  0.01757  &   0.00785 &  0.00181 
\\
\hline
\end{array}
$$

\noindent
$L=7$
$$
\begin{array}{|c|c|c|c|c|c|c|c|}
\hline
m&5&6&7&8&9&10 \\
\hline
\text{diff}&  1.2189   &  0.39856  &  0.12142 &   0.04167  &  0.01347  &  0.00463 
\\
\hline
\end{array}
$$

\noindent
$L=8$
$$
\begin{array}{|c|c|c|c|c|c|c|c|}
\hline
m&5&6&7&8&9&10 \\
\hline
\text{diff}  & 0.35597  & 0.13578  &  0.07187 &  0.01575 & 0.00523   & 0.00177
\\
\hline
\end{array}
$$

\noindent
$L=9$
$$
\begin{array}{|c|c|c|c|c|c|c|c|}
\hline
m&5&6&7&8&9&10 \\
\hline
\text{diff}&   1.4453 &  0.49133 &  0.17417 &  0.05943  &  0.01853  &  0.00706 
\\
\hline
\end{array}
$$

\noindent
$L=10$
$$
\begin{array}{|c|c|c|c|c|c|c|c|}
\hline
m&5&6&7&8&9&10 \\
\hline
\text{diff}  & 1.5157  & 0.55475  &   0.14687 &  0.05529  &  0.01544  &   0.00612
\\
\hline
\end{array}
$$
\subsection{Reducible polynomial}\label{subsec6.4}
\subsubsection{}
Let us consider the polynomial $f(x) =(x^2+x+1)(x^2+2x+2)$ with roots $\alpha_i$ such that
 $\alpha_i^2+ \alpha_i+1 =0 $ for $i=2,3$ and $\alpha_i^2+ 2\alpha_i+2 =0 $ for $i=1,4$. 
The field $\mathbb{Q}(f)$ is $\mathbb{Q}(\sqrt{-3},\sqrt{-1})$.
Vectors $\hat{\bm{m}}_i$ be
$$
\hat{\bm{m}}_1 = (1,0,0,1,-2),\hat{\bm{m}}_2 = (0,1,1,0,-1)\quad(t=2).
$$
Then we see that 
\begin{align*}
\bm{G} = \{&[1, 2, 3, 4], [1, 3, 2, 4], [2, 1, 4, 3], [2, 4, 1, 3], [3, 1, 4, 2], 
\\ 
&[3, 4, 1, 2], [4, 2, 3, 1], [4, 3, 2, 1]\},
\\
\hat{\bm{G}} = \{&[1, 2, 3, 4], [1, 3, 2, 4], [4, 2, 3, 1], [4, 3, 2, 1]\}
\end{align*}
and
$$
\mathfrak{D}(f,\sigma)=
\begin{cases}
\{(x_1,x_2,1-x_2,1-x_1)\mid 0\le x_1\le x_2 \le 1/2\} &\text{ if } \sigma \in\bm{G},
\\
\begin{cases}
\{(x,x,1-x,1-x)\mid 0\le x \le 1/2\}
\\
\text{ or }
\\
\{(1/2,1/2,1/2,1/2)\}
\end{cases}
&\text{ otherwise},
\end{cases}
$$
hence $\dim\mathfrak{D}(f,\sigma)=2\,(=n-t)$ if and only if $ \sigma \in\bm{G}$.
For $X=10^8$, we find
\begin{align*}
\frac{\#Spl_X(f,\sigma)}{\#Spl_X(f)}= 
\begin{cases}
1/2\,(=1/[\bm{G}:\hat{\bm{G}}])&\text{ if } \sigma\in {\bm{G}},
\\
0&\text{ otherwise}.
\end{cases}
\end{align*}
Three conditions $\sigma\in \bm{G}$, $Pr(f,\sigma)>0$ and $vol(\mathfrak{D}(f,\sigma) )>0$ are equivalent.

For the one-dimensional uniformity, we see that, with $X=10^8$
$$
\max_{k=0}^9|\frac{\sum_i\#\{p\in Spl_X(f)\mid k/10\le r_i/p<(k+1)/10\}}{4\#Spl_X(f)}-\frac{1}{10}|<0.000079.
$$

With respect to Conjecture \ref{conj2}, let $D$ be $\{(x_1,x_2,x_3,x_4)\mid x_1<1/3\}$.
The absolute value of the difference between
$$
\frac{\#\{p\in {Spl}_X(f,\sigma) \mid (r_1/p,\dots,r_n/p)\in D\}}
{\#{Spl}_X(f,\sigma)}
\text{ and }
\frac{ {vol}(  {D}\cap{\mathfrak{D}}(f,\sigma))}
{{vol}({\mathfrak{D}}(f,\sigma))}(=8/9)
$$
for $X=10^8$ and $\sigma\in \bm{G}$
is $0.0003$.
With respect to Conjecture \ref{conj4}, we see $k_1=k_2=1$ and $\bm{G}/\hat{\bm{G}}=
\{[1, 2, 3, 4], [2, 1, 4, 3]\}$.
In the following table,  the maximum of errors 
$$
\#\mathfrak{R}(f,\sigma,\{k_j\},L)\times|
(Pr(f,\sigma,\{k_j\},L,\{R_i\})-1/\#\mathfrak{R}(f,\sigma,\{k_j\},L)|
$$ over $\{R_i\}\in \mathfrak{R}(f,\sigma,\{k_j\},L)$
is given for $\sigma\in\bm{G}/\hat{\bm{G}}$, and $\#\mathfrak{R}$ is
$\#\mathfrak{R}(f,\sigma,\{k_j\},L)$,
which is independent of $\sigma$.
$$
\begin{array}{|r|c|c|r|}
\hline
L&[1, 2, 3, 4]&[2, 1, 4, 3]&\#\mathfrak{R}
\\
\hline
2&0.00123&0.00264&4
\\
\hline
3&0.00557&0.00354&9
\\
\hline
4&0.00885&0.00911&16
\\
\hline
5&0.02824&0.02858&100
\\
\hline
6&0.01451&0.01545&36
\\
\hline
7& 0.06906&0.05916&294
\\
\hline
8& 0.04457&0.03671&128
\\
\hline
9&0.06127&0.05443&243
\\
\hline
10&0.08031&0.08409&400
\\
\hline
11&0.13550&0.13511&1210
\\
\hline
12&0.04323&0.04952&144
\\
\hline
\end{array}
$$

\subsubsection{}\label{counterEx}
Let us consider the polynomial $f(x)=(x^2-2)((x-1)^2-2)$  with roots $\alpha_1=\sqrt{2},
\alpha_2=1+\sqrt{2},\alpha_3=-\sqrt{2},\alpha_4=1-\sqrt{2}$.
Then the matrix $(m_{j,i})$ and $m_i$ $(t=3)$ are
$$
\left(
\begin{array}{cccc}
1&0&0&1\\
1&-1&0&0\\
1&0&1&0\\
\end{array}
\right),
\quad
\left(
\begin{array}{c}
1\\
-1\\
0\\
\end{array}
\right),
$$
and $\bm{G}\!=\!\{ [1, 2, 3, 4], [1, 2, 4, 3], [2, 1, 3, 4],[2, 1, 4, 3], [3, 4, 1, 2], [3, 4, 2, 1],[4, 3, 1, 2]$, 
\linebreak
$[4, 3, 2, 1]\}$
and $\hat{\bm{G}}=\{[1, 2, 3, 4], [3, 4, 1, 2]\}$,
where a permutation $\sigma$ is identified with the vector $[\sigma(1),\dots,\sigma(4)]$ of images.
Then we see that
$$
\mathfrak{D}(f,\sigma)=
\begin{cases}
\{(x,x,1-x,1-x)\mid 0\le x \le 1/2\} &\text{ if } \sigma \in\bm{G},
\\
\{(1/2,1/2,1/2,1/2)\}&\text{ otherwise},
\end{cases}
$$
so two conditions $\dim\mathfrak{D}(f,\sigma)=1\,(=n-t)$ and $\sigma\in G$ are equivalent.
We find
\begin{align*}
Pr(f,\sigma)=
\begin{cases}
1&\text{ if } \sigma\in \hat{\bm{G}},
\\
0&\text{ otherwise},
\end{cases}
\end{align*}
since $p\in Spl(f,\sigma)$ means $r_{\sigma(1)}+r_{\sigma(4)}=p+1,r_{\sigma(1)}-r_{\sigma(2)}=-1,
r_{\sigma(1)}+r_{\sigma(3)}=p$ for a sufficiently large $p$, which implies $\sigma\in\hat{\bm{G}}$.
This  is an example such that $Pr(f,\sigma)=0$ but $vol(\mathfrak{D}(f,\sigma))>0$ for $\sigma \in \bm{G}\setminus \hat{\bm{G}}$
by $\#\bm{G}=8,\#\hat{\bm{G}}=2$, and   this supports Conjecture \ref{conj1}. 
We see easily that for $\sigma\in \bm{G}$
\begin{align*}
vol(D_{1,a}\cap\mathfrak{D}(f,\sigma))=vol(D_{2,a}\cap\mathfrak{D}(f,\sigma))&=
\begin{cases}
2a&\text{ if }a\le1/2,
\\
1&\text{ if }1/2<a<1,
\end{cases}
\\
vol(D_{3,a}\cap\mathfrak{D}(f,\sigma))=vol(D_{4,a}\cap\mathfrak{D}(f,\sigma))&=
\begin{cases}
0&\text{ if }a\le1/2,
\\
2a-1&\text{ if }1/2<a<1.
\end{cases}
\end{align*}
For the one-dimensional uniformity, we see that, with $X=10^8$
$$
\max_{k=0}^9|\frac{\sum_i\#\{p\in Spl_X(f)\mid k/10\le r_i/p<(k+1)/10\}}{4\#Spl_X(f)}-\frac{1}{10}|< 0.00017.
$$
The one-dimensional uniformity follows surely from \cite{DFI}, since   $f(x)$ is a product of
quadratic polynomials which define the same quadratic field $\mathbb{Q}(\sqrt{2})$.

With respect to Conjecture \ref{conj2}, let $D$ be $\{(x_1,x_2,x_3,x_4)\mid x_1<1/3\}$.
The absolute value of the difference between
$$
\frac{\#\{p\in {Spl}_X(f,\sigma) \mid (r_1/p,\dots,r_n/p)\in D\}}
{\#{Spl}_X(f,\sigma)}
\text{ and }
\frac{ {vol}(  {D}\cap{\mathfrak{D}}(f,\sigma))}
{{vol}({\mathfrak{D}}(f,\sigma))}(=\frac{2}{3})
$$
for $X=10^8$ and $\sigma=id$
is $0.00032$.
To check Conjecture \ref{conj4}, we see that $k_1=1,k_2=0,k_3=1$, $\sigma\in\hat{\bm{G}}$ hold if $Pr(f,\sigma,\{k_j\},L,\{R_i\})>0$, and for them the maximum $er$ of errors 
$$
\#\mathfrak{R}(f,\sigma,\{k_j\},L)\times|
(Pr(f,\sigma,\{k_j\},L,\{R_i\})-1/\#\mathfrak{R}(f,\sigma,\{k_j\},L)|
$$ over $\{R_i\}\in \mathfrak{R}(f,\sigma,\{k_j\},L)$
is as follows :
$$
\begin{array}{|c|c|c|c|c|c|c|c|c|}
\hline
L&2&3&4&5&6&7&8
\\
\hline
\#\mathfrak{R}(f,\sigma,\{k_j\},L)&2&6&8&20&12&42&16
\\
\hline
er&0.0002&0.001&0.001&0.006&0.002&0.008&0.003
\\
\hline
\end{array}
$$

\subsubsection{}
Let us consider the polynomial $f(x)=(x^2+1)(x^2-2)(x^4-2x^2+9)$, whose roots are
$\pm\sqrt{2},\pm\sqrt{-1},\pm\sqrt{2}\pm\sqrt{-1}$ and we number them as 
$\alpha_1 =  \sqrt{-1}$, $\alpha_2 = \sqrt{2} - \sqrt{-1}$, $\alpha_3 = \sqrt{2}$, 
$\alpha_4 = -\sqrt{2} -\sqrt{-1}$, $\alpha_5 = \sqrt{2} + \sqrt{-1}$, 
$\alpha_6 = -\sqrt{2}$, $\alpha_7 = -\sqrt{2} + \sqrt{-1}$, $\alpha_8 = -\sqrt{-1}$. 
The rows of the following matrix $M$ gives $\bm{m}_j$ $(j=1,\dots,6)$ with $m_1=\dots=m_6=0$, i.e.
$M\cdot{}^t(\alpha_1,\alpha_2,\alpha_3,\alpha_4,\alpha_5,\alpha_6,\alpha_7,\alpha_8)=0$ $(t=6)$:
$$
M=\begin{pmatrix}
1&0&0&0&0&0&0&1\\
0&1&0&0&0&0&1&0\\
0&0&1&0&0&1&0&0\\
0&0&0&1&1&0&0&0\\
1&1&-1&0&0&0&0&0&\\
1&0&1&0&-1&0&0&0&
\end{pmatrix}.
$$
The group $\bm{G}$ is, by Proposition \ref{prop4.1} 
\begin{align*}
&\bm{G}=\hat{\bm{G}}:=\{\sigma\mid M\cdot{}^t(\alpha_{\sigma(1)},\dots,\alpha_{\sigma(8)}) = 0\} 
\\
=\{&[1,2,3,4,5,6,7,8],[1,4,6,2,7,3,5,8],[3,4,8,7,2,1,5,6],[3,7,1,4,5,8,2,6],
 \\
 &[6,2,8,5,4,1,7,3],[6,5,1,2,7,8,4,3],[8,5,3,7,2,6,4,1],[8,7,6,5,4,3,2,1]\},
\end{align*}
where a permutation $\sigma$ is identified with the vector $[\sigma(1),\dots,\sigma(8)]$
of its images. 
The set of $\sigma$ satisfying  \eqref{eq6} for  some $p$ less than $ 10^8$ consists of $\sigma_iG$ 
for
\begin{align*}
&\sigma_1 = [1,2,3,5,4,6,7,8],
\sigma_2 = [1,2,4,3,6,5,7,8],
\sigma_3 = [1,4,6,2,7,3,5,8],\\
&\sigma_4 = [2,1,3,4,5,6,8,7],
\sigma_5 = [2,1,3,5,4,6,8,7],
\sigma_6 = [2,1,4,3,6,5,8,7],\\
&\sigma_7 = [3,1,4,2,7,5,8,6].
\end{align*}
The corresponding domain $\mathfrak{D}(f,\sigma_i)=\{\bm{x}=C+x_1\bm{v}_1+x_2\bm{v}_2\mid S\}$
are given by the following:
Vectors $C,\bm{v}_i$ are
$$
\begin{array}{|c|c|c|c|}
\hline
 i&C&\bm{v}_1&\bm{v}_2
 \\ \hline
 1&(0,0,0,0,1,1,1,1)&(1,0,1,2,-2,-1,0,-1)&(0,1,1,1,-1,-1,-1,0)
 \\\hline
 2&(0,0,1,0,1,0,1,1)&(1,0,-2,1,-1,2,0,-1)&(0,1,-1,1,-1,1,-1,0)
 \\\hline
3&(0,0,0,1,0,1,1,1)&(1,0,1,-2,2,-1,0,-1)&(0,1,1,-1,1,-1,-1,0)
 \\\hline
4&(0,0,0,1,0,1,1,0)&(1,0,1,-1,1,-1,0,-1)&(0,1,1,-2,2,-1,-1,0)
 \\\hline
5&(0,0,0,0,1,1,1,0)&(1,0,1,1,-1,-1,0,-1)&(0,1,1,2,-2,-1,-1,0)
\\\hline
6&(0,0,1,0,1,0,1,1)&(1,0,-1,1,-1,1,0,-1)&(0,1,-2,1,-1,2,-1,0)
\\\hline
\end{array}
$$
and for $i=7$, 
\begin{gather*}C=(0,0,1/2,1/2,1/2,1/2,1,1), \bm{v}_1=(1,0,-1/2,1/2,-1/2,1/2,0,-1),
\\
\bm{v}_2=(0,1,-1/2,-1/2,1/2,1/2,-1,0).
\end{gather*}

\noindent
The inequalities $S$ on $x_1,x_2$ are
\begin{align*}
(i=1)\quad &0\le x_1\le x_2\le1,4x_1+2x_2\le1,
\\
(i=2)\quad &0\le x_1\le x_2\le1,2x_1+2x_2\le1,3x_1+2x_2\ge1,
\\
(i=3)\quad &0\le x_1\le x_2\le1,3x_1+2x_2\le1,4x_1+2x_2\ge1,
\\
(i=4)\quad &0\le x_1\le x_2\le1, 2x_1+3x_2\le1,2x_1+4x_2\ge1,
\\
(i=5)\quad &0\le x_1\le x_2\le1,2x_1+4x_2\le1,
\\
(i=6)\quad &0\le x_1\le x_2\le1,x_1+3x_2\le1,2x_1+3x_2\ge1,2x_1+2x_2\le1,
\\
(i=7)\quad &0\le x_1\le x_2\le1,x_1+3x_2\le1.
\end{align*}

\noindent
The volumes  $[vol(\mathfrak{D}(f,\sigma_1)),\dots,vol(\mathfrak{D}(f,\sigma_7))]$
are $[1/4,3/40,1/20,3/40,1/8,\linebreak1/20,1/8]$ with $1/4+\dots+1/8=3/4$.
Let us give a brief proof on $\sigma_1$.
Since $(\bm{v}_1,\bm{v}_1)=12, (\bm{v}_1,\bm{v}_2)=6, (\bm{v}_2,\bm{v}_2)=6$,
we define the orthonormal  vectors $\bm{u}_i$ by $\bm{v}_1=2\sqrt{3}\bm{u}_1,
\bm{v}_2=\sqrt{3}(\bm{u}_1+\bm{u}_2)$.
Writing $x_1\bm{v}_1+x_2\bm{v}_2 = y_1\bm{u}_1+y_2\bm{u}_2$, we see $y_1=2\sqrt{3}x_1+\sqrt{3},y_2=\sqrt{3}x_2$ with Jacobian $6$.
The volume of the set of $(x_1,x_2)$ defined by 
 the condition $S$, i.e. $0\le x_1\le x_2\le1,4x_1+2x_2\le1$ is $1/24$, hence $vol(\mathfrak{D}(f,\sigma_1))= 6\cdot1/24=1/4$.
 
 \noindent
For $X=10^8$, the following is  a table of $\#Spl_X(f,\sigma_i)/\#Spl_X(f)$ on the first line and  
 $vol(\mathfrak{D}(f,\sigma_i))/\sum_i vol(\mathfrak{D}(f,\sigma_i))$ on the second   
 in the order of $i=1,\dots,7$, respectively
 $$
 \begin{array}{|c|c|c|c|c|c|c|}
 \hline
 0.33320&0.10003& 0.06653& 0.09999&0.16713& 0.06678& 0.16633
 \\\hline
0.33333&0.1&0.06666&0.1&0.16667& 0.06666&0.16667
\\
\hline
 \end{array}.
$$ 
Exact values on the second line  are 
$[1/3, 1/10, 1/15, 1/10, 1/6,1/15, 1/6]$ which are equal to $(3/4)^{-1}[1/4,3/40,1/20,3/40,1/8,1/20,1/8]$.
This supports Conjecture \ref{conj1}.

\noindent
For $p\le 10^8$, integers $k_j$ in \eqref{eq7}  are $k_1=\dots=k_4=1,k_5=k_6=0$ for all $\sigma_i$,
and for such integers $k_j$,  $D:=\{(x_1,\dots,x_8)\mid|\sum_i m_{j,i}x_{\sigma(i)}-k_j|<\epsilon \,(1\le{}^\forall j\le6)\}$
$(\epsilon>0)$ contains $\mathfrak{D}(f,\sigma)$ for $\sigma=\sigma_1,\dots,\sigma_7$, hence Conjectures \ref{conj2} is true for such domains.

For the one-dimensional uniformity, we see that, with $X=10^8$
$$
\max_{k=0}^9|\frac{\sum_i\#\{p\in Spl_X(f)\mid k/10\le r_i/p<(k+1)/10\}}{8\#Spl_X(f)}-\frac{1}{10}|< 0.000036.
$$

\noindent
For $L= 2$ and $\{k_j\}$ above, $\mathfrak{R}(f,\sigma_i,\{k_j\},2)$ with  $\#\mathfrak{R}(f,\sigma_i,\{k_j\},2)=4$ $(i=1,\dots,7)$ is in order
\begin{align*}
&\{[0, 0, 0, 0, 1, 1, 1, 1], [0, 1, 1, 1, 0, 0, 0, 1], [1, 0, 1, 0, 1, 0, 1, 0], [1, 1, 0, 1, 0, 1, 0, 0]\},
\\
&\{[0, 0, 1, 0, 1, 0, 1, 1], [0, 1, 0, 1, 0, 1, 0, 1], [1, 0, 1, 1, 0, 0, 1, 0], [1, 1, 0, 0, 1, 1, 0, 0]\},
\\
&\{[0, 0, 0, 1, 0, 1, 1, 1], [0, 1, 1, 0, 1, 0, 0, 1], [1, 0, 1, 1, 0, 0, 1, 0], [1, 1, 0, 0, 1, 1, 0, 0]\},
\\
&\{[0, 0, 0, 1, 0, 1, 1, 1], [0, 1, 1, 1, 0, 0, 0, 1], [1, 0, 1, 0, 1, 0, 1, 0], [1, 1, 0, 0, 1, 1, 0, 0]\},
\\
&\{[0, 0, 0, 0, 1, 1, 1, 1], [0, 1, 1, 0, 1, 0, 0, 1], [1, 0, 1, 1, 0, 0, 1, 0], [1, 1, 0, 1, 0, 1, 0, 0]\},
\\
&\{[0, 0, 1, 0, 1, 0, 1, 1], [0, 1, 1, 1, 0, 0, 0, 1], [1, 0, 0, 1, 0, 1, 1, 0], [1, 1, 0, 0, 1, 1, 0, 0]\},
\\
&\{[0, 1, 0, 0, 1, 1, 0, 1], [0, 1, 1, 1, 0, 0, 0, 1], [1, 0, 0, 1, 0, 1, 1, 0], [1, 0, 1, 0, 1, 0, 1, 0]\},
\end{align*}
where $\{R_i\}$ is written as $[R_1,\dots,R_8]$.

\noindent
For $X=10^8$, $\#Spl_X(f,\sigma_i,\{k_j\},2,\{R_i\})/\#Spl_X(f,\sigma_i,\{k_j\})$ is, 
corresponding to the above table
\begin{align*}
(i=1)&\,\,0.25030, 0.24960, 0.24953, 0.25056, \\
(i=2)&\,\,0.25118, 0.24711, 0.24956, 0.25216, \\
(i=3)&\,\,0.25013, 0.25103, 0.24831, 0.25053, \\
(i=4)&\,\,0.24914, 0.25041, 0.25020, 0.25025, \\
(i=5)&\,\,0.24941, 0.24956, 0.25081, 0.25023, \\
(i=6)&\,\,0.24950, 0.25072, 0.24924, 0.25054, \\
(i=7)&\,\,0.25039, 0.24860, 0.25156, 0.24946.
\end{align*}
The top left value $0.25030$ corresponds to $i=1$ and $\{R_i\}=[R_1,\dots,R_8]=[0, 0, 0, 0, 1, 1, 1, 1]$,
and so on. 
	All values are almost $1/4=0.25$ and this matches Conjecture \ref{conj4}.
%
%
%
\section{$M(f,\mu)$ and the decimal part $\left\{\frac{g(r)}{p}\right\}$ again}\label{sec8}

Let us give observations on $M(f,\mu)$ defined  in the subsection \ref{subsect4.1} and
problems derived from them.
Let us recall  definitions and properties
\begin{gather*}
\hat{{\bm G}}=\{\sigma\in S_n\mid\sum_i m_{j,i}\alpha_{\sigma(i)}=m_j  \,({}^\forall j)\},
\\
\bm{G}_0=\{ \mu\in S_n\mid \alpha_{\mu(i)}=\tilde{\mu}(\alpha_i)\,\,(1\le{}^\forall i\le n)\text{ for }
{}^\exists \tilde{\mu}\in Gal(\mathbb{Q}(f)/\mathbb{Q}) \},
\\
Spl_X(f,\sigma):= 
\{ p\in Spl_X(f)\mid\sum_{i=1}^nm_{j,i}r_{\sigma(i)}\equiv m_j\bmod p\,(1\le{}^\forall j\le t)\},
\\
M(f,\mu)=\{ p\in Spl(f) \mid   \alpha_i\equiv r_{\mu(i)}\bmod\mathfrak p\,\,\,(1\le{}^\forall i\le n) 
\text{ for }{}^\exists\mathfrak p\,|\,p  \},
\\
Spl(f,\sigma)\fallingdotseq Spl(f,\mu)\Leftrightarrow \sigma\hat{G}=\mu\hat{G},
\\
M(f,\sigma)\fallingdotseq M(f,\mu) \Leftrightarrow\sigma G_0= \mu G_0,
\\
Spl(f,\sigma)\fallingdotseq\cup_{\mu\in\sigma \hat{G}}M(f,\mu),
\end{gather*}
where $A\fallingdotseq B$ means that sets $A\setminus B,B\setminus A$ 
are finite.
We note that in case of $Gal(\mathbb{Q}(f)/\mathbb{Q})\cong S_n$, 
$\hat{G}=G_0= S_n$ follows, hence the set $M(f,\mu)$ is equal to $Spl(f)$ for every permutation $\mu$.

$\bullet$
Let us give an example  of  $M(f,\mu)\,(\ne Spl(f))$.
Let $f(x)=x^3-3x+1$ which is abelian, and let roots be $\alpha_1=\alpha, \alpha_2=-\alpha^2 -\alpha + 2, \alpha_3=\alpha^2 - 2$. Then we see that
$\hat{\bm{G}}=S_3$ and  $\bm{G}_0 =\{[1,2,3],[3,1,2],[2,3,1]\}$ where a permutation 
$\sigma$ is identified with $[\sigma(1),\sigma(2),\sigma(3)]$.
Thus we see that $Spl(f)=Spl(f,\sigma)=\cup M(f,\sigma)=M(f,id)\cup M(f,(2,3))$ for
the transposition  $(2,3)\not\in \bm{G}_0$.
Since for a prime $p\in M(f,id)$, we have $r_2\equiv\alpha_2=-\alpha^2-\alpha+2
\equiv-r_1^2-r_1+2\mod\frak{p}$,
primes $p\in Spl(f)$  are separated as follows: 

For $\mu=id.$,  $M(f,\mu)=\{p\in Spl(f)\mid r_2\equiv-r_1^2-r_1+2\bmod p \}
=\{37, 73, 89, 181,  233, 251, 269, 397, 449, 467, 521, 541, 557,  593,
613, 631, 683, 809,\linebreak 811, 919, \dots\}$.

For $\mu=(2,3)$,  $M(f,\mu)=\{p\in Spl(f)\mid r_2\equiv r_1^2-2\bmod p \}
=\{17, 19, 53, 71, \linebreak107, 109, 127, 163, 179, 197, 199, 271, 307,  359,
379, 431, 433, 487, 503, 523, 577, \linebreak647, 701, 719, 739, 757, 773, 827, 829, 863,   881,
883,\dots\}$.

\noindent
Their densities in $ Spl(f)$ seem to be $1/2=1/[\hat{\bm{G}}:\bm{G}_0]$.
The density of $M(f,\mu)$
in $ Spl(f,\mu)$  is conjectured  to be $\frac{[\mathbb{Q}(f):\mathbb{Q}]}{\#\hat{G}}$ for every permutation $\mu$ satisfying $\#M(f,\mu)=\infty$  just before Corollary \ref{cor4.2}.

$\bullet$ We introduce the equivalence relation $\sim$ among polynomials as follows :
For  monic integral polynomials $f(x),g(x)$ of degree $n$, 
\begin{center}
$f(x)\sim g(x)$ if and only if  $g(x) = \delta^n f(\delta x+m)$ 
\end{center}
with $\delta =
\pm1$ and $m\in\mathbb{Z}$ holds.   
Denote the roots of $f(x)$ by $\alpha_i$  $(i= 1,\dots,n)$ and assume that
any root $\alpha_i$ is not a rational integer.
Then $\beta_i:=\delta(\alpha_i-m)$ are roots of $g(x)$, and for local roots $r_i$ of $f(x)$ at
$p\in Spl(f)=Spl(g)$ $r'_i:= \delta(r_i-m)$ are roots of $g(x)\equiv 0 \bmod p$,
and $0<r_1-m\le\dots\le r_n-m<p$  if $p$ is sufficiently large.
Hence writing $R_i:=r'_i$ in the case of $\delta=1$, otherwise $R_i := r'_{n-i+1}$,
we see  that   $R_i$ $(i=1,\dots,n)$  are local roots of 
$g(x)$ if $p$ is sufficiently large.  
Then we see that neglecting small primes
\begin{align*}
M(g,\mu)&=\{p\in Spl(g)\mid \beta_i\equiv R_{\mu(i)}\bmod \mathfrak{p}\,(1\le {}^\forall i\le n)
\text{ for }
{}^\exists\mathfrak{p}\mid p\}
\\
&=
\left\{
\begin{array}{ccc}
M(f,\mu) & \text{if}&\delta=1,
\\
M(f,{\mu'}) &\text{if}&\delta=-1,
\end{array}
\right.
\end{align*}
where $\mu'(i):=n+1-\mu(i)$.
Thus the set of $M(f,\mu)$ $(\mu\in S_n)$  depends on    the equivalence class of polynomials.
In case of $f(x)=x^3-3x+1, x^4 + 2x^3 + 3x^2 - 3x + 1, x^4 + 3x^2 - 2x + 2, x^4 - x^3 + 2x^2 + x + 1,
x^4+x^3+x^2+x+1, (x^2)^2 + x^2 - 1, (x^2)^2 + 1, x^4 + 4x^3 + 2x^2 - 4x - 1=(x^2+2x)^2-2(x^2+2x)- 1$ which exhaust all types (with respect to $Gal(\mathbb{Q}(f)/\mathbb{Q})$ and the existence of 
non-trivial  linear relations among roots) of
polynomials of degree 3,4 with $[\hat{\bm{G}}:\bm{G}_0]>1$,
it is likely that   sets $M(f,\mu)$ $(\mu\in S_n)$ with $\#M(f,\mu)=\infty$ characterize the above 
equivalence class among polynomials $g(x)$ under assumptions that $\mathbb{Q}(g)=\mathbb{Q}(f)$,  
$\deg g=\deg f$ and $f,g$ have the same  linear relations among roots.
When does the set of $M(f,\mu)$ $(\mu\in S_n)$ characterize the equivalence class of polynomials above ?

$\bullet$ About the computer experiment on the problem on p.39 : 
Let us consider the density of primes $p\in Spl(f)$ satisfying 
\begin{equation*}
\left\{\frac{g_1(r)}{p}\right\}<\left\{\frac{g_2(r)}{p}\right\} \text{ for some root } r  \text{ of } 
f(x)\equiv 0 \bmod p,
\end{equation*}
where the notation $\{\,\,\}$ denotes the decimal part as usual.

In case of $\deg g_1=\deg g_2 = 1$, Proposition \ref{prop7.8} below says that
the above condition is equivalent to, writing $g_i(x)=a_ix+b_i$ $(i=1,2)$ 
$$
(r_1/p,\dots, r_n/p)\in\{(x_1,\dots,x_n)\mid \{a_1x_i\} < \{a_2x_i\}({}^\exists i)\},
$$
and Conjecture 2 is applicable  (cf. Subsection \ref{sec8.1}).
Let us consider the case of $\deg g_1>1 $ or $\deg g_2 > 1$,
for which Conjecture 2 is helpless.

Let $f(x)$ be an irreducible polynomial of degree $n$  and 
a polynomial $F(x)$ of degree $m$ the minimal polynomial of an algebraic integer $\beta$ satisfying $\mathbb{Q}(f)=\mathbb{Q}(\beta)$.
Let $\alpha_1,\dots,\alpha_n$ be roots of $f(x)$ and write
$$
\alpha_i = c_i(\beta)\quad(c_i(x)\in\mathbb{Q}[x], \deg c_i < m).
$$
\begin{lem}\label{lem7.2}
Let a  prime $p\in Spl(f)$ be sufficiently large and $r$  a root of $F(x)\equiv0\bmod p$.
Then all roots of  $f(x)\equiv0\bmod p$ are given by 
$c_i(r)\bmod p$
$(i=1,\dots,n)$.
\end{lem}
\proof{
If a prime $p\in Spl(f)$ is sufficiently large, then $\frak{p}:=(\beta-r,p)$ is a prime ideal of $\mathbb{Q}(f)=\mathbb{Q}(\beta)$ over $p$, hence we see that $f(c_i(r))\equiv f(c_i(\beta))\bmod\mathfrak{p}=f(\alpha_i) =0$, i.e. $f(c_i(r))\equiv 0\bmod p$.
If $c_i(r)\equiv c_j(r)\bmod  p$ holds, then we have $\alpha_i=c_i(\beta) \equiv c_i(r)\bmod\mathfrak{p}\equiv c_j(r)\equiv \alpha_j$,
therefore the prime ideal $\mathfrak{p}$ divides $\alpha_i-\alpha_j$.
The number of such prime ideals is finite.

\qed
}

 Let us consider  the density of the complement set, that is
 the problem to look for  the density of primes $p\in Spl(f)$ satisfying 
$$
\left\{\frac{g_1(r)}{p}\right\}>\left\{\frac{g_2(r)}{p}\right\} \text{ for every root } r  \text{ of } 
f(x)\equiv 0 \bmod p
$$
 for not necessarily monic   polynomials $g_1(x),g_2(x)$ over $\mathbb{Z}$.
By Lemma \ref{lem7.2} the above inequality is equivalent to
$$
\left\{\frac{g_1(c_{i}(r) )}{p}\right\}>\left\{\frac{g_2(c_{i}(r))}{p}\right\}\quad (i = 1,\dots n)
$$
 for some root  $r$ of $F(x)\equiv 0 \bmod p$ except finitely many primes.
Defining rational numbers $g_{l,i,k} $ by
$$
g_l(c_i(\beta))=\sum_{k=0}^{m-1}g_{l,i,k}\beta^k\quad(l=1,2),
$$
we expect that under the assumption that $g_{l,i,k}\in\mathbb{Z}$
for all $l,i,k$,

the density of primes $p\in Spl(f)$ satisfying 
$$
\left\{\frac{g_1(r)}{p}\right\}>\left\{\frac{g_2(r)}{p}\right\} \text{ for every root } r  \text{ of } f(x)\equiv 0 \bmod p
$$ 
is equal to the volume of  the set
$$
\mathcal{D}=\left\{\bm{x}\in[0,1)^{m-1}\left|\left\{\sum_{k=1}^{m-1}g_{1,i,k}x_k\right\}> \left\{\sum_{k=1}^{m-1}g_{2,i,k}x_k\right\}
(1\le{}^\forall i\le n)\right.\right\}.
$$ 
Here the reason why the term for $k=0$ is omitted is due to  Proposition \ref{prop7.8}
below.
The computer experiment supports the above  in the case 
that $f(x)$ is the Galois polynomial $ x^3 -3x+1$ with roots $\beta:=\alpha,-\alpha^2-\alpha+2,\alpha^2-2$, which implies that all coefficients $g_{l,i,k}$ are integral.

The above leads us to   the following conjecture :
\begin{quote}
for an irreducible
(not necessarily Galois)
 polynomial $f(x)$
the sequence of vectors 
 $$
\left (\left\{\frac{r}{p}\right\}, \dots,\left\{\frac{r^{n-1}}{p}\right\}\right)
 $$
 is uniformly distributed in $[0,1)^{n-1}$,
   where $r$ runs over $n$ roots of $f(x)\bmod p$ with $p\in Spl(f)$.
\end{quote}
In case that some coefficient $g_{l,i,k}$ is not integral, is the density still given by the volume of a certain figure?
For example, for $f(x)=x^3-4x^2-9x+4$, 
which gives  the subfield of degree $3$ in $\mathbb{Q}(1^{\frac{1}{43}})$, conjugates
of  a root $\alpha$ are $c_1(\alpha):=\alpha,c_2(\alpha):=-\alpha^2/2+3\alpha/2+5,c_3(\alpha):=\alpha^2-5\alpha/2-1$. Hence, for $g(x) \equiv x^2-x\bmod(2\mathbb{Z}x+2\mathbb{Z}x^2)$, polynomials $g(c_i(x))$ are integral, but for $g(x)\equiv x,x^2\bmod(2\mathbb{Z}x+2\mathbb{Z}x^2)$,
they are not necessarily integral. 
For polynomials in the second group, the density is not equal to the volume of $\mathcal{D}$ experimentally. 
There is no algebraic integer $\beta\in\mathbb{Q}(f)$ such that $\mathbb{Z}[\beta]$ contains all roots $c_i(\alpha)$ of $f(x)$.

The above conjecture implies, applying to the domain $\{\bm{x}\in[0,1)^{n-1}\mid a\le x_k\le b\} $
  the uniform distribution of the sequence $\left\{\frac{r^k}{p}\right\}$ $(p\in Spl(f))$ 
where $r$ runs $n$ distinct roots of $f(x)\equiv0\bmod p$ and $k$ is a fixed integer
with $1\le k \le n-1$.
We can not replace $Spl(f)$ by smaller sets $M(f,\sigma)$.

$\bullet$
In this paragraph, set $f(x) = x^4+1$ whose 
roots are $\alpha_1 = \frac{1+\sqrt{-1}}{\sqrt{2}}, \alpha_2 = \frac{1-\sqrt{-1}}{\sqrt{2}}=-\alpha_1^3, 
\alpha_3 = -\frac{1-\sqrt{-1}}{\sqrt{2}}=\alpha_1^3, \alpha_4 = -\frac{1+\sqrt{-1}}{\sqrt{2}}=-\alpha_1$,
where linear relations among roots over rationals are spanned by $\alpha_1 + \alpha_4=
\alpha_2 + \alpha_3=0$.
It is easy to see that $\#\hat{\bm{G}}=8,\#\bm{G}_0=4$ and the transposition  $(3,4)$ is  in $\hat{\bm{G}}\setminus\bm{G}_0$,
and $\mathfrak{D}(f,\sigma)=\{(x_1,\dots,x_4) \mid 0\le x_1\le\dots\le x_4\le1,x_1+x_4=x_2+x_3=1\}$
except finite points for $\sigma\in\hat{\bm{G}}$,
which is identified to $\{(x_1,x_2) \mid 0\le x_1\le x_2\le\frac{1}{2}\}$ and $\{(x_3,x_4) \mid
\frac{1}{2}\le x_3\le x_4\le1\}$ by projections $(x_1,\dots,x_4)\mapsto (x_1,x_2),(x_3,x_4)$, respectively.

For at most quadratic polynomials $g_1(x),g_2(x)\in\mathbb{Z}[x]$, let us consider
\begin{align*}
M_i(X):=\{p\in Spl_X(f) \mid \{ \scalebox{0.99}{$\frac{g_1(r_i)}{p} $} \}
<  \{\scalebox{0.99}{$\frac{g_2(r_i)}{p}$}\} \}\quad (i =1,\dots,4),
\end{align*}
where $r_i$'s are local roots of the polynomial $f(x)$ for the prime $p$ 
and $\{a\}$ denotes the decimal part of $a$, and write $M(X):=(M_1(X),\dots,M_4(X))$
and $\#M(X):=(\#M_1(X),\dots,\#M_4(X))$ simply.
\noindent
For  non-zero integers $m,n$, the set $I_{m,n}$ of $x\in(0,1)$ satisfying $\{mx\} < \{nx\}$ is a union of intervals.
Hence, for linear polynomials $g_1(x)= mx,g_2(x)=nx$ we see
\begin{align*}
M_i(X)&=\left\{p\in Spl_X(f)\mid\left\{g_1(r_i)/p\right\}<\left\{g_2(r_i)/p\right\}\right\}
\\
&=\left\{p\in Spl_X(f)\mid r_i/p\in I_{m,n}\right\},
\end{align*}
and the density $\lim_{X\to\infty}\frac{\#M_i(X)}{\#Spl_X(f)}$ is described by the domain of
$\{\bm{x}\in(0,1)^4\mid x_i\in I_{m,n}\}$ by Conjecture.
For example,  for polynomials $g_1(x)=nx,g_2(x)=2nx$ $(n\ne0,n\in\mathbb{Z})$, 
the computer experiment suggests 
$$
\frac{\#M(X)}{\#Spl_X(f)}\to
\left\{
\begin{array}{ll}
 (\frac{n + 1}{2n}, \frac{n + 1}{2n}, \frac{n -1}{2n},\frac{n -1}{2n})&\text{ if } 2\nmid n,
\\[2mm]
 (\frac{n + 1}{2n}, \frac{n -1}{2n}, \frac{n + 1}{2n}, \frac{n -1}{2n})&\text{ if } 2\mid n,
\end{array}
\right.
$$
which is equal to 
$$
\frac{1}{\frac{1}{8}}(vol(D_1),\dots,vol(D_4)),
$$
where
\begin{align*}
D_i = \{(x_1,x_2)\mid 0\le x_1 \le x_2 \le \frac{1}{2},\{g_1(x_i)\}<\{g_2(x_i)\}\}\quad(i=1,2),
\\
D_i = \{(x_3,x_4)\mid \frac{1}{2} \le x_3 \le x_4 \le 1,\{g_1(x_i)\}<\{g_2(x_i)\}\}\quad(i=3,4),
\end{align*}
and the denominator $\frac{1}{8}$ is the volume of the whole spaces $\{(x_1,x_2) \mid 
0\le x_1\le x_2\le \frac{1}{2}\}$, $\{(x_3,x_4)\mid \frac{1}{2}\le x_3 \le x_4 \le 1\}$,
which are identified with $\{(x_1,\dots,x_4)\mid 0 \le  x_1\le \dots\le  x_4\le 1,x_1+x_4=x_2+x_3=1\}$ as above.
These match the experiment above quite well and Conjecture \ref{conj2}.

\noindent
Other cases, i.e. cases of $\deg(g_1(x))$ or $\deg(g_2(x))=2$ are under the assumption that $\deg(g_1(x)) \le 2, \deg(g_2(x)) \le 2, g_1(0)
\linebreak[3]=g_2(0)=0$,
 (i)  $\deg(g_1(x)) = 2$  and $g_2(x) = nx\,(n\ne0)$, 
(ii) $\deg(g_1(x))=\deg(g_2(x)) = 2$ and $g_1(x)-g_2(x)=nx,(n\ne0)$, or
(iii)  $\deg(g_1(x))=\deg(g_2(x)) \linebreak[3]= \deg(g_1(x)-g_2(x))=2$,
and in case of (i),(ii) the density looks like
$$
\left\{
\begin{array}{ll}
(\frac{3n-2}{6n},\frac{3n-1}{6n},\frac{3n+1}{6n},\frac{3n+2}{6n})&\text{ if }2\nmid n,
\\[2mm]
(\frac{3n-2}{6n},\frac{3n+2}{6n},\frac{3n-2}{6n},\frac{3n+2}{6n}  )&\text{ if }2\mid n,
\end{array}
\right.
$$
and in case of (iii) it looks like $(\frac{1}{2},\frac{1}{2},\frac{1}{2},\frac{1}{2})$.
The author does not know how to elucidate this, in particular even the reason of 
$$
\lim_{X\to\infty}\frac{\#M_1(X) + \#M_4(X)}{\#Spl_X(f)}=\lim_{X\to\infty}\frac{\#M_2(X) + \#M_3(X)}{\#Spl_X(f)}=1.
$$

How about the case of polynomials of degree $3$ ?
Since we have $r_3\equiv r_1^3 \bmod p$ for $p \in M(f,id)$ as in  Example 2 below,
the inequality $\{\frac{g_1(r_1)}{p}\}<\{\frac{g_2(r_1)}{p}\}$ is equivalent to
$\{\frac{r_3}{p}\}<\{\frac{r_1^2}{p}\}$ for $g_1(x) = x^3,g_2(x)=x^2$, for example.
The density of such primes in $M(f,id.)$ looks like $\frac{1}{3}$.

It may be appropriate to start by $Spl(f,\sigma)$ or $M(f,\sigma)$ instead of $Spl(f)$ for a  general polynomial.

\vspace{2mm}
$\bullet$ Next, let us consider polynomials in many variables :
Write
\begin{align*}
M_X(f,\mu)&:=\{p\in M(f,\mu) \mid p<X\}
\\
&=\{ p\in Spl(f) \mid p<X,   \alpha_i\equiv r_{\mu(i)}\bmod\mathfrak p\,\,\,(1\le{}^\forall i\le n) \}
\end{align*}
Let $g(\bm{x}),h(\bm{x})\in\mathbb{Z}[x_1,\dots,x_n]$.
For  a permutation $\mu$,
what is the density
\begin{align*}
&D_f(g,h;\mu)
\\
:=
&\lim_{X\to\infty}
\frac
{
\#\left\{p \in M_X(f,\mu) \left|\left\{\frac{ g(r_{\mu(1)},\dots,r_{\mu(n)})}{p}\right\}<
\left\{\frac{ h(r_{\mu(1)},\dots,r_{\mu(n)})}{p}\right\}\right.\right\}
}
{
\#M_X(f,\mu)
}\,?
\end{align*}

First, let us give a remark. Let a polynomial $g(x_1,\dots,x_n)\in \mathbb{Q}[x_1,\dots,x_n]$ satisfy $g(r_1,\dots,r_n)\equiv 0\bmod p$ for infinitely many primes $p\in M(f,\mu)$.
Taking an element $\alpha\in\mathbb{Q}(f)$ such that $\mathbb{Q}(\alpha)=\mathbb{Q}(f)$ and write
$\alpha_i=h_i(\alpha)$ for a polynomial $h_i(x)\in \mathbb{Q}[x]$.
Then we have $g(r_1,\dots,r_n)\equiv g(\alpha_{\mu^{-1}(1)},\dots,\alpha_{\mu^{-1}(n)})=g(h_{\mu^{-1}(1)}(\alpha)
,\dots, h_{\mu^{-1}(n)}(\alpha))\bmod \mathfrak{p}$ for infinitely many primes $p\in M(f,\mu)$.
Therefore we see $g(h_{\mu^{-1}(1)}(\alpha)\dots, h_{\mu^{-1}(n)}(\alpha))=0$, that is the polynomial \newline
 $g(h_{\mu^{-1}(1)}(x),\dots, h_{\mu^{-1}(n)}(x))$ is divisible by the minimal polynomial of $\alpha$.

Let us take a polynomial  $g_1(\bm{x}) \in \mathbb{Q}[x_1,\dots,x_n]$ such that $ g_1(\alpha_1,\dots,\alpha_n)
=g(\alpha_1,\dots,\alpha_n)$.
Thus, we see that $g(r_{\mu(1)},\dots,r_{\mu(n)})\equiv g_1
(r_{\mu(1)},\dots,r_{\mu(n)}) \bmod p$ for a prime $p\in M(f,\mu)$.
In general, $g(r_1,\dots,r_n)\equiv g_1(r_1,\dots,r_n) \bmod p$ $({}^\forall p\in Spl(f))$
is not  necessarily true.
If we may suppose that  $g_1(\bm{x})$ and similarly $h_1(\bm{x})$ are linear with integer coefficients, 
then we have only to refer to Conjecture $2$ if $Spl(f,\mu)=M(f,\mu)$.
A sufficient condition to $g_1(\bm{x})$ being linear is that  $\mathbb{Q}(f)=\mathbb{Q}(\alpha_1)$ 
and $t = 1$, because they imply that
$[\mathbb{Q}(f):\mathbb{Q}]= n$ and $\dim\mathbb{Q}[\alpha_1,\dots,\alpha_n,1]=n$,
that is $\mathbb{Q}(f)=\mathbb{Q}[\alpha_1,\dots,\alpha_n,1]$.
Otherwise, what happens? Is the density the volume of some figure?
We give below two examples $f(x) = x^3+2$, which is non-Galois, and $f(x) = x^4+ 1$ in the case of $t=2>1$.

\noindent
{\bf Example 1} 
Let $f(x) = x^3 +a_2x^2 +a_1x+a_0\in\mathbb{Q}[x]$ be an irreducible polynomial with
roots $\alpha_1,\alpha_2,\alpha_3$ and suppose that  the extension $\mathbb{Q}(\alpha_1)/\mathbb{Q}$ is not a Galois extension.
Then for a polynomial $g(\bm{x})\in\mathbb{Q}[x_1,x_2,x_3]$, there are rational numbers $c_i$ such that $g(\alpha_1,\alpha_2,\alpha_3)=c_0 + c_1\alpha_1 + c_2\alpha_1^2+
c_3\alpha_2 + c_4\alpha_1\alpha_2 + c_5\alpha_1^2\alpha_2$. 
Note that the product $\alpha_1\alpha_2$  replaced by $\alpha_2^2$ by  the 
equation $\alpha_1\alpha_2 =-a_1-a_2(\alpha_1+\alpha_2)-\alpha_1^2-\alpha_2^2$.
As remarked at the beginning of this section, we have $M(f,\mu)=Spl(f)$ for every permutation $\mu$,
in particular $Spl(f)=M(f,id)$.

Let us make computer experiments with respect to the  density $D_f(g,h):=D_f(g,h;id)$
above for  $f(x)=x^3+2$.
We  suppose that polynomials $g(\bm{x}),h(\bm{x})$ are of the following ``reduced''  type
\begin{gather*}
g(\bm{x})=g_1x_1+g_2x_2+g_3x_1^2+g_4x_2^2+ g_5x_1^2x_2,
\\
h(\bm{x})=h_1x_1+h_2x_2+h_3x_1^2+h_4x_2^2+ h_5x_1^2x_2
\end{gather*}
with rational  integer coefficients $g_i,h_i$ without constant term by Proposition \ref{prop7.8}.
Denote its homogeneous part of degree $i$ by $g_i(\bm{x}),h_i(\bm{x})$.
For  polynomials $g(\bm{x}),h(\bm{x})$ with $g(\bm{x})\ne0,h(\bm{x})\ne0$ and $g(\bm{x})\ne h(\bm{x})$ as above, let us give observations.
If both of $g,h$ are linear, i.e.
$g(\bm{x})=g_1(\bm{x}),h(\bm{x})=h_1(\bm{x})$,
then we have only to refer to Conjecture 2, and so
we suppose that $g(\bm{x})$ or $h(\bm{x})$ is not linear.

\begin{enumerate}
\item    
Lemma \ref{lem7.4} implies
$$
D_f(g,h)+D_f(h,g)=1.
$$
\item
{The case of $g(\bm{x}) = g_1(\bm{x})$,  $h_2(\bm{x}) \ne0$ and $g_5 = h_5 = 0$.} 
\begin{equation*}
    D_f(g,h)= \frac{1}{2}+\left\{
    \begin{array}{cl}
    \frac{1}{4g_1}    &\text{ if }g_2=0 \text{ or }g_1=g_2, \\[1mm]
     -\frac{1}{4g_1}    &\text{ if }g_2=2g_1\text{ or }g_1+ g_2=0,\\[1mm]
    0 &\text{ otherwise},
    \end{array}
    \right.
\end{equation*}
which depends only on $g(\bm{x})$.
\item{The case of $g_2(\bm{x})h_2(\bm{x})\ne0$  and $g_5 = h_5 = 0.$}
The density $D_f(g,h)$  is, writing $k_i := g_i - h_i$
\begin{align*}
\frac{1}{2}+
\left\{
\begin{array}{ll}
 -\frac{1}{4k_1} &\text{ if } k_3=k_4=0 \text{ and }(k_2=0\text{ or }k_1=k_2),
\\
\frac{1}{4k_1} &\text{ if } k_3=k_4=0\text{ and }(k_2=2k_1\text{ or } k_1+k_2=0),
\\
0&\text{ otherwise},
\end{array}
\right.
\end{align*}
which depends only on $g(\bm{x})-h(\bm{x})$.
\item
{The case of $g_5\ne0$ or $h_5\ne0$.} Write $k_i := g_i - h_i$.
\begin{enumerate}
\item{The case of $k_3=k_4=k_5=0$.}
The density $D_f(g,h)$  is
\begin{align*}
\frac{1}{2}+
\left\{
\begin{array}{ll}
 -\frac{1}{4k_1} &\text{ if }k_2=0\text{ or } k_1=k_2,
\\
\frac{1}{4k_1} &\text{ if }   k_2=2k_1\text{ or }   k_1+k_2=0,
\\
0&\text{ otherwise}.
\end{array}
\right.
\end{align*}
\item In the case of ($k_3\ne0$ or $k_4\ne0$) and $ k_5=0$, $D_f(g,h)=\frac{1}{2}.$
\item{ The case of $k_5\ne0$ and $g_5\ne0$.} The density $D_f(g,h)$  is
\begin{align*}
\frac{1}{2}+
\left\{
\begin{array}{ll}
 -\frac{1}{4h_1} &\text{ if } h_1\ne0, h_3=h_4=h_5=0 \text{ and }
\\
&\hspace{3mm} ( h_2=0 \text{ or }h_1=h_2),
\\
\frac{1}{4h_1} &\text{ if }   h_1\ne0, h_3=h_4=h_5=0 \text{ and }
\\
&\hspace{3mm} (h_2=2h_1\text{ or }   h_1+h_2=0),
\\
0&\text{ otherwise}.
\end{array}
\right.
\end{align*}
\end{enumerate}
\end{enumerate}
It may be strange that the density seems to be given by the above for ``reduced'' polynomials for 
another polynomial $x^3+x^2+1$ and others. Of course, it is different for general polynomials, for example for $g(\bm{x})=x_1^3$,
$g(\bm{r})\equiv -2\bmod p$ if $f(x)=x^3+2$ and $g(\bm{r})\equiv -r_1^2-1\bmod p$ if $f(x)=x^3+x^2+1$.
The situation is unclear.

\noindent
{\bf Example 2} Set $f(x) = x^4+1$ and let
roots be $\alpha_1 := \frac{1+\sqrt{-1}}{\sqrt{2}}, \alpha_2 := \frac{1-\sqrt{-1}}{\sqrt{2}}=-\alpha_1^3, 
\alpha_3 := -\frac{1-\sqrt{-1}}{\sqrt{2}}=\alpha_1^3, \alpha_4 := -\frac{1+\sqrt{-1}}{\sqrt{2}}=-\alpha_1$ as before, and then the basis $\bm{m}_j$ of $LR_0\cap\mathbb{Z}^4$ 
are $(1,0,0,1),(0,1,1,0)$.
Noting that $\mathbb{Q}(f)=\mathbb{Q}[1,\alpha_1,\alpha_1^2,\alpha_2]$,
we see that for a polynomial $g(\bm{x})$  in $\mathbb{Q}[x_1,\dots,x_4]$, there are rational numbers
$c_i$ so that $g(\alpha_1,\dots,\alpha_4)=
c_0+c_1\alpha_1+c_2\alpha_1^2 + c_3\alpha_2 $.
So, we may assume that $g(\bm{x})=c_0 + c_1x_1+c_2x_1^2+ c_3x_2$.
For a permutation $\mu$, let us determine $M(f, \mu)$.
Since the condition $p\in M(f,\mu)$ is equivalent to $\alpha_i \equiv r_{\mu(i)} \bmod\frak{p}$, for $\sigma\in Gal(\mathbb{Q}(f)/\mathbb{Q})$ satisfying $\sigma(\alpha_{\mu^{-1}(1)})=\alpha_1$, it is equivalent to $\sigma(\alpha_i)\equiv r_{\mu(i)}\bmod\sigma(\frak{p})$, i.e. $\sigma(\alpha_{\mu^{-1}(i)})\equiv r_i\bmod\sigma(\frak{p})$,
 hence we have $M(f,\mu)=M(f,\mu')$ with $\mu'(1)=1$ for $\mu'$ defined by $\alpha_{\mu'(i)}=\sigma(\alpha_{\mu^{-1}(i)})$.
Hence we  assume $\mu(1)=1$. For a prime $p\in M(f,\mu)$, the conditions
$\alpha_1+\alpha_4=0,r_1+r_4=p$ imply $\mu(4)=4$, hence  $\mu(3)=2$ or $3$.
By $\alpha_3 = \alpha_1^3\equiv r_1^3\bmod\frak{p}$, we have $r_{\mu(3)}\equiv r_1^3\bmod p$.
Thus we have, supposing $\mu(1)=1$ $\mu = id.$ or the transposition $(2,3)$
and
$$
M(f,\mu)=\{p\in Spl(f)\mid r_{\mu(3)}\equiv r_1^3\bmod p\}.
$$
Note that $Spl(f) = Spl(f,id.)=M(f,id.)\cup M(f,(2,3))$.
Under the assumptions 
\begin{enumerate}
\item
the densities of $M(f,id.)$ and $M(f,(2,3))$ in $Spl(f)$ are equal to $\frac{1}{2}$,
\item
the densities $D_f(g,h;id.),D_f(g,h;(2,3))$ are equal,
\end{enumerate}
we see that $D_f(g,h;\mu)=\frac{1}{2}(D_f(g,h;id.)+D_f(g,h;(2,3)))$ is equal to
$$
\lim_{X\to\infty}\frac{\#\{p\in Spl_X(f)\mid \{\frac{g(\bm{r})}{p}\}<\{\frac{h(\bm{r})}{p}\}\}}
{\#Spl_X(f)}.
$$
Thus, in case that $g(\bm{x}),h(\bm{x})$ are linear, the density is given by Conjecture \ref{conj2}.
For example, let  $g_1(\bm{x})=2x_1-x_2,g_2(\bm{x})=x_2$; then the computer  experiment suggests 
$D_f(g_1,g_2;\mu)=\frac{1}{2}$,
which match $vol(\{ (x_1,x_2)\mid0<x_1<x_2<\frac{1}{2} ,\{2x_1-x_2\}<\{x_2\}  \})=vol(\{  (x_1,x_2)\mid
0<x_1<x_2<\frac{1}{2} ,x_1<x_2<2x_1 \})=\frac{1}{16}$ and $vol(\{  (x_1,x_2)\mid
0<x_1<x_2<\frac{1}{2} \})=\frac{1}{8}$.

We restrict polynomials to the following type similarly to the above example:
$$
g(\bm{x}) = g_1x_1 + g_2x_1^2 + g_3x_2,
$$
and write $\,g_1(\bm{x}):=g_1x_1 + g_3x_2,g_2(\bm{x}):=g_2x_1^2,$
and suppose  that $g(\bm{x}) \ne 0, h(\bm{x})\ne 0, g(\bm{x})\ne h(\bm{x})$
and either of $g_2$ or $h_2$ is not zero.
Writing $k_i:=h_i - g_i$, the condition $g(\bm{x})\ne h(\bm{x})$ is equivalent to
$k_i\ne 0$ for some $i$.
Then the observation is that, independently of $\mu$ $D_f(g,h)
:=D_f(g,h;\mu)$ $(\mu=id.$ or $(2,3))$ is
\begin{enumerate}
    \item 
$$
D_f(g,h)+D_f(h,g)=1.
$$ 
\item
{The case of $g(\bm{x}) = g_1x_1+g_3x_2,
h(\bm{x}) = h_1x_1+h_2x_1^2+h_3x_2$ with
 $ h_2(\bm{x})\ne0$.}
Then  $D_f(g,h) $ is   
$$
\frac{1}{2} +
\left\{
\begin{array}{ll}
\frac{-1}{3g_3}&\text{ if }  g_1=0,g_3\equiv0\bmod 2,
\\[2mm]
\frac{1}{6g_3} &\text{ if }  g_1=0,g_3\equiv1\bmod 2,
\\[2mm]
\frac{1}{3g_1}&\text{ if }  g_3=0,
\\[2mm]
\frac{-1}{3g_1}&\text{ if }  g_1 +g_3=0,
\\[2mm]
0&\text{ otherwise, i.e. if }
 g_1g_3(g_1 +g_3)\ne0,
\end{array}
\right.
$$
which depends only on $g(\bm{x})$.
\item
{The case of $g(\bm{x}) = g_1x_1+g_2x_1^2+g_3x_2,
h(\bm{x}) = h_1x_1+h_2x_1^2+h_3x_2$ with
$g_2h_2\ne0$.}
Then  $D_f(g,h) $ is
$$
\frac{1}{2}+
\left\{
\begin{array}{ll}
\frac{1}{3k_3}& \text{ if }k_1=0,k_2 =0 , k_3\equiv0\bmod2 ,
\\[2mm]
\frac{-1}{6k_3}&\text{ if } k_1=0,k_2 =0 , k_3\equiv1\bmod2 ,
\\[2mm]
\frac{-1}{3k_1}&\text{ if } k_2=0, k_3=0,
\\[2mm]
\frac{1}{3k_1}&\text{ if } k_2=0,k_1+k_3=0,
\\[2mm]
0&\text{ otherwise, i.e. if }
k_2\ne0\text{ or }k_1k_3(k_1+k_3)\ne0,\end{array}
\right.
$$
which depends only on $g(\bm{x})-h(\bm{x})$.
\end{enumerate}
The above seems to be true for another polynomial $f(x) = (x^2 + 2x)^2 - 2(x^2 + 2x) - 1$  with Galois group 
$\mathbb{Z}/4\mathbb{Z}$, too. 
The basic information is : Let a root of $f(x)$ be $\alpha$ and roots be
$\alpha_1:=\alpha,\alpha_2:=-\alpha^3-3\alpha^2+1,\alpha_3:=\alpha^3+3\alpha^2-3,\alpha_4:=-\alpha-2$.
Then linear relations among roots are spanned by $\alpha_1+\alpha_4=\alpha_2+\alpha_3=-2$, and we see $\hat{\bm{G}}=\langle(1,4),(2,3),(1,2)(3,4)\rangle$ and $r_1+r_4=r_2+r_3=p-2$. 
The condition $Spl(f,\sigma)\ne \emptyset$ is equivalent to $\sigma\in \hat{G}$ and $\hat{\bm{G}}/\bm{G}_0=\langle(1,4)\rangle$. A prime $p\in Spl(f,id)$ is in $M(f,id)$ if and only if $r_2\equiv -r_1^3-3r_1^2+1\bmod p$
(in $M(f,(1,4))$ if and only if $r_2\equiv r_1^3+3r_1^2-3\bmod p$).

Suppose that $g_2h_2\ne0$; then we see $D_f(g,h) = \frac{1}{2}$ in the case of
$g_2\ne h_2$, otherwise $D_f(g,h)=D_f(g-h,l)$ for any polynomial $l(\bm{x})$
with $l_2\ne 0$.


$\bullet$ In this paragraph, for integers $m,p$, we denote the integer $a$ satisfying $a\equiv m \bmod p$ and $0\le a<p$ by $[m]_p$, and let the polynomial $f(x)$ be irreducible and of degree $n\,(>1)$ with roots $\alpha_i$ as usual.
\begin{lem}\label{lem7.3}
Let $a$ be a positive integer.

\noindent
{\em(i)}
For a polynomial  $g(x)\in\mathbb{Z}[x]$,  there are only finitely many primes $p$ satisfies $ [g(r)]_p= a$ (resp. $ [g(r)]_p= p-a$) in the case of $g(x)-a$ ( resp. $g(x)+a$)
being not divisible by $f(x)$  in the ring $\mathbb{Q}[x]$.

\noindent
{\em(ii)}
For  a polynomial  $g(\bm{x})\in\mathbb{Z}[x_1,\dots,x_n]$,
there are only finitely many primes $p\in Spl(f)$ such that $[g(\bm{r})]_p= a$
(resp.  $ [g(r)]_p= p-a$) if  $g(\alpha_{\sigma(1)},\dots,\alpha_{\sigma(n)})\ne a$  (resp.  $g(\alpha_{\sigma(1)},\dots,\alpha_{\sigma(n)})\ne -a$) for any permutation $\sigma$.
Here $\bm{r}=(r_1,\dots,r_n)$
is the vector of  local roots for $p\in Spl(f)$. 
\end{lem}
\proof
{(i)
Suppose that there are infinitely many primes $p$ satisfying $  [g(r)]_p= a$.
Writing $g_0(x)=g(x)-a$, we see that there are infinitely many primes $p$ such that 
there is an integer $r$ satisfying $f(r)\equiv g_0(r) \equiv0\bmod p$.
Since the polynomial $f(x)$ is irreducible and  $g_0(x)$ is not divisible by $f(x)$, 
 there are integral polynomials $h_1(x),h_2(x)$ and a non-zero integer $h$ such that $h_1(x)f(x)+h_2(x)g_0(x)=h$.
So $h$ is divisible by infinitely many primes as above, which is the contradiction.
In the other case, we have only to put $g_0(x)=g(x)+a$.

\noindent
(ii) Suppose that there are infinitely many primes $p\in Spl(f)$ such that  $ [g(\bm{r})]_p=a$.
Then there is a permutation $\sigma$ such that $\alpha_{\sigma(i)}\equiv r_i
\bmod\frak{p}$ for  infinitely many primes $p\in Spl(f)$, where $\frak{p}$ is a
prime ideal over $p$ in $\mathbb{Q}(f)$, which implies $g(\alpha_{\sigma(1)},\dots,\alpha_{\sigma(n)})\equiv a\mod\frak{p}$, that is  $g(\alpha_{\sigma(1)},\dots,\alpha_{\sigma(n)})=a$.
It is the contradiction.
The proof of the other case is similar.
\qed
}
\begin{lem}\label{lem7.4}
Let $a$ be a positive integer.
{\em(i)} Let $g_i(x)$ $(i=1,2)$ be
integral polynomials for which $g_2(x)-g_1(x)-a$ is not divisible by  $f(x)$.
Then there are only finitely many primes $p\in Spl(f)$ such that
$$
 [g_2(r)]_p = [g_1(r)]_p+a\quad\text{for some root $r$ of }f(x)\equiv0\bmod p.  
$$

{\em(ii)} Let  $g_i(\bm{x})$ $(i=1,2)$ be polynomials in $\mathbb{Z}[x_1,\dots,x_n]$ and $a$ a positive integer.
Suppose that  $g_2(\alpha_{\sigma(1)},\dots,\alpha_{\sigma(n)}) - g_1(\alpha_{\sigma(1)},\dots,\alpha_{\sigma(n)})\ne a$ for every permutation $\sigma$.
Then there are only finitely many primes $p\in Spl(f)$ such that
$$
 [g_2(\bm{r})]_p= [g_1(\bm{r})]_p+a.
$$ 
\end{lem}
\proof
{(i) Suppose that there are infinitely  many primes $p\in Spl(f)$, which satisfies the equality in question.
By writing $g(x):=g_2(x)-g_1(x)$, the supposition means that there are infinitely many primes $p$ such that  $[g(r)]_p \equiv [g_2(r)]_p-[g_1(r)]_p=a\bmod p$, 
hence $[g(r)]_p=a$.
The previous lemma completes the proof.

\noindent
(ii) Assume that there are infinitely many primes $p$ satisfying $[g_2(\bm{r})]_p - [g_1(\bm{r})]_p = a$.
Writing $g(\bm{x}):=g_2(\bm{x})-g_1(\bm{x}) $, we see that $[g(\bm{r})]_p
\equiv[g_2(\bm{r})]_p - [g_1(\bm{r})]_p \equiv a\bmod p$, 
hence  $[g(\bm{r})]_p=a$. The previous lemma completes the proof.

\qed
}
\begin{prop}\label{prop7.8}
Let $a, b$ be integers.

\noindent
{\em (i)}
Let  $g_i(x)$ $(i=1,2)$ be
polynomials over $\mathbb{Z}$ such  that none of $g_1(x),g_2(x)$, $g_2(x)-g_1(x)$ are congruent to a constant modulo  $f(x)$.
Then for a prime $p$ with an integer $r$ satisfying $f(r)\equiv 0 \bmod p$,
the inequalities $\{\frac{g_1(r)+a}{p}\}<\{\frac{g_2(r)+b}{p}\}$ and $\{\frac{g_1(r)}{p}\}
<\{\frac{g_2(r)}{p}\}$ are equivalent except a finitely many primes.

\noindent
{\em (ii)}
Let  $g_i(\bm{x})$ $(i=1,2)$ be
polynomials in $\mathbb{Z}[x_1,\dots,x_n]$ such  that none of $g_1(\bm{x}),g_2(\bm{x})$, $g_2(\bm{x})-g_1(\bm{x})$ are   rational integers
at $\bm{x}=(\alpha_{\sigma(1)}, \dots,\alpha_{\sigma(n)})$ for every permutation $\sigma$.
Then for a prime $p\in Spl(f)$,
the inequalities $\{\frac{g_1(\bm{r})+a}{p}\}<\{\frac{g_2(\bm{r})+b}{p}\}$ and 
$\{\frac{g_1(\bm{r})}{p}\} <\{\frac{g_2(\bm{r})}{p}\}$ are equivalent except a finitely many primes.
\end{prop}
\proof
{(i)
For simplicity, the notation $A\overset{f}{\Leftrightarrow}B$ means that statements 
$A$ and $B$ are equivalent except finitely many primes $p$ in this proof.
We see that 
\begin{align*}
&\left\{\frac{g_1(r)+a}{p}\right\}<\left\{\frac{g_2(r)+b}{p}\right\} 
\\
\Leftrightarrow
&
\left\{\frac{[g_1(r)]_p+a}{p}\right\}<   \left\{\frac{[g_2(r)]_p+b}{p}\right\}
\\
\overset{f}{\Leftrightarrow}
&
[g_1(r)]_p+a<[g_2(r)]_p+b
\\
&\hspace{5mm}\text{ (since $0<[g_1(r)]_p+a,[g_2(r)]_p+b<p$ by Lemma \ref{lem7.3})}
\\
\Leftrightarrow
&
[g_2(r)]_p\le [g_1(r)]_p<[g_2(r)]_p+b-a, \text{ or }
\\
&    [g_1(r)]_p<\min([g_2(r)]_p,[g_2(r)]_p+b-a) 
\\
\overset{f}{\Leftrightarrow}
&[g_1(r)]_p<\min([g_2(r)]_p,[g_2(r)]_p+b-a)\quad\text{ (by Lemma \ref{lem7.4})},
\end{align*}
which is equivalent to $[g_1(r)]_p<[g_2(r)]_p$, i.e. $\{\frac{g_1(r)}{p}\}
 <\{\frac{g_2(r)}{p}\}$ in the case of  $b-a\ge0$.
Suppose  $b-a<0$; then the above inequality is 
$ [g_1(r)]_p<[g_2(r)]_p+b-a$. 
Lemma \ref{lem7.4} says that there are only  finitely many
primes $p$ satisfying $[g_2(r)]_p+b-a\le [g_1(r)]_p\le [g_2(r)]_p$, 
hence we have $[g_1(r)]_p<\min([g_2(r)]_p,[g_2(r)]_p+b-a)\overset{f}{\Leftrightarrow}[g_1(r)]_p<[g_2(r)]_p$.
Thus,
 we have $\{\frac{g_1(r)+a}{p}\}<\{\frac{g_2(r)+b}{p}\}   \overset{f}{\Leftrightarrow} [g_1(r)]_p<[g_2(r)]_p\,(\Leftrightarrow\{\frac{g_1(r)}{p}\}<\{\frac{g_2(r)}{p}\}) $.
The other case is similar.
\qed
}

\subsection{Polynomial without non-trivial linear relations among roots}\label{sec8.1}
In this subsection, assume that  the polynomial  $f(x)$ is of degree $\tilde{n}\,(>1)$ and has no non-trivial linear relations among roots, i.e. $t:=\dim LR=1$, and for linear forms $g_1(x)=mx,g_2(x)=nx$ with  distinct non-zero integers $m,n$, 
we study the density
$$
\lim_{X\to\infty}\frac{\#\{p\in Spl_X(f)\mid \{\frac{g_1(r)}{p}\}<\{\frac{g_2(r)}{p}\} \text{ for }^\exists r \text{ s.t. }f(r)\equiv0\bmod p \}}
{\#Spl_X(f)},
$$
where $\{\frac{g_i(r)}{p}\}$ denotes the decimal part of $\frac{g_i(r)}{p}$
as usual.

The condition $t=1$ implies $\hat{\bm{G}}=S_n$, and so $Spl(f,{}^\forall\sigma)
=Spl(f)$.
We see that
\begin{align*}
&\#\{p\in Spl_X(f)\mid \left\{\frac{mr}{p} \right\}  < \left\{\frac{nr}{p}\right\}\text{ for }^\exists r \in\mathbb{Z}\text{ s.t. }f(r)\equiv0\bmod p \}
\\
=&\#Spl_X(f)-\#\{p\in Spl_X(f)\mid \left\{\frac{nr}{p} \right\}  \le \left\{\frac{mr}{p}\right\}\text{ for }^\forall r \text{ s.t. }f(r)\equiv0\bmod p \},
\end{align*}
hence
\begin{align*}
d_{m,n}:=&\lim_{X\to\infty}\frac{\#\{p\in Spl_X(f)\mid \{\frac{mr}{p} \}  < \{\frac{nr}{p}\}\text{ for }^\exists r \text{ s.t. }f(r)\equiv0\bmod p \}}
{\#Spl_X(f)}
\\
=&\,1-\lim_{X\to\infty}\frac{\#\{p\in Spl_X(f)\mid \{\frac{nr}{p} \}  < \{\frac{mr}{p}\}\text{ for }^\forall r \text{ s.t. }f(r)\equiv0\bmod p \}}{\#Spl_X(f)}
\\
=&\,
1 - \frac{vol(I_{n,m}^{\tilde{n}}\cap\hat{\mathfrak{D}}_{\tilde{n}})}{vol(\hat{\mathfrak{D}}_{\tilde{n}})}
\quad(\text{under Conjecture \ref{conj2}}),
\end{align*}
since the number of the prime $p$ satisfying  $\{\frac{nr}{p} \}  = \{\frac{mr}{p}\}$ is finite
and  we write 
$$
I_{n,m} := \{x\in(0,1) \mid\{nx\}<\{mx\}\}.
$$
Let  us evaluate the ratio of volumes.
Note that the dimension $\dim I_{n,m}^{\tilde{n}}\cap\hat{\mathfrak{D}}_{\tilde{n}}$ is less than or equal to ${\tilde{n}}-1$ and the volume is considered as the $({\tilde{n}}-1)$-dimensional set under the assumption $\dim LR=1$.
We know $vol(\hat{\mathfrak{D}}_{\tilde{n}})=\frac{\sqrt{{\tilde{n}}}}{{\tilde{n}}!}$.
On the other hand, we see that
\begin{align*}
&\{(x_1,\dots,x_{\tilde{n}})\mid0\le x_i\le1, x_i\in I_{n,m}\,(^\forall i),\sum x_i\in\mathbb{Z}\}
\\
=&\mathop{\cup}_{\sigma\in S_{\tilde{n}}}\{(x_1,\dots,x_{\tilde{n}})\mid0\le x_{\sigma(1)}\le \dots\le x_{\sigma({\tilde{n}})}\le1,
 x_i\in I_{n,m}\,(^\forall i),\sum x_i\in\mathbb{Z}\},
\end{align*}
and the permutation induces the orthogonal transformation on $\mathbb{R}^{\tilde{n}}$, and the dimension of the intersection of the subsets corresponding to distinct permutation is less than ${\tilde{n}}-1$,
hence we see 
$$
vol(I_{n,m}^{\tilde{n}}\cap\hat{\mathfrak{D}}_{\tilde{n}})
=\frac{1}{{\tilde{n}}!}vol(
\{(x_1,\dots,x_{\tilde{n}})\in(0,1)^{\tilde{n}}\mid x_i\in I_{n,m},\sum x_i\in\mathbb{Z}\}
)
$$
and  
\begin{equation}\label{eq85}
\frac{vol(I_{n,m}^{\tilde{n}}\cap\hat{\mathfrak{D}}_{\tilde{n}})}{vol(\hat{\mathfrak{D}}_{\tilde{n}})}
=\frac{1}{\sqrt{{\tilde{n}}}}vol(
\{(x_1,\dots,x_{\tilde{n}})\in(0,1)^{\tilde{n}}\mid x_i\in I_{n,m},\sum x_i\in\mathbb{Z}\}
).
\end{equation}
In the rest of this subsection we will show the following :
\begin{thm} \label{thm7.1}
For  integers  $M,N$  satisfying $0<N<M$ and $(M,N)=1$, we have
\begin{align*}
&\frac{vol(I_{N,M}^{\tilde{n}}\cap\hat{\mathfrak{D}}_{\tilde{n}})}{vol(\hat{\mathfrak{D}}_{\tilde{n}})}
\\
=& \frac{1}{2^{\tilde{n}}}-
\frac{B_{\tilde{n}}(0)}{\tilde{n}!(MN(M-N))^{\tilde{n}}}
\biggl\{
3(MN(M-N))^{\tilde{n}}+M^{2\tilde{n}}+N^{2\tilde{n}}+(M-N)^{2\tilde{n}} 
\\
&\hspace{8mm}  -2((M-N)N)^{\tilde{n}} -2M^{\tilde{n}}\left((M-N)^{\tilde{n}}+N^{\tilde{n}}\right)
\biggr\}
,
\end{align*}
where $B_{\tilde{n}}(x)$ is the Bernoulli polynomial.
\end{thm}
The above matches computer experiment of $d_{M,N}$.
We note that 
\begin{itemize}
\item[(i)]
the ratio of volumes is independent of $M,N$ for  odd $\tilde{n}>1$,
since $B_{\tilde{n}}(0)=0$ for odd $\tilde{n}>1$.
\item[(ii)]
Let $f(x)$ be quadratic; then  the above shows  $\frac{vol(I_{N,M}^2\cap\hat{\mathfrak{D}}_2)}{vol(\hat{\mathfrak{D}}_{2})}=0$.
As a matter of fact, we show,  more strongly in Proposition \ref{prop7.3} that 
$I_{N,M}^2\cap\hat{\mathfrak{D}}_2$ is a finite set.
The set $\{p\in Spl_X(f)\mid \{\frac{nr}{p} \}  < \{\frac{mr}{p}\}\text{ for }^\forall r \text{ s.t. }f(r)\equiv0\bmod p \}$
is also a finite one as  in Proposition \ref{prop7.5}.
\end{itemize}
\subsubsection{Bernoulli polynomial}
In this subsection, let us  recall several facts on Bernoulli polynomials.
For non-negative integer $m$, we write
\begin{align*}
\Delta_m(f(x))&:=\frac{d^mf(x)}{dx^m}|_{x=0},
\end{align*}
for example
$$
\Delta_m(e^{rx})=r^m.
$$
We define 
Bernoulli polynomials $B_n(t)$  by
\begin{align*}
\sum_{n=0}^\infty B_n(t)\frac{x^n}{n!}&:=\frac{xe^{tx}}{e^x-1}=\left(\sum_{k=0}^\infty \frac{B_k(0)}{k!}x^k\right)
\left(\sum_{m=0}^\infty\frac{(tx)^m}{m!}\right),
\end{align*}
that is
\begin{align*}
\Delta_n\left(\frac{xe^{tx}}{e^x-1}\right)=
B_n(t)=\sum_{i=0}^n\binom{n}{i}B_{n-i}(0)t^i,\quad B_m(0)=\Delta_m\left(\frac{x}{e^x-1}\right).
\end{align*}
In this paper, the notation  $B_n$ denotes the Bernoulli polynomial, and so we use the notation  $B_n(0)$ 
for the  Bernoulli number.

We see that, for a real number $a$ and a non-negative integer $m$
\begin{align*}
&\frac{1}{m+1}(B_{m+1}(a+1)-B_{m+1}(a))
\\
=&\frac{1}{m+1}
\left(\Delta_{m+1}\left(\frac{xe^{(a+1)x}}{e^x-1}\right)-\Delta_{m+1}\left(\frac{xe^{ax}}{e^x-1}\right)\right)
\\
=&\frac{\Delta_{m+1}(xe^{ax})  }{m+1}
\\
=&\,a^m,
\end{align*}
hence for integers $ n, m\ge0$,
\begin{equation*}
\sum_{l=0}^n (a+l)^m =\frac{1}{m+1}(B_{m+1}(a+n+1) - B_{m+1}(a)).
\end{equation*}
Therefore, for a positive real number $a$ and a positive integer $n>1$, we see that
\begin{align}\nonumber
&\sum_{l\ge0} \max(0, a-l)^{n-1}
\\\nonumber
=&\sum_{0\le l \le \lfloor a\rfloor}(a - l)^{n-1}
\\\nonumber
=&\sum_{0\le l \le \lfloor a\rfloor}(\{a\}+\lfloor a\rfloor  - l)^{n-1}
\\\nonumber
=&\frac{1}{n}(B_n(\{a\}+\lfloor a\rfloor+1)-B_n(\{a\}))
\\ \label{eq86}
=&\frac{1}{n}(B_n(a+1)-B_n(\{a\})),
\end{align}
which we need to prove the theorem. 

\subsubsection{Proof of Theorem \ref{thm7.1}}
It is convenient to introduce the notation $\fallingdotseq$ : 
For sets $S,T$, we write $S\fallingdotseq T$ if and only if
$\#(S\setminus T) + \#(T\setminus S)<\infty$.
\begin{prop}\label{prop7.1}
Let $n,m$ be non-zero distinct integers, and write $n=dN$, $m = dM$ for $d:=(n,m)$.
Then we see that
\begin{equation}\label{eq87}
I_{n,m}= \{x\in(0,1) \mid\{nx\}<\{mx\}\}
\fallingdotseq \cup_{k=0}^{d-1}\left\{\left.\frac{k+\epsilon}{d} \right| \epsilon\in I_{N,M}\right\},
\end{equation}
and under the assumptions $M>N>0$ and $(M,N)=1$
\begin{equation}\label{eq88}
I_{N,M}\fallingdotseq \cup_{K=1}^{M-1}
\left\{
\begin{array}{ll}
(\frac{1}{M}(K-\frac{L-M}{N}),\frac{K}{M}) &\text{ if }L>M,
\\[2mm]
(\frac{1}{M}(K-\frac{M-L}{M-N}),\frac{K}{M}) &\text{ if }L<M,
\end{array}
\right.
\end{equation} 
where the integer $L$ is defined by
$$
L\equiv NK\bmod M,N\le L \le M+N-1.
$$
\end{prop}
\proof{
The condition $x\in I_{n,m}$ means  $\{nx\}<\{mx\}$, i.e. $\{dNx\}<\{dMx\}$, hence
$\{dx\}\in I_{N,M}$. We have only to  write $dx = k+\{dx\}$ for an integer $k$ for the proof of \eqref{eq79 }.
Let us show the equation \eqref{eq88}.
Take a number $x\in I_{N,M}$, which implies $Mx\not\in\mathbb{Z}$ and write $x=\frac{K-\epsilon}{M}$
$(1\le K\le M, K\in\mathbb{Z}, 0<\epsilon<1)$.
We note that the condition $K=M$ is equivalent to $L=M$ by the definition of $L$ and the assumption $M>N,(M,N)=1$.
Then it is easy to see that $\{Mx\}=1-\epsilon$ and $\{Nx\}= \{\frac{NK-N\epsilon}{M}\}=
\{\frac{L-N\epsilon}{M}\}$.
The inequalities $0<L-N\epsilon<2M$ show that $\{Nx\}=\frac{L-N\epsilon}{M}$ or $
\frac{L-N\epsilon}{M}-1$ according to  $ L-N\epsilon<M$ or $ L-N\epsilon\ge M$.
Let us  note that

(i) $\frac{L-M}{N}<\frac{2M-L}{M-N}$ which
follows from the assumption  $L<M+N$,  
 
 (ii) $\frac{L-N\epsilon}{M}-1<1-\epsilon\Leftrightarrow
 \epsilon<\frac{2M-L}{M-N}$,
and

 (iii) $\frac{L-N\epsilon}{M}<1-\epsilon\Leftrightarrow
\epsilon<\frac{M-L}{M-N}$.

\noindent
First, suppose $L<M$ : By $0< L-N\epsilon<M$, we have   $\{Nx\}=\frac{L-N\epsilon}{M}$,
hence $\{Nx\}<\{Mx\}$ $\Leftrightarrow$ $\frac{L-N\epsilon}{M} <1-\epsilon$ $\Leftrightarrow$ 
$0<\epsilon<\frac{M-L}{M-N}$ by (iii). 
Next, suppose $L\ge M$:
Under the supposition $\{Nx\}<\{Mx\}$, the condition $L-N\epsilon<M$ implies 
$\{Mx\}=1-\epsilon>\{Nx\}=\frac{L-N\epsilon}{M}$,
which is equivalent to $(0<)\,\epsilon<\frac{M-L}{M-N}$, which contradicts the assumption $L\ge M$.
Hence the condition $\{Nx\}<\{Mx\}$ implies $L-N\epsilon\ge M$, i.e. $\epsilon\le\frac{L-M}{N}$.
Conversely, suppose $\epsilon\le\frac{L-M}{N}$, which is equivalent to $L-N\epsilon\ge M$
and implies $\epsilon< \frac{ 2M-L}{M-N}$ by (i) , 
hence  $\{Nx\}=\frac{L-N\epsilon}{M}-1$,
and the property (ii) imply $\{Nx\}<\{Mx\}$.   
If $K=M$ happens, then we have $L=M$ and the contradiction $\{Nx\}=\{\frac{M-N\epsilon}{M}\}=1-\frac{N}{M}\epsilon
>1-\epsilon=\{Mx\}$.

\qed
}

Hereafter, we assume that integers $M,N$ satisfy $M>N>0$ and $(M,N)=1$ as in Theorem \ref{thm7.1}.

We write, for an integer $K$
\begin{equation*}
\tau(K):=
\left\{
\begin{array}{ll}
\frac{L-M}{N}&\text{ if }L\ge M,
\\[2mm]
\frac{M-L}{M-N}&\text{ if }L < M,
\end{array}
\right.
\end{equation*}
where $L$ is the integer defined by $L\equiv NK\bmod M$, and $ N\le L \le M+N-1$.
It is easy to see $\tau(0)=0,\tau(1)=1$
By Proposition \ref{prop7.1},
we have 
\begin{equation}\label{eq89}
I_{N,M}\fallingdotseq \cup_{K=1}^{M-1}S_{\tau(K)}
\quad\text{ with }S_{\tau(K)}:=\left(\frac{K-\tau(K)}{M},\frac{K}{M}\right).
\end{equation}
\begin{prop}\label{prop7.2}
The mapping $\tau$ is the bijection from $\{1,2,\dots,M-1\}$ to $\Sigma:=\Sigma_1\cup\Sigma_2$,
where $\Sigma_1:=\{\frac{1}{N},\dots,\frac{N-1}{N}\}$,   
$\Sigma_2:=\{\frac{1}{M-N},\dots,\frac{M-N}{M-N}\}$
and $\tau(K)=K$ holds in the ring $\mathbb{Z}/M\mathbb{Z}$. 
The set $\Sigma_1$ is empty if $N=1$.
Moreover, we have
\begin{align*}
K-\tau(K)=\left\{
\begin{array}{lll}
\frac{MA_1}{N} & (1\le {}^\exists A_1\le N-1)& \text{ if } \tau(K)\in\Sigma_1,
\\[2mm]
\frac{MA_2}{M-N}& (0\le {}^\exists A_2 \le M-N-1)& \text{ if } \tau(K)\in\Sigma_2,
\end{array}
\right.
\end{align*}
and $K-\tau(K)$ $(K=1,\dots,M-1)$ are distinct.
\end{prop}
\proof{
It is clear that the value $\tau(K)$ is  in the set  $\Sigma$ and $\#\Sigma=M-1$.
Since $(M,M-N)=(M,N) =  1$, we see that $\tau(K)=\frac{L}{N}=K$ in $\mathbb{Z}/M\mathbb{Z}$,
which implies the injectivity of $\tau$.
Write $NK=L+aM$ $(a\in \mathbb{Z})$.
Suppose that $L>M$; it implies $K-\tau(K)=\frac{M}{N}(a+1)$, and from inequalities $ N\le L \le M+N-1$
follows $\frac{N(K-1)}{M}-1+\frac{1}{M}\le a\le \frac{K-1}{M}N$,  hence $-1< a<N$.
If $a=N-1$  occurs, then $L=NK-(N-1)M$ holds and then the assumption $L>M$ implies $NK>NM$,
which contradicts $K\le M-1$.  
Next, suppose that $L<M$; then we see $K-\tau(K)=\frac{M}{M-N}(K-1-a)$ and inequalities $ N\le L \le M+N-1$
imply $K -\frac{N}{M}(K-1)-1\le K-1-a\le(1-\frac{N}{M})K+\frac{N-1}{M}$,
hence $0\le K-1-a \le(1-\frac{N}{M})(M-1)+\frac{N-1}{M}<M-N-1+2\frac{N}{M}-\frac{1}{M}< M-N+1$.
If $K-1-a=M-N$  occurs, then the assumption $L<M$ implies $NK-aM<M$, hence $K-M+N=a+1>\frac{NK}{M}$,
i.e. $(M-N)K>(M-N)M$, which is the contradiction.
Next, assume that $K_1-\tau(K_1)=K_2-\tau(K_2)$ for $1\le K_1<K_2\le M-1$; it implies the contradiction 
$1\le K_2-K_1=\tau(K_2)-\tau(K_1)<1$.

\qed 
}

When we change  the domain of the mapping $\tau$ from $\{1,\dots,M-1\}$ to  $\{0,1,\dots,M-1\}$,
the set $\{0,1,\dots,M-1\}$ corresponds , through $K\to K-\tau(K)$ to
$\Sigma':=\Sigma'_1\sqcup\Sigma'_2$ for
$$
\Sigma'_1:=\left\{\frac{MA_1}{N}\mid0\le A_1\le N-1\right\},
\Sigma'_2:=\left\{\frac{MA_2}{M-N}\mid0\le A_2\le M-N-1\right\},
$$
where $K=0,1$ are supposed to correspond to $0\in\Sigma'_1,0\in\Sigma'_2$, respectively.
Now we see that
\begin{align}\nonumber
&vol(
\{(x_1,\dots,x_{\tilde{n}})\in(0,1)^{\tilde{n}}\mid x_i\in I_{N,M},\sum x_i\in\mathbb{Z}\}
)
\\\nonumber
=\,&\sum_{K_1,\dots,K_{\tilde{n}}=1}^{M-1}vol(
\{(x_1,\dots,x_{\tilde{n}})\in(0,1)^{\tilde{n}}\mid x_i\in S_{\tau(K_i)},\sum x_i\in\mathbb{Z}\}
\\\nonumber
\intertext{writing $x_i=\frac{K_i-\epsilon_i}{M}$ $(0<\epsilon_i<\tau(K_i))$}
=\,&
\frac{1}{M^{{\tilde{n}}-1}}
\sum_{K_i=1\atop(1\le i \le \tilde{n})}^{M-1}vol
\left(
\left\{
(\epsilon_1,\dots,\epsilon_{\tilde{n}})\in(0,1)^{\tilde{n}}\left|
\begin{array}{l}
0<\epsilon_i<\tau(K_i)  ,\sum \epsilon_i\in\mathbb{Z},
\\\label{eq90}
\sum \epsilon_i\equiv\sum K_i \bmod M
\end{array}
\right.
\right\} 
\right).
\end{align}
Before the  evaluation of the volume of the above,
let us show the special case of $\tilde{n}=2$.
\begin{prop}\label{prop7.3}
$I_{N,M}^{2}\cap\hat{\mathfrak{D}}_2$ is a finite set.
\end{prop}
\proof{
By \eqref{eq89}, we have only to sow that $(\cup_{K=1}^{M-1}S_{\tau(K)})\cap\hat{\mathfrak{D}}_2$ is empty.
Suppose that it is not empty and for $x_i=\frac{K_i-\tau(K_i)+\epsilon_i}{M}$ with $0<\epsilon_i<\tau(K_i),1\le K_i
\le M -1$ $(i=1,2)$, assume that $x_1+x_2\in\mathbb{Z}$. 
The equalities $0\le\sum_{i=1}^2 \frac{K_i-\tau(K_i)}{M}<x_1+x_2<\sum_{i=1}^2\frac{K_i}{M}<2$ imply $x_1+x_2=1$,
hence $K_1+K_2=M+\tau(K_1)-\epsilon_1+\tau(K_2)-\epsilon_2$.
It implies $M<K_1+K_2<M+\tau(K_1)+\tau(K_2)\le M+2$, that is $K_1+K_2=M+1$,
hence 
\begin{equation}\label{eq91}
\tau(K_1)+\tau(K_2)-\epsilon_1-\epsilon_2=1.
\end{equation}
Let $L_i$ $(i=1,2)$ be integers defining $\tau(K_i)$, that is $L_i\equiv NK_i\bmod M, N\le L_i\le M+N-1$;
then we have $L_1+L_2\equiv N(K_1+K_2)\equiv N(M+1) \equiv N\bmod M$.
Writing $L_1+L_2=N+kM$ $(k\in\mathbb{Z})$, we see $N \le L_2=N+kM-L_1\le M+N-1$ and $L_1\le kM\le L_1+M-1
\le 2M+N-2$, which implies $k =1,2$.
First, suppose $k=1$; then $L_1+L_2=N+M$ and $0\le L_1-N=M-L_2$, hence $\tau(K_1)=\frac{M-L_1}{M-N}$ by definition of $\tau$, similarly  $\tau(K_2)=\frac{M-L_2}{M-N}$, which imply $\tau(K_1)+\tau(K_2)=1$ and so the contradiction $\epsilon_1+\epsilon_2=0$ by \eqref{eq91}.
Next, suppose $k=2$; we see that $L_1=(N+M-L_2) +M>M$ and $\tau(K_1)=\frac{L_1-M}{N}$, and similarly
$\tau(K_2)=\frac{L_2-M}{N}$.
They also imply the contradiction $\tau(K_1)+\tau(K_2)=1$.

\qed
}
\begin{prop}\label{prop7.4}
Let $f(x)$ be a monic integral polynomial without rational roots, and let $n$ be a non-zero integer, and $A$ a positive number.
If a prime $p$ is sufficiently large, then for any integer $r$ satisfying $f(r)\equiv0\bmod p$, $\{\frac{nr}{p}\}\ge\frac{A}{p}$ holds.
\end{prop}
\proof{
Suppose that for a  prime $p$ and an integer $r$,  $f(r)\equiv0\bmod p$ and $\{\frac{nr}{p}\}<\frac{A}{p}$ hold.
We may assume that $0\le r<p$, and
writing
$$
nr=kp+R \quad(k,R\in\mathbb{Z},0\le R<p),
$$
we see that $|k|p<(|n|+1)p$, and $0\le R<A$ by the assumption $\left\{\frac{nr}{p}\right\}<\frac{A}{p}$, hence
 the possibility of $k,R$ is finite.
Assume that there are infinitely many such primes $p$; then for some $k_0,R_0$ there are infinitely many primes $p$ such that $nr=k_0p+R_0$ with $f(r)\equiv0\bmod p$.
Then, for the integral polynomial $g(x):=n^{\deg f}f(\frac{x}{n})$, we see that $g(R_0)\equiv g(nr)\equiv n^{\deg f}f(r)
\equiv0\bmod p$, that is $g(R_0)=0$, which is the contradiction.
\qed
}
\begin{prop}\label{prop7.5}
Let $f(x)=x^2-ax+b$ be an integral irreducible polynomial, and let $m,n$ be non-zero distinct integers.
Then there are only finitely many  primes $p$ such that $f(x)\equiv0\bmod p$ has two roots $r_1,r_2\in\mathbb{Z}$ such that $r_1\not\equiv r_2\bmod p$
\begin{equation}\label{eq92}
\left\{\frac{nr_i}{p}\right\}<\left\{\frac{mr_i}{p}\right\}\quad (i=1,2).
\end{equation}
\end{prop}
\proof{
Suppose that  the inequalities \eqref{eq2} for a sufficiently large prime $p$; then by $r_1+r_2\equiv a\bmod p$ we have 
$\left\{\frac{nr_2}{p}\right\}<\left\{\frac{mr_2}{p}\right\}\Leftrightarrow$
$\left\{\frac{n(a-r_1)}{p}\right\}<\left\{\frac{m(a-r_1)}{p}\right\}$$\Leftrightarrow\left\{\frac{na}{p}-\left\{\frac{nr_1}{p}\right\}\right\}<\left\{\frac{ma}{p}-\left\{\frac{mr_1}{p}\right\}\right\}\Leftrightarrow$
$ \left\{\frac{nr_1}{p}\right\}-\frac{na}{p}>\left\{\frac{mr_1}{p}\right\}\linebreak[3]-\frac{ma}{p}$
by Proposition \ref{prop7.4}.
Suppose that there are infinitely many primes satisfying \eqref{eq92}.
Then, writing $nr_1=kp+R,\,mr_1=k'p+R'$ $(k,k'\in\mathbb{Z},0\le R,R'<p)$, the above inequality and the assumption \eqref{eq2} imply $R-na> R'-ma$,
 i.e. $(n-m)a<R-R'=p(\{\frac{nr_1}{p}\}-\{\frac{mr_1}{p}\})<0$.
 Therefore there is  an integer $r_0$ with $(n-m)a<r_0<0$ such
 that $(n-m)r_1\equiv R-R' \equiv r_0\bmod p$ for infinitely many primes $p$.
Writing $g(x)=(n-m)^{\deg f}f(\frac{x}{n-m})\in\mathbb{Z}$, we have $g(r_0)\equiv (n-m)^{\deg f}f(r_1)\equiv0\bmod p$
 for infinitely many primes $p$, which contradicts the irreducibility of the polynomial $f(x)$.

\qed
}


Let us begin the evaluation of the volume.
Write $\tilde{M}(x):=\max(0,x)$ for simplicity.
\begin{lem}
Let $a,b,m$ be real numbers and suppose  $a\le b$ and $m\ne -1$.
Then we have
\begin{equation}
\int_a^b \tilde{M}(t - w)^mdw=\frac{1}{m+1}\left(\tilde{M}(t-a)^{m+1}-\tilde{M}(t-b))^{m+1}\right).
\end{equation}
\end{lem}
\proof{
By writing $t-w = W$, the left hand is equal to
\begin{align*}
\int_{t-a}^{t-b}\tilde{M}(W)^m(-dW)
&=-\left(
\int_{-\infty}^{t-b}\tilde{M}(W)^mdW - \int_{-\infty}^{t-a}\tilde{M}(W)^mdW
\right)
\\
&=\int_{0}^{t-a}\tilde{M}(W)^mdW - \int_{0}^{t-b}\tilde{M}(W)^mdW
\\
&=
\frac{1}{m+1}\left( \tilde{M}(t-a)^{m+1}-\tilde{M}(t-b)^{m+1}\right).
\end{align*}
\qed
}

Here we introduce the notation, for $x,\tau_1,\dots,\tau_n\in\mathbb{R}$
\begin{align*}
W_l(x;\tau_1,\dots,\tau_n)&:=
\sum_{0\le k \le n}\sum_{1\le i_1<\dots<i_k\le n}(-1)^k\tilde{M}(x-\tau_{i_1}-\dots-\tau_{i_k})^l
\\
&=\sum_{S}(-1)^{|S|}\tilde{M}(x-\sum_{i\in S}\tau_{i})^l,
\end{align*}
where $S$ runs over all $2^n$ subsets of $\{1, \dots,n\}$ and $|S|=\#S$.
It is obvious that
\begin{equation*}
W_l(x;\tau_1,\dots,\tau_{n})=W_l(x;\tau_{\sigma(1)},\dots,\tau_{\sigma(n)})
\end{equation*}
for any permutation $\sigma$. 
 Moreover, we see
\begin{align}\nonumber
&W_l(x;\tau_1,\dots,\tau_{n})
\\ \label{eq94}
=\,&W_l(x;\tau_1,\dots,\tau_{n-1})-W_l(x-\tau_n;\tau_1,\dots,\tau_{n-1}).
\end{align}
In particular, the equation  \eqref{eq94} says that
\begin{equation}\label{eq95}
W_l(x;\tau_1,\dots,\tau_n)=0 \text{ if }{}^\exists\tau_i=0.
\end{equation}
\begin{prop}
For positive numbers $\tau_1,\dots,\tau_n$, the volume  of  the set
\begin{equation}\nonumber
V_n(x;\tau_1,\dots,\tau_n):=\{(x_1,\dots,x_n)\mid 0\le x_i \le \tau_i,\sum_{i=1}^n x_i\le x\}
\end{equation}
is
\begin{align}\label{eq96}
&\frac{1}{n!}W_n(x;\tau_1,\dots,\tau_n).
\end{align}
\end{prop}
\proof{
We use the induction on $n$.
Write $U_n(x)=vol(V_n(x;\tau_1,\dots,\tau_n))$ simply.
For $n = 1$, it is easy to see that
$$
U_1(x)=\tilde{M}(x)-\tilde{M}(x-\tau_1)=\frac{1}{1!}W_1(x;\tau_1).
$$
Supposing \eqref{eq96}, we see that $U_{n+1}(x)$ is 
\begin{align*}
&\int_{x_{n+1}=0}^{\tau_{n+1}} \dots\int_{x_{1}=0,\atop\sum x_i\le x}^{\tau_{1}} 1\,dx_1\dots dx_{n+1}
\\
=&\int_{x_{n+1}=0}^{\tau_{n+1}} U_n(x-x_{n+1})dx_{n+1}
\\
=&\int_{x_{n+1}=0}^{\tau_{n+1}}
\frac{1}{n!}\sum_{k=0}^n(-1)^k
\left\{ 
\sum_{1\le i_1<\dots<i_k\le n}\tilde{M}(x-x_{n+1}-\tau_{i_1}-\dots-\tau_{i_k})^n
\right\}dx_{n+1}
\\
=&\frac{1}{(n+1)!}\sum_{k=0}^n(-1)^k\sum_{1\le i_1<\dots<i_k\le n}
\left\{ 
\tilde{M}(x-\tau_{i_1}-\dots-\tau_{i_k})^{n+1}\right.
\\
&\hspace{50mm}\left.
-\tilde{M}(x-\tau_{i_1}-\dots-\tau_{i_k} -\tau_{n+1})^{n+1}
\right\}
\\
=&
\frac{1}{(n+1)!}\sum_{k=0}^n (-1)^k   \{ \sum_{1\le i_1<\dots<i_k\le n}
\tilde{M}(x-\tau_{i_1}-\dots-\tau_{i_k})^{n+1}\}
\\
&+
\frac{1}{(n+1)!}\sum_{j=1}^{n+1} (-1)^j   \{ \sum_{1\le i_1<\dots<i_j=n+1}
\tilde{M}(x-\tau_{i_1}-\dots-\tau_{i_j})^{n+1}\}
\\
=& \,U_{n+1}(x).
\end{align*}
\qed}

\noindent
{\bf{Example}}
It is easy to check, by drawing the figure
$$
vol(V_2(x;\tau_1,\tau_2))=\frac{1}{2}\{\tilde{M}(x)^2-\tilde{M}(x-\tau_1)^2-\tilde{M}(x-\tau_2)^2+\tilde{M}(x-\tau_1-\tau_2)^2\}.
$$
We note that the proposition implies the following equation :
 \begin{align*}
\sum_{S} (-1)^{|S|}    
(c-\sum_{i\in S}\tau_i)^{l}
=\,
\left\{
\begin{array}{cll}
0&\text{ if }&0\le l\le n-1,
\\
n!\prod_i\tau_i&\text{ if }&l=n,
\end{array}
\right.
\end{align*}
where $\tau_1,\dots,\tau_n$ and  $c$ are variables and $S$ runs over all $2^n$ subsets of $\{1,\dots,n\}$.

\begin{prop}\label{prop7.7}
For positive numbers $\tau_1,\dots,\tau_n$, the $n-1$-dimensional volume  of 
the set
\begin{equation}
S_n(x;\tau_1,\dots,\tau_n):=\{(x_1,\dots,x_n)\mid 0\le x_i \le \tau_i,\sum_{i=1}^n x_i= x\}
\end{equation}
is
\begin{equation}\label{eq98}
\frac{\sqrt{n}}{(n-1)!}W_{n-1}(x;\tau_1,\dots,\tau_n).
\end{equation}
\end{prop}
\proof{
Using  orthnormal basis 
\begin{align*}
{\bm{f}}_1&:=\frac{1}{\sqrt{1+1^2}}(1,-1,0,\dots,0),
\\
\vdots
\\
{\bm{f}}_k&:=\frac{1}{\sqrt{k+k^2}}(1,\dots,1,-k,0,\dots,0),
\\
\vdots
\\
{\bm{f}}_{n-1}&:=\frac{1}{\sqrt{(n-1)+(n-1)^2}}(1,\dots,1,-(n-1)),
\\
{\bm{f}}_n&:=\frac{1}{\sqrt{n}}(1,\dots,1),
\end{align*}
and the transformation $(y_1,\dots,y_n)=(x_1,\dots,x_n)({}^t\bm{f}_1,\dots,{}^t\bm{f}_n)$,
that is $\sqrt{n}y_n=x_1+\dots+x_n$,
we see, denoting $S_n(x;\tau_1,\dots,\tau_n)$ by $S_n(x)$  simply
$$
vol(V_n(x;\tau_1,\dots,\tau_n)=\int_{\bm{x}\in V_n(x)}1\,d\bm{x}=
\int_{\sqrt{n}y_n\le x}vol(S_n(\sqrt{n}y_n))dy_n,
$$
hence $U_n'(x)=\sqrt{n}^{-1}vol(S_n(x))$, i.e. $vol(S_n(x))=\sqrt{n}\,U_n'(x)$.

\qed
}

\noindent
{\bf{Example}}
The length of $S_2(x;\tau_1,\tau_2)$
is easily checked to be
\begin{align*}
\sqrt{2}\{\tilde{M}(x)-\tilde{M}(x-\tau_1)-\tilde{M}(x-\tau_2)+\tilde{M}(x-\tau_1-\tau_2)\}=\sqrt{2}U_2'(x).
\end{align*}

We note that the volume of $S_n(x;\tau_1,\dots,\tau_n)$ is positive if and only if $0<x<\sum \tau_i$ holds,
and that $vol(S_n(x;\tau_1,\dots,\tau_n))=vol(S_n(x;\tau_{\sigma(1)},\dots,\tau_{\sigma(n)}))$ for every permutation
$\sigma\in S_n$, since a permutation induces the orthogonal transformation.
In particular, for positive numbers $\tau_1,\dots,\tau_n$ and a real number $x$,
$W_{n-1}(x;\tau_1,\dots,\tau_n)>0$ if and only if $0<x<\sum \tau_i$ holds.
\vspace{5mm}

Now we see that, using \eqref{eq85}, \eqref{eq82 }
\begin{align*}
&\frac{vol(I_{N,M}^{\tilde{n}}\cap\hat{\mathfrak{D}}_{\tilde{n}})}{vol(\hat{\mathfrak{D}}_{\tilde{n}})}
\\
=\,&\frac{1}{\sqrt{{\tilde{n}}}    }vol(
\{(x_1,\dots,x_{\tilde{n}})\in(0,1)^{\tilde{n}}\mid x_i\in I_{N,M},\sum x_i\in\mathbb{Z}\}
)
\\
=\,&
\frac{1}{\sqrt{{\tilde{n}}} M^{{\tilde{n}}-1} }\sum_{1\le K_1,\dots,K_{\tilde{n}}<M}
\sum_{l\in\mathbb{Z}}vol(S_n(\sum K_i-Ml;\tau_1,\dots,\tau_n))
\\
=\,&
\frac{1}{(\tilde{n}-1)! M^{{\tilde{n}}-1} }    \sum_{1\le K_1,\dots,K_{\tilde{n}}<M}\sum_{l\in\mathbb{Z}}
W_{\tilde{n}-1}\left(\sum K_i-Ml;\tau_1,\dots,\tau_{\tilde{n}}\right)
\end{align*}
where $\tau_j=\tau(K_j)$, and note that
$$
W_{\tilde{n}-1}\left(\sum K_i-Ml;\tau_1,\dots,\tau_{\tilde{n}}\right)>0
\Leftrightarrow 0<\sum K_i-Ml<\sum \tau_i,
$$
which implies  $0\le l \le \tilde{n}$,
and continuing the above
\begin{align*}
=\,&
\frac{1}{(\tilde{n}-1)! M^{{\tilde{n}}-1} }    \sum_{0\le K_1,\dots,K_{\tilde{n}}<M}\sum_{l=0}^{\tilde{n}}
W_{\tilde{n}-1}\left(\sum K_i-Ml;\tau_1,\dots,\tau_{\tilde{n}}\right)
\\
\intertext{by $\tau(0)=0$ and \eqref{eq95} }
=\,&
\frac{1}{(\tilde{n}-1)! M^{{\tilde{n}}-1} }    \sum_{0\le K_1,\dots,K_{\tilde{n}}<M}\sum_{l=0}^{\tilde{n}}
\sum_{S}(-1)^{|S|}\tilde{M}\left(\sum K_i-Ml -\sum_{i\in S}\tau_i\right)^{\tilde{n}-1}
\end{align*}
where $S$ runs over all subsets of $\{1,\dots,\tilde{n}\}$.
Thus we see that 
\begin{align*}
&(\tilde{n}-1)! M^{{\tilde{n}}-1}\frac{vol(I_{N,M}^{\tilde{n}}\cap\hat{\mathfrak{D}}_{\tilde{n}})}{vol(\hat{\mathfrak{D}}_{\tilde{n}})}
\\
=\,&\sum_{S}(-1)^{|S|} \sum_{0\le K_i\le M-1\atop(1\le i \le\tilde{n})}\sum_{l=0}^{\tilde{n}}
\tilde{M}\left(\sum_{i\not\in S} K_i +\sum_{j\in S}(K_j-\tau_j)-Ml\right)^{\tilde{n}-1},
\end{align*}
where $K_j-\tau_j$ runs over $\Sigma':=\Sigma'_1\sqcup\Sigma'_2$ as stated after Proposition \ref{prop7.2} and continuing
\begin{align*}
=\,&\sum_{k=0}^{\tilde{n}}(-1)^{k}\binom{\tilde{n}}{k}
 \sum_{0\le K_i\le M-1(i> k), \atop\tau'_j \in \Sigma'(j\le k)}\sum_{l=0}^{\tilde{n}}
\tilde{M}\left(\sum_{i>k} K_i +\sum_{j\le k}\tau'_j-Ml\right)^{\tilde{n}-1}
\\
=\,&{M^{\tilde{n}-1}}\sum_{k=0}^{\tilde{n}}(-1)^{k}\binom{\tilde{n}}{k}
 \sum_{0\le K_i\le M-1(i> k), \atop\tau'_j \in \Sigma'(j\le k)}\sum_{l=0}^{\tilde{n}}
\tilde{M}\left(\frac{\sum_{i>k} K_i +\sum_{j\le k}\tau'_j}{M}-l\right)^{\tilde{n}-1}
\\
=\,&
\frac{M^{\tilde{n}-1}}{\tilde{n}}\sum_{k=0}^{\tilde{n}}(-1)^{k}\binom{\tilde{n}}{k}
 \sum_{0\le K_i\le M-1(i> k), \atop\tau'_j \in \Sigma'(j\le k)} 
\\
 &\hspace{10mm}\left(
 B_{\tilde{n}}\left(\frac{\sum_{i>k} K_i +\sum_{j\le k}\tau'_j}{M}+1\right)
-B_{\tilde{n}}\left(\left\{\frac{\sum_{i>k} K_i +\sum_{j\le k}\tau'_j}{M}\right\}\right) 
  \right),
\end{align*}
by \eqref{eq86}.
For an integer $k$ $(0\le k \le \tilde{n})$, we see that
\begin{align*}
 T_{1,k}&:=\sum_{0\le K_i\le M-1(i> k), \atop\tau'_j \in \Sigma'(j\le k)} 
B_{\tilde{n}}\left(\frac{\sum_{i>k} K_i +\sum_{j\le k}\tau'_j}{M}+1\right)
\\
&=
\sum_{0\le K_i\le M-1(i> k), \atop\tau'_j \in \Sigma'(j\le k)} 
\Delta_{\tilde{n}}\left(\frac{x}{e^x-1}e^{\left(\frac{\sum_{i>k} K_i +\sum_{j\le k}\tau'_j}{M}+1\right)x}  \right)
\\
&=
\Delta_{\tilde{n}}\left(\frac{xe^x}{e^x-1}\cdot
\left(\sum_{0\le K\le M-1}e^{\frac{K }{M}x}\right)^{\tilde{n}-k}
\cdot
\left(\sum_{\tau' \in \Sigma'} e^{\frac{\tau'}{M}x}\right)^k\right),
\end{align*}
where 
\begin{align*}
\sum_{0\le K\le M-1}e^{\frac{K }{M}x}=\frac{e^x  -   1}{e^{\frac{x}{M}}-  1  }
\end{align*}
and
\begin{align*}
\sum_{\tau' \in \Sigma'} e^{\frac{\tau'}{M}x}&=\sum_{A_1=0}^{N-1}e^{\frac{A_1}{N}x}
+\sum_{A_2=0}^{M-N-1}e^{\frac{A_2}{M-N}x}
\\
&=\frac{e^x-1}{e^{\frac{x}{N}}-1}
+\frac{e^x-1}{e^{\frac{x}{M-N}}-1}.
\end{align*}
Therefore we see that 
\begin{align*}
T_{1,k}&=\Delta_{\tilde{n}}\left( 
\frac{xe^x}{e^x-1}\cdot\left(  \frac{e^x-1}{e^{\frac{x}{M}}-1}     \right)^{\tilde{n}-k}    
\cdot\left( \frac{e^x-1}{e^{\frac{x}{N}}-1}
+\frac{e^x-1}{e^{\frac{x}{M-N}}-1}  \right)^k
\right),
\end{align*}
and so
\begin{align*}
&\sum_{k=0}^{\tilde{n}}(-1)^k\binom{\tilde{n}}{k}T_{1,k}
\\
=&
\Delta_{\tilde{n}}\left( \frac{xe^x}{e^x-1}
\sum_{k=0}^{\tilde{n}}(-1)^k\binom{\tilde{n}}{k}
\cdot\left(  \frac{e^x- 1}{e^{\frac{x}{M}}-1}     \right)^{\tilde{n}-k}    
\cdot\left( \frac{e^x-1 }{e^{\frac{x}{N}}-1}
+\frac{e^x-1}{e^{\frac{x}{M-N}}-1}  \right)^k
\right)
\\
=&
\Delta_{\tilde{n}}\left( 
\frac{xe^x}{e^x-1}
\left(
  \frac{e^x-1}{e^{\frac{x}{M}}-1}
  -
  \frac{e^x-1}{e^{\frac{x}{N}}-1}
-\frac{e^x-1}{e^{\frac{x}{M-N}}-1}
\right)^{\tilde{n}}
\right)
\\
=&
\Delta_{\tilde{n}}
\left(
f(x)g(x)^{\tilde{n}}
\right),
\end{align*}
where
$$
f(x):=\frac{xe^x}{e^x-1},\quad g(x):= \frac{e^x-1}{e^{\frac{x}{M}}-1}
  -
  \frac{e^x-1}{e^{\frac{x}{N}}-1}
-\frac{e^x-1}{e^{\frac{x}{M-N}}-1}.
$$
We see that $\Delta_0(f)=f(0)=1,\Delta_0(g)=g(0)=0,\Delta_1(g)=\frac{1}{2}$, and
\begin{align*}
\Delta_{\tilde{n}}(fg^{\tilde{n}})&=\sum_{k_1+\dots+k_{\tilde{n}+1}=\tilde{n},\atop{}^\forall k_i\ge0}
\frac{\tilde{n}!}{k_1!\dots k_{\tilde{n}+1}!}\Delta_{k_1}(f)\Delta_{k_2}(g)\dots\Delta_{k_{\tilde{n}+1}}(g)
\\
&=\tilde{n}!\,\Delta_1(g)^{\tilde{n}}=\frac{\tilde{n}!}{2^{\tilde{n}}},
\end{align*}
thus
\begin{equation*}
\sum_{k=0}^{\tilde{n}}(-1)^k\binom{\tilde{n}}{k}T_{1,k}=\frac{\tilde{n}!}{2^{\tilde{n}}}.
\end{equation*}
Next, for an integer $k$ $(0\le k \le \tilde{n})$, let us evaluate
\begin{align*}
 T_{2,k}&:=\sum_{0\le K_i\le M-1(i> k), \atop\tau'_j \in \Sigma'(j\le k)} 
B_{\tilde{n}}\left(\left\{
\frac{\sum_{i>k} K_i +\sum_{j\le k}\tau'_j}{M}
\right\}\right),
\end{align*}
where $\frac{\tau'_j}{M}$ runs over the set $\Sigma'':=\Sigma''_1 \sqcup \Sigma''_2 $ for
$$
\Sigma''_1:=\left\{\frac{A_1}{N}\mid0\le A_1\le N-1\right\},
\Sigma''_2:=\left\{\frac{A_2}{M-N}\mid0\le A_2\le M-N-1\right\}.
$$
First of all, we note that, for  positive integers $d,d_1,d_2$ with $(d_1,d_2)=1$
\begin{equation}\label{eq99}
\sum_{j=0}^{d-1}B_{\tilde{n}}\left(
\frac{j}{d}
\right)
=\frac{B_{\tilde{n}} (0)  }{d^{\tilde{n}-1}},   
\sum_{0\le j_i\le d_i-1,\atop(i=1,2)}B_{\tilde{n}}\left(
\left\{\frac{j_1}{d_1}+\frac{j_2}{d_2}\right\}
\right)
=\frac{B_{\tilde{n}} (0)  }{(d_1d_2)^{\tilde{n}-1}}.
\end{equation}
Because,
\begin{align*}
&\sum_{j=0}^{d-1}B_{\tilde{n}}\left(
\frac{j}{d}
\right)
=\sum_{j=0}^{d-1}\Delta_{\tilde{n}}\left(
\frac{xe^{\frac{jx}{d}}}{e^x-1}
\right)
=\Delta_{\tilde{n}}\left(
\frac{x}{e^x-1}
\sum_{j=0}^{d-1}
e^{\frac{jx}{d}}
\right)
\\
=&
\Delta_{\tilde{n}}\left(
\frac{x}{e^{\frac{x}{d}}-1}
\right)=
\Delta_{\tilde{n}}\left(
d\sum_{k=0}^\infty\frac{B_k(0)}{k!}\left(\frac{x}{d}\right)^k
\right)
=
\frac{B_{\tilde{n}} (0)  }{d^{\tilde{n}-1}},
\end{align*}
and
$$
\sum_{0\le j_i\le d_i-1,\atop(i=1,2)}B_{\tilde{n}}\left(
\left\{\frac{j_1}{d_1}+\frac{j_2}{d_2}\right\}
\right)
=
\sum_{0\le j \le d_1d_2-1}B_{\tilde{n}}\left(
\frac{j}{d_1d_2}
\right).
$$
We see that
\begin{align*}
T_{2,k}&=
\sum_{0\le K_i\le M-1\atop(i> k)}
\sum_{l=0}^k\binom{k}{l} 
\sum_{\tau''_i\in\Sigma''_1(i\le l),\atop\tau''_j\in\Sigma''_2(l+1\le j\le k)}
\\
&\hspace{10mm}
B_{\tilde{n}}\left(\left\{
\frac{\sum_{i>k} K_i}{M} +\sum_{1\le i\le l}\tau''_i+\sum_{l+1\le j\le k}\tau''_j
\right\}\right)
\\
&=
\sum_{0\le K_i\le M-1\atop(i> k)}
\biggl\{
\sum_{1\le a_2\le M-N}
B_{\tilde{n}}\left(\left\{
\frac{\sum_{i>k}K_i}{M} +\frac{a_2}{M-N}
\right\}\right)
\\
&\hspace{12mm}
\times\#\{(A_{1},\dots,A_{k})\mid1\le A_i\le M-N ,\sum _i A_i\equiv a_2\bmod M-N\}
\\
&\hspace{8mm}+\sum_{l=1}^{k-1}\binom{k}{l} 
\sum_{1\le a_1\le N,\atop1\le a_2\le M-N}
\Bigl\{B_{\tilde{n}}\left(\left\{
\frac{\sum_{i>k}K_i}{M} +\frac{a_1}{N} +\frac{a_2}{M-N}
\right\}\right)
\\
&\hspace{12mm}
\times\#\{(A_{1},\dots,A_{l})\mid1\le A_i\le N ,\sum _i A_i\equiv a_1\bmod N\}
\\
&\hspace{12mm}
\times\#\{(A_{l+1},\dots,A_{k})\mid 1\le A_i\le M-N,\sum A_i\equiv a_2\bmod M-N\}\hspace{-1mm}
\Bigr\}
\\
&\hspace{8mm}+\sum_{1\le a_1\le N}
B_{\tilde{n}}\left(\left\{
\frac{\sum_{i>k}K_i}{M} +\frac{a_1}{N}
\right\}\right)
\\
&\hspace{12mm}
\times
\#\{(A_{1},\dots,A_{k})\mid 1\le A_i\le N,\sum A_i\equiv a_1\bmod N\}
\biggr\}
\\&=
\sum_{0\le K_i\le M-1\atop(i> k)}
\biggl\{
\sum_{1\le a_2\le M-N}
B_{\tilde{n}}\left(\left\{
\frac{\sum_{i>k}K_i}{M} +\frac{a_2}{M-N}
\right\}\right)
(M-N)^{k-1}
\\
&\hspace{8mm}+\sum_{l=1}^{k-1}\binom{k}{l} 
\sum_{1\le a_1\le N,\atop1\le a_2\le M-N}
B_{\tilde{n}}\left(\left\{
\frac{\sum_{i>k}K_i}{M} +\frac{a_1}{N} +\frac{a_2}{M-N}
\right\}\right)
\\
&\hspace{40mm}
\times N^{l-1}(M-N)^{k-l-1}
\\
&\hspace{8mm}+\sum_{1\le a_1\le N}
B_{\tilde{n}}\left(\left\{
\frac{\sum_{i>k}K_i}{M} +\frac{a_1}{N}
\right\}\right)
N^{k-1}
\biggr\}.
\end{align*}
Hence, we see that
\begin{align*}
T_{2,\tilde{n}}
&=
\sum_{1\le a_2\le M-N}
B_{\tilde{n}}\left(\left\{
\frac{a_2}{M-N}
\right\}\right)
(M-N)^{\tilde{n}-1}
\\
&\hspace{8mm}+\sum_{l=1}^{\tilde{n}-1}\binom{\tilde{n}}{l} 
\sum_{1\le a_1\le N,\atop1\le a_2\le M-N}
B_{\tilde{n}}\left(\left\{
\frac{a_1}{N} +\frac{a_2}{M-N}
\right\}\right)
\\
&\hspace{40mm}
\times N^{l-1}(M-N)^{\tilde{n}-l-1}
\\
&\hspace{8mm}+\sum_{1\le a_1\le N}
B_{\tilde{n}}\left(\left\{
\frac{a_1}{N}
\right\}\right)
N^{\tilde{n}-1}
\biggr\}
\\
&=
B_{\tilde{n}}(0)+\sum_{l=1}^{\tilde{n}-1}\binom{\tilde{n}}{l} 
\frac{B_{\tilde{n}}(0)}{N^{\tilde{n}-l}(M-N)^{l}}
+B_{\tilde{n}}(0)
\\
&=
\left(
2+
\left(
\frac{1}{N}+\frac{1}{M-N}
\right)^{\tilde{n}}
-\frac{1}{N^{\tilde{n}}}-\frac{1}{(M-N)^{\tilde{n}}} 
\right)B_{\tilde{n}} (0)
\end{align*}
by \eqref{eq99} and for $1\le k<\tilde{n}$ we have
\begin{align*}
T_{2,k}&=
\frac{B_{\tilde{n}}(0)}{M^{k}(M-N)^{\tilde{n}-k}}
+\sum_{l=1}^{k-1}\binom{k}{l}\frac{B_{\tilde{n}}(0)}{M^{k}N^{\tilde{n}-l}(M-N)^{\tilde{n}-k+l}}
\\
&+\frac{B_{\tilde{n}}(0)}{M^kN^{\tilde{n}-k}}
\\
&=\biggl(\frac{1}{M^{k}(M-N)^{\tilde{n}-k}}
+\frac{M^k-N^k-(M-N)^k}{M^{k}(N(M-N))^{\tilde{n}}}+\frac{1}{M^{k}N^{\tilde{n}-k}}
\biggr)B_{\tilde{n}}(0)
\end{align*}
and
\begin{align*}
T_{2,0}&=\sum_{0\le K_i\le M-1,\atop(i\ge1)}B_{\tilde{n}}\left(\left\{
\frac{\sum_{i\ge1}K_i}{M}
\right\}\right)
=B_{\tilde{n}}(0).
\end{align*}
Thus we have
\begin{align*}
&\sum_{k=0}^{\tilde{n}}(-1)^k\binom{\tilde{n}}{k}T_{2,k}
\\=&
\biggl\{
1
\\
+&
\sum_{k=1}^{\tilde{n}-1}(-1)^k\binom{\tilde{n}}{k}
\biggl(\frac{1}{M^{k}(M-N)^{\tilde{n}-k}}
+\frac{M^k-N^k-(M-N)^k}{M^{k}(N(M-N))^{\tilde{n}}}+\frac{1}{M^{k}N^{\tilde{n}-k}}
\biggr)
\\
+&
(-1)^{\tilde{n}}\biggl(
2+
\left(
\frac{1}{N}+\frac{1}{M-N}\right)^{\tilde{n}}
-\frac{1}{N^{\tilde{n}}}-\frac{1}{(M-N)^{\tilde{n}}} 
\biggr)
\biggr\}B_{\tilde{n}}(0)
\\=&
\biggl\{
1
\\
+&
\frac{N^{\tilde{n}}}{(M(M-N))^{\tilde{n}}} + \frac{(M- N)^{\tilde{n}}}{(MN)^{\tilde{n}}} -\frac{1}{(MN)^{\tilde{n}}}-\frac{1}{(M(M-N))^{\tilde{n}}}
\\
-&\frac{2(-1)^{\tilde{n}}}{M^{\tilde{n}}}-\frac{1}{(M-N)^{\tilde{n}}}-\frac{1}{N^{\tilde{n}}}+\frac{(-1)^{\tilde{n}}}{(M(M-N))^{\tilde{n}}}+\frac{(-1)^{\tilde{n}}}{(MN)^{\tilde{n}}}+\frac{1-(-1)^{\tilde{n}}}{(N(M-N))^{\tilde{n}}}
\\
+&
(-1)^{\tilde{n}}\biggl(
2+
\left(
\frac{1}{N}+\frac{1}{M-N}\right)^{\tilde{n}}
-\frac{1}{N^{\tilde{n}}}-\frac{1}{(M-N)^{\tilde{n}}} 
\biggr)
\biggr\}B_{\tilde{n}}(0)
\\
=&
\frac{1}{(MN(M-N))^{\tilde{n}}}\biggl\{
3(MN(M-N))^{\tilde{n}}+M^{2\tilde{n}}+N^{2\tilde{n}}+(M-N)^{2\tilde{n}} 
\\
&\hspace{30mm}  -2((M-N)N)^{\tilde{n}}
 -2M^{\tilde{n}}((M-N)^{\tilde{n}}+N^{\tilde{n}})
\biggr\}
B_{\tilde{n}}(0),
\end{align*}
where we replaced $(-1)^{\tilde{n}}$ by $1$ by virtue of  $B_n(0)=0$ for odd $n>1$.
These complete the proof of Theorem \ref{thm7.1}.


%
%
\section{Other direction}\label{sec9}
We know that for a prime $p$, the condition  $p\in Spl(f)$ and  the condition that $f(x)$ has $n(=\deg f(x))$ roots 
in the local field $\mathbb{Q}_p$ are equivalent except finitely many primes.
So, it may be natural to ask,  for a natural number $s$  how  the distribution of roots of $f(x)\equiv 0\bmod p^s$ is,
 how the distribution of the coefficient of $p^s$ of
the standard $p$-adic expansion of roots in $\mathbb{Q}_p$ is,
and  relations among them of several $s$'s.
   
With respect to the distribution of roots with a congruence condition, we may consider the distribution of roots  $r_{i,s}$ for a positive integer $s$ such that
\begin{equation*}
\left\{
\begin{array}{l}
f(r_{i,s})\equiv0\bmod p^s\,\,\,\,\,\,(1\le {}^\forall i\le n),
\\
r_{i,s}\equiv R_i\bmod L\quad\quad(1\le {}^\forall i\le n),
\\
0\le r_{1,s}\le\dots\le r_{n,s}<p^sL
\end{array}
\right.
\end{equation*}
with
\begin{equation*}
\sum_{i=1}^n m_{j,i}r_{\sigma(i),s}=m_j+k_jp^s\quad (1\le{}^\forall j\le t).
\end{equation*}
This problem is somewhat  easy compared with the previous one, judging from the case of degree $1$.

We fixed the integer $s$ in the above, but contrary to it, we may ask  the density of $s$ ($s\to\infty$)
 such that there exist integers $r_{i,s}$ satisfying above equations for a fixed prime $p\in Spl(f)$,
 $L,R_i$
(cf. 3.2 in \cite{K6}).
To make data in this case by computer, it takes much time.

One may ask about the equi-distribution similar to Conjecture \ref{conj2} with congruence conditions.

\end{document}